\documentclass[10pt]{amsart}
\usepackage[margin=3 cm]{geometry}
\usepackage[foot]{amsaddr}
\usepackage{tikz-cd} 
\usepackage{graphicx}
\usepackage{amssymb}

\usepackage{lineno}

\usepackage{comment}

\usepackage[centertags]{amsmath}
\usepackage{amsfonts}
\usepackage{amssymb}
\usepackage{bm}
\usepackage{amsthm}
\usepackage[normalem]{ulem}
\usepackage{dsfont}
\usepackage{xcolor}
\usepackage{cancel}
\usepackage{enumitem}
\usepackage{verbatim}
\usepackage{bbm} 

\usepackage[colorlinks]{hyperref}
\usepackage{todonotes} 

\newtheorem{appxlem}{Lemma}[section]
\newtheorem{thm}{Theorem}[section]
\newtheorem{prop}[thm]{Proposition}
\newtheorem{defn}[thm]{Definition}
\newtheorem{lem}[thm]{Lemma}
\newtheorem{oss}[thm]{Remark}

\newtheorem{cor}[thm]{Corollary}
\newtheorem{sublem}[thm]{Sublemma}
\newtheorem*{Not}{Notation}
\numberwithin{equation}{section} 
\newcommand{\nzero}{q_0}
\newcommand\ray{R}

\newcommand\cA{{\mathcal A}}
\newcommand\cB{{\mathcal B}}
\newcommand\cC{{\mathcal C}}
\newcommand\cD{{\mathcal D}}

\newcommand\cF{{\mathcal F}}

\newcommand\cH{{\mathcal H}}
\newcommand\cI{{\mathcal I}}

\newcommand\cL{{\mathcal L}}
\newcommand\cO{{\mathcal O}}
\newcommand\cM{{\mathcal M}}
\newcommand\cN{{\mathcal N}}
\newcommand\cT{{\mathcal T}}
\newcommand\cU{{\mathcal U}}
\newcommand\cK{{\mathcal K}}
\newcommand\cZ{{\mathcal Z}}
\newcommand\cR{{\mathcal R}}

\newcommand\bC{{\mathbb C}}

\newcommand\bF{{\mathbb F}}
\newcommand\bG{{\mathbb G}}
\newcommand\bJ{{\mathbb J}}
\newcommand\bI{{\mathbb I}}
\newcommand\bL{{\mathbb L}}
\newcommand\bN{{\mathbb N}}

\newcommand\bR{{\mathbb R}}
\newcommand\bT{{\mathbb T}}

\newcommand\bZ{{\mathbb Z}}
\newcommand{\Id}{\mathbbm{1}}
\newcommand{\bbb}{\mathbbm{b}}
\newcommand{\bbc}{\mathbbm{c}}
\newcommand{\bbd}{\mathbbm{d}}

\newcommand{\bbk}{\mathbbm{k}}
\newcommand{\bbr}{\mathbbm{r}}
\newcommand\ve{\varepsilon}
\newcommand\eps{\epsilon}
\newcommand\vf{\varphi}

\newcommand\tz{\tilde{z}}
\newcommand\tN{\widetilde{\mathfrak{N}}}
\newcommand{\fra}{\mathfrak{a}}
\newcommand\frh{{\mathfrak h}}
\newcommand\frH{{\mathfrak H}}

\newcommand{\frq}{\mathfrak{q}}

\newcommand\Leb{\operatorname{Leb}}
\newcommand{\balpha}{\bar{\alpha}}
\newcommand{\bvarsigma}{\bar\varsigma}

\newcommand{\vcs}{\zeta}
\newcommand{\ovm}{{\overline m}}

\newcommand{\Const}{C_\sharp}
\newcommand{\const}{c_\sharp}
\newcommand\fC{{\mathbf{C}}}

\newcommand{\wF}{\widetilde F}
\newcommand{\qdr}{\operatorname{\boldsymbol q}}
\newcommand{\sign}{\operatorname{sign}}
\newcommand{\xip}{\xi^\perp}
\newcommand{\supp}{\operatorname{supp}}

\begin{document}
\title[Quantitative statistical properties]{Quantitative statistical properties of two-dimensional partially hyperbolic systems}
\author{Roberto Castorrini}
\address{Roberto Castorrini\\
Laboratoire de Probabilit\'es, Statistique et Mod\'elisation (LPSM),
Sorbonne Universit\'e, Universit\'e de Paris,
4 Place Jussieu, 75005 Paris, France
}
\email{{\tt castorrini@lpsm.paris}}
\author{Carlangelo Liverani}
\address{Carlangelo Liverani\\
Dipartimento di Matematica\\
II Universit\`{a} di Roma (Tor Vergata)\\
Via della Ricerca Scientifica, 00133 Roma, Italy.}
\email{{\tt liverani@mat.uniroma2.it}}
\begin{abstract}
We study a class of two dimensional partially hyperbolic systems, not necessarily skew products, in an attempt to develop a general theory. As a main result, we provide explicit conditions for the existence of finitely many physical measures (and SRB) and prove exponential decay of correlations for mixing measures. In addition, we obtain precise information on the regularity of such measures (they are absolutely continuous with respect to Lebesgue with density in some Sobolev space). To illustrate the scope of the theory, we show that our results apply to the case of fast-slow partially hyperbolic systems, and for such systems we obtain more precise results on the structure of the SRB measures.

\end{abstract}
\thanks{This work was partially supported by the PRIN Grant ``Regular and stochastic behaviour in dynamical systems" (PRIN 2017S35EHN) and by the European Research Council (ERC)
under the European Union's Horizon 2020 research and innovation program (grant agreement No 787304). C.L. acknowledges the MIUR Excellence Department Project awarded to the Department of Mathematics, University of Rome Tor Vergata, CUP E83C18000100006. R.C. acknowledges the Gran Sasso Science Institute (L'Aquila 67100, Italy), where most of the work was done.  CL is member of the GNFM of INDAM. We thank Florestan Martin-Baillon who contributed to the very first stage of this project. It is also a pleasure to thank Viviane Baladi, Mark Demers, Peyman Eslami, Jacopo de Simoi, Dmitry Dolgopyat, S\'ebastien Gou\"ezel, Oliver Butterley and Masato Tsujii for many helpful discussions on related topics.} 
\maketitle
\tableofcontents
\section{Introduction}\label{sec:intro}
One of the main challenges of the field of Dynamical Systems is to understand the ergodic properties of {\it partially hyperbolic}
systems. Substantial progress has been made in the study of ergodicity starting
with~\cite{GPS94, pugh-shub97, Wilk} until establishing very general results, e.g. \cite{BuWilk}, in the case of volume preserving diffeomorphisms. Nevertheless, if the invariant measure is not a priori known, then establishing the existence of SRB measures is a serious challenge in itself, see \cite{BV, ABV, Tsu 2} for some important partial results. Moreover, it is well known, at least since the work of Krylov~\cite{Krylov79}, that for many applications ergodicity does not suffice, and mixing (usually in the form of effective quantitative estimates on the decay of correlations) is of paramount importance. 
Some results on correlation decay exist in the case of mostly expanding central direction~\cite{ALP05}, and mostly contracting central
direction~\cite{Dimacontract, deCastro02}. Such results, albeit important, are often not easy to apply since it is very difficult to estimate the central Lyapunov exponent.

For a central direction with zero Lyapunov
exponents (or close to zero) there exist quantitative results on exponential decay of
correlations only for group extensions of Anosov maps and Anosov flows~\cite{Dima-group,
Che98, Dima98, Liverani04, Tsujii10}, but none of them apply to an open class (with the notable exception of  \cite{BuWa, TZ20}; also some form of rapid mixing is known to be typical for large classes of
flows~\cite{FMT07, Melbourne}). 
Hence, the problem of effectively studying the quantitative mixing properties of partially hyperbolic systems is wide open.

Recently, motivated by deep physical reasons \cite{Dima-L, BHLLO, Liverani18}, the second author has proposed the study of a simple class of partially hyperbolic systems with the goal of developing a theory applicable to a large class of fast-slow systems. Some encouraging results exist \cite{DeLi1, DeLi2, DeLiPoVo}. However, the amount of work needed to prove the above partial results has proven rather daunting, and to extend such an approach to more realistic systems seems extremely challenging. To attain substantial progress, it seems necessary to introduce new ideas supplementing the approaches developed so far.

In the last years, starting with \cite{BKL, GoLi, BT}, a powerful method to investigate the statistical properties of hyperbolic systems has evolved: {\it the functional approach}. It consists in the study of the spectral properties of the transfer operator on appropriate Banach spaces. Although the basic idea can be traced back, at least, to Von Neumann ergodic theorem, the new ingredient consists in the understanding that  non standard functional spaces must be used, and in the insight of how to embed the key geometrical properties of the system in the topology of the Banach space. See \cite{Babook2} for a recent review of this approach, and \cite{DKL21} for an introduction.

This point of view has produced many important results, e.g. see \cite{Liverani04, KL05, KL06, GLP,FT17,DZ17, BDL} just to cite a few.
It is then natural to investigate if the functional approach can be extended to partially hyperbolic systems. Some results that hint at this possibility already exist (e.g. \cite{AvGoTs, F11}), however, a general approach is totally missing. Nonetheless, the idea that some quantitative form of accessibility should play a fundamental role has slowly emerged, e.g. see \cite{Tsu 0, NTW, BuEs}.

In this paper we attempt to further the latter approach by combining ideas from \cite{AvGoTs} and \cite{GoLi}. We find checkable conditions that imply the existence of finitely many physical (and SRB) measures  for a large class of two dimensional endomorphisms; in addition, such measures are exponentially mixing (see Theorem \ref{main-result-sv}). This implies all the standard statistical results (CLT, Large deviations, deterministic stability, etc; see \cite{Babook1, Babook2} and references therein for details). Next, we show that the hypothesis of Theorem \ref{main-result-sv} are fulfilled for an open set of physically relevant systems (fast-slow systems), see Theorem \ref{main-result-fs}. Moreover, for such systems, we obtain some precise quantitative information on the SRB measures (Theorem \ref{thm:srb-fast-slow}). In addition, we show how the results obtained here provide detailed information on the structure of the peripheral eigenfunctions of the transfer operator, see Theorem \ref{thm:peripheral}, which hopefully should allow further progress. Indeed,  we believe that this approach can be further refined and extended to produce similar results in a much more general class of systems.

It is customary to think that the constants appearing in Lasota-Yorke type inequalities are largely irrelevant. This is certainly not the case for the fast-slow systems discussed in section \ref{section application to map F} where the constant in front of the weak norm is highly non-uniform. This is reflected in the possible concentration of the invariant measures. Moreover, the possibility to consider the class of maps discussed there as a perturbation of a limiting case depends crucially on the size of such constants. It was then essential to try to push the estimates to their extreme in order to find out if perturbative ideas could be applied. It turns out that our estimates are not sharp enough to do so. However, we have identified precisely the obstructions to this approach, hence clarifying the focus of future research.

 {\bf The paper is organized as follows}: in the next section we describe the systems under consideration: we  call them SVPH, for general partially hyperbolic systems. We will also study, in more detail, a special case: fast-slow systems. Since many results for fast-slow are obtained by refining some estimates already established for the SVPH systems, we made an effort to separate clearly the results for the two classes of systems. Hence, the reader that is not interested in the more technical part of the paper can easily skip it. In section \ref{sec:results-ph} we state our main results on the invariant measures of the systems: Theorem \ref{main-result-sv} and \ref{main-result-fs}, which are, respectively, direct consequences of the two technical results Theorems \ref{Main Theorem 1} and \ref{thm check assump of thm SVPH} on the transfer operators stated in section \ref{sec:results-tr}. Finally, the latter section contains two further results (Theorem \ref{thm:peripheral} and  \ref{finitely many SRB}) about the spectrum of the transfer operators for fast-slow systems.  In section \ref{sec:3} we introduce the necessary notation and prove several facts needed to define the Banach spaces we are interested in. In particular, sections \ref{sec:superpalla} and \ref{sec:distortion} contain most of the hard estimates needed in the following and are rather technical, so we postponed the proofs to Appendix \ref{appG}; which can be skipped in a first reading without losing the logic of the argument. In section \ref{section A first LY ineq} we prove a first Lasota-Yorke inequality. Unfortunately, the spaces considered in this section do not embed compactly in each other, hence one cannot deduce the quasi-compactness of the operator from such inequalities. Sections \ref{sec:5} and \ref{section A second Lasota-Yorke inequality} are the core of the paper where some inequalities relating the previous norms to the Sobolev norms $\cH^s$ are obtained.
In section \ref{section:The final Lasota-Yorke inequality} we harvest the work done and prove Theorem \ref{Main Theorem 1} which implies Theorem \ref{main-result-sv}. In section \ref{sec:8} we show that fast-slow systems satisfy the hypotheses of Theorem \ref{Main Theorem 1}, hence our theory applies and Theorem \ref{main-result-fs} follows. Also, we take advantage of the peculiarities of the fast-slow systems to prove some sharper results on the SRB measure, and, more generally, the spectral projections (or resonances). \\
The paper also includes seven appendices that contain some necessary technical results which would have disrupted the flow of the argument if included in the main text.

\begin{Not}\label{rmk:constants} 
As we would like to apply our results to open sets of maps $F$, all the constants appearing in the text are really functions of $F$. We will call a constant {\rm uniform} if it depends continuously only on the $\cC^r$ norm  of the map $F$, on $(\lambda_--\mu_+)^{-1}$, $\chi_c^{-1}$, $(1-\chi_u)^{-1}$, $(1-\iota_\star)^{-1}$  and $C_\star$ (see hypothesis ({\bf H1}) for the definition of $\lambda_-$, $\mu_+$, $\chi_c$, $\iota_\star$ and $C_\star$).\footnote{ The name is motivated by the fact that our results are thought for application to families of maps for which $(\lambda_--\mu_+)^{-1}$, $\chi_c^{-1}$, $(1-\chi_u)^{-1}$, $(1-\iota_\star)^{-1}$  and $C_\star$ are uniformly bounded, hence for such families our constants will be uniform in the usual sense.}

In order to make the reading more fluid, we will use the notation $f \lesssim g$ to mean that there exists a uniform constant $C_\sharp>0$, such that $f\le C_\sharp g$. The values of the constants $C_\sharp$ can change from one occurrence to the next.  Moreover, in the following we will use $C_{a,b,\dots}, c_{a,b,\dots}$ to designate constants that depend also on the quantities $a,b,\dots$. When the quantities in the subscripts are fixed, these constants are uniform; hence, since no confusion can arise, we will call them uniform as well. 

Note that $\chi_c,\chi_u\in (0,1)$, which determine the size of the central and unstable cone, respectively, are not uniquely determined by the map. Given our convention, we must keep track of how the constants depend on $\chi_u^{-1}$ and we cannot hide such a dependency inside a constant $C_\sharp$. Indeed, in the next sections it will be apparent that it may be convenient to choose $\chi_u$ as small as possible while it is convenient to choose $\chi_c$ as large as possible.\\
Finally, to simplify notations, we use $\{a,b, \dots\}^+$ and $\{a,b, \dots\}^-$ to designate the maximum and minimum between the quantities $a,b,\dots$, respectively.
\end{Not}
\section{Partially hyperbolic systems}\label{sec:mainTs}
In this section we introduce the class of systems we are interested in, the main assumptions and some definitions necessary to present the results. In this work $\bT^2$ and $\bT$ represent the quotients $\bR^2/\bZ^2$ and $\bR/\bZ$ respectively. For a local diffeomorphism  $F:\bT^2\to \bT^2$ we define the functions $\mathfrak{m}_F^*, \mathfrak{m}_F:\bT^{2}\times \bR^2\setminus \{0\}\to \bR_+$ as~\footnote{ By $\|\cdot\|$ we mean the Riemannian metric in $\bT^2$ induced by the Euclidean norm in $\mathbb{R}^2$.}
\begin{equation}\label{eq:m+_}
	\mathfrak{m}_F(z, v)=\frac{\|D_zF v\|}{\|v\|} \quad;\quad  \mathfrak{m}^*_F(z, v)=\frac{\|(D_zF)^{-1} v\|}{\|v\|}.
\end{equation}
\subsection{Strongly dominated vertical partially hyperbolic systems (SVPH)}\label{sec:SV}\ \\
The systems we are interested in are defined in Definition \ref{def:SVPH}, but before that we need to introduce some notation.\\
Let $r \ge 4$ and $F:\mathbb{T}^2\mapsto \mathbb{T}^2$ be a surjective $\mathcal{C}^{r}$ local diffeomorphism. We call $F$ a \textit{partially hyperbolic system}\footnote{ In the present case the term \textit{partially expanding} would be more appropriate, as there is only an expanding direction which is dominant.} if there exist a continuous splitting, not necessarily invariant, of the tangent bundle into subspaces $\cT\bT^2=E^c \oplus E^u$, $\sigma>1$ and $c>0$ such that for each $n\in\bN$
\begin{equation}
\label{partial hyperbolicity}
\begin{split}
&\|DF^n_{|E^u}\|> c\sigma^{n}\\
&{\|DF^n_{|E^c}\|}<c\sigma^{-n}\|DF^n_{|E^u}\|.
\end{split}
\end{equation}
Notice that for non-invertible maps the unstable direction is not necessarily unique, nor backward invariant. It is then more convenient to work with cones instead than distributions. Indeed, it is well known (see e.g \cite{HiPuSh}) that the above conditions are equivalent to the existence of smooth invariant transversal cone fields $\mathbf{C}_u(z), \mathbf{C}_c(z)$, which satisfy conditions equivalent to \eqref{partial hyperbolicity}. 

To simplify the following arguments we will restrict ourselves to maps without critical points. To further simplify matters we restrict to orientation preserving maps (if not, one can always consider $F^2$)
\\ \vskip-6pt
({\bf H0}) for all $p\in\bT^2$ we have $\det(D_pF)>0$.
\\ \vskip-6pt
In addition, to simplify notations, we make the assumption that the cone fields can be chosen constant, since this hypothesis applies to all the examples we have in mind.\footnote{ One can reduce to such a case by a change of variables.} Hence we require the following hyperbolicity hypothesis:\\ 
\vskip-6pt
({\bf H1}) There exist $\chi_u\in (0,1)$, $\chi_c\in (0,1]$ and $ 0<\mu_-<  1<  \mu_+ < \lambda_-\le\lambda_+$ such that, setting
\begin{equation}\label{eq:cone-def}
\begin{split}
&\mathbf{C}_u:= \lbrace (\xi, \eta)\in \cT_z\mathbb{T}^2  :  |\eta|\le \chi_u|\xi | \rbrace \\
&\mathbf{C}_c:= \lbrace (\xi, \eta)\in  \cT_z\mathbb{T}^2 :  |\xi|\le \chi_c |\eta | \rbrace,
\end{split}
\end{equation}
defining (recall equations \ref{eq:m+_})
\begin{equation}\label{def of lamb^+ mu^+}
\begin{split}
&\lambda^-_n(z):= \inf_{v\in \bR^2\setminus \fC_c }\mathfrak{m}_{F^{n}}(z, v) \qquad \qquad\qquad\; \lambda^+_n(z):= \sup_{v\in \bR^2\setminus \fC_c }\mathfrak{m}_{F^{n}}(z, v),\\
&\mu^-_n(z):= \inf_{v\in \fC_c\setminus{\lbrace 0 \rbrace}}\mathfrak{m}^*_{F^{n}}(z, v) \quad\quad\qquad\qquad\mu^+_n(z)
:= \sup_{v\in \fC_c\setminus{\lbrace 0 \rbrace}}\mathfrak{m}^*_{F^{n}}(z, v),
\end{split}
\end{equation}
and letting $\lambda^-_n=\inf_z\lambda^-_n(z)$ and $\lambda^+_n=\sup_z\lambda^+_n(z)$ we assume the following:\\
There exist uniform constants $C_\star\geq 1$ and $\iota_\star \in (0,1)$ such that, for all $z\in\bT^2$ and $n\in\bN$,\footnote{ $A\Subset B$ means $\overline{A}\subset \mbox{int}(B)\cup \lbrace 0 \rbrace$.}
\begin{align}
&{D_zF\mathbf{C}_u \subset \{(\xi,\eta): |\eta|\le \iota_\star\chi_u|\xi|\}\Subset \fC_u;} \quad
(D_zF)^{-1}\mathbf{C}_c \subset \{(\xi,\eta): |\xi|\le \iota_\star\chi_c|\eta|\} \label{invariance of cone},\\
&C_\star^{-1}\mu_-^n\leq \mu^-_n(z)\leq \mu^+_n(z)\leq C_\star\mu_+^n \;;\quad C_\star^{-1}\lambda_-^n\leq\lambda^-_n(z)\leq  \lambda^+_n(z)\leq C_\star \lambda_+^n \,.\label{partial hyperbolicity 2}
\end{align}
\\ \vskip-6pt
From now on we set $\mu:={\lbrace \mu_+,\mu_-^{-1}\rbrace}^+> 1$. The above conditions imply, in particular, $\det(DF)\neq 0$.\\
Up to now we have just described a rather general two dimensional partially hyperbolic map.
Next, we impose a constraint on the topology of the map\\ 
\vskip-6pt
({\bf H2}) Let $\Upsilon$ be the family of closed curve $\gamma \in \cC^r(\bT, \bT^2)$ such that~\footnote{ As usual we consider equivalent two curves that differ only by a $\cC^r$ non-singular reparametrization. In the following we will mostly use curves that are parametrized by vertical length.}
\begin{itemize}
\item[c0)] $\gamma'\neq 0$,\label{item c0}
\item[c1)]  $\gamma$ has homotopy class $(0,1),$\label{item c1}
\item[c2)] $\gamma'(t)\in \fC_c,$ for each $t\in \bT.$\label{item c2}
\end{itemize}
We assume that for each $\gamma\in\Upsilon$ there exist $\{\gamma_i\}_{i=1}^N\subset \Upsilon$ such that $F^{-1}(\Upsilon)=\{\gamma_i\}_{i=1}^N$.\\
\vskip-6pt
We also need a {\it pinching} condition\\
\vskip-6pt
({\bf H3}) Let 
\begin{equation}\label{def of zetar}
\zeta_r:=6(r+1)!\,.
\end{equation}
We assume that $F$ satisfies
\begin{equation}\label{pinching condition}
\mu^{\zeta_r}<\lambda_-.
\end{equation}

\begin{oss}
Notice that condition \eqref{pinching condition} implies in particular that $\mu<\lambda_-$. The presence of the factorials in \eqref{def of zetar} is probably not optimal. This is a condition we did not try to optimize since it is irrelevant for our main application in which $\mu$ is close to 1.
\end{oss}
We will call a partially hyperbolic system satisfying \eqref{pinching condition} \textit{strongly dominated}.
\begin{oss}\label{rem:omegaper}
Note that, since $F$ is a local diffeomorphism, then it can be lifted to a diffeomorphism $\bF$ of $\bR^2$ with the projection $\pi$ map being $\mod 1$, so that $\pi(0)=0$. Then we can define $\bG(x,\theta)=\bF(x,\theta)-(0,\theta)$ and write  $F\circ\pi(x,\theta)=\pi(\bG(x,\theta)+(0,\theta))$. Thus in the following, with a slight abuse of notation, we will often confuse the map with his covering and write 
\begin{equation}
\label{the map F}
F(x,\theta)=(f(x,\theta),\theta +\omega(x,\theta)).
\end{equation}
In addition, note that if the map satisfies condition ({\bf{H2}}) then for each $x\in\bR^2$ the curve $\gamma_x(t)=(x,t)$, $t\in\bT$ has a preimage  $\nu\in \Upsilon$ homotopic to the curve $\bar \gamma_p(t)=p+(0,t)$, $p\in\nu$, $F(p)=(x,0)$. This implies that $F(\bar \gamma_p(t))$ is a curve homotopic to $\gamma_x$. Thus for each $(x,\theta)\in\bR^2$ the lift has the property $\bF(x,\theta+1)=\bF(x,\theta)+(0,1)$, which implies that $\bG$, and hence $\omega$, is periodic in the second variable. 
\end{oss}
In the following we will need some quantitative information on the Lipschitz constant of the graphs describing the ``unstable manifolds." To simplify matters, we prove the needed results in Lemma \ref{lem:vec_regularity}. We require then that our maps satisfy the hypotheses of such a Lemma. However, be aware that such hypotheses are not optimal and the following condition is used only in Lemma \ref{lem:vec_regularity}. Hence, the next assumption becomes superfluous if in a given system one can prove Lemma \ref{lem:vec_regularity} independently. Also note that, in some cases, it is implied by ({\bf H3}).
\\
\vskip-6pt
({\bf H4}) With the notation \eqref{the map F} we require, for each $p\in\bT^2,$
\[
\partial_xf(p)>\left\{ 2(1+\|\partial_x\omega\|_\infty),|\partial_\theta f(p)|\right\}^+.
\]
\begin{oss}\label{rem:chi_c}
Note that one can always have $\chi_c=1$ by a linear change of variables, yet we prefer to  keep track of $\chi_c$ since it may be useful in cases where assumption ({\bf H4}) is not used.
\end{oss}
\begin{defn}[\bfseries SVPH systems]\label{def:SVPH}
We call a map $F$ a \textit{strongly dominated vertical partially hyperbolic system} (SVPH for simplicity) if it satisfies assumptions ({\bf H0}),.., ({\bf H4}).
\end{defn}
\begin{oss}\label{rem:lambda2} Note that if $F$ satisfies ({\bf H0}), ({\bf H1}) and ({\bf{H2}}), then so does $F^n$, $n\in\bN$. Thus it may be convenient to consider $F^n$, instead of $F$, to check  ({\bf{H3}}) and ({\bf{H4}}).
\end{oss}
\noindent From now on we will write a SVPH in the form \eqref{the map F} when convenient.\\
Here we provide simple explicit conditions implying ({\bf H0}),.., ({\bf H4}). The proof is in Appendix \ref{sec:gen_ex}.
\begin{lem}
\label{lem check SVPH}\label{lem:gen_ex}
Let $\lambda:=\inf_{\bT^2}\partial_xf, \Lambda:=\sup_{\bT^2} \partial_xf$ and suppose that:
\begin{enumerate}
\item  $\partial_xf(p)>\left\{ 2(1+\|\partial_x\omega\|_\infty),|\partial_\theta f(p)|\right\}^+,\quad \forall p\in\bT^2$\label{eq:condition ftheta 0}, 
\item $\|\partial_x\omega\|_{\infty}+\|\partial_\theta\omega\|_\infty< \frac 12$\label{item:no-vertical},
\item  $\|\partial_\theta \omega\|_\infty<\frac{1+\|\partial_x\omega\|_\infty}{\lambda-1}$\label{eq:omegatheta-cond},
\item $ 1+\|\partial_\theta f\|_\infty+\|\partial_\theta \omega\|_\infty+\|\partial_x\omega\|_\infty<\lambda$, \label{cond phi positive}
\item $\|\partial_{\theta}f\|_{\infty}<\frac 12\left(-1+\sqrt{1+ 2\lambda^2\Lambda^{-1}}\right)$,  \label{condition ftheta 1}
\item $\chi_c\|\partial_x \omega\|_\infty+\|\partial_\theta\omega\|_\infty
<\frac{\ln \lambda}{4\,\zeta_r},$\label{eq:pinching}
\end{enumerate}
with $\zeta_r$ as in \eqref{def of zetar}. Then $F$ satisfies assumptions ({\bf H0}),..,({\bf H4})
with $\chi_u$ given by \eqref{chi_u}, \eqref{eq:separate}, $\chi_c$ given by \eqref{chi_c}, in particular $\chi_c=1$ is allowed,  and
\begin{equation}\label{rem:mu_choice}
\mu:=\{(1-\chi_c\|\partial_x\omega\|_\infty-\|\partial_\theta\omega\|_\infty)^{-1}, e^{\chi_c\|\partial_x\omega\|_\infty+\|\partial_\theta\omega\|_\infty)}\}^+.
\end{equation}
\end{lem}
In the next section we introduce an interesting example of SVPH. Despite of their simple form, the endomorphisms we are going to consider still include a large class of physically relevant systems.
\subsection{Fast-slow systems}\label{section fast slow syst}\ \\
We are specially interested in the following class of systems, introduced in \cite{DeLi3} (and inspired by the more physically relevant model in \cite{Dima-L}). We will call these systems {\em fast-slow}.
\begin{defn}[\bfseries Fast-Slow systems] Let $F_0(x,\theta)=(f(x, \theta), \theta)$ be $\mathcal{C}^r(\mathbb{T}^2,\mathbb{T}^2)$, for $r\ge 4$, such that  $\inf_{(x,\theta)\in \mathbb{T}^2}\partial_xf(x,\theta)\geq \lambda>2$. For any $\omega\in \cC^r(\bR^2, \bR)$, periodic of period one, and $\ve>0$,  we define
\begin{equation}
\label{Fve}
F_{\ve}(x,\theta)=(f(x,\theta),\theta +\ve \omega(x,\theta)).
\end{equation}
We call the maps $F_\ve$ fast-slow if they satisfy assumption \eqref{condition ftheta 1}  of Lemma \ref{lem check SVPH}. 
\end{defn}
In the following we will need also the next definition.
\begin{defn}
\label{x-constant}
The function  $\omega\in\cC^0(\bT^2,\bR)$ is called $x$-constant with respect to $F_0$ if there exist $\theta\in \mathbb{T}$, $\Phi_{\theta}\in \cC^0(\bT, \bR)$ and a constant $c \in \mathbb{R}$ such that,  for each $x\in \bT$,
\begin{equation*}
\omega(x,\theta)=\Phi_{\theta}(f(x,\theta))-\Phi_{\theta}(x)+c.
\end{equation*}
\end{defn}
Note that, given a specific $F_\ve$, one can often check that $\omega$ is not $x$-constant by looking at the shortest periodic orbits. Moreover, according to \cite{DeLi2}, being not $x$-constant is a generic condition.\\
Our main results on $F_\ve$ hold under the assumption that $\omega$ is \textit{not} $x$-constant.
%
\section{Main Results}\label{sec:results-ph}
Here we detail our main results, first for general systems, then for fast-slow systems.
\subsection{SVPH systems}
\noindent A \textit{physical measure} is an $F$-invariant probability measure $\mu_{ph}$ such that the set
\begin{equation*}
B(\mu_{ph}):=\Bigl\{ p\in \bT^2: \frac{1}{n}\sum_{k=0}^{n-1}\delta_{F^k(p)}\to \mu_{ph} \quad \text{weakly as} \quad n\to \infty  \Bigr\}
\end{equation*}
has positive Lebesgue measure.\\
To state our first result we introduce a quantity inspired by \cite{Tsu 0}. Given $y\in \mathbb{T}^2$ and a line $L$ in $\mathbb{R}^2$ passing through the origin, define
 \begin{equation}
 \label{def of tildeN}
 \widetilde{\mathcal{N}}_{F}(n, y, L):=\sum_{\substack{z\in F^{-n}(y) \\ DF^{n}(z)\fC_{u}\supset L }}  |\det DF^{n}(z)|^{-1},
 \end{equation}
and we set $\widetilde{\mathcal{N}}_{F}(n)=\sup_{y\in \mathbb{T}^2} \sup_{L} \widetilde{\mathcal{N}}_F(n, y, L)$. 

In addition, for each integer $1\le s\le r-1$ we define\footnote{ Note that in \eqref{alphas betas zetas} $\alpha\in (0, 1)$, thanks to hypothesis \eqref{pinching condition}.}   
\begin{equation}
\label{alphas betas zetas}
\begin{split}
&\alpha= \frac{\log(\lambda_-\mu^{-2})}{\log (\lambda_+)}\\
&\alpha_s:=2(2+s-\alpha) \quad ;\quad \beta_s:=2(s+2).
\end{split}
\end{equation}

\begin{thm}\label{main-result-sv}
Let $F\in \cC^r(\bT^2,\bT^2)$ be SVPH. We assume that there exist {$n_1\in\bN$} and $\nu_0\in(0,1) $ such that, for some $1\le s \le r-1$,
\begin{equation}
\label{main assumption with N}
\Big\{ \mu^{\zeta_s}{\lambda_-^{-1}}, \sqrt{{\widetilde\cN}_F(\lceil\alpha {n_1}\rceil)\mu^{\alpha_s {n_1}+\beta_s \ovm_{\chi_u}} } \Big\}^+< \nu_{0},
\end{equation}
where $\ovm_{\chi_u}$ is defined in \eqref{def of mrho and m}.
 Then there exist finitely many ergodic physical measures and they are absolutely continuous with densities in the Hilbert space $\cH^s(\bT^2)$, hence they are SRB measures as well.\footnote{In general non invertible systems do not have an unstable manifold, but many of them, depending on the past history selected, so the SRB should be absolutely continuous when restricted to all such manifolds; this is the case since the physical measures are absolutely continuous.} Moreover, each mixing physical measure $\mu_{ph}$ enjoys exponential decay of correlations for H\"older observables $\phi,\varphi$, namely there exist $\nu>0$ such that,
\[
\left| \mu_{ph}\left(\phi \cdot \varphi \circ F^{n}\right)- \mu_{ph}(\phi) \mu_{ph}(\varphi)\right| \leq C_{\phi,\varphi} e^{-\nu n}.
\]
\end{thm}
The proof of Theorem \ref{main-result-sv} can be found at the end of subsection \ref{subsub:fast}.
\begin{oss}\label{rmk:generic}
Under the assumption ({\bf H3}), condition \eqref{main assumption with N} is automatically satisfied if $\widetilde \cN_F$ grows sub-exponentially with $n$. According to \cite{Tsu 2}, this latter fact holds generically for partially hyperbolic systems in two dimensions (for more details see Remark \ref{rmk:Tsuji-generic}). 
Hence, the result and all the consequences of Theorem   \ref{main-result-sv} hold generically. Unfortunately, this does no say much about a specific map $F$. However,  given a map $F$, \eqref{main assumption with N} is an explicit condition about some power $\alpha n_1$ of $F$ that one can try to check. If successful, then Theorem \ref{main-result-sv} applies to $F$.
\end{oss}
\subsection{Fast-Slow systems}
Even though it is generic and checkable, condition \eqref{main assumption with N} of Theorem \ref{main-result-sv}, may be quite laborious to check and it may entail some computer assisted strategy. It is then interesting to consider less general examples in which such a condition can be easily verified. An important example is given by the fast-slow systems $F_\ve$ introduced in section \ref{section fast slow syst}. For fast-slow sytems condition \eqref{main assumption with N} is directly related to the condition of $\omega$ not being $x$-constant (see Definition \ref{x-constant}).\footnote{This relation was already remarked in \cite{Tsu 2} and \cite{BuEs} in the special case of skew-products.}

\begin{thm}
\label{main-result-fs}
There exists $\ve_*$ such that, if $\omega$ is not $x$-constant, then for each $1\le s \le r-1$ and $\ve \in (0, \ve^*)$, $F_\ve$ has only finitely many physical measures. The physical measures are absolutely continuous, with densities in the Hilbert space $\cH^s(\bT^2)$, hence they are SRB measures as well. Moreover, there exist $\nu>0$ such that each mixing physical measure $\mu^\ve_{ph}$ enjoys exponential decay of correlations for H\"older observables $\phi,\varphi$, namely,
\[
\left| \mu^\ve_{ph}\left(\phi \cdot \varphi \circ F_\ve^{n}\right)- \mu^\ve_{ph}(\phi) \mu_{ph}^\ve(\varphi)\right| \leq C_{\phi,\varphi, \ve} e^{-\nu n}.
\]
\end{thm}
The proof of Theorem \ref{main-result-fs} can be found at the end of subsection \ref{subsub:fast}.\\
In this case, we can also establish some refined properties of SRB measures. 
\begin{thm}\label{thm:srb-fast-slow}
Let $h_{\ve,1}\in L^1$, $\|h_{\ve,1}\|_{L^1}=1$, be the density of a physical measure. Then, setting $\gamma_\ve(\theta)= \int h_{\ve,1}(y,\theta)dy$, we have
\begin{align}
& \operatorname{dist}_{W_1}\left( h_{\ve,1}\cdot\Leb,h_*\cdot \gamma_\ve\cdot \Leb\right)\leq \Const \ve(\ln\ve^{-1})^2.\label{eq:projectors4}
\end{align}
where $\operatorname{dist}_{W_1}$ is the Wasserstein distance and $\Leb$ is the Lebesgue measure on $\bT^2$, and $h_*(\cdot, \theta)$ is the unique invariant probability density of $f(\cdot, \theta)$. In addition, for all $\beta>\frac{13}{2}$, there exists $C_\beta>0$ :
\[
\|h_{\ve,1}\|_{\cH^1}\leq  C_\beta \ve^{-\beta}.
\]
\end{thm}

The proof of Theorem \ref{thm:srb-fast-slow} can be found at the end of subsection \ref{subsub:fast} where, in fact, some more precise results are stated.


\section{Transfer operators}\label{sec:results-tr}
\noindent We will prove Theorems \ref{main-result-sv} and \ref{main-result-fs}, \ref{thm:srb-fast-slow} using a Transfer operator.
\begin{defn}
\label{defn of transf oper}
Given a map $F:\bT^2\to \bT^2$, we define  $\mathcal{L}_F: L^1(\mathbb{T}^2)\to {L}^1(\mathbb{T}^2)$, the transfer operator associated to $F$, as
\begin{equation}
\label{Trasfer Operator}
\begin{split}
\mathcal{L}_F u(z)=\sum_{y\in F^{-1}(z)} \frac{u(y)}{|\det(D_y F)|}.
\end{split}
\end{equation}
\end{defn}
\noindent Iterating (\ref{Trasfer Operator}) yields, for all  $n\in\bN$,
\begin{equation}
\label{Trasfer Operator iteration}
\begin{split}
\mathcal{L}^n_F u(z)=\sum_{y\in F^{-n}(z)} \frac{u(y)}{|\det(D_y F^n)|} \,.
\end{split}
\end{equation}
\noindent By a simple change of variables it follows that $\|\cL_F u\|_{L^1}\le \|u\|_{L^1}$.\\

\subsection{Transversality of unstable cones}\label{Transversality of unstable cones}\ \\
Starting from \cite{Tsu 2} and following \cite{F11},\cite{Tsu 0}, \cite{BuEs} and \cite{Zha}, a link between the mixing property of a partially hyperbolic system\footnote{Although restricted to cases in which the central direction is unidimensional.} and a transversality condition of unstable cones clearly emerges. Let us recall the following notion of transversality introduced in \cite{Tsu 0} by Tsujii.  
\begin{defn}
\label{transversality}
Given $n\in \mathbb{N}$, $y\in \mathbb{T}^2$ and  $z_1,z_2 \in F^{-n}(y)$ , we say that  $z_1$ is  \textit{transversal} to $z_2$ (at time $n$) if $D_{z_1} F^n \mathbf{C}_u \cap D_{z_2} F^n \mathbf{C}_u=\lbrace 0 \rbrace$, and we write $z_1 \pitchfork z_2$. 
\end{defn}
To make the notion of transversality more quantitative, for each $y\in \mathbb{T}^2$ and $z_1\in F^{-n}(y)$ we define
\begin{equation}
\label{def of N(y)}
\mathcal{N}_{F}( n, y, z_1):=\sum_{\substack{z_2 \not{\pitchfork} z_1 \\ z_2\in F^{-n}(y)}}  |\det D_{z_2}F^{n}|^{-1}
\end{equation}
 and set
$
\mathcal{N}_{F}(n)=\sup_{y\in \mathbb{T}^2} \sup_{z_1\in F^{-n}(y)}\mathcal{N}_F( n, y, z_1).
$
\begin{oss}
Note that if all the preimages are non-transversal, then the sum in \eqref{def of N(y)} corresponds to the classical transfer operator applied to one, $\cL^n_F1$ (see \eqref{Trasfer Operator iteration}). 

In essence, $\cL_F^n1-\mathcal{N}_{F}(n)$ provides a quantitative version of the notion of  {\em accessibility} in our systems.
\end{oss}
In Lemma \ref{lemma on the relation N and tild N} we explain the relation between $\cN_F$ and $\widetilde\cN_F$ (defined in \ref{def of tildeN}), while in section \ref{section:Transversality} we explore the properties of $\widetilde\cN_F$.
\subsubsection{ \bfseries Partially hyperbolic systems}\label{subsub:partial}
We are now ready to state the main technical result for SVPH.  The proof of the following Theorem is in Section \ref{section:The final Lasota-Yorke inequality}. 
\begin{thm}
\label{Main Theorem 1}
Let $F\in \cC^r(\bT^2,\bT^2)$ be SVPH that satisfies \eqref{main assumption with N}.
Then there exists a Banach space $\cB_{s,*}$, $\cC^{r-1}(\bT^2)\subset \cB_{s,*}\subset \cH^{s}(\bT^{2})$ such that $\cL_F(\cB_{s,*})\subset \cB_{s,*}$.\footnote{ By $\cH^s(\bT^2)$ we mean the usual Sobolev space (see Appendix \ref{appendix space Hs }).} The restriction of $\cL_F$ to $ \cB_{s,*}$ is a bounded quasi-compact operator, with spectral radius one and essential spectral radius smaller than $\nu_0$.
\end{thm}

\begin{proof}[\bfseries Proof of Theorem  \ref{main-result-sv}]
Note that if $\cL_F h=h$, then the measure $h\Leb$ is invariant for $F$. On the other hand if $\mu_{ph}$ is a physical measure, then there exists a set $K\subset \bT^2$, $\Leb(K)>0$, $\mu_{ph}(K)=1$ such that, for each $x\in K$, $\lim_{n\to\infty}\frac 1n\sum_{k=1}^{n-1}\delta_{F^k(x)}$ converges weakly to $\mu_{ph}$. By Lusin theorem and the density of $\cC^\infty$ in $\cC^0$, for each $\ve>0$ there exists $g_\ve\in\cC^\infty$ such that $\int_{\bT^2}|\Id_K-g_\ve|\leq \ve$.
We can approximate it weakly by measures $g_\ve \Leb$. Then, for each $\phi\in\cC^r$,
\[
\begin{split}
\mu_{ph}(\phi)&=\frac{1}{\Leb(K)}\int_{K}\lim_{n\to\infty}\frac 1n\sum_{k=1}^{n-1}\Id_K\phi\circ F^k=\lim_{n\to\infty}\frac{1}{n\Leb(K)}\sum_{k=1}^{n-1}\int_{K}\Id_K(x)\phi\circ F^k(x)\\
&=\lim_{n\to\infty}\frac{1}{n\Leb(K)}\sum_{k=1}^{n-1}\int_{\bT^2}\cL_F^kg_\ve \phi+\cO(\|\phi\|_\infty \ve)
\end{split}
\]
where the second equality follows by Lebesgue dominate convergence Theorem. Since Theorem \ref{Main Theorem 1} implies that $\lim_{n\to\infty}\frac{1}{n}\sum_{k=1}^{n-1}\cL_F^kg_\ve$ converges in $\cH^s$ to the projector $\Pi$ which projects on the finite dimensional eigenspace associated to the eigenvalue 1, we have
\[
\mu_{ph}(\phi)=\int_{\bT^2} \Pi g_\ve \cdot \phi.
\]
Thus $\mu_{ph}$ is a convex combination of $\{h_i\}$ such that $\cL_F h_i=h_i$. Hence, there are finitely many ergodic physical measure and they are absolutely continuous with density in $\cH^s$. Obviously their supports are disjoint, hence if one is mixing with density $h$, we have that, if $P_*$ is the projector on the eigenspace associated to the eigenvalues of modulus one, then $P_*(\phi h) =\mu_{ph}(\phi) h$ for each $\phi\in\cC^r$. Hence,
\[
\mu_{ph}(\phi\cdot \vf\circ F^n)=\int_{\bT^2} \cL_F^n(\phi h) \vf=\mu_{ph}(\phi) \mu_{ph}(\vf) +\cO(e^{-\nu n})
\]
where $e^{-\nu}$ is larger than  the modulus of the largest eigenvalue not on the unit circle.
\end{proof}
\subsubsection{ \bfseries Fast-Slow systems}\label{subsub:fast}
 As already remarked, the condition that $\omega$ is not $x$-constant is much easier to check than \eqref{main assumption with N}. The following theorem is proven in section \ref{section:proof of thm check ass}.
\begin{thm}
\label{thm check assump of thm SVPH}
There exists $\ve_*$ such that the map $F_\ve$ is a SVPH (see Definition \ref{def:SVPH}) for any $\ve\in(0,\ve_*)$. In addition, if $\omega$ is not $x$-constant, then there exists $\sigma_\star\in (0,1)$ such that the transfer operator $\cL_{F_\ve}$ is quasi compact on the spaces $\cB_{s,*}$, with spectral radius one and essential spectral radius bounded  by $\sigma_\star$ for all $\ve\in(0,\ve_*)$.
\end{thm}
Consider the operator $P:L^1(\bT^2)\to L^1(\bT^2)$ defined by
\[
P h(x,\theta)= h_*(x, \theta)\int_{{\bT}}dy h(y,\theta).
\]
The following Theorem is proved in section \ref{sec:last}.
\begin{thm}\label{thm:peripheral}
In the hypothesis of Theorem \ref{thm check assump of thm SVPH},
we have the decomposition $\cL_{F_\ve}=\Pi+Q$ where $\Pi Q=Q\Pi=0$, $\Pi$ is the finite rank projector on the eigenspace associated to the eigenvalues of modulus one, and $Q$ has spectral radius strictly smaller than one.
Moreover,
\begin{equation}
\|\Pi -P\Pi \|_{L^1\to (\cC^1)'}\leq \Const \ve[\ln\ve^{-1}]^2.\label{eq:projectors1}
\end{equation}
Finally, for each $\tau>0$, let $h_\nu$, $\|h_\nu\|_{L^1}=1$, be an eigenfunction associated to the eigenvalue $\nu$ with $|\nu|\ge e^{-\ve^{\tau}}.$ Then, setting $\beta(\theta)= \int h_\nu(y,\theta)dy$, we have
\begin{align}
& \left\| h_\nu-h_*\beta\right\|_{(\cC^1)'}\leq \Const \ve^{\min\{1,\tau\}}(\ln\ve^{-1})^2.\label{eq:projectors3}
\end{align}
\end{thm}
\begin{oss}\label{rem:compare}The above Theorem is much stronger than the results in \cite{Tsu 2} (where only the existence of the physical measure is discussed and the results hold only generically) or \cite{BV, ABV} (where no information on the SRB measure is provided and its existence is obtained under an additional condition on the contraction or the expansion in the center foliation, even though for more general systems). 
\end{oss}
However, the papers \cite{DeLi1, DeLi2, DeLi3, DeLiPoVo} show that, using the standard pair technology and investigating limit theorems, in some special cases it is possible to obtain considerably more detailed information on the system. Unfortunately, on the one hand the necessary arguments in \cite{DeLi2} are rather involved and, on the other hand, the conclusions concerning the physical measure in \cite{DeLi1} hold only for mostly contracting systems (contrary to the present case that holds in full generality). It is then very important to investigate if the present strategy can provide further information. 

First of all we have an explicit bound on the regularity of the eigenfunctions. The reader can find the proof of the following theorem at the end of section \ref{section conseq of theorem LY for SVPH}. 

\begin{thm}
\label{finitely many SRB}
If $\omega$ is not $x-$constant, then there exists $c_\star>0$ such that, for each $\ve>0$ small enough, and $\bbr\in (0,1)$, if $\nu\in\sigma_{ \cB_{1,*}}(\cL_{F_\ve})\cap\{z\in\bC\;:\; 1-\bbr c_\star[\ln\ve^{-1}]^{-1}\leq |z|\}$, and $u$ is an eigenvector of $\cL_{F_\ve}$ with eigenvalue $\nu$, and $\|u\|_{\cB_0}=1$,\footnote{ See Section \ref{section A first LY ineq} for the definition of the space $\cB_0$.} then for all $\beta>\frac{13}{2}$, 
\[
\|u\|_{\cH^1}\leq  C_\beta \ve^{-(1-\bbr)^{-1}\beta}.
\]
\end{thm} 
\begin{oss}\label{rem:wishfull} It is not clear how sharp the above Theorem is. Certainly some form of blow-up is inevitable. For example: let $f_\theta(\cdot)=f(x,\theta)$ and call $h_*(\cdot,\theta)$ the unique invariant probability density of $f_\theta$.  Let $\bar\omega(\theta)=\int_{\bT} \omega(x,\theta)h_*(x,\theta)dx$. If $\bar\omega$ has non degenerate zeroes $\{\theta_i\}_{i=1}^N$ such that $\bar\omega'(\theta_i)<0$, then \cite{DeLiPoVo} (see also the discussion below) implies that there must exist an 
eigenfunction $u$ essentially concentrated in the $\sqrt\ve$ neighborhood of each $\theta_i$. This implies that $\|u\|_{\cH^1}\geq \Const \ve^{-\frac 14}$.
However, there is a large gap between such a lower bound and the upper bound provided by Theorem \ref{finitely many SRB}. In particular, much more information on the spectrum could be obtained if one could establish an upper bound of the type $\ve^{-\beta}$ with $\beta<1$. We regard this as an open problem.
\end{oss}
Finally, in the setting of Remark \ref{rem:wishfull}, let $\widehat P:L^1\to (\cC^0)'$ be the finite rank operator defined by: for all $\vf\in\cC^0$
\[
\int_{\bT^2} \vf(x,\theta) [\widehat P h](dx, d\theta)=\sum_j\int_{{\bT}}dx\vf(x,\theta_j) h_*(x, \theta_j)\int_{{\bT}\times U_j}dy\,ds h(y,s),
\]
where $U_j$ is the basin of attraction of the stable equilibrium point $\theta_j$ of the averaged dynamics
\begin{equation}\label{eq:averaged}
\begin{split}
&\dot{\bar\theta}=\bar\omega(\bar\theta)\\
&\bar\theta(0, \theta)=\theta.
\end{split}
\end{equation}
Then, an immediate consequence of Theorem \ref{thm check assump of thm SVPH} and \cite[Proposition 4]{DeLiPoVo} is that the eigenfunctions $h$ for the eigenvalue $1$ satisfy, for $\gamma\in(0,\frac 14)$,
\begin{equation}\label{eq:final-equation}
\|h-\widehat P h\|_{L^1\to (\cC^1)'}\leq \left(\Const\varepsilon^{1 / 2-2 \gamma}+\Const\varepsilon \ln \varepsilon^{-1}\right).
\end{equation}

\begin{oss}\label{rem:comp}
Note that the results of \cite{DeLiPoVo} are conditional to the existence of the physical measure which has been previously proven only generically or in special cases, see Remark \ref{rem:compare}. On the contrary here the existence of the physical measures is ensured by Theorems \ref{thm check assump of thm SVPH}, \ref{thm:peripheral}, regardless of the value of the central Lyapunov exponent.
\end{oss}
\begin{proof}[\bfseries Proof of Theorems \ref{main-result-fs} and \ref{thm:srb-fast-slow}]
The proof follows from Theorems \ref{thm check assump of thm SVPH} by the same exact arguments used in the proof of Theorem  \ref{main-result-sv}. Theorem \ref{thm:srb-fast-slow} follows trivially  from Theorem \ref{finitely many SRB}.
\end{proof}

\section{Preliminary estimates}\label{sec:3}
In this section we provide several basic definitions and we prove many estimates that will be extensively used in the following. To deal with the fast-slow systems $F_\ve$, and in particular to prove Theorem \ref{finitely many SRB}, we will need to know explicitly how the constants appearing in this section will behave depending on $\ve$. This makes the computations much longer and tedious than in the general case, where it is not essential to explicitly bound so many constants. For this reason, we will keep the sharper estimates needed for $F_\ve$ as much as possible separated from general ones for the SVPH. 
The reader can thus skip them if not interested in the results, proven in section \ref{sec:8}, about $F_\ve$.

For convenience, and for possible future use, we will consider the fast-slow case as a special case of a larger class of systems, the ``SVPH$^{\,\sharp}$ systems ". 
\subsection{\texorpdfstring{SVPH$^{\,\sharp}$}{SVPHs} systems}\label{section fast slow syst svphs}\ \\
The definition of SVPH$^{\,\sharp}$ systems is a bit technical and is motivated by the need to prove Lemma  \ref{lem stable vertic curves_sharp}. In particular, it depends on uniform constants $ \bar c_\flat,  c^-_2$ introduced in Proposition \ref{prop on DF^-1} and equation \eqref{eq:barc2} respectively, and the constants $ \bar n, C_\flat$ introduced in Lemma \ref{lem stable vertic curves}.

\begin{defn}[\bfseries SVPH$^{\,\sharp}$ systems]\label{def:SVPHsharp}
A map $F$ is called SVPH$^{\,\sharp}$ if it is a SVPH and it satisfies the following additional conditions: there exist $\hat n \in \bN$ uniform with $\hat n >\bar n$ such that, given $B_{\hat n}:=B_{\hat n,  \bar c_\flat,  c^-_2, C_\flat}$ defined in \eqref{hatB} of Lemma \ref{lem stable vertic curves_sharp}, we have\footnote{ Remark that condition \eqref{eq:condition-hatn} can be satisfied thanks to hypothesis ({\bf H3}).}
\begin{align}
&B_{\hat n}^{\frac{1}{\hat{n}}} \mu^{30}\lambda_{-}^{-1}\leq \frac 12\label{eq:condition-hatn}\\
&\hat n< c^-_2\ln\chi_u^{-1}.\label{eq:eps-cond}
\end{align}
\end{defn}

We will see in section \ref{section:Fve is SVPH} that, for $\ve$ small enough, fast-slow systems are SVPH$^{\,\sharp}$.
\subsection{\texorpdfstring{$C^r$}{Lg}-norm}\ \\
\label{section:C12}
Since we will need to work with high order derivatives, it is convenient to choose 
a norm $\|\cdot\|_{\mathcal{C}^{r}}$ equivalent to the standard one, 
which ensures our spaces to be  Banach Algebras. 
We thus define the weighted norm in $\cC^{r}(\bT^2, \cM(m,n))$, where $\cM(m,n)$ 
are the $m\times n$ matrices,\footnote{ According with the previous notations 
we set $x_1=x$ and $x_2=\theta$.}
\begin{equation}
\label{def of C norm}
\begin{split}
&\|\varphi \|_{\mathcal{C}^0}=\sup_{x \in \mathbb{T}^2}\sup_{i\in\{1,..,n\}}\sum_{j=i}^{m}|\varphi_{i,j}(x)| \\
&\|\varphi\|_{\mathcal{C}^{\rho+1}}=2^{\rho+1}\|\varphi\|_{\cC^0}+\sup_{i}\|\partial_{x_i}\varphi\|_{\cC^\rho}.
\end{split}
\end{equation}
where, for a multi-index $\alpha=(\alpha_1,..., \alpha_k)$ with $\alpha_k\in \{1,2\}$, and we will use the notation $|\alpha|=k$ and $\partial^{\alpha}=\partial_{x_{\alpha_1}}\cdots\partial_{x_{\alpha_k}}.$\footnote{ Notice that this is at odd with the usual multi-index definition in PDE, however we prefer it for homogeneity with the case, treated later, of non-commutative vector fields.}
The above definition implies
\begin{equation}
\label{inductive def of the rho-norm}
\|\varphi\|_{\mathcal{C}^{\rho}}\le\sum_{k=0}^{\rho}2^{\rho-k} \sup_{|\alpha|=k}\|\partial^{\alpha}\varphi\|_{\mathcal{C}^0}.
\end{equation}
We will often need to compute the $\cC^\rho$ norm of $\varphi$ along a curve $\nu\in\cC^r(\bT, \bT^2)$. In this case we use the notation $\|\varphi\|_{\cC^\rho_\nu}:=\|\varphi\circ \nu\|_{\cC^{\rho}}.$\\
The following Lemma is proven in Appendix \ref{appendix proof of lem prop of the C norm}. Note that the estimates in the Lemma are not sharp, however they try to optimize the balance between simplicity and usefulness.\footnote{ See \cite{Bru, HerMus} for precise, but much more cumbersome, formulae.}
\begin{lem}
\label{properties of the C norm}
For every $\rho,n,m,s\in\bN_0, \psi\in \cC^\rho(\bT^2, \cM(n,m))$ and $\varphi\in \cC^\rho(\bT^2, \cM(m,s))$ we have 
\[
\|\varphi\psi\|_{\cC^{\rho}}\le \|\varphi\|_{\cC^{\rho}}\|\psi\|_{\cC^{\rho}}.
\]
Also, there exists $C_j^\star>0$, $j\in\bN$, such that, if $\varphi\in \cC^\rho(\bT^2, \cM(n,m))$ and $\psi\in\cC^{\rho}(\bT^2,\bT^2)$,
\begin{equation}
\label{formula 1 for norm Crho}
\|\vf\circ \psi\|_{\cC^{\rho}}\leq C_\rho^\star\sum_{s=0}^{\rho}\|\vf\|_{\cC^{s}}
\sum_{k\in\cK_{\rho,s}}\prod_{l=1}^\rho \|D\psi\|^{k_l}_{\cC^{l-1}}
\end{equation}
where $\cK_{\rho,s}=\{k\in \bN_0^{\rho}\;:\; \sum_{l=1}^{\rho} k_l= s, \sum_{l=1}^{\rho} lk_l\leq \rho\}$.
\end{lem}
Using the above Lemma it follows that there exists a constant $\Lambda>1$ such that
\begin{equation}
\label{definition of Lambda}
\|DF^{n}\|_{\cC^r}+\|(DF^{n})^{-1}\|_{\cC^r}\le \Lambda^{n}, \quad \forall n\in \bN.
\end{equation} 

\subsection{Admissible curves}\ \\
\label{subsec admiss curv}
In this section we introduce the notion of \textit{admissible curve} in order to define important auxiliary spaces and norms in the next section. We start by fixing some notations and defining exactly what we mean by {\it inverse branch}.
\begin{lem}
\label{lem inverse branches}
Let  $\gamma$ be a differentiable closed curve in the homotopy class $(0,1)$ such that  $ \gamma'(t)\not\in\fC_u$ for each $t\in\bT$ and $F^{-1} \gamma=\bigcup_{k=1}^d \nu_k,$ where the $\nu_k$ are disjoint closed curves in the homotopy class $(0,1)$. Then, there exist open sets $\Omega_\gamma, \Omega_{\nu_k}$,  with $\bar \Omega_\gamma=\bT^2$,  and diffeomorphisms (the inverse branches) $\frh_{\nu_k}:\Omega_{\gamma}\to \Omega_{\nu_k}$ satisfying,
\begin{itemize}
\item $F\circ \frh_{\nu_k}=Id|_{\Omega_{\gamma}},$
\item If $\nu_{k}, \nu_{j}\in F^{-1}\gamma$, $k\neq j$, then  $\Omega_{\nu_k} \cap \Omega_{\nu_{j}}=\emptyset,$
\item $\bigcup_{\nu_k\in F^{-1}\gamma}\overline \Omega_{\nu_k}=\bT^2.$
\end{itemize}
\end{lem}  
\begin{oss} 
If $\gamma\in\Upsilon$, then the hypotheses of the Lemma are satisfied thanks to hypothesis {\bf (H2)}.
\end{oss}

\begin{proof}[{\bfseries Proof of Lemma \ref{lem inverse branches}}] The circle $\frq=\{(a,0)\}_{a\in\bT}$ intersects each $\nu_k$ in only one point $p_k=\nu_k\cap \frq$. Indeed, by the backward invariance of the complement of $\fC_u$, $\nu_k$ is locally monotone so it can meet twice $\frq$ only if it wraps around the torus more than once, which cannot happen since $\nu_k$ belongs to the homotopy class $(0,1)$. We can then label the $\nu_k$ so that the map $k\to p_k$ is orientation preserving $(\!\!\!\mod d\,)$, let us call it positively oriented.\footnote{ This definition is ambiguous if $d=2$, but in such a case the ambiguity is irrelevant.}
Also, calling $\tilde \gamma$ the curve obtained by translating $\gamma$ by $\frac 12$ in the horizontal direction, we consider  
$A:=F^{-1}(\tilde\gamma)\cap \frq$. Since $F$ is a local diffeomorphism, if $\tilde p\in A$, in a neighborhood of $\tilde p$ the set $F^{-1}(\tilde\gamma)$ 
consists of a curve with derivative outside $\fC_u$, hence transversal to $\frq$. Accordingly $A$ is a finite collection of points. Suppose that 
$\tilde p_k\in A$ is between $p_k$ and $p_{k+1}$, then $\bT^2\setminus \nu_k$ is a cylinder and $\nu_{k+1}$ separates the cylinder in two disjoint
 regions (by Jordan curve theorem), thus $\tilde p_k$ belongs to a cylinder defined by the curves $\nu_k,\nu_{k+1}$. We can then follow the curve in 
 $F^{-1}\tilde \gamma$ starting from $\tilde p_k$, such a curve cannot exit the cylinder (since $\gamma$ and $\tilde \gamma$ are disjoint). If it intersects 
 again $\frq$ at a point $p'$ then the image, under $F$, of the segment of $\frq$ between $\tilde p_k$ and $p'$ is an unstable curve that starts and ends 
 at $\tilde \gamma$, hence it must cross $\gamma$, contrary to the hypothesis. It follows that $p'=\tilde p_k$, 
 that is $F^{-1}\tilde \gamma=\bigcup_{k=1}^d \tilde \nu_{k}$, where the $\tilde \nu_k$ are disjoint closed curves, of homotopy type $(0,1)$, 
 and $\tilde p_{k}=\tilde \nu_{k}\cap \frq$. As before, we can label the curves so that the $\tilde p_k$ are positively oriented and so are $\tilde p_{k-1},p_k,\tilde p_k$, where the indexes are $\!\!\!\mod d$.
Next, for $i\in \{1,\cdots, d\}$ and $q\in \nu_i$, we define the horizontal segment $\left\{\xi_q(t)\right\}_{t\in (-\delta_{-}(q), \delta_{ +}(q))}$ where $\xi_q(t)=q+e_1t$, $\xi_q(\delta_+(q))\in \tilde  \nu_{i}$ and $\xi_q(-\delta_-(q))\in \tilde  \nu_{i-1}$.
We then define the regions
\begin{equation}
\label{domains Omega_nu}
\Omega_{\nu_i}=\bigcup_{q\in\nu_i}\xi_q.
\end{equation}
Clearly, $\Omega_{\nu_i} \cap \Omega_{\nu_{j}}=\emptyset$ if $i\neq j$, and $\bigcup_{i}\overline \Omega_{\nu_i}=\bT^2$.
Note that $F:\Omega_{\nu_i}\cup\tilde\nu_{i-1}\to\bT^2$ is a bijection, although the inverse is not continuous. However, if we restrict the map to the set $\Omega_{\nu_i}$ then it is a diffeomorphism between  $\Omega_{\nu_i}$ and $\Omega_{\gamma}=\bT^2\setminus\{\tilde\gamma\}$. Thus it is well defined the diffeomorphism $\frh_{\nu_i}: \Omega_{\gamma}\to \Omega_{\nu_i}$ such that $F\circ\frh_{\nu_i}=Id|_{\Omega_{\gamma}}$.
\end{proof}

From now on we call $\frh_\nu$ the inverse branch of $F$ associated to $\nu$ and simply $\frh$ when the curve $\nu$ is clear from the context. We denote by $\frH$ the set of inverse branches of $F$. Likewise, for each $n\in\bN$ we denote with $\frH_n$ the set of inverse branches of $F^n$.
As usual, we wish to identify the elements of $\frH_n$ as compositions of elements of $\frH.$ Unfortunately, Lemma \ref{lem inverse branches} tells us that each $\frh\in\frH$ is defined on a domain obtained by removing a curve in $\Upsilon$ from $\bT^2$. Therefore the composition of two inverse branches in $\frH$ may not be well defined. We can however consider the following sets: denoting as $\cD_\frh$ and $\cR_\frh$ the domain and the range of $\frh$ respectively. For a curve $\gamma \in \Upsilon$ and $n\in\bN$ we define\footnote{ Here we are using the notation $\frH^n=\underbrace{\frH\times\cdots \times \frH}_{\text{n-times}}$ and $\tilde \gamma=\gamma+(1/2,0)$.}
\begin{equation}\label{def:frh^star-frH}
\begin{split}
&\frH_{\gamma, n}:=\{\frh\in \frH_n : \cD_\frh=\bT^2\setminus \{\tilde \gamma\}\},\\
& {\frH^n_{*,\gamma}:=\left\{ \frh_n=(\frh_{1}^*,\cdots, \frh_n^*)\in \frH^n : \cD_{\frh_{j}^*}\subset \cR_{\frh_{j-1}^*}, j\in\{2,..,n\}, \cD_{\frh_1^*}=\bT^2\setminus \{\tilde \gamma\}\right\}.}
\end{split}
\end{equation}
In $\frH^n_{*,\gamma}$ there exists the obvious equivalence relation $\frh_n\sim \frh_n'$ if $\frh_{n}^*\circ \cdots\circ \frh_1^*=\frh_{n}^{'*}\circ \cdots\circ  \frh_1^{'*}$ and the quotient of $\frH^n_{*,\gamma}$ is naturally isomorphic to $\frH_{\gamma,n}$. In the following we will use the two notations interchangeably.
Finally, we define
\[
\frH^{\infty}_\gamma=\left\{\frh=(\frh_{1}^*,\cdots)\in \frH^{\bN}\;:\;  \cD_{\frh_{j+1}^*}\subset \cR_{\frh_{j}^*}, j\in\bN\;; 
\cD_{\frh_1^*}= \bT^2\setminus\{\tilde \gamma\}\right\}.
\]
For $\frh\in \frH_\gamma^{\infty}$, the symbol $\frh_n$ will denote the restriction of $\frh$ to $\frH^n_{*,\gamma}$ and we will say that $\frh\sim\frh'$ iff their restrictions are equivalent for each $n\in\bN$.\footnote{ As it is messy to define infinite compositions, we define the equivalence relation indirectly.} 

In the following we will often suppress the subscripts $\gamma,\nu$ if it does not create confusion.
\subsubsection{\bfseries Some further notation}\label{section Some further notation}
For technical reasons it is convenient to work with cones which are slightly smaller than $\fC_u$ and $\fC_c$. Take $\epsilon\in (0, 1-\iota_\star)$ small and,\footnote{ During the following sections $\epsilon$ will have to satisfy different conditions. However, it is important to note that, once the conditions are satisfied, the value of $\epsilon$ is fixed and it can be considered a uniform constant. Thus for the time being we will keep track of $\eps$, but we will stop when there is no danger of confusion.} setting $\epsilon^*=1-\epsilon,$ let us consider the cone
\begin{equation}
\label{definition of C_epsu}
\fC_{\epsilon, u}=\{(x,y)\in \bR^2: |y|\le \chi_u\epsilon^*|x|\},
\end{equation}
which is strictly contained in $\fC_u$.  In the same way it is defined $\fC_{\epsilon, c}$.
For each $p\in \bT^2$ let $\frH^{n}_p:=\{\frh\in\frH^n\;:\; p\in \cD_{\frh}\}$.
By the expansion of the unstable cone under backward dynamics and the backward invariance of the central cone we can define $m_{\chi_u}(p,\frh):\bT^2\times \frH_p^{\infty}\to \bN$ and $m_{\chi_u}\in \bN$ as
\begin{equation}
\label{def of mrho and m}
\begin{split}
m_{\chi_u}(p, \frh)&=\min\{ n\in \bN: D_p \frh_n(\bR^2\setminus\fC_{\epsilon, u})\subset \mathbf{C}_{\epsilon, c}\}\\
m_{\chi_u}(p)&=\inf_{\frh\in\frH^{\infty}}m_{\chi_u}(p, \frh)\;;\quad m_{\chi_u}=\inf_{p\in\bT^2}m_{\chi_u}(p)\\
m^+_{\chi_u}(p)&=\sup_{\frh\in\frH^{\infty}}m_{\chi_u}(p, \frh)\;;\quad m^+_{\chi_u}=\sup_{p\in\bT^2}m^+_{\chi_u}(p).
\end{split}
\end{equation}
Note that, by \eqref{invariance of cone}, $\fC_{\epsilon, c}\supset D_p\frh\fC_c$, for each $\frh\in\frH_p^1$.\\
By a direct computation (see Sub-Lemma \ref{sublem:chiulambda} for the details) equation (\ref{def of mrho and m}) implies
\begin{align}
&\lambda^-_{m_{\chi_u}(p, \frh)}(p)^{-1}{ \mu^{m_{\chi_u}}}<\epsilon^*\chi_c\chi_u, \quad \forall p\in\bT^2, \frh\in\frH^\infty, \label{expansion with mchi_u}\\
&c_2^- \log\chi_u^{-1}\leq m_{\chi_u}\leq m^+_{\chi_u}< c^+_2 (\log\chi_u^{-1}+1),
\label{eq:barc2}
\end{align}
for some uniform constants $c_2^-, c^+_2>0$.
Next, consider a vector $v=(1, u_0)\in \fC_u$, so that $|u_0|\in [-\chi_u, \chi_u]$. By forward invariance of the unstable cone, there exist continuous functions $\Upsilon_n,\Xi_n:\bN\times \bT^2\times [-\chi_u, \chi_u]\to \bR$ such that 
\[D_pF^n v=\Upsilon_n(p, u_0)(1, \Xi_n(p, u_0)),\]
where $\|\Xi_n\|_\infty \le \chi_u.$ 
We are interested in the  evolution of the slope field $\Xi_n$. For this purpose it is convenient to introduce the dynamics $\Phi(p,u_0)=\left( F(p), \Xi(p,u_0) \right)$, for $p\in \mathbb{T}^2$, $u_0\in [-\chi_u,\chi_u]$ and where we use the notation $\Xi=\Xi_1$. The map $\Phi$ describes how the slopes of the cones change while iterating $F$. Note that 
\begin{equation}\label{Phifrh in the gener case}
\Phi^{n}(p, u_0)=\left( F^{n}(p), \Xi_{n}(p, u_0) \right).
\end{equation}
Finally, for $n\in \mathbb{N}$ and $\frh\in \frH^{\infty}$, let us define the function
\begin{equation}
\label{slope of the cone general case}
u_{\frh,n}(p,u_0)=\pi_2\circ \Phi^n(\frh_n(p),u_0):\mathbb{T}^2\times [-\chi_u, \chi_u]\to [-\chi_u,\chi_u],
\end{equation}
where $\pi_2$ is the projection on the second coordinate. By Lemma \ref{lem:vec_regularity}, applied with $u=u'=u_0$ and $\ve_0=1$, we see that $u_{\frh,n}(p,u_0)$ is Lipschitz and the Lipschitz constant can be computed using \eqref{eq:lip-eq-for-u}.

\subsubsection{\bfseries Admissible central and unstable curves} In the following
$\pi_{k}:\mathbb{T}^2\to \mathbb{T}$ will denote the projection on the $k^{\text{th}}$ component, for $k=1,2$. Also, for $\varphi\in \cC^r(\bT, \bC)$ we use the notation $(\varphi)^{(j)}(t)=\frac{d^j}{dt^j}\varphi(t)$ and $\varphi'$ in the case $j=1$.
\begin{defn}
\label{adm cent curve}
Let $c$ be a positive constant, then $\Gamma_j(c)$ is the set of the $\mathcal{C}^r$ closed curves $\gamma:\mathbb{T} \to \mathbb{T}^2$ which are parametrized by vertical length, i.e. $\gamma(t)=(\gamma_1(t), t)$, satisfy conditions c0), c1) and c2) of assumption {\bf (H2)}, and:
\begin{itemize}
\item[c3)] for every $2\le \ell\le j$: $\|\gamma^{(\ell)}(t)\|\le c^{(\ell-1)!}$\label{item c3}.
\end{itemize}
Given $\bbc>0$ and $j\leq r$ we will call $\gamma\in \Gamma_j(\bbc)$ a \textit{$(j,\bbc)$-admissible central curve} (or simply admissible curve if the context is clear). We will choose $\bbc$ in Corollary \ref{cor:stable_curves_inv}.\\
Similarly, a curve  $\eta \in \mathcal{C}^r(I,\mathbb{T}^2)$ of length $\delta$ defined on a compact interval $I=[0,\delta]$ of $\mathbb{T}$ is called an \textit{admissible unstable curve} if $\eta'(t)\in \mathbf{C}_u$, it is parametrized by horizontal length and its $j$-derivative is bounded by $c^{(j-1)!}$.
\end{defn}
The basic objects used in the paper are integrals along admissible (or pre-admissible) curves. To estimate precisely such objects are necessary the technical estimates developed in the next subsections.
\subsection{Estimates for derivatives: SVPH case}\ \\
We start with the following simple, but very helpful, proposition.
\begin{prop}
\label{prop on the det}
There exists a {uniform} constant $C_*\ge1$ such that, for every $z\in \mathbb{T}^2$, any $n\in\bN$, any vectors $v^u\in \mathbf{C}_u$ and $v^c\in \mathbf{C}_c$ such that $(a,b):=D_zF^{n}v^c\not\in \mathbf{C}_u$, we have :
\begin{equation*}
 C_*^{-1}\frac{\|D_zF^n v^u\|}{\|v^u\|}\frac{ |b|}{\|v^c\|}\le |\det D_zF^n|\le C_* \frac{\|D_zF^n_z v^u\|}{\|v^u\|}\frac{ |b|}{\|v^c\|}.
\end{equation*}
\end{prop}
\begin{proof}
Recall that for a matrix $D\in GL(2,\mathbb{R})$ and vectors $v_1,v_2\in \mathbb{R}^2$ linearly independent\footnote{ By $\measuredangle(v,w)$ we mean the absolute value of the angle between $v$ and $w$, hence it has value in   $[0,\pi]$.}
\begin{equation}
\label{wadge formula}
|\det D|=\frac{|Dv_1\wedge Dv_2|}{|v_1\wedge v_2|}=\frac{\|D v_1\|}{\|v_1\|}\frac{\|D v_2\|}{\|v_2\|}\frac{\sin(\measuredangle(Dv_1,Dv_2))}{\sin(\measuredangle(v_1,v_2))}.
\end{equation}
Let $\theta=\measuredangle(DF^nv^u, DF^nv^c)$, $\theta_1=\measuredangle(DF^nv^u, e_1)$, $\theta_2=\measuredangle(DF^nv^c, e_1)$  and $\theta_{u} =\arctan\chi_u$. Since $D_zF^n v^u\in DF\fC_u$ we have $|\theta_1|\leq c\,\theta_{u} $, for some fixed $c\in(0,1)$. On the other hand, by hypothesis, $|\theta_2|\geq \theta_{u} $. Thus
\[
\begin{split}
\frac{|\theta|}{|\theta_2|}&\leq \frac{|\theta_2|+|\theta_1|}{|\theta_2|}\leq 1+c\\
\frac{|\theta|}{|\theta_2|}&\geq \frac{|\theta_2|-|\theta_1|}{|\theta_2|}\geq 1-c.
\end{split}
\]
The Lemma follows since $\|DF^nv^c\|\sin \theta_2=b$.
\end{proof}
Next, we introduce the following quantities for each $n\in\bN,  { p\in\bT^2}$ and some constants $C_{\sharp}>0$:
\begin{align}
&C_{\mu, n}:=\Const\frac{1-\mu^{-n}}{\mu-1}\le C_{\sharp}\min\{n, (\mu-1)^{-1}\} \label{Cmu,n}; \quad C_{\mu,0}=0.
\end{align}
\begin{oss}
\label{rmk on the constant Cnmu in the generic case}
Note that we can always estimate $C_{\mu, n}$ by $(\mu-1)^{-1},$ which is independent on $n$, and we will do so if we need estimates uniform in $n$. Nevertheless, it can be useful to keep track of constants which deteriorate when $\mu$ approaches one (as for $C_{\mu,n}$), in view of the fast-slow case.
\end{oss}
\noindent Next, we provide sharp estimates of various quantities relevant in the next sections. 

\begin{prop}
\label{prop on DF^-1}
There exist $\bar c_\flat\geq 1$ such that, for any $n\in \bN$ and $p\in\bT^2$, we have: 
\begin{equation}\label{lambda_+<Clambda_-}
\begin{split}
&\lambda^+_n(p)\le \bar c_\flat \lambda^-_n(p)\\
& \|(DF^n)^{-1}\|_{ \cC^0(\bT^2)}\leq \Const  \mu^{n}.
\end{split}
\end{equation}

\end{prop}

\begin{proof}
Let $v^c \in \cT_{F^{n}(p)}\bT^2$ with $v^c \in \fC_c$ unitary, and $w_u\in \fC_u$. Define
\[
\tilde{w}_u=\frac{D_{p}F^nw_u}{\|D_{p}F^nw_u\|}\in \fC_u.
\] 
For each $v\in \cT_{F^{n}(p)}\bT^2$ we can write $v=\alpha v^c+\beta\tilde w_u$, then
\[
\|(D_{F^n p}F^n)^{-1} v\|\le |\alpha|\|(D_{F^n p}F^n)^{-1}v^c\| +|\beta|\|(D_{F^n p}F^n)^{-1}\tilde{w}_u\|
\] 
By (\ref{def of lamb^+ mu^+}) and (\ref{partial hyperbolicity 2}) we have the following
\begin{enumerate}
\item $\|(D_{F^n p}F^n)^{-1} \tilde{w}_u\|\le C_\star\lambda_-^{-n},$
\item $\|(D_{F^n p}F^n)^{-1} v^c\|\le C_\star\mu^n.$
\end{enumerate}
Hence,
\[
\|(D_{F^n p}F^n)^{-1} v\|\le C_\star\mu^n|\alpha|+C_\star\lambda_-^{-n}|\beta|,
\] 
A direct computation shows
\[
\{|\alpha|,|\beta|\}^+\leq \frac{1+|\langle v^c,\tilde w_u\rangle|}{1-\langle v^c,\tilde w_u\rangle^2}\|v\|\leq \frac {1+\cos\vartheta}{1-(\cos\vartheta)^2}\|v\|
\]
where
\[
\cos \vartheta:= \cos\left[\inf_{v\in \fC_u, w\in \fC_c}\{|\measuredangle(v,w)|\}\right]\leq \frac{1+\chi_u}{\sqrt{2(1+\chi_u^2)}}<1.
\]
From the above the second statement of \eqref{lambda_+<Clambda_-} follows. The strategy for proving the first of \eqref{lambda_+<Clambda_-} is similar. We take $w_1,w_2\not\in \fC_c$ unitary and $v^c=(0,1)\in \fC_c,$ and we set  $\tilde{v}_c=\frac{(D_{F^n p}F^n)^{-1}v^c}{\|(D_{F^n p}F^n)^{-1}v^c\|}\in \fC_c$. Notice that $\|D_p F^n \tilde v^c\|\le C\mu^{n}.$ Let $w_2=\alpha w_1+\beta \tilde v^c$. By \eqref{invariance of cone} it follows that there exists a minimal angle between $w_1\not \in \fC_c$ and $\tilde v^c\in (DF)^{-1}\fC_c$, thus $|\alpha|+|\beta|\leq C$ for some constant $\Const>0$. Hence,
\[
\|D_pF^n w_1-D_pF^n w_2\|\le |1-\alpha|\| D_pF^n w_1\|+\Const\mu^n\leq (1+\Const)\| D_pF^n w_1\|+\Const\mu^n.
\]
Since $\| D_pF^n w_1\|\ge C\lambda_n^{-}(p)$, it follows that
\[
\left| 1-\frac{\|D_pF^n w_2\|}{\| D_pF^n w_1\|}\right |\le
\left| \frac{D_pF^n w_1}{\|D_pF^n w_1\|}-\frac{D_pF^n w_2}{\| D_pF^n w_1\|}\right |\leq (1+\Const)+\Const\frac{\mu^n}{\lambda_n^-(p)}.
\]
Equation (\ref{lambda_+<Clambda_-}) follows by the arbitrariness of $w_1, w_2$ and since $\mu<\lambda_-$.

\end{proof}
\subsection{Sharp estimates for derivatives: \texorpdfstring{SVPH$^{\,\sharp}$}{SVPHs} case}
In this section we provide sharper estimates of the first derivatives. To do so, it is convenient to introduce the following quantities for each $n\in\bN, m\le n, { p\in\bT^2}$ and some constants $C_{\sharp}>0$:
\begin{align}
& { \bvarsigma_{n,m}(p):=\{C_{\mu, n-m},\lambda^+_{m}( F^{n-m}(p))\}^+}\label{def:bvarsigma-chiu}\\
&{\varsigma_{n,m}(p):=\{1,C_{\mu,m}(\lambda_{n-m}^-(p))^{-1}+ C_F\bvarsigma_{n,m}(p)\}^+};\quad \varsigma_{n,n}:=\varsigma_n,\label{def:varsigma-chiu} \\
& C_F:=\chi_u+\|\omega\|_{\cC^r}, \label{def:C_F}
\end{align}
and we will use the notation $\bvarsigma_{n,m}:=\|\bvarsigma_{n,m}\|_{\infty}$ and $\varsigma_{n,m}=\|\varsigma_{n,m}\|_{\infty}$.

We have the following improvement of Proposition \ref{prop on DF^-1}.
\begin{prop}
\label{prop on DF^-1S}
For each $\bbc>0$, $m\leq n$ and $\nu\in \Gamma_2(\bbc)$ such that $DF^{n-m}\nu'\in\fC_c$,\footnote{ Recall Section \ref{section:C12} for the definition of  $\|\cdot\|_{\cC^r_\nu}$ and \eqref{def:varsigma-chiu} for the notations used.}
\begin{equation}\label{norm of DF^N^-1}
\begin{split}
&\|DF^n\|_{\cC^0_\nu}\leq \Const \lambda_n^+\\
&\|DF^n\|_{\cC^1_\nu}\leq \Const  \lambda_{n}^+{ \bvarsigma_{n,m}}\mu^{n-m}\\
&\|DF^n\|_{\cC^2_\nu}\leq \Const (\lambda_n^+{ \bvarsigma_{n,m}}\mu^{n-m})^2+\Const\lambda_{n}^+{ \bvarsigma_{n,m}}\mu^{n-m}\bbc\\
& \|\frac{d}{dt} (D_{\nu(t)}F^n)^{-1}\|\leq \Const \mu^{2n-m}\varsigma_{n,m}\circ\nu(t)\\
&  \|\frac{d^2}{dt^2} (D_{\nu(t)}F^n)^{-1}\|\leq  \mu^{n}C_{\mu,n} (1+C_F\lambda_n^+)\circ\nu(t)\left\{\mu^{2n}C_{\mu,n}\lambda_m^+\circ\nu(t)+\bbc\right\}.
\end{split}
\end{equation}
\end{prop}
\begin{proof}
By \eqref{def of lamb^+ mu^+}, we have 
\begin{equation}
\label{eq: derivative of Fs}
 \|D_xF^k\| \leq \Const \lambda^+_k(x). 
\end{equation}
Moreover, for each $n,k\in\bN$, we have
\begin{equation}\label{eq:dtDF-dt^2DF}
\begin{split}
&\frac{d}{dt}D_{\nu(t)}F^n=\sum_{s=1}^2\sum_{k=0}^{n-1}D_{F^{k+1}(\nu(t))}F^{n-k-1}
\partial_{x_s} (D_{F^k(\nu(t))}F)D_{\nu(t)}F^{k} (D_{\nu(t)}F^k\nu')_s\\
&\frac{d}{dt}(D_{\nu(t)}F^n)^{-1} =\sum_{s=1}^2\sum_{k=0}^{n-1}(D_{\nu(t)}F^k)^{-1}\left[\partial_{x_s} (DF)^{-1}(D_{F(\cdot)}F^{n-k-1})^{-1}\right]\circ F^k(\nu(t))\\ &\phantom{\frac{d}{dt}(D_{\nu(t)}F^n)^{-1} =}
\cdot(D_{\nu(t)}F^k\nu')_{s}.
\end{split}
\end{equation}
The above, also differentiating once more, implies that, for $k\leq n$,
\begin{equation}\label{eq:der12}
\begin{split}
&\|\frac{d}{dt}(D_{\nu(t)}F^k)\|\leq  \sum_{j=0}^{k-1} \lambda_{k-j-1}^+\lambda_k^+ \mu^{\{j,n-m\}^-}\lambda_{\{j-n+m,0\}^+}^+\\
&\phantom{\|\frac{d}{dt}(D_{\nu(t)}F^k)\|}
\leq \Const \lambda^+_k\mu^{n-m}\{C_{\mu,n-m},\lambda^+_m\}^+= \Const \lambda_{k}^+{ \bvarsigma_{n,m}}\mu^{n-m},\\
&\|\frac{d^2}{dt^2}(D_{\nu(t)}F^n)\|= \|\sum_{\ell,s}(\partial_{x_\ell}\partial_{x_s}D_xF^n)\nu_\ell'\nu_s'+\sum_{s}\partial_{x_s}D_xF^n\nu_s''\|\\
&\phantom{\|\frac{d^2}{dt^2}(D_{\nu(\cdot)}F^n\circ\nu)\|=}
\leq\Const (\lambda_n^+\mu^{n-m}{ \bvarsigma_{n,m}})^2+\Const\lambda_{n}^+{ \bvarsigma_{n,m}}\mu^{n-m}\bbc. 
\end{split}
\end{equation}
To estimate the second of \eqref{eq:dtDF-dt^2DF}, note that there exist $\xi, \tilde \xi\in\cC^{r-1}(\bT^2,\bR^2)$, $\|\xi\|_{\cC^{r-1}}\leq\Const$, such that, for all $w\in\bR^2$, and $0<|\alpha|\leq r-1$,
\begin{equation}\label{eq:ve-special-inv}
\begin{split}
&\left\|\partial^{\alpha}(DF)^{-1}w-e_1 \langle\partial^\alpha \xi,w\rangle\right\|\leq \Const\|w\|\|\omega\|_{\cC^{|\alpha|+1}}\\
&\left\|\partial^{\alpha}(DF)w-e_1 \langle\partial^\alpha \tilde \xi,w\rangle\right\|\leq \Const\|w\|\|\omega\|_{\cC^{|\alpha|+1}}.
\end{split}
\end{equation}
Thus, setting $\eta_k(p)=D_pF^ke_1\|D_pF^ke_1\|^{-1}$, we have $\|\eta_k-e_1\|\leq \Const \chi_u$ and, for all $w\in\bR^2$,
\begin{equation}\label{eq:ve-special-inv1}
\begin{split}
\left\|(D_xF^k)^{-1}\partial_{x_i}(D_{p}F)^{-1}w\right\|&\leq \left\|(D_xF^k)^{-1}\eta_k(x) \langle \partial_{x_i}\xi,w\rangle\right\|\\
&\phantom{\leq}
+\left\|(D_xF^k)^{-1}\left[\partial_{x_i}(D_{p}F)^{-1}w-\eta_k(x) \langle \partial_{x_i}\xi,w\rangle\right]\right\|\\
&\leq \Const \frac{\|w\|}{\lambda_k^-(x)}+\Const\mu^k\|w\|(\chi_u+\|\omega\|_{\cC^2}),
\end{split}
\end{equation}
where $F^k(x)=p$.
Hence, recalling \eqref{def:C_F} and using the above and  \eqref{lambda_+<Clambda_-}, we have, for all $k\leq n$
\begin{equation}\label{eq:first_der_gen}
\begin{split}
&\| \frac{d}{dt}((D_{\nu}F^k)^{-1})\|\leq \Const\sum_{j=0}^{\{k, n-m\}^--1}\mu^{k-1}\{(\lambda_j^-\circ \nu)^{-1}+C_F{\mu^j}\}\\
&\phantom{\| \frac{d}{dt}((D_{\nu}F^k)\|}
+\Const\sum_{j=n-m}^{k-1}\mu^{k-j-1}\{(\lambda_j^-\circ \nu)^{-1}+C_F{ \mu^j}\}\mu^{n-m}\lambda_{m-n+j}^+\circ F^{n-m}\circ\nu\\
&\le \Const \mu^{k+n-m}\left[1+ C_{\mu,m}(\lambda_{n-m}^-\circ \nu)^{-1}+C_F\{C_{\mu, n-m}, \lambda_{k-n+m}^+\circ F^{n-m}\circ\nu\}^+\right].
\end{split}
\end{equation}
Therefore we have
\begin{equation}
\label{eq: first derivative of DF(-1)}
\| \frac{d}{dt}((D_{\nu(t)}F^n)^{-1})\|\leq \Const \mu^{2n-m}\varsigma_{n,m}\circ\nu(t),
\end{equation}
which yields the statement for the first derivative.
Next, differentiating once more the second of \eqref{eq:dtDF-dt^2DF},
\[
\begin{split}
&\frac{d^2}{dt^2}(D_{\nu}F^n)^{-1}=\sum_{s=1}^2\sum_{k=0}^{n-1}\left[\frac{d}{dt}(D_{\nu(t)}F^k)^{-1}\right]\left[\partial_{x_s} (DF)^{-1}(D_{F(\cdot)}F^{n-k-1})^{-1}\right]\circ F^k(\nu)\\
&\cdot(D_{\nu(t)}F^k\nu')_{s}
+\sum_{s,\ell=1}^2\sum_{k=0}^{n-1}(D_{\nu}F^k)^{-1}\left\{\partial_{x_\ell}\left[\partial_{x_s} (DF)^{-1}(D_{F(\cdot)}F^{n-k-1})^{-1}\right]\right\}\circ F^k(\nu)\\
&\cdot (D_{\nu}F^k\nu')_{\ell}(D_{\nu}F^k\nu')_{s}
+\sum_{s=1}^2\sum_{k=0}^{n-1}(D_{\nu}F^k)^{-1}\left[\partial_{x_s} (DF)^{-1}(D_{F(\cdot)}F^{n-k-1})^{-1}\right]\circ F^k(\nu)\\ &
\cdot\left\{\left[\frac{d}{dt}D_{\nu}F^k\right]\nu'+D_{\nu}F^k\nu''\right\}.
\end{split}
\]
We estimate the three sums above separately. By \eqref{eq:ve-special-inv1} and \eqref{eq: first derivative of DF(-1)}, the first one is bounded by
\[
\begin{split}
\Const\sum_{k=0}^{n-m-1}&\mu^{k+2n-m-1}\varsigma_{n,m}\circ \nu+\Const \sum_{k=n-m}^{n-1}\mu^{n+k-m}\mu^{k}\varsigma_{n,m}\circ \nu\mu^{n-k-1}\mu^{n-m}\lambda_{m-n+k}^+\circ F^{n-m} \nu\\
&\le \Const \mu^{3n-2m}(\varsigma_{n,m}\bar\varsigma_{n,m})\circ \nu .
\end{split}
\]
The second one is equal to
\[
\begin{split}
&\sum_{k=0}^{n-1}(D_{\nu}F^k)^{-1}\left\{\partial^2_{x_\ell,x_s}(DF)^{-1}\cdot (D_{F(\cdot)}F^{n-k-1})^{-1}\right\}\circ F^k(\nu)\cdot (D_{\nu}F^k\nu')_{\ell}(D_{\nu}F^k\nu')_{s}\\
&+(D_{\nu}F^k)^{-1}\left\{\partial_{x_s}(DF)^{-1}\partial_{x_\ell} (D_{F(\cdot)}F^{n-k-1})^{-1}\right\}\circ F^k(\nu)\cdot (D_{\nu}F^k\nu')_{\ell}(D_{\nu}F^k\nu')_{s},
\end{split}
\]
so we can use \eqref{eq:dtDF-dt^2DF}, \eqref{eq:ve-special-inv}, \eqref{eq:ve-special-inv1} to get the bound
\[
\begin{split}
&\Const \sum_{k=0}^{n-1}\left[(\lambda^-_k\circ \nu(t))^{-1}+ C_F\mu^k\right]
\mu^{n-k}\left\{ 1 +\left[C_{\mu,n-k}\mu^{n-k}+ C_F\lambda^+_{n-k}\circ \nu\right]\right\}\mu^{2\{k,n-m\}^-}\\
&\cdot (\lambda^+_{\{0,k-n+m\}^+}\circ F^{n-m}\circ \nu(t))^2 \leq \mu^{5n-2m}C_{\mu,n}^2 (1+C_F\lambda_m^+)^2.
\end{split}
\]
For the last term we estimate as above and, recalling \eqref{eq:der12}, we obtain the bound
\[
\mu^nC_{\mu,n}\left\{1+ C_F \lambda^+_n(\nu(t))\right\}({ \bvarsigma_{n,m}}\mu^{n-m}+\bbc).
\]
Collecting the above estimates, the last of the \eqref{norm of DF^N^-1} readily follows.
\end{proof}

\subsection{Iteration of curves: SVPH case}\label{sec:superpalla}\ \\
A key point in the following arguments is to check how the central admissible curves behave under iteration.
As they are just merely technical, we postpone the proofs, except for Corollary \ref{cor:stable_curves_inv} and \ref{cor:stable_curves_inv_sharp} , to Appendix \ref{appG} to make the reading more fluent.
\begin{lem}
\label{lem stable vertic curves}
Let $F$ be SVPH. There exist uniform constants $\bar n\in\bN$,\footnote{ The  constant $\bar n$ is chosen in \eqref{def:eta and barn}.} $C_\flat>1$ and  $\eta<1$ such that, if $\bbc> \frac 94 C_\flat^3 \mu^{3\bar n}$, for each  $c_\star>\bbc/2$, $\gamma\in \Gamma_\ell(c_\star)$, $1\le\ell\leq r$, and $n\geq \bar n$, setting $\nu_n\in F^{-n}\gamma$, there exist diffeomorphisms $h_{n,\nu}=:h_n \in \mathcal{C}^r(\mathbb{T})$ such that:\\
The curve $\hat\nu_n=\nu_n\circ h_n$ is in $\Gamma_\ell(\eta^{n} c_\star +\bbc/2)$ and
\begin{equation}
\label{Crho norm of h}
\|h_{n}\|_{\cC^\ell}\le
\begin{cases}  C_\flat \mu^{n} \quad  &\text{if}\quad \ell=1\\
 C_\flat^4 (c_\star C_F +1)C_{\mu, n}\mu^{2n}\quad &\text{if}\quad \ell=2\\
C_\flat^{2\ell!} \left( c_{\star}^{ \ell!} C_F+1\right)C_{\mu, n }^{a_\ell}\mu^{\ell! n} \quad &\text{if} \quad \ell> 2,
\end{cases}
\end{equation}
where $a_\ell=(\ell-1)!\sum_{k=0}^{\ell-1}\frac{1}{k!}$, and $C_{\mu, n}$ as in \eqref{Cmu,n}.\\
\end{lem}

Lemma \ref{lem stable vertic curves} implies immediately the following key result. 
\begin{cor}\label{cor:stable_curves_inv}
Let $\bbc=\chi_u^{-\varpi_{\bar n, \chi_u}}$, for some $\varpi_{\bar n, \chi_u}\ge 1$. Under the hypothesis of Lemma \ref{lem stable vertic curves},
for each $n$ such that $\eta^n\leq \frac 12$, we have  the inclusion $F^{-n}\Gamma_\ell(\bbc)\subset \Gamma_\ell(\bbc)$.\\

\end{cor}

\begin{proof}
The statement follows by Lemma \ref{lem stable vertic curves} choosing $\varpi_{\bar n, \chi_u}$ and $c_\star$ such that
\begin{equation}\label{eq:chose-varpi}
c_\star=\bbc=\chi_u^{-\varpi_{\bar n,\chi_u}}=\frac 92 C_\flat^3 \mu^{3 \bar n}.
\end{equation}
The result then follows since $\eta^n\leq \frac 12$. 

\end{proof}
\begin{oss}\label{rmk:extra-conditions}
From now on we will use $\Gamma$ to denote $\Gamma_r(\bbc)$ where $\bbc$ is defined in Corollary \ref{cor:stable_curves_inv} and has thus the stated invariance property. As it clear from the proof of the above Corollary, the hypothesis of Lemma \ref{lem stable vertic curves} suffice to guarantee the backward invariance of the central curves for a SVPH system with the choice $\bbc=\chi_u^{-\varpi_{\bar n, \chi_u}}$.
\end{oss}

\indent The above results tell us that the space of admissible central curves is stable under backward iteration of the map.
Arguing as above, but forward in time (see \cite[Lemma 3.2]{Tsu 2}), it can be proven that the space of admissible unstable curves is stable under the iteration of $F^n$, for $n$ greater than $\bar{n}$. In particular, if $\eta:I\to \mathbb{T}^2$ is an admissible unstable curve, and $\eta_n$ is the image of $\eta$ under $F^n$, then there exists a diffeomorphism $p_{n,\eta}=:p_n$ such that
\begin{equation}
\label{property of p}
p'_{n}(t)=\frac{\|D_{{\eta}}F^{n} \cdot {\eta}'(t)\|}{\|{\eta}'(t)\|},
\end{equation}
and $\eta_n\circ p_n =F^n \circ \eta $ is an admissible unstable curve. Moreover, as $F$ acts as an expanding map along those curves, we have the following standard distortion estimate for each $n\ge 1:$
\begin{equation}
\label{distortion for unsable curve}
\frac{p'_n(t)}{p'_n(s)}\leq \Const, \quad \forall t,s\in I.
\end{equation}

We will need to control the backward evolution also of curves not in the center cone. The last result of this section is the following Lemma, which proof can be found in Appendix \ref{appG}, Section \ref{secG3}.
\begin{lem}
\label{lem unst curve pull back}
Let $F$ be a SVPH and $\Delta_{\gamma}\in L^\infty({\bT},[1,+\infty])$ and consider any closed curve $\gamma\in\cC^r$, homotopic to $(0,1)$, such that $\|\gamma'(t)\|=1$ and $\|\gamma^{(j+1)}(t)\|\le\Delta_{\gamma}(t)^j$,\footnote{ We will apply this Lemma with $\Delta_{\gamma}(t)$ given by \eqref{derivative of extended curve tildegamma}.} for all $j\in\{1,\dots,r\}$ and $t\in\bT$.
For $\frh\in\frH^{\infty}$ let $n_0\geq 0$ and $m$ be the smallest integers such that, for all $t\in\bT$,
\begin{equation}
\label{condition for m}
D_{\gamma(t)}\frh_{n_0}\gamma'(t)\not\in \mathbf{C}_u \quad \text{and} \quad D_{\gamma(t)}\frh_{m}\gamma'(t)\in \operatorname{Int}\left(\mathbf{C}_c\right),
\end{equation}
and assume that $m>\{\bar n,n_0\}^+$.\footnote{ The other possibilities are already covered by Corollary \ref{cor:stable_curves_inv}.}
Let  $\nu_{k}=\frh_k(\gamma)$, $k\in\bN$. For $k\geq n_0$ then there exists a reparametrization $h_k$ such that, setting  $\hat \nu_k= \nu_k\circ h_k$, $\pi_2\circ \hat\nu_k(t)=t$.

If, for some $\frh$, we have $m<\infty$, then: \\
For $\eta<1$ given in Lemma \ref{lem stable vertic curves},\footnote{ See \eqref{def:eta and barn} for a precise definition of $\eta$.} $\Lambda$ as in \eqref{definition of Lambda}, $\overline  m=\sigma m$, where
\begin{equation}\label{def of sigma}
\sigma=\left\{1,\left\lceil \frac{\ln(\mu\|\Delta_{\gamma}\|_{\infty}\Lambda)}{\ln\eta^{-1}} \right\rceil\right\}^+,
\end{equation}
and $\bbc$ as in Corollary \ref{cor:stable_curves_inv}, we have $\hat{\nu}_{\overline m}\in\Gamma_j(\bbc)$ for each $j\ge 1$, and the $\cC^j$-norm of $h_{\overline{m}}$ 
satisfies (\ref{Crho norm of h}) with 
$c_{\star}=\chi_u^{-2} \|\Delta_{\gamma}\|_\infty (\mu\Lambda)^{m}$.\\
\end{lem}

 \subsection{Iteration of curves: the \texorpdfstring{SVPH$^{\,\sharp}$}{SPVHs} case}
In this sub-section we provide an improvement of Lemma \ref{lem unst curve pull back}, for lower derivatives, in the  SVPH$^{\,\sharp}$ case. 
Lemma \ref{lem stable vertic curves_sharp} deals with the derivatives of $\hat\nu_n$, while Lemma \ref{lem unst curve pull back_sharp} bounds the derivative of the reparametrization.
\begin{lem}\label{lem stable vertic curves_sharp}
In the hypothesis of Lemma \ref{lem unst curve pull back}, let $F$ be a SVPH$^{\,\sharp}$ as in Definition \ref{def:SVPHsharp} where\footnote{ Recall that condition \eqref{eq:condition-hatn} can be satisfied thanks to hypothesis ({\bf H3}). Also note that for the present Lemma $\mu^3$ would suffice, however we will need $\mu^{30}$ in \eqref{eq:ovm_choice}.}
\begin{equation}\label{hatB}
B_{\hat n}:=  \bar c_\flat^3 C_\flat^9(1+\chi_c^2)^{\frac 32}(1+6\varsigma_{\hat n}C_{\mu,\hat n}).
\end{equation}
 
Then there exist $C_3, C_4 \ge 1$ uniform, such that for all $\gamma\in \Gamma_\ell(c_\star)$, $c_\star>0$, and $n_\star\in\{\hat n, \dots,c^-_2\ln\chi_u^{-1}\}$, setting $c_{n_\star}=c_\flat ^{1/n_\star}$, $c_{\flat}=\bar c_\flat a_{n_\star}^{n_\star}$, $ a_{n_\star}=((1+\chi_c^2)^{1/2}C_\flat^3)^{n_\star^{-1}}$,\footnote{ Recall \eqref{def:varsigma-chiu} for the definition of $\varsigma_{n}$ while $\bar c_\flat$ is defined in Proposition \ref{prop on DF^-1S}.}\
\begin{equation}\label{def of mathbbms}
\begin{split}
&b_{n_\star}:=(C_4 C_\flat^5\varsigma_{n_\star}C_{\mu,n_*})^{\frac{1}{n_\star}}\\
&\overline{\mathbbm{s}}_{n_\star}= C_\flat \mu^{4{n_\star}}\varsigma_{n_\star}^2C_{\mu,n_\star}+ C_\flat\varsigma_{n_\star} C_{\mu, n_\star}\mu^{5n_\star}C_3+ C_{\mu, n_\star}^2\mu^{6n_\star}C_3^2\\
&\mathbbm{s}_{n_\star}=\overline{\mathbbm{s}}_{n_\star}+6\varsigma_{n_\star}C_{\mu,n_\star}^3\mu^{6n_\star}C_3^2+b_{n_\star}^{2n_\star}C_{\mu,n_\star}\mu^{6n_\star}C_3,
\end{split}
\end{equation}
we have, for all $n\in\bN$, 
\begin{equation}\label{eq:nu2-sharp}
\begin{split}
&\|\hat\nu_n''(t)\|\leq c_\flat c_{n_\star}^n\mu^{2n}\lambda_n^-(\gamma\circ h_n(t))^{-1}c_\star+C_{\mu,n_\star}\mu^{3n_\star}C_3 \\
&\|\hat \nu_n'''\|\le c_\flat^2(1+6\varsigma_{n_\star}C_{\mu,n_\star})c_{n_\star}^{n}\mu^{3n}(\lambda_{n}^{-}(\gamma\circ h_{n}(t)))^{-1}c_\star^2\\
&\phantom{\|\hat \nu_n'''\|\le }
+c_\flat^2b_{n_\star}^{n_\star} \mu^{2n+2n_\star}\lambda_{n}^-(\gamma\circ h_{n}(t))^{-1}c_\star+\mathbbm{s}_{n_\star}.
\end{split}
\end{equation}
\end{lem}
The proof of Lemma \ref{lem stable vertic curves_sharp} can be found in section \ref{secG2} of Appendix \ref{appG}. \\
The above Lemma implies the following sharper version of Corollary \ref{cor:stable_curves_inv}.
\begin{cor}\label{cor:stable_curves_inv_sharp}
 If $\ell\in\{2,3\}$ then, in the hypothesis of Lemma \ref{lem stable vertic curves_sharp}, for all $n\geq\hat n$, we have $F^{-n}\Gamma_\ell(\bbc)\subset \Gamma_\ell(\bbc)$ with the sharper choice
 \begin{equation}\label{eq:cgamma-l2}
 \bbc=2\sqrt{\mathbbm{s}_{n_\star}}.
 \end{equation}
 
\end{cor}

\begin{proof}

We use Lemma \ref{lem stable vertic curves_sharp}. By \eqref{eq:nu2-sharp} and \eqref{eq:condition-hatn} we have, for each $n\geq \hat n$,
\[
\begin{split}
&\|\hat\nu_{n}''(t)\|\leq \frac 12 \bbc+C_{\mu,n_\star}\mu^{3n_\star}C_3 \leq \bbc\\
&\|\hat \nu_{n}'''\|\le \frac 12\bbc^{2}+ c_{\flat}^{2} b_{n_\star}^n \mu^{4n}(\lambda_n^{-}(\gamma \circ h_n))^{-1}\bbc
+\mathbbm{s}_{n_\star}\\
&\phantom{\|\hat \nu_{n}'''\|}
\le \frac 12\bbc^{2}+\frac{\sqrt{\mathbbm{s}_{n_\star}}}{2}\bbc+\mathbbm{s}_{n_\star}\le \bbc^2,
\end{split}
\]
where in the last line we used our choice of $\bbc$.
\end{proof}

\noindent To improve Lemma \ref{lem unst curve pull back} and control the backward evolution of curves not in the center cone, we introduce a further quantity.
Given a smooth curve $\gamma$ such that $\pi_1\circ \gamma'(t)\neq 0$ for each $t\in\bT$, let 
\begin{equation}\label{def of vartheta_gamma}
\begin{split}
&\vartheta_\gamma(t)= \left\{\frac{|\pi_2\circ\gamma'(t)|}{|\pi_1\circ\gamma'(t)|},\chi_u \right\}^+\\
&\vartheta_\gamma= \inf_t \{\vartheta_\gamma(t))\}.
\end{split}
\end{equation}
The last result of this section is the next Lemma, which proof is in Appendix \ref{appG}, section \ref{secG4}.
\begin{lem}\label{lem unst curve pull back_sharp}
In the case $j\in \{1,2\}$ we have the following sharper version of Lemma \ref{lem unst curve pull back}:\\
 If $F$ is a SVPH$^{\,\sharp}$, for each $p\in\gamma$ and $n_\star\in\{\hat n,\cdots, c^-_2\log \chi_u^{-1}\}$, let 
$\overline{m}(p,\frh,  n_\star)\equiv \ovm$ be the minimum integer such that
\begin{equation}
\label{eq:cond for bar m}
\begin{split}
& c_\flat c_{n_\star}^{\ovm}\mu^{2\ovm}\lambda_{\ovm-m}^+(\hat \nu_m\circ  h_{\ovm-m}(t))^{-1}M_{m,n_0}(t)\le C_{\mu, n_\star}\mu^{3 n_\star} C_3,\\
&c_\flat^2b_{n_\star}^{n_\star} \mu^{2\ovm+2n_\star}\lambda_{\ovm-m}^-(\gamma\circ h_{\ovm-m}(t))^{-1}M_{m,n_0}(t)\leq \frac 12\mathbbm{s}_{n_\star}\\
& c_\flat^2(1+6\varsigma_{n_\star}C_{\mu,n_\star})c_{n_\star}^{\ovm}\mu^{3\ovm}(\lambda_{\ovm-m}^{-}(\gamma\circ h_{\ovm-m}(t)))^{-1}\overline M_{m,n_0}(t)\le \frac 12 \mathbbm{s}_{n_\star},
\end{split}
\end{equation}
where
\begin{equation}\label{def of eta_nstar and M(t,n_0,m)}
\begin{split}
& M_{m,n_0}(t):= \Const\left[ \Lambda^{2n_0}\mu^{2m}\Delta_{\gamma}\circ h_{n_0}(t)+ 
C_{\mu,m}\mu^{3m}\vartheta_{\hat\nu_{n_0}}^{-1}\right] \left\{1+ \mu^{2m}\vartheta_{\hat\nu_{n_0}}^{-1}C_F\right\}\\
&\overline{M}_{m,n_0}(t):= \mu^{8m}C_{\mu,m}^2(1+C_F\mu^{m}\vartheta_{\hat\nu_{n_0}}^{-1})^2\\
&\phantom{\overline{M}_{m,n_0}(t):= }
\times\left[ 1+\Lambda^{2n_0}\Delta_\gamma\circ h_{n_0}\vartheta_{\hat\nu_{n_0}}+\Lambda^{3n_0}(\Delta_\gamma\circ h_{n_0})^2\vartheta_{\hat\nu_{n_0}}\right]\vartheta_{\hat\nu_{n_0}}^{-2}.
\end{split}
\end{equation}
and $a_{n_\star}, b_{n_\star} ,c_{n_\star}$, $\mathbbm{s}_{n_\star}$ are defined in Lemma \ref{lem stable vertic curves_sharp}.\\
Then $\hat \nu_{\overline m}\in \Gamma_3(\bbc)$ and, setting $\bar h_{k-n_0}= h^{-1}_{n_0}\circ h_k$,
\begin{equation}\label{eq: h_ovm' h_ovm''}
\begin{split}
&C_\sharp\Lambda^{-n_0}\vartheta_{\hat\nu_{n_0}}(\bar h_{\ovm-n_0}(t))^{-1}\mu^{-\ovm+n_0}\le|h_\ovm'(t)|\leq C_\sharp\Lambda^{n_0}\vartheta_{\hat\nu_{n_0}}(\bar h_{\ovm-n_0}(t))^{-1}\mu^{\ovm-n_0}\\
&|h_\ovm''|\le   C_\sharp \mu^{3\ovm}\vartheta_{\hat\nu_{n_0}}^{-1}\Lambda^{3n_0}\left\{\Delta_\gamma\vartheta_{\hat\nu_{n_0}}^{-1}+M_{m,n_0}(C_FC_{\mu,\ovm}+1)+C_{\mu,\ovm}\right\}.
\end{split}
\end{equation}

\end{lem}


\subsection{Distortion: the SVPH case}\label{sec:distortion}\ \\
We conclude this section with some technical distortion results needed in the following.
\begin{lem}
\label{lem local expans}
For all $n\in\bN$, $\nu\in F^{-n}(\Gamma( \bbc))$ and $x,y\in\nu$, we have
\begin{equation}
\label{eq: local expansion ratio}
 e^{-\mu^n C_{\mu,n}\|x-y\|} \le \frac{\lambda^+_n(x)}{\lambda^+_n(y)} \le e^{\mu^n C_{\mu,n}\|x-y\|}.
\end{equation}
\end{lem}
\begin{proof}
We prove it by induction. To start with, let $x=\nu(t_1), y=\nu(t_2)$ such that $\|x-y\|\leq\tau_n$ for some $\tau_n$ to be chosen shortly. For $n=1$ we have, for all unit vector $v\not \in \fC_c$,
\[
\frac{\|D_xF v\|}{\|D_yF v\|}\leq e^{\ln \left[1+\frac{\|D_xFv-D_yFv\|}{\|D_yFv\|}\right] }
\leq e^{\frac{\|D_xFv-D_yFv\|}{\|D_yFv\|} }.
\]
\begin{equation}\label{eq:distor0}
\|D_xFv-D_yFv\|\leq \int_{t_1}^{t_2}\|\frac{d}{ds}D_{\nu(s)}Fv\|ds\leq  \Const|t_2-t_1|\leq C_\sharp \|x-y\|,
\end{equation}
the case $n=1$ follows. Assume it is true for each $k<n$, then, by the triangular inequality
\[
\begin{split}
&\|D_yF^nv-D_xF^nv\|\le \sum_{k=0}^{n-1}\|D_{F^{k+1}y}F^{n-k-1}(D_{F^ky}F-D_{F^kx}F)D_xF^k v\|\\
&\leq\Const  \sum_{k=0}^{n-1}\lambda^+_{n-k-1}(F^ky)\lambda^+_k(x) \|D_{F^ky}F-D_{F^kx}F\|
\leq \Const  \sum_{k=0}^{n-1}\lambda^+_{n-k-1}(F^ky)\lambda^+_k(x)\mu^{k}\|x-y\|.
\end{split}
\]
Since $\nu\in F^{-n}(\Gamma(c))$, $\|D_{F^ky}F-D_{F^kx}F\|\le C_\sharp \mu^{k}\|x-y\|$. Also remark that \eqref{lambda_+<Clambda_-} and the induction hypothesis imply 
\[
\lambda^+_{n-k}(F^ky)\lambda^+_k(x)\le e^{\mu^kC_{\mu,k}\|x-y\|}\lambda^+_{n-k}(F^ky)\lambda^+_k(y)\leq \Const\lambda^-_n(y),
\]
provided we have chosen $\tau_n$ small enough. Accordingly, since $\|D_yF^nv\|\geq \lambda^-_n(y)$,
\[
\frac{\|D_xF^nv\|}{\|D_yF^nv\|}
\leq e^{\frac{\|D_xF^nv-D_yF^nv\|}{\|D_yF^nv\|} } 
\leq e^{C_\sharp \sum_{k=0}^{n-1}\mu^k\|x-y\|}.
\]
We can now choose $v$ such that $\|D_xF^nv\|=\lambda^+_n(x)$ so
\[
\frac{\lambda^+_n(x)}{\lambda^+_n(y)} \leq \frac{\|D_xF^nv\|}{\|D_yF^nv\|}\leq e^{C_\sharp \sum_{k=0}^{n-1}\mu^k\|x-y\|}\leq e^{C_{\mu,n}\mu^n \|x-y\|},
\]
which proves the upper bound, for points close enough. Next, for all $x,y\in \nu$ we can consider close intermediate points $\{x_i\}_{i=0}^l$, $x_0=x$, $x_l=y$, to which the above applies, hence
\[
\frac{\lambda^+_n(x)}{\lambda^+_n(y)} \leq \frac{\|D_xF^nv\|}{\|D_yF^nv\|}=\prod_{i=0}^{l-1}\frac{\|D_{x_i}F^nv\|}{\|D_{x_{i+1}}F^nv\|}
\leq e^{\mu^n C_{\mu,n}\sum_{i=0}^{l-1}\|x_{i+1}-x_i\|}.
\]
Taking the limit for $l\to \infty$ we have the distance, along the curve, between $x$ and $y$ which is bounded by $\Const\|x-y\|$. This proves the upper bound.
The lower bound is proven similarly.
\end{proof}
Next, we prove other two distortion Lemmata, 
inspired by Lemma 6.2 in \cite{GoLi}. Even though the basic idea of the proof is the same, the presence of the central direction creates some difficulties. 
\begin{lem}\label{lemma on det^-1}
For each $c_\star\geq\bbc$, $\gamma\in \Gamma(c_\star)$, $n>\bar n$ and $0\le \rho\le r-1,$ we have
\begin{equation}
\label{bound of det^-1}
\begin{split}
&\sum_{\nu_n\in F^{-n}\gamma}\left\| \frac{h'_{n}}{\det D_{\hat\nu_n}F^{n}} \right\|_{\mathcal{C}^{\rho}(\bT)}\le C_{\sharp} 
c_\star^{\rho!-1}(1+C_Fc_\star^{\beta_\rho})^{\{\rho+1,(\rho+1)\rho/2\}^+}C_{\mu, n}^{\tilde{a}_\rho} \mu^{\tilde{b}_\rho n}\\
& \sum_{\nu_n\in F^{-n}\gamma}\left\| \frac{1}{\det D_{\hat\nu_n}F^{n}} \right\|_{\mathcal{C}^{\rho}(\bT)}\le C_{\sharp} 
c_\star^{\rho!-1}(1+C_Fc_\star^{\beta_\rho})^{\{\rho+1,(\rho+1)\rho/2\}^+}C_{\mu, n}^{\tilde{a}_\rho}\mu^{(\tilde{b}_\rho +1)n}
\end{split}
\end{equation}
where $\beta_\rho=1$ for $\rho\leq 2$ and $\beta_\rho=\rho!$ otherwise; $\tilde a_\rho=\rho+1$, $\tilde b_\rho=\rho!-1$, for $\rho\leq 2$, and\footnote{ Recall the definition of $a_\rho$ in Lemma \ref{lem stable vertic curves}}  $\tilde a_\rho=a_\rho \rho(\rho+1)/2+1$, $\tilde b_\rho=\rho!\rho(\rho+1)/2+1$, otherwise.
\end{lem}
\begin{proof}
For every $\nu\in F^{-n}\gamma$ define
\begin{equation*}
\Psi_{\nu_n}(t)=\frac{h'_{n}(t)}{\det D_{{\hat{\nu}_n(t)}}F^{n}},
\end{equation*} 
and recall that in dimension one holds $\|\Psi_{\nu_n}\|_{\mathcal{C}^0}\le \|\Psi_{\nu_n}\|_{L^1}+\|\Psi'_{\nu_n}\|_{L^1}$. We then first look for a bound of the $W^{1,1}(\mathbb{T})$-norm of $\Psi_{\nu_n}$.
Since $e_1=(1,0)\in \fC_{u}$,  $D_{{\hat{\nu}_n}}F^n\hat{\nu}_n'\not \in \fC_{u}$ and recalling that $F^{n}{\hat{\nu}_n}=\gamma\circ h_n$, we have
\[
h'_{n}D_{{\hat{\nu}_n}}F^{n}e_1\wedge \gamma'\circ h_n=D_{{\hat{\nu}_n}}F^{n}e_1\wedge D_{{\hat{\nu}_n}}F^{n}\hat\nu_n'=\det(D_{{\hat{\nu}_n}}F^n)e_1\wedge\hat\nu_n'.
\]
Thus we have the equation\label{ Since two forms are (Hodge) dual to zero forms, we can treat them as functions.}
\begin{equation}\label{eq:alltheway}
\frac{|h'_{n}(t)|}{|\det D_{{\hat{\nu}_n(t)}}F^{n}|}=\frac{\|e_1\wedge\hat\nu_n'(t)\|}{\|D_{{\hat{\nu}_n(t)}}F^{n}e_1\wedge \gamma'\circ h_n(t)\|}.
\end{equation}
Since $\|\gamma'\|\ge 1$ we have, recalling definition \eqref{def of vartheta_gamma},  
\begin{equation}
\label{eq:DTvartheta}
\|D_{{\hat{\nu}_n}}F^{n}e_1\wedge \gamma'\circ h_n\|\geq \Const \vartheta_\gamma\circ h_n \|D_{{\hat{\nu}_n}}F^{n}e_1\|.
\end{equation}
Therefore, since $\|\hat{\nu}_n'\|^2\leq 1+\chi_c^2$, we have
\begin{equation}
\label{bound with unst derivative}
\sum_{{{\nu_n}}\in F^{-n}\gamma} \|\Psi_{{\nu_n}}\|_{L^1}\lesssim \sum_{{\nu_n}\in F^{-n}\gamma} \left\| \frac{1}{ \vartheta_\gamma\circ h_n \|D_{{{\hat{\nu}_n}}}F^{n} \cdot e_1\|}\right\|_{L^1}.
\end{equation}
Recall that, by Lemma \ref{lem inverse branches}, for each ${\hat \nu_n}$ we have an inverse branch $\frh_{\hat \nu_n}:\Omega_\gamma\to \Omega_{\hat \nu_n}$ such that $F^{n}\circ \frh_{\hat \nu_n}=Id_{\Omega_\gamma}$. Note that  the domain $\Omega_{\hat \nu_n}$ can be written as $\bigcup_{t\in\bT}\xi_{t,{\hat \nu_n}},$ where $\xi_{t,{\hat \nu_n}}(s)={\hat{\nu}_n}(t)+se_1$ are horizontal segments defined on an interval $I_t$ of length $\delta_{{\hat \nu_n(t)}}$ whose images are unstable curves $\xi^{\sharp}_{t,\gamma}$ with $\operatorname{length}(\xi^{\sharp}_{t,\gamma})=\delta^{\sharp}_{t,\gamma}\ge 1$.  Let $p_{n,{  \xi_{t, {\nu_n}}}}$ be the  diffeomorphism associated to $\xi_{t, {\nu_n}}$, see formula \eqref{property of p}.  By equation \eqref{distortion for unsable curve} $p'_{n,{  \xi_{t, {\nu_n}}}}(s)\lesssim p'_{n,{  \xi_{t, {\nu_n}}}}(0)=\|D_{\hat \nu_n{  (t)}}F^{n} e_1\|$. It follows 
\[
1\leq \delta^{\sharp}_{t,\gamma}=\int_{I_t}\left\| \frac{d}{ds}F^{n}(\xi_{t,{\hat \nu_n}}(s)) \right\|ds\le {  \Const}{  \delta_{{\hat \nu_n(t)}}} p'_{n,{  \xi_{t, {\nu_n}}}}(0)= {  \Const}{  \delta_{{\hat \nu_n(t)}}}\|D_{\hat \nu_n{  (t)}}F^{n} e_1\|,
\]
from which
\begin{equation}
\label{estimate of 1/DFe}
\|D_{{\hat{\nu}_n}(t)}F^{n}e_1\|\gtrsim \frac{1}{{  \delta_{{\hat \nu_n(t)}}}}.
\end{equation}
Since by Lemma \ref{lem inverse branches} the $\Omega_{\nu_n}$ are all disjoints and the $\nu_n$ are parametrized vertically, by (\ref{estimate of 1/DFe}) we have\footnote{ Here ${m}(A)$ is the Lebesgue measure of a set $A$.}
\begin{equation}
\label{estm of L1 norm Psi} 
\sum_{\nu_n \in F^{n}\gamma}\left\| \frac{1}{\|D_{\hat{\nu}_n(t)}F^{n}e_1\|} \right\|_{L^1}\lesssim \sum_{{\hat \nu_n} \in F^{n}\gamma}  \int_{{\bT}}\delta_{{\hat \nu_n(t)}}=\sum_{{\nu_n} \in F^{n}\gamma} {m}(\Omega_{{{\hat \nu_n}}})= {m}(\mathbb{T}^2)= 1.
\end{equation}
Using this in (\ref{bound with unst derivative}) yields
\begin{equation}
\label{control of the L1 norm}
\sum_{{{\nu_n}}\in F^{n}\gamma} \|\Psi_{{\nu_n}}\|_{L^1}\le  \Const  \vartheta_\gamma^{-1}\le \Const,
\end{equation}
since $|\pi_1\circ \gamma'(t)|^{-1}\ge \chi_c^{-1}>1>\chi_u$ implies $\vartheta_\gamma^{-1}\le 1.$
To bound the $L^1$ norm of the derivative we can notice that:
\begin{equation}
\label{Psi'_L1}
\|\Psi'_{{\nu_n}}\|_{L^1}\le \left\| \frac{\Psi'_{{\nu_n}}}{\Psi_{\nu_n}} \right\|_{\mathcal{C}^0}\| \Psi_{\nu_n} \|_{L^1}.
\end{equation}
For each $0\le i\le n$, let $\nu_{n-i}=F^i\nu_n$ and $h_i$ be the diffeomorphism such that $\hat \nu_i=\nu_i\circ h_i$ is parametrized by vertical length. Define the diffeomorphisms $h_i^*$ by 
\begin{equation}\label{eq: defn of nu*}
\hat \nu_{i}=F\circ \hat \nu_{i+1}\circ (h_{i+1}^*)^{-1},
\end{equation}
where $\nu_0=\gamma$ and $h^*_0=h_0=Id$.
Note that $h_i=h^*_{1}\circ \cdots \circ h_i^*$. 
We can then write
\begin{equation*}
\begin{split}
\Psi_{{\nu_n}}(t)&=\frac{\frac{d}{dt} h_n(t)}{\det D_{\hat\nu_n(t)}F^{n}}=\frac{\prod_{i={1}}^{n} (h_i^*)'\circ h_{i+1}^*\circ\cdots\circ h_n^*}{\prod_{i=1}^{n} (\det D_{\hat \nu_i}F)\circ h_{i+1}^*\circ\cdots\circ h_n^*}(t)\\
&=\prod_{i=1}^{n} (\psi_i\circ h_{i+1}^*\circ\cdots\circ h_n^*)(t),
\end{split}
\end{equation*}
where $\psi_{i}(t)=(h_i^*)'(t)\cdot(\det D_{\hat \nu_i(t)}F)^{-1}$. Hence,
\begin{equation}\label{eq:Psi'/Psi}
\left| \frac{\Psi_{{\nu_n}} '}{\Psi_{{\nu_n}}}\right|\le\sum_{i=1}^{n} \left| \left( \frac{\psi '_i}{\psi_i}\circ h_{i+1}^*\circ\cdots\circ h_n^*\right) 
(h_{i+1}^*\circ\cdots\circ h_n^*)' \right|.
\end{equation}
Differentiating twice \eqref{eq: defn of nu*}, yields
\[
\hat \nu_i '\circ h_{i+1}^*\cdot(h_{i+1}^*)'=D_{\hat \nu_{i+1}}F \hat\nu_{i+1}'.
\]
Since $\hat\nu_n\in \Gamma(c_\star)$, multiplying by $e_2$ and remembering \eqref{the map F} we have
\[
(h_{i+1}^*)'=\langle e_2, D_{\hat \nu_{i+1}}F\hat\nu_{i+1}'\rangle=(\partial_x\omega)\circ \hat \nu_{i+1}\cdot \langle e_1, \hat \nu_{i+1}'\rangle+1+(\partial_\theta\omega)\circ \hat \nu_{i+1}
\]
Thus $\|\psi_i\|_{\cC^{0}}\le \Const$ and, by \eqref{formula 1 for norm Crho},  it follows that $\|\psi_i\|_{\cC^{1}}\le  \Const (1+C_Fc_\star)$, while for $\rho\geq \ell>1$, $\|\psi_i\|_{\cC^{\ell}}\le \Const ( c_\star^{(\ell-1)!}+C_Fc_\star^{\ell !})$. Thus, setting $h_{i,n}=h_{i+1}^*\circ\cdots\circ h_n^*$, by \eqref{formula 1 for norm Crho} we have
\begin{equation*}
\begin{split}
\left\| \frac{\Psi_{{\nu_n}}'}{\Psi_{{\nu_n}}} \right\|_{\mathcal{C}^{\ell-1}} &\lesssim\sum_{i=0}^{n-1} \left\|  \left( \log \psi_i \circ h_{i,n} \right)'\right\|_{\mathcal{C}^{\ell-1}}\lesssim   \sum_{i=0}^{n-1}\left\|   \log \psi_i \circ h_{i,n} \right\|_{\mathcal{C}^{\ell}}\\
&\lesssim \sum_{i=0}^{n-1} \|\log \psi_i\|_{\cC^{\ell}}\sum_{j=0}^{\ell-1}\|h_{i,n}\|^{j}_{\cC^{\ell}}\\
\end{split}
\end{equation*}
which, recalling \eqref{Crho norm of h}, yields
\begin{equation}\label{eq:ratio-Psi}
\left\| \frac{\Psi_{{\nu_n}}'}{\Psi_{{\nu_n}}} \right\|_{\mathcal{C}^{\ell-1}}\lesssim
\begin{cases}
C_{\mu,n}(C_Fc_\star+1) &\textrm{for }\ell=1\\
C_{\mu,n} c_\star(C_Fc_\star+1)&\textrm{for }\ell=2\\
c_\star^{(\ell-1)!}(C_Fc_\star^{\ell !}+1)^{\ell}C_{\mu,n}^{\ell a_\ell} \mu^{n\ell \ell !}&\textrm{for }2<\ell\leq \rho.
\end{cases}
\end{equation}
In particular, the above estimates in the case $\ell=1$ and \eqref{Psi'_L1} gives
\[
\sum_{{{\nu_n}}\in F^{n}\gamma} \|\Psi_{{\nu_n}}'\|_{L^1}\le \Const C_{\mu,n}\mu^{n}(C_Fc_\star+1)\sum_{{{\nu_n}}\in F^{n}\gamma} \|\Psi_{{\nu_n}}\|_{L^1}\le  \Const C_{\mu,n}\mu^{n}(C_Fc_\star+1),
\]
which gives 
\[
\sum_{{\nu_n} \in F^{-n}\gamma}\| \Psi_{{\nu_n}}\|_{\mathcal{C}^{0}}\lesssim C_{\mu,n}(1+C_Fc_\star). 
\]
Then
\[
\begin{split}
\sum_{{\nu_n} \in F^{-n}\gamma}\| \Psi_{{\nu_n}}\|_{\mathcal{C}^{1}}&\lesssim \sum_{{\nu_n} \in F^{-n}\gamma}\| \Psi'_{{\nu_n}}\|_{\mathcal{C}^{0}} 
\lesssim \sum_{{\nu_n}\in F^{-n}\gamma}\left\|\frac{\Psi'_{\nu_n}}{\Psi_{\nu_n}}\right\|_{\cC^{0}}\|\Psi_{\nu_n}\|_{\cC^{0}}\lesssim C_{\mu,n}^2(1+C_Fc_\star)^2\\
\sum_{{\nu_n} \in F^{-n}\gamma}\| \Psi_{{\nu_n}}\|_{\mathcal{C}^{2}}& 
\lesssim \sum_{{\nu_n}\in F^{-n}\gamma}\left\|\frac{\Psi'_{\nu_n}}{\Psi_{\nu_n}}\right\|_{\cC^{1}}\|\Psi_{\nu_n}\|_{\cC^{1}}\lesssim C_{\mu,n}^3c_\star(1+C_Fc_\star)^3.
\end{split}
\]
To conclude we can obtain the general case $\rho\in \{3, \dots, r-1\}$ by induction as follows:
\begin{equation*}
\begin{split}
\sum_{{\nu_n} \in F^{-n}\gamma}\| \Psi_{{\nu_n}}\|_{\mathcal{C}^{\rho}}&\lesssim \sum_{{\nu_n} \in F^{-n}\gamma}\| \Psi'_{{\nu_n}}\|_{\mathcal{C}^{\rho-1}} \lesssim \sum_{{\nu_n}\in F^{-n}\gamma}\left\|\frac{\Psi'_{\nu_n}}{\Psi_{\nu_n}}\right\|_{\cC^{\rho-1}}\|\Psi_{\nu_n}\|_{\cC^{\rho-1}} \\
& \lesssim c_\star^{(\rho-1)!}(C_Fc_\star^{\rho !}+1)^{\rho}C_{\mu,n}^{\rho a_\rho} \mu^{n\rho \rho !}\sum_{{\nu_n} \in F^{-n}\gamma}\| \Psi_{{\nu_n}}\|_{\mathcal{C}^{\rho-1}}\\
&\lesssim c_\star^{\rho!-2}(C_Fc_\star^{\rho !}+1)^{(\rho+1)\rho/2-3}C_{\mu,n}^{a_\rho \sum_{k=0}^{\rho}k } \mu^{n \rho! \sum_{k=0}^{\rho}k }\sum_{{\nu_n} \in F^{-n}\gamma}\| \Psi_{{\nu_n}}\|_{\mathcal{C}^{2}}\\
&\lesssim c_\star^{\rho!-1} (C_Fc_\star^{\rho !}+1)^{(\rho+1)\rho/2}C_{\mu,n}^{a_\rho \frac{\rho(\rho+1)}{2}+1}\mu^{\rho! \frac{\rho(\rho+1)}{2}+1)n}.
\end{split}
\end{equation*}
The procedure to prove the second of \eqref{bound of det^-1} is analogous, with the difference that, by \eqref{eq:alltheway} and \eqref{eq: h_n'}, the estimate for $\rho=0$ gives another $\Const \mu^n$, while the computation for $\rho\ge 1$ is exactly the same, but using $\psi_i=(\det D_{\hat \nu_i(t)}F)^{-1}$ instead.
\end{proof}
The next result is a refinement of the previous Lemma in the more general case in which $\gamma'$ is not contained neither in $\fC_c$ nor in $\fC_u$, and it is parametrised by arc-length. Let $h_0$ be such that $\hat\nu_0=\gamma\circ h_0$ is parametrised vertically. To state the result it is convenient to define the following quantities
\begin{equation}\label{eq:Thetag}
\begin{split}
&\vartheta_{\hat\nu_0,m}(t)=\inf_{|s-t|\leq \const m \|\omega\|_\infty}\vartheta_{\hat\nu_0}(s)\\
\end{split}
\end{equation}
\begin{lem}\label{lem on det^-1-for-unstable-curve}
In the same hypothesis of Lemma \ref{lem unst curve pull back} with $n_0=0$, we have
\begin{equation}
\label{bound of C2 norm of det^-1}
\begin{split}
&\sum_{\nu_\ovm\in F^{-\ovm}\gamma}\left\| \frac{h'_{\ovm}}{\det D_{\hat\nu_\ovm}F^{\ovm}} \right\|_{\mathcal{C}^{\rho}(\bT)}\le \Const (\chi_u^{-1} \Delta_\gamma )^{\const} \vartheta_\gamma^{-1} \Lambda^{\const \ovm}, \quad \rho \ge 2.
\end{split}
\end{equation}
\end{lem}
\begin{proof}
We use the same notations of the proof of Lemma \ref{lemma on det^-1}. Note that \eqref{bound with unst derivative} holds true for each $\gamma'\not\in\fC_u$, so we can proceed exactly as in the proof of above Lemma \ref{lemma on det^-1}, but with the following differences: $\vartheta_\gamma$ in \eqref{control of the L1 norm} is not bounded by a uniform constant $\Const$, as $\gamma$ is not an admissible central curve. Moreover, we use the estimates in \eqref{eq:ratio-Psi} with $c_\star=\chi_u^{-2} \|\Delta_{\gamma}\|_\infty (\mu\Lambda)^{m}$ provided by Lemma \ref{lem unst curve pull back}, so that the last of \eqref{bound of C2 norm of det^-1} follows. 
\end{proof}
By the above we can also obtain a key estimate of the $L^\infty$ norm of $\cL^n1$.
\begin{cor}
\label{cor of inf norm of transfer}
 Let $\cL:=\cL_{F}$ be the transfer operator defined in (\ref{Trasfer Operator}). Then, for each $n\in\bN$,
\begin{equation}
\label{inf norm of transf ineq}
\| \mathcal{L}^{n}1\|_{L^{\infty}(\mathbb{T}^{2})}\le C_{\mu,n}(1+C_F\bbc) \mu^{n}.
\end{equation}
\end{cor}
\begin{proof}
For any $x\in \mathbb{T}^2$ we want to estimate the quantity
\begin{equation}
\label{trans oper in branches}
\mathcal{L}^n{1}(x)=\sum_{y\in F^{-n}x} \frac{1}{|\det D_yF^n|}.
\end{equation}
Recall the notation in Section \ref{subsec admiss curv} and take $y\in \gamma$, where $\gamma\in \Gamma(\bbc)$ is an admissible central curve. Then, for every $x\in F^{-n}(y)$, there exist $t\in \mathbb{T}$ and $\nu\in F^{-n}\gamma$ such that $x=\nu(h_n(t))=\hat{\nu}(t)$. Hence
\begin{equation*}
\begin{split}
\sup_{y\in\gamma}\sum_{x\in F^{-n}(y)}\left| \frac{1}{\det D_{x}F^{n}} \right| \le \sum_{\nu\in F^{-n}\gamma} \left\| \frac{h'_{\nu,n}}{\det D_{\hat{\nu}}F^{n}}\right\|_{\mathcal{C}^0}\|(h_{\nu,n}')^{-1}\|_{\mathcal{C}^{0}}.
\end{split}
\end{equation*}
By equations \eqref{eq: h_n'} and (\ref{partial hyperbolicity 2}) we know that $\|(h_n')^{-1}\|_{\cC^0}\le  \Const\mu^{n}$, for every $\nu$ and $n$.  Moreover, Lemma \ref{lemma on det^-1} gives the bound 
\[
\sum_{\nu\in F^{-n}\gamma}\left\| \frac{h'_{\nu}}{\det D_{\hat{\nu}}F^{n}}\right\|_{\mathcal{C}^0}\le C_{\mu,n}(1+C_F\bbc) \mu^{n}.
\]
\end{proof}
\begin{oss}
With some extra work the estimate (\ref{inf norm of transf ineq}) can be made sharper, however the above bound is good enough for our current purposes. We will need an improvement, provided in Lemma \ref{shadowing}, in Section \ref{section application to map F}.
\end{oss}

\subsection{Distortion: the \texorpdfstring{SVPH$^{\,\sharp}$}{SPVHs} case}

In the case of SVPH$^{\,\sharp}$ we need an improvement of Lemma \ref{lem on det^-1-for-unstable-curve} for the first derivatives. To state the result it is convenient to define the following quantities:
\begin{equation}\label{eq:ThetagS}
\begin{split}
&\vartheta_{\hat\nu_0,m}(t)=\inf_{|s-t|\leq \const m \|\omega\|_\infty}\vartheta_{\hat\nu_0}(s)\\
&\bJ_{\gamma,n}=\int_{{\bT}}\left[\vartheta_{\hat\nu_0,n}(s)\right]^{-1} ds,\\
&\bI_{\gamma, m}=C_{\mu,m}\mu^{4m}\left[1+\mu^m(\vartheta_{\hat\nu_0}\circ\bar h_{m})^{-1}C_F)\right](\vartheta_{\hat\nu_0}\circ\bar h_{m})^{-1}
 \Delta_\gamma.
\end{split}
\end{equation}
\begin{lem}\label{lem on det^-1-for-unstable-curve_sharp}
In the same hypothesis of Lemma \ref{lem unst curve pull back_sharp} with $n_0=0$, we have
\begin{equation}
\label{bound of C2 norm of det^-1S}
\begin{split}
&\sum_{\nu_\ovm\in F^{-\ovm}\gamma}\left\| \frac{h'_{\ovm}}{\det D_{\hat\nu_\ovm}F^{\ovm}} \right\|_{\mathcal{C}^{0}(\bT)}\le \Const\bI_{\gamma, \ovm}\bJ_{\gamma,\ovm}\\
& \sum_{\nu_\ovm\in F^{-\ovm}\gamma}\left\| \frac{h'_{\ovm}}{\det D_{\hat\nu_\ovm}F^{\ovm}} \right\|_{\mathcal{C}^{1}(\bT)}
\le \Const (\bI_{\gamma, \ovm})^2\bJ_{\gamma,\ovm}
\\
&\sum_{\nu_\ovm\in F^{-\ovm}\gamma}\left\| \frac{h'_{\ovm}}{\det D_{\hat\nu_\ovm}F^{\ovm}} \right\|_{\mathcal{C}^{2}(\bT)}\le    C_{\mu,\ovm}^4\mu^{5\ovm}(1+C_F\mu^{m}\vartheta_{\hat\nu_0}^{-1})O^\star_m\bI_{\gamma,m}^3\bJ_{\gamma,\ovm}\\
\end{split}
\end{equation}
where $O^\star_m=M_{m,0}(1+C_FM_{m,0})^3$.\footnote{ Where $M_{m,n}$ is defined in \eqref{def of eta_nstar and M(t,n_0,m)}.}
\end{lem}
\begin{proof}
We use the same notations of the proof of Lemma \ref{lemma on det^-1}. To clarify the argument we use the more precise notation $h_{\hat\nu_n}$ rather than $h_n$, since the dependence on the branch is relevant in the following. Thus, setting, for each $n\leq \ovm$, $\Psi_{\nu_n}(t)=\frac{h'_{\nu_n}(t)}{\det D_{\hat\nu_n(t)}F^{n}} $ and recalling \eqref{estimate of 1/DFe}, we can write
\[
\begin{split}
&\sum_{\nu_{n}\in F^{-n}\gamma} \left\|\Psi_{\hat\nu_n} \right\|_{L^1}
\leq \Const \sum_{{\nu_n}\in F^{-n}\gamma} \int_{{\bT}}\frac{\delta_{\hat\nu_n}(s)}{ \vartheta_\gamma\circ h_{\hat\nu_n}(s)} ds.
\end{split}
\]
Setting $F^n(\hat \nu_n(t))=\hat\nu_0\circ \bar h_{\hat\nu_n}(t)$ we have $h_{\hat\nu_n}(t)=h_0\circ \bar h_{\hat\nu_n}(t)$ and $\hat\nu_0=\gamma\circ h_0$, $\pi_2(\hat \nu_0')=1$.
Then, since $|\pi_2(F^n(x,\theta))-\theta|\leq n\|\omega\|_\infty$ it follows that, 
\[
|\bar h_{\hat\nu_m}(t)-t|=|\pi_2(F^m(\hat\nu_{m}(t)))-t))|\leq m\|\omega\|_\infty.
\]
Accordingly, 
\[
\vartheta_\gamma\circ h_{\hat \nu_n}(s)=\vartheta_{\hat\nu_0}\circ \bar h_{\hat \nu_n}(s)\geq \inf_{|s-\tau|\leq \const n \|\omega\|_\infty}\vartheta_{\hat\nu_0}(\tau)= \vartheta_{\hat\nu_0,n}(s).
\]
Thus
\begin{equation}\label{eq:Psi-L^1}
\begin{split}
\sum_{\hat \nu_{n}\in F^{-n}\gamma} \left\|\Psi_{\hat\nu_n}\right\|_{L^1}
&\leq \Const  \int_{{\bT}}\frac{\sum_{{\hat\nu_n}\in F^{-n}\gamma}\delta_{{\hat \nu_n(s)}}}{ \vartheta_{\hat\nu_0,n}(s)} ds
\leq \Const  \int_{{\bT}}\left[\vartheta_{\hat\nu_0,n}(s)\right]^{-1} ds .
\end{split}
\end{equation}
Next, let $m_\star(t)$ be the smallest integer $k$ for which $\hat\nu_k'\in\fC_c$. Note that, setting $m_1=\ovm-m_\star$  and 
$\tilde \Psi_{\hat\nu_\ovm}(t)=\frac{\bar h'_{\hat\nu_\ovm,\hat\nu_{m_\star}}}{\det D_{\hat\nu_\ovm(t)}F^{m_1}} $, where $\bar h_{\hat\nu_n,\hat\nu_m}$ is such that $F^{n-m}\hat\nu_m=\hat\nu_n\circ \bar h_{\hat\nu_n,\hat\nu_m},$
\begin{equation}\label{eq:Psi-multiplicative}
\begin{split}
\sum_{\hat\nu_{\ovm}\in F^{-\ovm}\gamma} \Psi_{\hat\nu_\ovm}(t)&
= \hskip-6pt\sum_{F^{m_\star}\hat\nu_{m_\star}\in \gamma}\sum_{F^{m_1}\hat\nu_{\ovm}\in \hat\nu_{m_\star}}\hskip-6pt\frac{ h'_{\hat\nu_{m_\star}}\circ \bar h_{\hat\nu_\ovm,\hat\nu_{m_\star}}\cdot  \bar h'_{\hat\nu_\ovm,\hat\nu_{m_\star}}}{\det D_{\hat \nu_{\ovm}}F^{m_1}\cdot \det D_{\hat \nu_{m_\star}\circ \bar h_{\hat\nu_\ovm,\hat\nu_{m_\star}}}F^{m_\star}}\\
& =\sum_{\hat\nu_{m_\star}\in F^{-m_\star}\gamma} \Psi_{\hat \nu_{m_\star}}\circ\bar h_{\hat\nu_\ovm,\hat\nu_{m_\star}} \sum_{ \hat\nu_{\ovm}\in F^{-m_1}\hat \nu_{m_\star}}\tilde \Psi_{\hat \nu_{\ovm}}\,.
\end{split}
\end{equation}
Since $\hat\nu_{m_\star}'\in \fC_c$, we apply Lemma \ref{lemma on det^-1}, with $c_\star=\|\hat\nu_{m_\star}''\|$ and remember \eqref{eq:h_m0''-hat nu''_m}, we write
\begin{equation}\label{eq:easy_part}
\begin{split}
&\sum_{ \hat\nu_{\ovm}\in F^{-m_1}\hat \nu_{m_\star}}\left\|\tilde \Psi_{\hat \nu_{\ovm}}\right\|_{\cC^0}\leq  C_{\mu,\ovm}(1+C_FM_{m_\star,0}) \\
&\sum_{ \hat\nu_{\ovm}\in F^{-m_1}\hat \nu_{m_\star}}\left\|\tilde \Psi_{\hat \nu_{\ovm}}\right\|_{\cC^1}\leq  C_{\mu,\ovm}^2(1+C_FM_{m_\star,0})^2 \\
&\sum_{ \hat\nu_{\ovm}\in F^{-m_1}\hat \nu_{m_\star}}\left\|\tilde \Psi_{\hat \nu_{\ovm}}\right\|_{\cC^2}\leq  C_{\mu,\ovm}^3\mu^{\ovm}M_{m_\star,0}(1+C_FM_{m_\star,0})^3.
\end{split}
\end{equation}
Next, using \eqref{eq:alltheway},\footnote{ Remark that two forms are isomorphic to functions, and since here they are never zero we can disregard the norm.} 
\begin{equation}\label{eq:psiuno}
\begin{split}
\frac{\Psi'_{{\hat\nu_{m_\star}}}}{\Psi_{\hat\nu_{m_\star}}}=&\frac{e_1\wedge\hat\nu_{m_\star}''}{e_1\wedge\hat\nu_{m_\star}'}\\
&- \frac{\partial_t\left(D_{{\hat{\nu}_{m_\star}}}F^{{m_\star}}e_1\right)\wedge \gamma'\circ h_{m_\star}+D_{{\hat{\nu}_{m_\star}}}F^{{m_\star}}e_1\wedge \gamma''\circ h_{m_\star}\cdot h_{m_\star}'}{D_{{\hat{\nu}_{m_\star}}}F^{{m_\star}}e_1\wedge \gamma'\circ h_{m_\star}}\\
=&- \left(\left[(D_{{\hat{\nu}_{m_\star}}}F^{{m_\star}})^{-1}\partial_t\left(D_{{\hat{\nu}_{m_\star}}}F^{{m_\star}}\right)\right]e_1\right)\wedge \hat{\nu}'_{m_\star}\\
&-e_1\wedge (D_{{\hat{\nu}_{m_\star}}}F^{{m_\star}})^{-1} \gamma''\circ h_{m_\star}\cdot (h_{m_\star}')^2,
\end{split}
\end{equation}
where we used that $e_1\wedge\hat\nu_{m_\star}''=0$, since the vectors are parallel, and $e_1\wedge \hat{\nu}'_{m_\star}=1$.
Letting $\eta$, $\|\eta\|=1$ such that  $(DF^{{m_\star}})^{-1}\eta\wedge e_1=0$ we can write $e_1=a\eta+b e_2$, with $|b|\leq\chi_u$. Hence, using \eqref{eq:dtDF-dt^2DF}, \eqref{eq:ve-special-inv}, \eqref{eq: Fq nu=gamma h}, \eqref{eq: h_ovm' h_ovm''} and Sublemma \ref{sublem:chiulambda} we have
\begin{equation}\label{eq:derivatives_varie}
\begin{split}
&\|\left[(D_{{\hat{\nu}_{m_\star}}}F^{{m_\star}})^{-1}\partial_t\left(D_{{\hat{\nu}_{m_\star}}}F^{{m_\star}}\right)\right]e_1\|\leq \\
&\leq\sum_{k=0}^{{m_\star}-1}
\|(D_{\hat\nu_{m_\star}}F^{k+1})^{-1}
\left[\partial_{x_i}D_{F^k(\hat\nu_{{m_\star}})}F\right]D_{\hat\nu_{m_\star}}F^{k}e_1\|\cdot \|D_{\hat\nu_{m_\star}}F^{k}\hat\nu_{m_\star}'\|\\
&\leq \sum_{k=0}^{{m_\star}-1}\Const (1+\mu^k C_F  \lambda_{k}^+ )
\|(D_{F^k\hat\nu_{m_\star}}F^{{m_\star}-k})^{-1}D_{\hat\nu_{m_\star}}F^{{m_\star}}\hat\nu_{m_\star}'\|\\
& \leq C_{\mu,{m_\star}}\mu^{{m_\star}}(1+\lambda_{m_\star}^+(\hat\nu_{m_\star})C_F)|h_{{m_\star}}'|\\
&\leq C_{\mu,{m_\star}}\mu^{2{m_\star}}(1+\mu^{m_\star}(\vartheta_{\hat\nu_0}\circ\bar h_{{m_\star}})^{-1}C_F)(\vartheta_{\hat\nu_0}\circ\bar h_{{m_\star}})^{-1}.
\end{split}
\end{equation}
On the other hand, writing, as before, $\gamma''=\eta+\alpha e_2$ with $\|\eta\|\leq\|\gamma''\|$, $|\alpha|\leq C_F\|\gamma''\|$ and $DF^{m_\star}e_1\wedge \eta=0$, yields
\[
|e_1\wedge (D_{{\hat{\nu}_{m_\star}}}F^{{m_\star}})^{-1} \gamma''|=|\langle e_2,  (D_{{\hat{\nu}_{m_\star}}}F^{{m_\star}})^{-1} \gamma''\rangle|\leq  \Const\mu^{m_\star}(1+\mu^{m_\star}C_F\vartheta_{\hat\nu_0})\vartheta_{\hat\nu_0}\Delta_\gamma.
\]
Thus, recalling \eqref{Crho norm of h},
\begin{equation}\label{Psi'-Psi}
\left|\frac{\Psi'_{{\hat\nu_{m_\star}}}}{\Psi_{\hat\nu_{m_\star}}}\circ \bar h_{\hat\nu_\ovm,\hat\nu_{m_\star}}\cdot \bar h_{\hat\nu_\ovm,\hat\nu_{m_\star}}'\right|\leq\Const\bI_{\gamma, {m_\star}}.
\end{equation}
By \eqref{eq:Psi-multiplicative}, \eqref{Psi'-Psi}  and  \eqref{eq:Psi-L^1} it follow, for all $n\in\{m_\star,\dots,\ovm\}$,
\begin{equation}\label{Psi'_Czero}
\begin{split}
\sum_{\hat\nu_{n}\in F^{-n}\gamma}\|\Psi_{\hat\nu_n}\|_{\cC^0}&\leq   \sum_{\hat\nu_{n}\in F^{-n}\gamma}\left(1+\left\| \frac{\Psi'_{{\hat\nu_n}}}{\Psi_{\hat\nu_n}}\right\|_{\cC^0}\right)\|\Psi_{\hat\nu_n} \|_{L^1}\\
&\leq  \,\Const\bI_{\gamma, n}\sum_{\hat\nu_{n}\in F^{-n}\gamma}\|\Psi_{\hat\nu_n} \|_{L^1}\leq \,\Const\bI_{\gamma, n} \int_{{\bT}}\left[\vartheta_{\hat\nu_0,n}(s)\right]^{-1} ds .
\end{split}
\end{equation}
which proves the first of \eqref{bound of C2 norm of det^-1}.
Next,
\[
\|\Psi'_{{\hat\nu_\ovm}}\|_{\cC^0}\le \left\| \frac{\Psi'_{{\hat\nu_\ovm}}}{\Psi_{\hat\nu_\ovm}} \right\|_{\mathcal{C}^0}\| \Psi_{\hat\nu_\ovm} \|_{\cC^0}.
\]
leads immediately to the second of \eqref{bound of C2 norm of det^-1S}.

To conclude the lemma we must compute $\Psi''_{\hat\nu_{m_\star}}$, which can be obtained by \eqref{eq:psiuno}:
\begin{equation}\label{eq:Psidue0}
\begin{split}
&\frac{\Psi''_{\hat\nu_{m_\star}}}{\Psi_{\hat\nu_{m_\star}}}=\left[\frac{\Psi'_{\hat\nu_{m_\star}}}{\Psi_{\hat\nu_{m_\star}}}\right]'+\left[\frac{\Psi'_{\hat\nu_{m_\star}}}{\Psi_{\hat\nu_{m_\star}}}\right]^2\\
&=-\left\{\partial_t\left[(D_{{\hat{\nu}_{m_\star}}}F^{m_\star})^{-1}\partial_t\left(D_{{\hat{\nu}_{m_\star}}}F^{m_\star}\right)\right]e_1\right\}\wedge \hat{\nu}'_{m_\star}-\left\{\left[(D_{{\hat{\nu}_{m_\star}}}F^{m_\star})^{-1}\partial_t\left(D_{{\hat{\nu}_{m_\star}}}F^{m_\star}\right)\right]e_1\right\}\wedge \hat{\nu}''_{m_\star}\\
&-e_1\wedge \left[\partial_t(D_{{\hat{\nu}_{m_\star}}}F^{m_\star})^{-1}\right] \gamma''\circ h_{m_\star}\cdot (h_{m_\star}')^2
-e_1\wedge(D_{{\hat{\nu}_{m_\star}}}F^{m_\star})^{-1} \gamma'''\circ h_{m_\star}\cdot (h_{m_\star}')^3\\
&- 2 e_1\wedge(D_{{\hat{\nu}_{m_\star}}}F^{m_\star})^{-1} \gamma''\circ h_{m_\star}\cdot h_{m_\star}'h_{m_\star}''
+\left[\frac{\Psi'_{\hat\nu_{m_\star}}}{\Psi_{\hat\nu_{m_\star}}}\right]^2.
\end{split}
\end{equation}
We estimate the lines of \eqref{eq:Psidue0} one at a time. 
For the first line, by \eqref{eq:dtDF-dt^2DF} we can write
\[
\begin{split}
(D_{{\hat{\nu}_{m_\star}}}F^{{m_\star}})^{-1}\partial_t(D_{{\hat{\nu}_{m_\star}}}F^{{m_\star}})= \sum_{s=1}^2\sum_{k=0}^{{m_\star}-1}(D_{{\hat{\nu}_{m_\star}}}F^{k+1})^{-1}
\partial_{x_s} (D_{F^k(\hat\nu_{m_\star}(t))}F)D_{\hat\nu_{m_\star}(t)}F^{k} (D_{\hat\nu_{m_\star}(t)}F^k\hat\nu'_{m_\star})_s.
\end{split}
\]
We can thus use the fourth of \eqref{norm of DF^N^-1}, with $n=m=m_\star$ and $\bbc=M_{m_\star,0}$, and argue as in \eqref{eq:derivatives_varie} to bound the first line of \eqref{eq:Psidue0} by
\begin{equation}\label{eq:rigadue}
\begin{split}
 \Const&(1+\mu^{m_\star}C_F\vartheta_{\hat\nu_0}^{-1})\left[ (\lambda_{m_\star}^+\circ \hat\nu_{m_\star})^2
 +(\lambda_{m_\star}^+\circ \hat\nu_{m_\star})^2+\lambda_{m_\star}^+\circ \hat\nu_{m_\star}M_{m_\star,0} \right]\\
 &\le \Const(1+\mu^{m_\star}C_F\vartheta_{\hat\nu_0}^{-1})^2 \mu^{4m_\star}\vartheta_{\hat\nu_0}^{-2}(1+\Delta_\gamma\vartheta_{\hat\nu_0}).
\end{split}
\end{equation}
To estimate the second line we use the second line of \eqref{eq:dtDF-dt^2DF}, arguing as above, and \eqref{eq: h_ovm' h_ovm''}
\begin{equation}\label{eq:rigatre}
\begin{split}
&C_{\mu,m_\star}\mu^{4m_\star}(1+\mu^{m_\star}C_F\vartheta_{\hat\nu_0}^{-1})^2\Delta_\gamma\vartheta_{\hat\nu_0}^{-1} +\mu^{4m_\star}(1+\mu^{m_\star}C_F\vartheta_{\hat\nu_0}^{-1})\Delta_\gamma^2 \vartheta_{\hat\nu_0}^{-2}.
\end{split}
\end{equation}
Finally, again by \eqref{eq: h_ovm' h_ovm''} and recalling \eqref{Psi'-Psi}, the last line is estimated by  
\begin{equation}\label{eq:rigaquattro}
\mu^{4m_\star}(1+\mu^{m_\star}C_F\vartheta_{\hat\nu_0}^{-1}){M}_{m_\star,0}\vartheta_{\hat\nu_0}^{-2}+\Const\bI_{\gamma, {m_\star}}^2\mu^{-2m_\star}
\end{equation}
Thus
\begin{equation}\label{eq:Psidue01}
\left|\frac{\Psi''_{\hat\nu_{m_\star}}}{\Psi_{\hat\nu_{m_\star}}}\right|\leq \mu^{7m_\star}C_{\mu,m_\star}^2(1+\mu^{m_\star}C_F\vartheta_{\hat\nu_0}^{-1})^2
(1+\Delta_\gamma^2+\mu^{m_\star}\vartheta_{\hat\nu_0}^{-1})\vartheta_{\hat\nu_0}^{-2}.
\end{equation}
To conclude note that
\begin{equation}\label{eq:palla_aggiuntiva}
\begin{split}
&\left|\frac{(\Psi_{\nu_{m_\star}}\circ \bar h_{\hat\nu_\ovm,\hat\nu_{m_\star}})''}{\Psi_{\nu_{m_\star}}\circ \bar h_{\hat\nu_\ovm,\hat\nu_{m_\star}}}\right|
=\left| \frac{\Psi_{\nu_\ovm}''\circ \bar h_{\hat\nu_\ovm,\hat\nu_{m_\star}}}{\Psi_{\nu_\ovm}\circ \bar h_{\hat\nu_\ovm,\hat\nu_{m_\star}}}\left(\bar h_{\hat\nu_\ovm,\hat\nu_{m_\star}}'\right)^2+ \frac{\Psi_{\nu_\ovm}'\circ \bar h_{\hat\nu_\ovm,\hat\nu_{m_\star}}}{\Psi_{\nu_\ovm}\circ \bar h_{\hat\nu_\ovm,\hat\nu_{m_\star}}}\bar h_{\hat\nu_\ovm,\hat\nu_{m_\star}}''\right|\\
&\phantom{\left|\frac{(\Psi_{\nu_{m_\star}}\circ \bar h_{\hat\nu_\ovm,\hat\nu_{m_\star}})''}{\Psi_{\nu_{m_\star}}\circ \bar h_{\hat\nu_\ovm,\hat\nu_{m_\star}}}\right|}
\leq  \left|\frac{\Psi_{\nu_\ovm}''}{\Psi_{\nu_\ovm}}\right|\mu^{2\ovm}+C_{\mu,\ovm}\left|\frac{\Psi_{\nu_\ovm}'}{\Psi_{\nu_\ovm}}\right| (1+C_FM_{m_\star,0})\mu^{2\ovm}\\
&\phantom{\left|\frac{(\Psi_{\nu_{m_\star}}\circ \bar h_{\hat\nu_\ovm,\hat\nu_{m_\star}})''}{\Psi_{\nu_{m_\star}}\circ \bar h_{\hat\nu_\ovm,\hat\nu_{m_\star}}}\right|}
\leq C_{\mu,m_\star}^3 \mu^{10m_\star}(1+C_F\mu^{m_\star}\vartheta_{\hat\nu_0}^{-1})^3\Delta_\gamma^2\vartheta_{\hat\nu_0}^{-2}\\
&\phantom{\left|\frac{(\Psi_{\nu_{m_\star}}\circ \bar h_{\hat\nu_\ovm,\hat\nu_{m_\star}})''}{\Psi_{\nu_{m_\star}}\circ \bar h_{\hat\nu_\ovm,\hat\nu_{m_\star}}}\right|}
\leq C_{\mu,m_\star}\mu^{2m_\star}(1+C_F\mu^{m_\star}\vartheta_{\hat\nu_0}^{-1})\bI_{\gamma,m}^2,
\end{split}
\end{equation}
where, in the last lines we have used \eqref{Crho norm of h}, with $c_\star=M_{m_\star,0}$, and \eqref{eq:Psidue01},  \eqref{Psi'-Psi}.
We finally have, by \eqref{eq:easy_part}, \eqref{eq:palla_aggiuntiva}, and \eqref{Psi'_Czero}, 
\[
\begin{split}
\sum_{{{\nu_m}}\in F^{m}\gamma} \|\Psi_{{\nu_\ovm}}\|_{\cC^2}&\le \sum_{\hat\nu_{m_\star}\in F^{-m_\star}\gamma} \left\|\Psi_{\hat \nu_{m_\star}}\circ\bar h_{\hat\nu_\ovm,\hat\nu_{m_\star}}\right\|_{\cC^2} \sum_{ \hat\nu_{\ovm}\in F^{-m_1}\hat \nu_{m_\star}}\left\|\tilde \Psi_{\hat \nu_{\ovm}}\right\|_{\cC^2}\\
&\leq\sum_{\hat\nu_{m_\star}\in F^{-m_\star}\gamma}\left\{\bI_{\gamma, {m_\star}}+ \left\|\frac{[\Psi_{\hat \nu_{m_\star}}\circ\bar h_{\hat\nu_\ovm,\hat\nu_{m_\star}}]''}{\Psi_{\hat \nu_{m_\star}}\circ\bar h_{\hat\nu_\ovm,\hat\nu_{m_\star}}}\right\|_{\cC^0}\right\}\|\Psi_{\hat \nu_{m_\star}}\circ\bar h_{\hat\nu_\ovm,\hat\nu_{m_\star}}\|_{\cC^0}\\
&\phantom{\leq\ \ \ }
\times C_{\mu,\ovm}^3\mu^{\ovm}M_{m_\star,0}(1+C_FM_{m_\star,0})^3\\
&\leq C_{\mu,\ovm}^4\mu^{5\ovm}(1+C_F\mu^{m_\star}\vartheta_{\hat\nu_0}^{-1})M_{m_\star,0}(1+C_FM_{m_\star,0})^3\bI_{\gamma,m}^3\bJ_{\gamma,\ovm}.
\end{split}
\]
The Lemma follows since $m\leq m_\star$.
\end{proof}
\section{A first Lasota-Yorke inequality}\label{section A first LY ineq}
We define a class of geometric norms inspired by \cite{GoLi} and \cite{AvGoTs}.
Given $u\in \mathcal{C}^r(\bT^2,\bR)$ and an integer $\rho <r$,  we denote by $\mathcal{B}_{\rho}$ the completion of $\mathcal{C}^r(\mathbb{T}^2,\bR)$ with respect to the norm:
\begin{equation}
\label{geom norm}
\| u \|_{\rho}:=\max_{|\alpha|\le \rho}\sup_{\gamma \in \Gamma(\bbc)} \sup_{\substack{\phi \in \mathcal{C}^{|\alpha|}(\mathbb{T})\\\|\phi\|_{\mathcal{C}^{|\alpha|}}=1}} \int_{\mathbb{T}} \phi(t)(\partial ^{\alpha} u)(\gamma(t))dt.
\end{equation}
This defines a decreasing sequence of Banach spaces continuously embedded in $L^1$, namely
\begin{equation}
\label{relation norm rho rho' and L1}
\|u\|_{L^1}\le {C}\|u\|_{\rho_1}\le {C}\|u\|_{\rho_2}, \qquad \mbox{for every} \quad 0\le \rho_1\le \rho_2\le r-1.
\end{equation}
To see this we observe that, since {$\sigma_x(t)=(x,t)\in \Gamma$},
\[
\begin{split}
{\|u\|_{L^1}=}\sup_{\|\phi\|_{\cC^0(\bT^2)}\le 1}&\int_{\bT} dx\int_{\bT} dy \phi(x,y)u(x,y)\le \int_{\bT} dx \sup_{\|\phi\|_{\cC^0(\bT^2)}\leq 1} \int_{\bT} dy \phi(x,y)u(x,y)\\
&\le \int_{\bT} dx \sup_{\|\phi\|_{\cC^0(\bT)}\leq 1} \int_{\bT} dt\; \phi(t)u({ \sigma_x}(t))\le \int_{\bT} dx \|u\|_0=\|u\|_0.
\end{split}
\]
The above proves the first inequality of \eqref{relation norm rho rho' and L1}, the others being trivial.\\
Next, we prove a Lasota-Yorke type inequality between the spaces $\cB_{\rho}$ and $\cB_{\rho-1}.$
\begin{thm} 
\label{lem MBB}
Let $F\in \cC^r(\bT^2, \bT^2)$ be a SVPH. Let $\cL:=\cL_{F}$ be the transfer operator defined in (\ref{Trasfer Operator}), and $\bar{n}$ be the integer given in Lemma \ref{lem stable vertic curves}. For each $\rho \in [1,r-1]$ and $n>\bar{n}$, there exists $C_{n,\rho}$ such that
\begin{align}
&\|\mathcal{L}^n u \|_{0} \le (1+C_F\bbc){C_{\mu, n}  \mu ^n }\|u\|_{0}\label{LY B^0}\\
&\|\mathcal{L}^n u \|_{\rho} \le \bbc_{\rho,F} \frac{ C_{\mu, n}^{\bar{a}_\rho}\mu^{\bar{b}_\rho n}}{\lambda_-^{\rho n}} \|u\|_{\rho}  +C_{n,\rho}\|u\|_{\rho-1} \label{L-Y Brho} 
\end{align}
where $\bar{a}_\rho=1+2a_\rho\rho^2+\rho$, $\bar b_\rho= (\rho!+1)(\rho^2+2)$ and $\bbc_{\rho,F}:=\bbc^{\rho!(2\rho+1)} (1+C_F\bbc^{\rho!})^{\bar c_\rho}$ where $\bar c_\rho=1+\{\rho+1, \rho(\rho+1)/2\}^+$.
\end{thm}
\noindent We postpone the proof of Theorem \ref{lem MBB} to section  \ref{section Proof of Theorem lem MBB}. First we need to develop several results on the commutators between differential operators and transfer operators. 

\subsection{Differential Operators}\ \\
\label{section diff oper}
For $s, \rho\in \bN$ we denote by $P_{s,\rho}$ a differential operator of order at most $\rho$ defined as a finite linear combination of compositions of at most $\rho$ vector fields, and we write
\begin{equation}
\label{vector fields}
P_{s, \rho}u=\sum_{j=0}^{s}\sum_{{ \alpha \in A\subset \bN^j}} v_{j,\alpha_{1}} \cdots v_{j, \alpha_{j}} u,
\end{equation}
where $A$ is a finite set and for every $i\le j$, $v_{j, \alpha_i}$ are vector fields in $\cC^{\rho+j-s}$, with the convention that $v_{j, \alpha_1}\cdots v_{j, \alpha_j}u=u$ if $j=0.$ We denote by $\Psi^{s,\rho}$ the set of differential operators $P_{s,\rho}$. 
For a function $u\in \cC^r(\bT^2,\bR)$ and a smooth vector field $v$, we denote $\partial_v u (x)=\langle \nabla_x u, v(x) \rangle$.

We start by studying the structure of the commutator between $\cL$ and the differential operators. Next, we will estimate the coefficients of the commutator.
\begin{prop}
\label{comm of vector fields}
Given smooth vector fields $v_1,\cdots,v_s\in \mathcal{C}^{\rho}$, we have
\[
{\partial_{v_s}\cdots \partial_{v_1} \mathcal{L}^n=\mathcal{L}^n\partial_{F^{n^{\ast}}v_s}\cdots \partial_{F^{n^{\ast}}v_1}+\cL^n P_{s-1, \rho},}
\]
where $F^{\ast}v(x):=(D_xF)^{-1}v(F(x))$ is the pullback of $v$ by the map $F$ and $P_{s-1, \rho}\in \Psi^{s-1,\rho}$ whose coefficients may depend on $n$.
\end{prop} 
\begin{proof}  Let us start with $s=1$. Let $v_1\in \cC^{\rho}(\bT^2, \bT^2)$ and define
\begin{equation}
\label{definition of logdet}
\begin{split}
J_n(p)=(\det D_pF^n)^{-1}; \quad
\phi_n(p)=\log |\det D_pF^n|.
\end{split}
\end{equation}
For each $\frh\in \frH^n$ we have
\[
\begin{split}
\langle \nabla \left[ J_n\circ \frh\cdot u\circ \frh\right],v_1\rangle&=\langle J_n\circ \frh (D\frh)^*\nabla u\circ \frh, v_1 \rangle -\langle  (D\frh)^*\nabla(\det DF^n)\circ \frh J_n^2\circ \frh u\circ \frh, v_1 \rangle\\
& =J_n\circ \frh  \langle (D\frh)^*\nabla u\circ \frh, v_1 \rangle -J_n\circ\frh\langle  (D\frh)^*\nabla\phi_n\circ \frh  u\circ \frh, v_1 \rangle.
\end{split}
\]
Then, since $DF^n\circ \frh D\frh=\operatorname{Id}_{\cR_\frh}$, for each $\frh\in \frH^n$ and $x\in \cD_\frh $ \footnote{ Recall that $\cD_\frh,\cR_\frh$ indicate respectively the domain and the range of $\frh$.}
\begin{equation}\label{first derivative of L in frh}
\langle \nabla \left[ J_n\circ \frh\cdot u\circ \frh\right](x),v_1(x)\rangle=J_n\circ \frh (x) \left[  \partial_{F^{n*}v_1} u- \partial_{F^{n*}v_1} \phi_n u \right]\circ \frh(x).
\end{equation}
Observing that 
\begin{equation}\label{tranf oper in frh}
\cL^n u=\sum_{\frh\in \frH^n}u\circ \frh J_n\circ \frh \mathds{1}_{\cR_\frh}\circ \frh,
\end{equation}
it follows 
\begin{equation}
\label{first derivative of L}
\begin{split}
&\langle\nabla_x \mathcal{L}^n u, v_1(x)\rangle = \mathcal{L}^n\left(\partial_{F^{n^{\ast}}v_1} u \right) (x) -\mathcal{L}^n(\partial_{F^{n^{\ast}}v_1} \phi_n\cdot  u)(x),
\end{split}
\end{equation}
which prove the result since the multiplication operator $P_{0,\rho}:=-\partial_{F^{n^{\ast}}v_1} \phi_n\in\Psi^{0,\rho}$.
Next, we argue by induction on $s$:
\begin{equation}
\begin{split}
&\partial_{v_{s+1}}\cdots \partial_{v_1} \mathcal{L}^nu=\partial_{v_{s+1}}\left[\mathcal{L}^n\partial_{F^{n^{\ast}}v_s}\cdots \partial_{F^{n^{\ast}}v_1}u+\cL^n P_{s-1, \rho}u\right]\\
&=\mathcal{L}^n\partial_{F^{n^{\ast}}v_{s+1}}\cdots \partial_{F^{n^{\ast}}v_1}u+
\mathcal{L}^n(\partial_{F^{n^{\ast}}v_{s+1}} \phi_n\cdot \partial_{F^{n^{\ast}}v_{s}}\cdots \partial_{F^{n^{\ast}}v_1}u)\\
&\phantom{=}
+\cL^n \partial_{F^{n^{\ast}}v_{s+1}} P_{s-1, \rho}u+\mathcal{L}^n(\partial_{F^{n^{\ast}}v_{s+1}} \phi_n\cdot P_{s-1, \rho}u),
\end{split}
\end{equation}
which yields the Lemma with
\begin{equation}\label{eq:P-induction}
\begin{split}
P_{s, \rho}=&\partial_{F^{n^{\ast}}v_{s+1}} P_{s-1, \rho}+\partial_{F^{n^{\ast}}v_{s+1}}\phi_n\cdot \left[ \partial_{F^{n^{\ast}}v_{s}}\cdots \partial_{F^{n^{\ast}}v_1}
+ P_{s-1, \rho}\right]\\
&+\partial_{F^{n^{\ast}}v_{s+1}} P_{s-1, \rho}.
\end{split}
\end{equation}
\end{proof}
\noindent In the case $v_j\in \{e_1, e_2\}$ for each $j$, we have the following Corollary as an immediate iterative application of formulae \eqref{first derivative of L in frh} and \eqref{first derivative of L}.
\begin{cor}
For each $t\ge 1$, $n\in\bN$ $\alpha=(\alpha_1,..,\alpha_t)\in \{1,2\}^t$ and $\frh\in \frH^n$,
\begin{equation}\label{def operator Pfrh}
\begin{split}
&\partial^{\alpha}[J_n\circ \frh\cdot u\circ \frh]=J_n\circ \frh\cdot[P^{\alpha}_{n,t}u]\circ \frh,
\end{split}
\end{equation}
in particular
\begin{equation}
\label{eq: LnP=PLn}
\partial^{\alpha} \cL^n u=\cL^n P^{\alpha}_{n, t} u,
\end{equation}
the operators $ P^{\alpha}_{n, t}$ being defined by the following relations, for each $u\in \cC^t$,
\begin{equation}
\begin{cases}
\label{definition of P}
&P^{\alpha}_{n, 0}u=u,\\
& P^{\alpha}_{n, 1}u=A_n^{\alpha_1}u-A_n^{\alpha_1}\phi_n\cdot u,\\
&P^{\alpha}_{n,t}u=A^{\alpha}_{n,1}u
-\sum_{k=1}^{t}A^{\alpha}_{n,k+1}((A^{\alpha_{k}}_n \phi_n)
\cdot  P^{\alpha}_{n, k-1}u) \quad \textit{ for } t\ge 2,\\
\end{cases}
\end{equation}
where $A_n^{\alpha_i}=\partial_{F^{n^*}e_{\alpha_i}}$, $A_{n,k}^{\alpha}:=A_n^{\alpha_t}\cdots A_n^{\alpha_{k}}$, $A_{n, t+1}^{\alpha}=\text{Id}$ and $\phi_n$ is defined in (\ref{definition of logdet}).
\end{cor}
\begin{prop}
\label{operator bound rho norm}
For each $n\in\bN$ let $P^\alpha_{n,t}\in \Psi^{t, t}$ given by \eqref{definition of P}. For any $0\le t<r$, $\psi\in \cC^r(\bT^2, \bC)$ with $\text{supp}\,\psi \subset U=\mathring{U}\subset\bT^2$, $\nu\in \Gamma(\bbc)$ such that $DF^{n-m}\nu'\in\fC_c$, $\varphi\in \cC^{t}(\bT, \bC)$ with $\|\varphi\|_{\cC^t}\le 1$, multi-index $\alpha$, $|\alpha|=t$ and $u\in \mathcal{C}^r(\mathbb{T}^2)$ we have 
\begin{equation}\label{Pnt}
\int_{\bT} \varphi(\tau) P^{\alpha}_{n,t}(\psi u)(\nu(\tau))d\tau \le \tilde C(t,n,m)\|\psi\|_{\cC^t(U)}\|u\|_{t},
\end{equation}
where $\tilde C(t,n{, m}) \le  \Const \Lambda^{\const t n}$.
\end{prop} 
\begin{proof} 
For simplicity we set $\partial_{k}=\partial_{x_{k}}$ for $k\in \{1,2\}$. First of all notice that, if we set $d_{k,i}=\langle (DF^n)^{-1}e_{k}, e_i\rangle$, then $A_n^{\alpha_j} =\sum_{i=1}^2d_{\alpha_j,i}\partial_{x_i}$. Furthermore, by formula (\ref{definition of Lambda}),  $\|d_{j,i}\|_{\cC^t}\le \|(DF^{n})^{-1}\|_{\cC^t}\le \Lambda^{n},$ for each $1\le t\le r.$
We are going to prove (\ref{Pnt}) by induction on $t$. For $t=0$ it is obvious, and for $t=1$ it follows from $\|(DF^n)^{-1}\|\leq \Const \mu^n$. Let us assume it for any $k\le t-1$. By (\ref{definition of P}) the integral in (\ref{Pnt}) splits into\footnote{ Unless differently specified, in the following all the integrals are on $\bT$.}
\begin{equation}
\label{splits of the integral in P}
\begin{split}
&\int \varphi(\tau) P^{\alpha}_{n,t}(\psi u)(\nu(\tau))d\tau\\
&=\int  \varphi \left[A_n^{\alpha_t}\cdots A_n^{\alpha_1}(\psi u)\right]\circ \nu-\int \varphi \sum_{k=1}^{t}\left[A^{\alpha}_{n, k+1}((A^{\alpha_{k}}_n \phi_n)\cdot  P^{\alpha}_{n, k-1}(\psi u))\right]\circ \nu.
\end{split}
\end{equation}
The first integral is equal to
\begin{equation}
\label{eq: integral 1}
\begin{split}
\sum_{\substack{i_1,\cdots,i_t \\ i_l\in\{1,2\}}}\sum_{J,J_0,J_1, .., J_t}\int \varphi\cdot \big( \prod_{j\in J}\partial_{j} u \big) \big( \prod_{j\in J_0}\partial_{j} \psi \big)\big( \prod_{j\in J_1}\partial_{j} d_{\alpha_1,i_1} \big)\cdots \big( \prod_{j\in J_t}\partial_{j} d_{\alpha_t,i_t} \big),
\end{split}
\end{equation}
where the second sum is made over all the partitions $J,J_0,J_1,..,J_t$ of $\{1,..,t\}$ such that $J_j\subset \{j+1,..,t\}, j\ge 1$.\footnote{ We use the conventions $\prod_{j\in\emptyset} \partial_jA=A$ and $\sharp B$ denote the cardinality of the set $B$.} Note that 
$\left\|( \prod_{k=1}^t \Pi_{j\in J_k}\partial_{j})d_{\alpha_k,i_k}\right\|_{\cC^{\sharp J}_{\nu}}\le \Lambda^{n\{\sharp J+\sum_{k=1}^{t} \sharp J_k}\}$ and
$\left\| \Pi_{j\in J_0}\partial_{j}\psi  \right\|_{\cC^{\sharp J}_{ \nu}}\lesssim \|\psi\|_{\cC^{\sharp J+\sharp J_0}}\leq \|\psi\|_{\cC^t}$.
Consequently, from (\ref{eq: integral 1}) and the definition \eqref{geom norm}, we have
\begin{equation}
\label{eq: estimate An}
\left|\int  \varphi(\tau) A_{n,1}^{ \alpha}(\psi u)(\nu(\tau))d\tau\right|\le \Const \Lambda^{c_\sharp t n }\|\psi\|_{\cC^t}\|u\|_t.
\end{equation}
To bound the second integral in \eqref{splits of the integral in P} we first note that
\begin{equation}
\label{eq: A n alphat of phin}
\begin{split}
A_n^{\alpha_{k}} \phi_n(x) &=\sum_{j=0}^{n-1} \langle (D_{x}F^{j})^*\nabla \phi_1\circ F^j(x), (D_{x}F^n)^{-1} e_{\alpha_{k}} \rangle\\
&=\sum_{j=0}^{n-1} \langle \nabla \phi_1, (D F^{n-j})^{-1} e_{\alpha_{k}} \rangle\circ F^j(x),
\end{split}
\end{equation}
thus \eqref{formula 1 for norm Crho} implies
\begin{equation}\label{eq:ALambda}
\|A_n^{\alpha_{k}} \phi_n\|_{\cC^l}\le \Const\sum_{j=0}^{n-1}\|(DF^{n-j})^{-1}\|_{\cC^l}\Lambda^{nl}\le \Const\sum_{j=0}^{n-1}\Lambda^{c_\sharp (n-j+l)}\le \Const\Lambda^{c_\sharp n}.
\end{equation}
We can then use \eqref{eq: estimate An} to estimate 
\begin{equation}
\label{eq:P-alpha-last estimate}
\begin{split}
\left|\int \varphi A^{\alpha}_{n, k+1}((A^{\alpha_{k}}_n \phi_n)\cdot  P^{\alpha}_{n, k-1}(\psi u))\right|\leq &\Const \Lambda^{\const n}\|A^{\alpha_{k}}_n \phi_n\|_{\cC^{t-k-1}_{ \nu}}\|P^{\alpha}_{n, k-1}(\psi u))\|_{t-k-1}\\
&\leq \Const \Lambda^{\const n}\|P^{\alpha}_{n, k-1}(\psi u))\|_{t-k-1}.
\end{split}
\end{equation}
To bound the last term we take $\phi\in\cC^{t-k-1},\|\phi\|_{\cC^{t-k-1}}=1, \gamma\in \Gamma$, and we consider
\[
\int \phi \partial^{t-k-1}[P^{\alpha}_{n, k-1}(\psi u)]\circ \gamma.
\]
We can then split the integral as in \eqref{splits of the integral in P}, although this time $\alpha=(\alpha_1,\cdots,\alpha_{k-1}).$ For the first term we take $t-k-1$ derivatives in \eqref{eq: integral 1} and, arguing as we did to prove \eqref{eq: estimate An}, we have
\begin{equation*}
\left|\int  \phi(\tau) \partial^{t-k-1}A_{n,1}^{ \alpha}(\psi u)(\gamma(\tau))d\tau\right|\le \Const \Lambda^{c_\sharp t n }\|\psi\|_{\cC^t}\|u\|_t.
\end{equation*}
The second term is estimated in the same way, using the inductive assumption.  
The statement of the Proposition then follows using this in \eqref{eq:P-alpha-last estimate}. \end{proof}

\subsection{Differential operators: the \texorpdfstring{SVPH$^{\,\sharp}$}{SPVHs} case} 
When treating SVPH$^{\,\sharp}$ systems we will need an improved estimate of the constant $\tilde C(t,n{, m})$ appearing in Proposition \ref{operator bound rho norm} for low derivatives.

\begin{prop}
\label{operator bound rho norm_sharp}
If $F$ is SVPH$^{\,\sharp}$, then the constants $\tilde C(t,n{, m})$ of Proposition \ref{operator bound rho norm} satisfy
\begin{equation}
\label{eq: constant in the s=2 case}
\tilde C(t,n{, m}) \le 
\begin{cases}
1 \qquad &t=0\\
\Const\mu^n&t=1\\ 
C_{\mu,n}^3\mu^{4n}\sup\limits_{\zeta\in\supp(\varphi)\cap\nu(\bT)}(1+C_F\lambda^+_n(\zeta))\{\lambda^+_m \circ F^{n-m}(\zeta)+\bbc\}\qquad &t=2.
\end{cases}
\end{equation}
\end{prop} 
\begin{proof}
We use the same notation of Proposition \ref{operator bound rho norm} and of its proof, where the cases $t=0$ and $t=1$ have already been handled. The special case $t=2$ corresponds to $\alpha=(\alpha_1, \alpha_2)$,\footnote{ We use the following notation: $\Phi_1$ equals the third line from the bottom, the other $\Phi_i$ are, ordered, the terms in the second line from the bottom.}
\[
\begin{split}
P^{\alpha}_{n,2}(\psi u)&=A^{\alpha}_{n,1}(\psi u)-A^{\alpha_2}_n(\phi_n)A^{\alpha_1}_n(\psi u)-A^{\alpha}_{n,1}(\phi_n)\psi u-A^{\alpha_1}_n\phi_n A^{\alpha_2}_n (\psi u)\\
&\phantom{=}
-A^{ \alpha_1}_n\phi_n A^{ \alpha_2}_n\phi_n \cdot \psi u \\
&=\left\{ A^{\alpha}_{n,1}\psi -A^{ \alpha_2}_n\phi_nA^{ \alpha_1}_n\psi-(A^{ \alpha_1}_n\phi_nA^{ \alpha_2}_n\phi_n+A^{\alpha}_{n,1}\phi_n )\psi \right\} u\\
&\phantom{=}
-\left\{A^{\alpha_2}_n\psi+\psi A^{ \alpha_2}_n\phi_n  \right\}A^{ \alpha_1}_n u-\left\{A^{\alpha_1}_n\psi+\psi A^{ \alpha_1}_n\phi_n\right\}A^{ \alpha_2}_n u +\psi A^{\alpha}_{n,1}u\\
&=:\Phi_1 + \Phi_2+\Phi_3 +\Phi_4.
\end{split}
\] 
We then want to integrate the above terms along the curve $\nu$ against a test function $\varphi\in \cC^2$. Recalling that the coefficients of the differential operators $A^{\alpha_j}_n$ have $\cC^r$ norm bounded by $\|(DF^{n})^{-1}\|_{\cC^r}$, we thus have
\[
\begin{split}
\int \varphi \Phi_1\circ\nu \le C_\sharp\max_{i,j}\{\| &A^{\alpha}_{n,1}\psi\|_{\cC^0_\nu}, \|\psi A^{\alpha}_{n,1}\phi_n\|_{\cC^0_\nu},\\
& (1+\|A^{\alpha_i}_{n}\phi_n\|_{\cC^0_\nu})^2\|\psi\|_{\cC^0},
\|A^{ \alpha_i}_n\phi_nA^{ \alpha_j}_n\psi\|_{\cC^0_\nu} \} \|u\|_0.
\end{split}
\]
The bounds for  $\Phi_2$ and $\Phi_3$ are similar:
\[
\left|\int \varphi\Phi_2 \circ\nu\right|+\left|\int \varphi\Phi_3 \circ\nu\right|\le C_\sharp\|(DF^{n})^{-1}\|_{\cC^1_\nu} \max_i\{\|A^{\alpha_i}_n\psi\|_{\cC^1_\nu},\|A^{ \alpha_i}_n\phi_n \|_{\cC^1_\nu} \|\psi\|_{\cC^1} \}\|u\|_1.
\]
Next, for any two vector $v,w\in\bR^2$, $i,j\in\{1,2\}$ and $x=(x_1,x_2)\in \bT^2$,\footnote{ Here we denote $\left[ (DF)^{-1}w \right]_k :=\langle  (DF)^{-1}w, e_k \rangle$.}
\[
\begin{split}
&\partial_{F^*v}(\partial_{F^{*}w}u)=\partial_{F^*v}(\langle \nabla u, (DF)^{-1}w \rangle)=\langle \nabla (\langle \nabla u, (DF)^{-1}w \rangle), (DF)^{-1}v  \rangle\\
&= \sum_{j, k} \partial^2_{x_jx_k} u\cdot  \left[ (DF)^{-1}v  \right]_k  [(DF)^{-1}w ]_j+\sum_{j,k} \partial_{x_k}u \, \partial_{x_j} [(DF)^{-1}w]_{k} \cdot [(DF)^{-1}v]_{j}.
\end{split}
\]
Thus, 
\[
\int \varphi \Phi_4 \circ \nu\le C_\sharp\{\mu^{2n}\|\psi\|_{\cC^2},\mu^n \|(DF^n)^{-1}\|_{\cC^2_\nu}\|\psi\|_{\cC^1}, \|(DF^n)^{-1}\|^2_{\cC^1_\nu}\|\psi\|_{\cC^1}\}^+\|u\|_2.
\]
It follows by the property of the $\cC^r$ norm and \eqref{relation norm rho rho' and L1}, that
\[
\begin{split}
\int& \varphi P^{\alpha}_{2,n}(\psi u)\circ \nu\le C_\sharp\{1,\|A^{\alpha}_{n,1}\phi_n\|_{\cC^0_\nu},
\max_{i\in\{1,2\}}\|A^{\alpha_i}_{n}\phi_n\|_{\cC^0_\nu}^2, \\
&\max_{i\in\{1,2\}}\|(DF^n)^{-1}\|_{\cC^1_\nu} \|A^{ \alpha_i}_n\phi_n \|_{\cC^1_\nu},
\mu^n\|(DF^n)^{-1}\|_{\cC^2_\nu}, \|(DF^n)^{-1}\|^2_{\cC^1_\nu}\}^+\|\psi\|_{\cC^2_\nu}\|u\|_2.
\end{split}
\]
We have thus proved that 
\[
\begin{split}
\tilde C(2,n)=\Const \big\{&1,\|A^{\alpha}_{n,1}\phi_n\|_{\cC^0_\nu},
\max_{i\in\{1,2\}}\|A^{\alpha_i}_{n}\phi_n\|_{\cC^0_\nu}^2,\max_{i\in\{1,2\}}\|(DF^n)^{-1}\|_{\cC^1_\nu} \|A^{ \alpha_i}_n\phi_n \|_{\cC^1_\nu},\\
&\mu^n\|(DF^n)^{-1}\|_{\cC^2_\nu}, \|(DF^n)^{-1}\|^2_{\cC^1_\nu}\big\}^+.
\end{split}
\]
To conclude we need a bound of the above quantity. It is enough
to find estimates for $ \|(A^{\alpha}_{n,1}\phi_n)\|_{\cC^0_\nu}$ and $\|A^{ \alpha_1}_n\phi_n \|_{\cC^1_\nu}\cdot \|(DF^{n})^{-1}\|_{\cC^1_\nu}$, the other quantities being already estimated in Proposition \ref{prop on DF^-1S}. First, we can use formulae \eqref{eq: A n alphat of phin} and \eqref{norm of DF^N^-1},
\begin{equation}
\label{A1phin}
|\partial_{F^{n^{\ast}}e_\ell}\phi_n(x)|\le  \Const\sum_{j=0}^{n-1}\mu^{n-j}\le C_{\mu, n}\mu^n.
\end{equation}
In particular $\|A^{ \alpha_1}_n\phi_n \|_{\cC^0_\nu} \le  C_{\mu, n}\mu^n.$
Next, we take another derivative of \eqref{eq: A n alphat of phin} in the direction of $F^{n^*}e_q$ and, setting $g_{\ell,n,j}(x)=\langle \nabla \phi_1, (DF^{n-j})^{-1} e_\ell \rangle (x),$ we have
\begin{equation}
\label{second derivative of Dn}
\begin{split}
\partial_{F^{n^{\ast}}e_q}(\partial_{F^{n^{\ast}}e_\ell}\phi_n(x))&=\sum_{j=0}^{n-1} \langle \nabla (g_{\ell,n,j}\circ F^j(x))  ,(D_xF^n)^{-1}e_q \rangle\\
&=\sum_{j=0}^{n-1} \langle (D_xF^j)^{*}\nabla g_{\ell,n,j}\circ F^j(x)  ,(D_xF^n)^{-1}e_q \rangle\\
&=\sum_{j=0}^{n-1} \langle \nabla g_{\ell,n,j}\circ F^j(x)  ,(D_xF^{n-j})^{-1}e_q \rangle.
\end{split}
\end{equation}
By a direct computation we see that, recalling \eqref{eq:dtDF-dt^2DF} and \eqref{eq:ve-special-inv}, 
\[
\|\nabla g_{\ell,n,j}\|\le \Const \max_{i}\{\|\partial_{x_i}(DF^{n-j})^{-1}\|\}\le C_{\mu,n-j} (1+C_F\mu^{n-j}\lambda^+_{n-j})\mu^{n-j}.
\]
We use this in (\ref{second derivative of Dn}) obtaining
\[
\|A^{\alpha}_{n,1}\phi_n\|_{\cC^0_\nu}\leq \sup_{q,\ell}\|\partial_{F^{n^{\ast}}e_q}(\partial_{F^{n^{\ast}}e_\ell}\phi_n)\|_{\cC^0_\nu}\le C_{\mu,n}^2 (1+C_F\mu^{n}\lambda^+_{n})\mu^{2n}. 
\] 
Finally, we use \eqref{eq: A n alphat of phin} to compute
\[
\begin{split}
\left|\frac{d}{dt} (A^{ \alpha_1}_n\phi_n\circ \nu) \right|&\le \sum_{j=0}^{n-1}|\langle D(\nabla \phi_1)\circ F^j (D_\nu F^j)\nu',(D F^{n-j})^{-1}\circ F^j 
\circ\nu e_{\alpha_1}\rangle|\\
&+|\langle \nabla \phi_1\circ (F^j \nu), [(D F^{n-j})^{-1}\circ (F^j \nu)]'e_{\alpha_1} \rangle|\\
&\le C_{\mu, n}\lambda_m^+\circ (F^{n-m}\nu)\mu^n+C_{\mu,n}\mu^n(1+C_F\lambda_m^+\circ (F^{n-m}\nu))\\
&\leq C_{\mu, n}\lambda_m^+\circ (F^{n-m}\nu)\mu^n,
\end{split}
\]
where we have used \eqref{norm of DF^N^-1} and that, recalling \eqref{def:varsigma-chiu}, $\varsigma_{n,m}\leq C_{\mu,n}(1+C_F\lambda^+_m)$. We thus obtain
\[\|A^{ \alpha_1}_n\phi_n \|_{\cC^1_\nu}\cdot \|(DF^{n})^{-1}\|_{\cC^1_\nu}\le   C_{\mu,n}\mu^{3n}\sup\limits_{\zeta\in\supp(\varphi)\cap\nu(\bT)}[(1+C_F\lambda^+_m)\lambda^+_m]\circ F^{n-m}(\zeta) .
\]
The proposition follows collecting all the above estimates and recalling again \eqref{norm of DF^N^-1} for the estimate of $\|(DF^n)^{-1}\|_{\cC^1_\nu}, \|(DF^n)^{-1}\|_{\cC^2_\nu}$ and noticing that $\varsigma_{n,m}\leq C_{\mu,n}(1+C_F\lambda^+_m)$. \end{proof}

\subsection{Conclusion of the first Lasota-Yorke inequality}\label{section Proof of Theorem lem MBB}\ \\
This section is devoted first to the proof of Theorem \ref{lem MBB} and then is concluded by a useful Corollary.
\begin{proof}[\bfseries Proof of Theorem \ref{lem MBB}]
Given Lemma \ref{lemma on det^-1}, the proof of Theorem \ref{lem MBB} is almost exactly the same as in \cite{GoLi}, hence we provide the full proof for $\rho=0,1$ and give a sketched proof for the case $1<\rho\le r-1$.\\
Let us prove (\ref{LY B^0}) first, since it is an immediate consequence of Lemma \ref{lemma on det^-1} and Definition \ref{geom norm} in the case $\rho=0$. Indeed, by changing the variables and recalling the notation of Section \ref{subsec admiss curv} and Lemma \ref{lemma on det^-1}, we have,
\begin{equation*}
\begin{split}
&\int_{\mathbb{T}} \phi(t) \mathcal{L}^n u(\gamma(t)) dt=\sum_{\nu\in F^{-n}\gamma} \int_{\mathbb{T}}|\det D_{\nu(t)}F^{n}|^{-1}\cdot (u\circ \nu)(t)\cdot \phi(t)dt \\ 
&=\sum_{\nu\in F^{-n}\gamma} \int_{\mathbb{T}}|\det D_{\hat{\nu}}F^{n}|^{-1}\cdot (u\circ \hat{\nu})(t)\cdot (\phi\circ h_n)(t)h_n'(t)dt\\
&\le \sum_{\nu\in F^{-n}\gamma} \left\| h'_n \left|\det D_{\hat{\nu}}F^{n}\right|^{-1}\right\|_{\mathcal{C}^0} \|u\|_{0}\leq (1+C_F\bbc)C_{\mu, n}\|\phi\|_{\cC^0(\bT)}\|u\|_{0}.
\end{split}
\end{equation*}
Let us now proceed with the case $\rho=1$, from which we deduce the general case by similar computations. We must bound the quantity, 
for all $\|\phi\|_{\cC^1}\leq 1$,
\begin{equation*}
\int_{\mathbb{T}}\phi(t) (\partial_v\mathcal{L}^n u)(\gamma(t))dt=\int_{\mathbb{T}}\phi(t) \langle \nabla(\mathcal{L}^n u)(\gamma(t)), v\rangle dt,
\end{equation*}
where now $\phi \in \mathcal{C}^{1}(\mathbb{T})$ with norm one and $v$ is a unitary $\mathcal{C}^{r}$ vector field. From Proposition \ref{comm of vector fields} the above quantity is equal to 
\begin{equation}
\label{first bound }
\sum_{\nu\in F^{-n}\gamma}\int_\bT \frac{1}{|\det D_{\nu}F^{n}|}\phi\cdot \partial_{F^{n^*}v}u(\nu)+\int_\bT \cL^n(Q_{n,0}u)\phi,
\end{equation}
where $Q_{n,0}$ is an operator of multiplication by a $\cC^1$ function.\\
By Proposition \ref{operator bound rho norm} applied with $\psi=1$, plus the result for $\rho=0$, the last term is then bounded by $C_n\|u\|_{1}$. 

In order to bound the first term of (\ref{first bound }) we need an analogous of Lemma 6.5 in \cite{GoLi}. The idea is to decompose the vector field $v$ into a vector tangent to the central curve $\gamma$ and a vector field approximately in the unstable direction so that the first one can be integrated by parts, while for the other we can exploit the expansion. The proof of the following Lemma follows that of the aforementioned paper, since the key point is the splitting of the tangent space in two directions, one of which is expanding. Once more, however, the presence of the central direction creates difficulties. We give the proof adapted to our case in Appendix \ref{appendix proof of lem on decomposition}.
\begin{lem}
\label{lem on decomposition}
Let $\bar{n}$ be the integer provided by Lemma \ref{lem stable vertic curves}. For every $n>\bar{n}$, $\gamma\in \Gamma_\rho(\bbc)$, $\rho\leq r$, $\nu\in F^{-n}\gamma$, and any vector field $v\in \mathcal{C}^\rho$, with $\|v\|_{\cC^\rho}\le 1$, defined in some neighborhood $M(\gamma)$ of $\gamma$, there exist a neighborhood $M'(\gamma)$ of $\gamma$ and a decomposition
\begin{equation}\label{eq:v=vu+vc}
v= \hat{v}^c+ \hat{v}^u,
\end{equation}
where $ \hat{v}^c$ and $\hat{v}^u$ are $\cC^{\rho}( M'(\gamma))$ vector fields such that, setting $F^n(N(\nu))=M'(\gamma)$,\footnote{ The constants $a_\rho$ are defined in Lemma \ref{lem stable vertic curves}.}
\begin{itemize}
\item $\hat v^c(\gamma(t))=g(t)\gamma'(t),$ where $g\in\cC^r$ and $\|g\|_{\cC^\rho}\le\Const \bbc^{\rho!}C_{\mu,n}^{a_\rho}\mu^{\rho!n}$,
\item $ \|{(F^{n})^{\ast} \hat{v}^u}\|_{\mathcal{C}^{\rho}(N(\nu))}\le \lambda^{-n}_- \bbc^{\rho!\rho}
 C_{\mu, n}^{\rho a_{\rho}}\mu^{\rho\rho! n}$,
\item $\|{(F^{n})^{\ast} \hat{v}^c}\|_{\mathcal{C}^{\rho} (N(\nu))}\le \bbc^{\rho!\rho}C_{\mu, n }^{2\rho a_{\rho}+1} \mu^{[(\rho+1) (\rho!+1)+\rho!]n}$,
\item $\|\hat{v}^u\|_{\cC^\rho(M'(\gamma))}+\|\hat{v}^c\|_{\cC^\rho(M'(\gamma))}\le C_n$.
\end{itemize}
\end{lem}

\noindent By the above decomposition, the addends in the first term in  (\ref{first bound }) become
\begin{equation}
\label{split of integrals}
\begin{split}
\int \frac{1}{|\det D_{\nu}F^{n}|}\phi\cdot \partial_{F^{n^*}\hat{v}^c}u(\nu) +\int \frac{1}{|\det D_{\nu}F^{n}|}\phi\cdot \partial_{F^{n^*}{\hat{v}}_u}u(\nu).
\end{split}
\end{equation}
Since $\gamma(t)=F^{n}\nu(t)$ we have $g(t)D_{\nu(t)}F^{n}\cdot \nu'(t)=\hat{v}^c(F^{n}\nu(t))$, hence:
\begin{equation*}
g(t)\nu'(t)=(D_{\nu(t)}F^{n})^{-1}\cdot \hat{v}^c(F^{n}\nu(t))=F^{n^*}\hat{v}^c(\nu(t)),
\end{equation*} 
hence $|g'|\leq\|F^{n^*}\hat{v}^c\|_{\cC^1}$.
Accordingly,
\begin{equation*}
\begin{split}
&\int \frac{\phi(t)}{|\det D_{\nu(t)}F^{n}|} \partial_{F^{n^*}\hat{v}^c}u(\nu(t))dt=\int \frac{g(t)\phi(t)}{|\det D_{\nu(t)}F^{n}|} \frac{d}{dt} (u(\nu(t)))dt  \\
&=\int \frac{g(t)\phi(t)}{|\det D_{\hat\nu\circ h_n^{-1}(t)}F^{n}|} \left[\frac{d}{dt} (u\circ \hat \nu)\right]\circ h_n^{-1}(t)\left[ h_n^{-1}(t)\right]'dt\\
&=\int \frac{[g\phi]\circ h_n(t)}{|\det D_{\hat\nu(t)}F^{n}|} (u\circ \hat\nu)'(t)dt=-\int \frac{d}{dt}\left( \frac{[g\phi]\circ h_n(t)}{|\det D_{\hat \nu(t)}F^{n}|}  \right) u(\hat\nu(t))dt\\
&\le \left\|\frac{[g\phi]\circ h_n}{\det D_{\hat \nu}F^{n}}\right\|_{\mathcal{C}^1}\|u\|_{0}.
\end{split}
\end{equation*}
Summing over $\nu\in F^{-n}\gamma$ and using Lemma \ref{lemma on det^-1}
 we obtain \begin{equation}
\label{first term}
\sum_{\nu\in F^{-n}\gamma} \int \phi(t) \partial_{F^{n^*}\hat{v}^c}u(\nu(t))dt\lesssim C_{\mu,n}^5\mu^{7n}(1+C_F\bbc)^2 \|u\|_0 .
\end{equation}
The second term of \eqref{split of integrals} is
\begin{equation}
\label{second term}
\begin{split}
\int  \frac{\phi}{|\det D_{\nu}F^{n}|} \partial_{F^{n^*}{\hat{v}^u}}u(\nu)&=\int \frac{\phi}{|\det D_{\nu}F^{n}|}\langle \nabla u, F^{n^*}{\hat{v}^u} \rangle\circ \nu \\
&\le \Const  \left\| \frac{\phi\circ h_n h_n'}{\det D_{\hat{\nu}}F^{n}}\right\|_{\mathcal{C}^1}  \|F^{n^*}\hat{v}^u\circ \hat \nu\|_{\mathcal{C}^{1}}\|u\|_{1} \\
&\le  \Const  \|h_n\|_{\cC^1}\left\| \frac{h'_{n}}{\det D_{\hat{\nu}}F^{n}}\right\|_{\mathcal{C}^1}  \bbc\lambda_-^{-n}C_{\mu, n}\mu^{n}\|u\|_{1},
\end{split}
\end{equation}
where we made the usual change of variables $t=h_n(s)$ and used Lemma \ref{lem on decomposition}. Finally, using (\ref{first term}) and (\ref{second term}) in (\ref{split of integrals}), and recalling \eqref{Crho norm of h}, we have by Lemma \ref{lemma on det^-1}, with $\rho=1$,
\begin{equation}
\label{central vector in case rho=1}
\begin{split}
\|\mathcal{L}^n u\|_1\le   \bbc \lambda_-^{-n} C_{\mu, n}^3\mu ^{2n}(1+C_F\bbc) \|u\|_{1}+  C_n\|u\|_{{0}}.
\end{split}
\end{equation}
For the general case $1 \le  \rho \le r-1$ one has to control the term $\int_{\mathbb{T}}\phi(t) \partial_{v_s}\cdots \partial_{v_1}\mathcal{L}^n u(\nu(t))dt$, for vector fields $v_j\in\mathcal{C}^{\rho},j=1,...,s$ and $s\le \rho$. Using again Propositions \ref{comm of vector fields} and \ref{operator bound rho norm}, the latter is bounded by
\begin{equation}
\label{integral for generic rho}
\sum_{\nu\in F^{-n}\gamma}\int \frac{1}{|\det D_{\nu}F^{n}|}\phi\cdot \partial_{F^{n^*}v_s\cdots F^{n^*}v_{1} }u(\nu)+C_{n,\rho}\|u\|_{{\rho-1}}.
\end{equation}
Now the strategy is exactly the same as before. We use Lemma \ref{lem on decomposition} to decompose each $v_j=\hat{v}_j^{u}+\hat{v}_j^{c}$. We take $\sigma \in\{u, c\}^{s}$, $k=\#\left\{i | \sigma_{i}=c\right\} $ and let  $\pi$ be a permutation of $\{1, \ldots, s\}$ such that $\pi\{1, \ldots, k\}=\left\{i | \sigma_{i}=c\right\}$. Using integration by parts, we can write the integral in (\ref{integral for generic rho}) as
\[
\begin{split}
&\int\frac{\phi}{\det D_\nu F^n} \partial_{F^{n^*}v_{s}} \ldots \partial_{F^{n^*}v_{1}} u(\nu)=\sum_{\sigma \in\{u,c\}^{s}}\int\frac{\phi}{\det D_\nu F^n} \left(\prod_{s=1}^1\partial_{F^{n^*}\hat v^{\sigma_i}_{i}}\right) u(\nu)\\
&=\sum_{\sigma \in\{u,c\}^{s}}\int \frac{\phi}{\det D_{\nu}F^s} \prod_{i=s}^{k} \partial_{F^{n^*}\hat v^c_{\pi(i)}} \prod_{i=k+1}^{1} \partial_{F^{n^*}\hat v^u_{\pi(i)} } u(\nu)+C_{n,\rho}\|u\|_{\rho -1}\\
&=\sum_{\sigma \in\{u,c\}^{s}}(-1)^k\int \prod_{i=k+1}^{s} \partial_{F^{n^*}\hat v^u_{\pi(i)}} u(\nu) \prod_{i=k}^{1} \partial_{F^{n^*}\hat v^c_{\pi(i)} }\left(\frac{\phi}{\det D_{\nu(t)}F^n}\right) +C_{n,\rho}\|u\|_{\rho-1}.
\end{split}
\]
By Lemma \ref{lem on decomposition}, $\| {F^{n^*}\hat v^c_{\pi(i)} }\|_{\mathcal{C}^{\rho}(\nu)}\le  \bbc^{\rho!\rho} C_{\mu, n }^{2\rho a_{\rho}+1}\mu^{(\rho+1) (\rho!+1)+\rho!}$ while \[\|\prod_{i=k+1}^{s} {F^{n^*}\hat v^u_{\pi(i)}} \|_{\cC^{\rho}(\nu)}\le \bbc^{\rho!\rho} C\lambda_-^{-(s-k)n}
 (C_{\mu, n}^{\rho a_{\rho}}\mu^{\rho\rho! n})^{s-k}.\] It follows by Lemma \ref{lemma on det^-1}, equation \eqref{Crho norm of h} and the fact that $\|\phi\|_{\cC^r}\le 1$, that\footnote{ Notice that the coefficient in front of the strong norm is obtained in the case $s=\rho$ and $k=0$, while all the other terms are bounded again by $C_{n,\rho}\|u\|_{\rho-1}$.}
\[
\begin{split}
\sum_{\nu\in F^{-n}\gamma}\int &\frac{\phi}{\det DF^n} \partial_{F^{n^*}v_{1}} \ldots \partial_{F^{n^*}v_{\rho}} u\circ \nu\\
&\le  \bbc^{2\rho!\rho} \lambda_-^{-\rho n}C_{\mu, n}^{\rho^2 a_{\rho}}\mu^{\rho^2\rho! n}\|h_n\|_{\cC^\rho} \sum_{\nu\in F^{-n}\gamma} \left\| \frac{h'_{n}}{\det D_{\hat \nu}F^{n}}\right\|_{\mathcal{C}^\rho}\|u\|_\rho+C_{n,\rho}\|u\|_{\rho-1}\\
&\le  \bbc^{\rho!(2\rho+1)}(1+C_F\bbc^{\rho!})^{\bar c_\rho}  \lambda_-^{-\rho n} C_{\mu, n}^{ 2\rho^2 a_\rho+\tilde a_\rho}\mu^{ [(\rho!+1)(\rho^2+1) +\tilde{b}_\rho]n}\|u\|_{\rho}+C_{n,\rho}\|u\|_{\rho-1},
\end{split}
\]
hence, recalling the definitions of $\tilde a, \tilde b$ in Lemma \eqref{lemma on det^-1}, we have (\ref{L-Y Brho}) with $\bar{a}_\rho=1+2a_\rho\rho^2+\rho$, $\bar b_\rho= (\rho!+1)(\rho^2+2)$ and $\bar c_\rho=1+\{\rho+1, \rho(\rho+1)/2\}^+$
\end{proof}

The last result of this section is the following Corollary of Theorem \ref{lem MBB}.
\begin{cor} 
\label{Corollario 1}
Let us assume that, for every integer $0\le \rho\le r-1$,\footnote{ Note that the following is implied by {\bf (H3)}.} 
\begin{equation}
\label{mu lambda^rho<lambda^ -1/2}
\lambda_{-}^{-1}\mu^{6(\rho+1)!}<1.
\end{equation}
Let $\delta_\ast \in (\lambda_{-}^{-1}\mu^{6(\rho+1)!}, 1)$. Then, for each $n\in \mathbb{N}$,
\begin{equation}
\label{L-Y B2} 
\|\mathcal{L}^n u \|_{\rho} \le  C_\sharp\delta_\ast^{n} \|u\|_{\rho}+ C_{\mu, n}\mu^{n}\|u\|_{0}.
\end{equation}
\end{cor}
\begin{proof} The strategy is to take the integer $\bar n$ large enough such that the coefficient of the strong norm in \eqref{L-Y Brho} is smaller than $1$, and then to iterate the estimates. First, recalling \eqref{eq:chose-varpi}, we can estimate\footnote{ As  previously mentioned, all the estimates for the exponent of $\mu$ are far from being optimal.} $\bbc_{\rho,F}\le \Const \mu^{\frac{3}2\rho![\rho^2+5\rho+6]\bar n}$. Hence, a direct computation yields 
\[
\bbc_{\rho,F}\mu^{\bar b_\rho \bar n}\le \Const \mu^{[6 (\rho+1)!\rho+13]\bar n}.
\]
Hence,  by \eqref{mu lambda^rho<lambda^ -1/2}, we can choose $\delta\in (\lambda_{-}^{-1}\mu^{6(\rho+1)!}, 1)$. 
Accordingly, since  $C_{\mu, n}$ grows only linearly in $n$, we can choose $\bar{n}\in \bN$ large enough such that $\bbc_{\rho,F}C_{\mu, \bar n}^{\bar{a}_\rho}\mu^{\bar{b}_\rho \bar{n}}\lambda_-^{-{\rho \bar{n}}}<\delta^{\rho \bar{n}}$ for every $\rho\in [1, r-1]$. \\
Let us proceed by induction on $\rho$. For $\rho=0$ the statement is implied by (\ref{LY B^0}). Let us assume it true for each integer smaller then or equal to $\rho-1$. By Theorem \ref{lem MBB} and (\ref{mu lambda^rho<lambda^ -1/2}), we have
\begin{equation}
\label{LY rho/rho-1}
\|\mathcal{L}^{\bar{n}} u \|_{\rho} \le C_\sharp  \delta^{\rho\bar{ n}}\|u\|_\rho  +C_{\bar{n}}\|u\|_{\rho-1}.
\end{equation}
 For every $m\in \bN$ we write $m=\bar{n}q+r$, $0\le r<\bar{n}$, and iterate (\ref{LY rho/rho-1}) to have
\begin{equation*}
\begin{split}
&\|\mathcal{L}^{m} u \|_{\rho} =\|\cL^{\bar{n}}(\cL^{m-\bar{n}}u)\|_{\rho}\le  C_\sharp\delta^{\rho\bar{n}}\|\cL^{m-\bar{n}} u\|_{\rho}+C_{\bar{n}}\|\cL^{m-\bar{n}} u\|_{\rho-1}\le \cdots \\
&\cdots \le \Const\delta^{q \rho  \bar{n}}\|\cL^r u\|_{\rho}+ \Const\sum_{k=0}^{q-1} \delta^{k \rho  \bar{n}}\|\cL^{m-(k+1)\bar{n}}u\|_{\rho-1}\le  \Const\delta^{\rho m}\| u\|_{\rho}+ C_{\mu, m} \mu^m \|u\|_{\rho-1}, 
\end{split}
\end{equation*}
where we used $\|\cL^{m-(k+1)\bar{n}}u\|_{\rho-1}\le C_{\mu, m}\mu^{m-(k+1)}\|u\|_{\rho-1}$ by the inductive assumption. We iterate the last inequality $\rho$ times and obtain
\[
\begin{split}
\|\mathcal{L}^{\rho m} u \|_{\rho} &\le C_{\mu, m}^{\rho-1}(\mu^{\rho-1} \delta^{\rho})^m\|u\|_\rho+ C_{\mu, m}^\rho\mu^{\rho m}\|u\|_{0}\\
&\le C_{\mu, m}^{\rho}(\mu^{\rho} \delta^{\rho})^m\|u\|_\rho+ C_{\mu, m}^\rho\mu^{\rho m}\|u\|_{0}.
\end{split} 
\]
We then consider the above inequality for $m$ such that $\rho m={\bar{n}}$, so that $C_{\mu, m}^\rho(\mu^\rho \delta^\rho)^m<\tilde\delta^{\bar{n}},$ for some $\tilde \delta \in (\delta, 1)$. Hence,
\begin{equation}
\label{ly with delta mu}
\|\mathcal{L}^{\bar{n}} u \|_{\rho}\le \tilde\delta^{\bar{n}}\|u\|_\rho+C_{\mu, \bar{n}}\mu^{\bar{n}}\|u\|_0.
\end{equation}
Finally, we iterate once again (\ref{ly with delta mu}) and we obtain the result for some $\delta_\ast\in (\tilde\delta, 1)$.
\end{proof} 
\begin{oss}\label{eq:noncompact}
Although Corollary \ref{Corollario 1} provides a Lasota-Yorke inequality, a fundamental ingredient is missing: the embedding of $\mathcal{B}_{\rho}$ in $\mathcal{B}_{0}$ is not compact.  
\end{oss}
\section{A second Lasota-Yorke inequality: preliminaries}\label{sec:5}
The main result of the following two sections is the second step towards the proof of Theorem \ref{Main Theorem 1}, namely a Lasota-Yorke type inequality between the Hilbert space $\mathcal{H}^s$ and $\mathcal{B}_{\rho}$.\footnote{ See Appendix \ref{appendix space Hs } for definitions and the needed properties of $\cH^s(\bT^{2}).$}  We will see in Corollary \ref{cor conseq  of thm for SVPH} that this solves the compactness problem mentioned in Remark \ref{eq:noncompact}. First we state some results on the $\cH^s$-norm of the transfer operator.
\subsection{\texorpdfstring{$\cH^s$}{Lg}-norm of \texorpdfstring{$\cL$}{Lg}}\ \\
\begin{lem}
\label{cor norm Hs of cL}
 Let $F\in\cC^r(\bT^2,\bT^2)$ satisfying {\bf (H1)}. For each $n\in \bN$ and $1\le s\le r$, there exist $A_s, Q(n,s)>0$ such that, for every $u\in \cH^s(\bT^2, \bR)$, 
\begin{align}
&\| \cL^n u\|_{L^2}\le \|\cL^n1\|_{\infty}^{\frac 12}\|u\|_{L^2}\label{L2 norm transfer}\\
&\|\mathcal{L}^n u\|^2_{\cH^s}  \le  A_s\mu^{2sn}
\|\cL^n1\|_{\infty} \|u\|^2_{\cH^s}+Q(n,s)\|u\|^2_{\cH^{s-1}}\label{LY Hs and Hs-1},
\end{align}
where 
\begin{equation}\label{eq:Q1_est}
Q(n,1)\le C_{\mu, n}^{3} \mu^{3n}(1+C_F\bbc).
\end{equation}
\end{lem}
\begin{proof} First of all notice that
\begin{equation}
\begin{split}
\|\cL^n u\|^2_{L^2}&\le \|u\|_{L^2}\left( \int (\cL^{n}u\circ F^n)^2\right)^{\frac{1}{2}}\le \|u\|_{L^2}\left( \int (\cL^{n}u)^2\cL^n1\right)^{\frac{1}{2}}\\
&\le \|u\|_{L^2}\|\cL^n1\|_{\infty}^{\frac{1}{2}}\|\cL^n u\|_{L^2},
\end{split}
\end{equation}
hence (\ref{L2 norm transfer}). Next, by (\ref{eq: LnP=PLn}) and (\ref{definition of P}) we have, for each $v_i\in\{e_1, e_2\},$
\begin{equation}
\label{L2 norm of pullback}
\begin{split}
\|\partial_{v_s}\cdots \partial_{v_1} \mathcal{L}^n u\|^2_{L^2}&\le\|\mathcal{L}^n(\partial_{F^{n \ast}v_s}\cdots \partial_{F^{ n \ast}v_1}) u\|^2_{L^2}\\
&+\sum_{k=1}^{s}\| \cL^n(A^{\alpha}_{n,k}((A^{\alpha_{k}}_n \phi_n)\cdot P^\alpha_{k-1} u))\|^2_{L^{2}}.
\end{split}
\end{equation}
Let us analyze the first term above when $s=2$. Notice that
\[
\begin{split}
\partial_{F^{n \ast}v_2}(\partial_{F^{ n \ast}v_1} u)&=\langle \nabla \left( \langle \nabla u, (DF^n)^{-1}v_1 \rangle \right),  (DF^n)^{-1}v_2 \rangle\\
&= \langle  (DF^n)^{-1}v_1 D^2  u, (DF^n)^{-1}v_2 \rangle +  \langle D( (DF^n)^{-1}v_1) \nabla u, (DF^n)^{-1}v_2 \rangle.
\end{split}
\]
where $D^2 f$ indicates the Hessian of a function $f$ and $D(V)$ is the Jacobian of the vector field $V$. The term with higher derivatives of $u$ has coefficients bounded by $\|(DF^n)^{-1}\|^2,$  while the other term is a differential operator of order one applied to $u$.
In the general case we have some $\overline P_{s-1, \rho}$ such that
\begin{equation}
\label{bound of pull backs}
\begin{split}
|\mathcal{L}^n(\partial_{F^{n \ast}v_s}\cdots \partial_{F^{ n \ast}v_1}) u| &\le \|(DF^n)^{-1}\|^s \sup_{w_1,\cdots, w_s\in\{e_1,e_2\}} \cL^{n}( |\partial_{w_s}\cdots \partial_{w_1}  u|)\\
& +|\cL^n \overline P_{s-1, \rho}u|.
\end{split}
\end{equation}
Hence, by  (\ref{L2 norm transfer}), (\ref{equivalence of norm H L}) and (\ref{lambda_+<Clambda_-}), there exists a constant $C_1(n,s)$ such that
\begin{equation}
\label{strong +weak}
\|\mathcal{L}^n(\partial_{F^{n \ast}v_s}\cdots \partial_{F^{ n \ast}v_1}) u\|^2_{L^2} \le \Const \|\cL^n1\|_{\infty}\mu^{2sn} \|u\|^2_{\cH^s}+C_1(n,s)\|u\|^2_{\cH^{s-1}}.
\end{equation}
Similarly there exists $C_2(s,n)$ such that
\begin{equation}
\label{weak norm 2}
\sum_{k=1}^{t}\| \cL^n(A^{\alpha}_{n,k}((A^{\alpha_{k}}_n \phi_n)\cdot P^\alpha_{n,k-1} u))\|^2_{L^{2}}\le C_2(n,s)\|u\|^2_{\cH^{s-1}}.
\end{equation}
By (\ref{L2 norm of pullback}), (\ref{strong +weak}) and (\ref{weak norm 2}) we obtain 
\[
\|\mathcal{L}^n(\partial_{F^{n \ast}v_1}\cdots \partial_{F^{ n \ast}v_s}) u\|^2_{L^2}  \le
\Const \|\cL^n1\|_{\infty}\mu^{2sn}
\|u\|^2_{\cH^s}+Q(n,s)\|u\|^2_{\cH^{s-1}}.
\]
It remains to prove that in the case $s=1$ we have an explicit bound on $Q(n,1)$.
Recall that by (\ref{first derivative of L}) and (\ref{L2 norm transfer}) we have, for any $v\in \{e_1, e_2\},$
\begin{equation}
\label{H norm 1}
\begin{split}
\|\langle\nabla \mathcal{L}^n u, v\rangle\|_{L^2} &\le  \|\mathcal{L}^n\langle\nabla u,\left(D F^n\right)^{-1} v\rangle\|_{L^2}  +\|\mathcal{L}^n( \langle \nabla \phi_n, (D F^n)^{-1} v\rangle u)\|_{L^2},\\
&\le  \|\cL^n1\|_{\infty}^{\frac 12}\left(\|\langle\nabla u,\left(D F^n\right)^{-1} v\rangle\|_{L^2}  +\|\langle \nabla \phi_n, (D F^n)^{-1} v\rangle u\|_{L^2}\right).
\end{split}
\end{equation}
A bound for the first term is straightforward, since by (\ref{lambda_+<Clambda_-})
\begin{equation}
\label{H norm 2}
\|\langle\nabla u,\left(D F^n\right)^{-1} v\rangle)\|_{L^2} \le \Const\mu^{n} \|\nabla u\|_{L^2}.
\end{equation}
For the second term we use formula (\ref{eq: A n alphat of phin}) and we have
\begin{equation}
\label{H norm 3}
\begin{split}
\|\langle \nabla \phi_n, (D F^n)^{-1} v\rangle u\|_{L^2}&\le \sum_{j=0}^{n-1}\| \langle \nabla \phi_1\circ F^j(x), (D_{F^jx}F^{n-j})^{-1} v \rangle\|_\infty \|u\|_{L^2}\\
&\le \Const\sum_{j=0}^{n-1}\mu^{n-j}\|u\|_{L^2}\le C_{\mu, n} \mu^n \|u\|_{L^2} ,
\end{split}
\end{equation}
By (\ref{H norm 1}), (\ref{H norm 2}), (\ref{H norm 3}) and (\ref{inf norm of transf ineq}) we obtain \eqref{eq:Q1_est}.
\end{proof}
\subsection{Transversality}\label{section:Transversality}\ \\
In this Section we investigate the quantities $\cN_F, \widetilde \cN_F$ defined in section \ref{Transversality of unstable cones}. Recall that 
\begin{equation}
\begin{split}
\label{def of N}
&\mathcal{N}_{F}(n)=\sup_{y\in \mathbb{T}^2} \sup_{z_1\in F^{-n}(y)}\mathcal{N}_F( n, y, z_1)\\
& \widetilde{\mathcal{N}}_{F}(n)=\sup_{y\in \mathbb{T}^2} \sup_{L} \widetilde{\mathcal{N}}_F(n, y, L).
\end{split}
\end{equation}
Both $\mathcal{N}_F$ and $\widetilde{\mathcal{N}}_F$ depend on the map $F$, however in the following we will often drop the $F$ dependence to ease notation. 
An important advantage of $\widetilde{\cN}$ over $\cN$ is the following 
\begin{prop}
\label{submultiplicativity}
$\widetilde{\mathcal{N}}(n)$ is sub-multiplicative, i.e $\widetilde{\mathcal{N}}(n+m)\le \widetilde{\mathcal{N}}(n)\widetilde{\mathcal{N}}(m)$, for every $n,m\in \mathbb{N}$.
\end{prop}
\textit{Proof}. For any $y\in \mathbb{T}^2$, and line $L$ we have
\begin{equation*}
\begin{split}
\widetilde{\mathcal{N}}(y,L,n+m) &=  \sum_{\substack{z\in F^{-n-m}(y) \\ DF^{n+m}(z)\fC_{u}\supset L }}  |\det DF^{n+m}(z)|^{-1 
 } \\
 &=\sum_{\substack{\hat{z}\in F^{-n}(y) \\ DF^{n}(\hat{z})\mathbf{C}_{u}\supset L} }\;\; \sum_{\substack{z\in F^{-m}(\hat{z}) \\ DF^{m}(z)\mathbf{C}_{u}\supset \left( DF^{n}(\hat{z})\right)^{-1}L}} \frac{1}{|\det DF^{m}(z) \det DF^{n}(\hat{z})|} \\
 &\le \sum_{\substack{\hat{z}\in F^{-n}(y) \\ DF^{n}(\hat{z})\mathbf{C}_{u}\supset L }} \frac{1}{|\det DF^{n}(\hat{z})|} \sup_{\tilde{z}}\sup_{L'} 
 \sum_{\substack{z\in F^{-m}(\tilde{z} ) \\ DF^{m}(z)\mathbf{C}_{u}\supset L' }} \frac{1}{|\det DF^{m}(z) |}.
\end{split}
\end{equation*}
Taking the sup over $y\in \mathbb{T}^2$ and $L$ we get the claim. $\qed$
\begin{oss}
\label{rmk on submult}
The above Proposition, in spite of its simplicity, turns  out to be  pivotal. The sub-multiplicativity of the sequence $\widetilde{\mathcal{N}}(n)$ implies the existence of  $\lim_{n\to \infty}\widetilde{\mathcal{N}}(n)^{\frac{1}{n}}$. Also, an estimate of $\widetilde{\mathcal{N}}(n_0)$ for some $n_0\in \mathbb{N}$ yields an estimate for all $n\in \mathbb{N}$.
\end{oss}
\begin{flushleft}
The result below, inspired by \cite{BuEs}, provides the relation between $\mathcal{N}$ and $\widetilde{\mathcal{N}}$.
\end{flushleft}
\begin{lem}
\label{lemma on the relation N and tild N}
Let $\alpha= {\frac{\log(\lambda_-\mu^{-2})}{\log (\lambda_+)}\in ( 0,1)}$ and ${m_0}={m_0}(n)=\lceil \alpha n \rceil$ we have, for all $n\in\bN$
\begin{equation*}
\label{bound of N}
\mathcal{N}(n)^{\frac{1}{n}}\le \|\cL^{n-m_0}1\|_{\infty}^{\frac{1}{n}}\left( \widetilde{\mathcal{N}}({m_0})^{\frac{1}{ {m_0} }}\right)^{\alpha}.
\end{equation*}
\end{lem}
\textit{Proof.} Given $y \in \mathbb{T}^2$, we consider $z_1,z_2 \in F^{-n}(y)$ such that $D_{z_1}F^{n}\mathbf{C}_u\cap D_{z_2}F^{n}\mathbf{C}_u\neq  \{0\}$ and the line $ L:=L(z_1):=D_{z_1}F^{n}\left( \mathbb{R}\times \lbrace 0 \rbrace \right)$.\\
Let $v_{\pm}=(1,\pm \chi_u)\in \fC_u$ and $\theta_n:=\max_\pm\measuredangle\left( D_{z_1}F^{n}e_1,  D_{z_1}F^{n}v_{\pm} \right)$. 
Notice that, for $n=0$, $|\cos\theta_0|^{-1}=|\cos(\arctan(\chi_u))|^{-1}=\sqrt{1+\chi_u^2}=:a_0\le 2$. In fact, by the invariance of the unstable cone  $|\cos\theta_n|^{-1}\leq a_0$, for all $n\in\bN$. On the other hand, by formula \eqref{wadge formula}, Proposition \ref{prop on the det} and condition \eqref{partial hyperbolicity 2} we have 
\begin{equation}\label{eq:wadge-1}
\begin{split}
|\tan\theta_n|\le a_0|\sin \theta_n|&=a_0|\sin\measuredangle(e_1,v_{\pm})|\frac{|\det D_{z_1}F^{n}|\|v_{\pm}\| }{\|D_{z_1}F^{n}e_1\|\|D_{z_1}F^{n}v_{\pm}\|}\\
&\le |\sin(\arctan \chi_u) |a_0 C_*C_\star^2\mu_+^n\lambda_-^{-n}\\
&=  a_0C_*C_\star^2 \chi_u \mu_+^n\lambda_-^{-n},
\end{split}
\end{equation}
where we have used that $\sin(\arctan x)=x(\sqrt{1+x^2})^{-1}$.

Next, note that in the projective space $\mathbb{R}\mathbb{P}^{2}$ the cones $D_{z_1}F^{n}\mathbf{C}_u$ and $D_{z_2}F^{n}\mathbf{C}_u$ are canonically identified with two intervals $I_1=[a_1,b_1]$ and $I_2=[a_2,b_2]$, respectively. While the line is a point that we also denote by $L$. From the assumption on the cones, and \eqref{eq:wadge-1}, we have that the projective distance between $L$ and each one of the extremal points of $I_2$ is bounded by
\begin{equation}
\label{d(L,a)<lambda}
\min\lbrace \text{dist}(L,a_2),\text{dist}(L,b_2) \rbrace\le {a_0 C_*C_\star^2}\chi_u\lambda_-^{-n}\mu_+^n.
\end{equation}
Let us now take $m<n$ to be chosen later and, for  $\tilde{z}=F^{n-m}(z_2)$, consider the cone $D_{\tilde{z}}F^{m}\mathbf{C}_u$ corresponding to the interval $I_3$ in the projective space. By the forward  invariance of the unstable cone it is clear that $D_{\tilde{z}}F^{m}\mathbf{C}_u \supset D_{z_2}F^{n}\mathbf{C}_u$, meaning that $I_3\supset I_2$. 
We are going to prove that $L\in I_3$. {Let $w_{n,m}:=D_{z_2}F^{n-m}v_\pm$}. Arguing as before, but remembering also condition \eqref{invariance of cone}, we have
{
\begin{equation}\label{eq:wadge-2}
\begin{split}
|\sin\measuredangle\left( D_{\tilde z}F^{m}w_{n,m},  D_{\tilde z}F^{m}v_{\pm} \right)|&=|\sin\measuredangle(w_{n,m},v_{\pm})|\frac{|\det D_{\tilde z}F^{m}|\|v_{\pm}\|\|w_{n,m}\| }{\|D_{\tilde z}F^{m}w_{n,m}\|\|D_{\tilde z}F^{m}v_{\pm}\|}\\
&\ge C_*^{-1}C_\star^{-2}|\sin(\arctan(\iota_\star\chi_u))|\lambda_+^{-m}\mu_-^n\\
&= C_*^{-1}C_\star^{-2}\frac{\iota_\star\chi_u}{\sqrt{1+(\iota_\star\chi_u)^2}}\lambda_+^{-m}\mu_-^n.
\end{split}
\end{equation}
It follows that, setting $B_{\iota_\star}:=  a_0 C_*^{2}C_\star^{4}{\iota_\star^{-1}}\ge 1$, if $\lambda_+^{-m}\mu_-^n\ge B_{\iota_\star}\lambda_-^{-n}\mu_+^n,$ then $L\in I_3$. By a direct computation, and recalling that $\mu:=\{\mu_+,\mu_-^{-1}\}^+$, we see that the choice $m=\lceil \alpha n -\beta_{\iota_\star} \rceil,$ with $\alpha:= \frac{\log(\lambda_-\mu^{-2})}{\log (\lambda_+)}>0$ and $\beta_{\iota_\star}:=\frac{\log B_{\iota_\star}}{\log \lambda_+}\ge 0$ yields the wanted inequality. Also, note that $\alpha<1$ since $\lambda_-<\lambda_+\mu^2$ .}\\
The above computation shows that, given $z_1\in F^{-n}(y)$, for every $z_2 \in F^{-n}(y)$ which is non-transversal to $z_1$, the line $L$ is contained in the cone $D_{\tilde{z}}F^{m}\mathbf{C}_u$, for $\tilde{z}=F^{n-m}(z_2).$ In particular,  for every $y\in \mathbb{T}^2$, one has
\begin{equation*}
\begin{split}
&\sup_{z_1\in F^{-n}(y)}\sum_{\substack{z_2\in F^{-n}(y) \\  z_2 \not{\pitchfork} z_1 }}  |\det D_{z_2}F^{n}|  ^{-1}\le
\sup_{z_1\in F^{-n}(y)} \sum_{\substack{z_2\in F^{-n}(y) \\  D_{F^{n-m}(z_2)}F^{m}\mathbf{C}_u\supset D_{z_1}F^{n}\left( \mathbb{R}\times \lbrace 0 \rbrace \right) }}  |\det D_{z_2}F^{n}|^{-1} \\
&\le   \sup_{L\subset \mathbb{R}\mathbb{P}^2}\sum_{\substack{\tilde z\in F^{-m}(y) \\  D_{\tilde z}F^{m}\mathbf{C}_u\supset L }}  |\det D_{\tilde z}F^{m}|^{-1}  \sum_{z_2\in F^{-n+m}(\tilde z)}  |\det D_{z_2}F^{n-m}|^{-1} \\
&\le\cL^{n-m}1( y) \sup_{L\subset \mathbb{R}\mathbb{P}^2}\sum_{\substack{z\in F^{-m}(y) \\  D_{z}F^{m}\mathbf{C}_u\supset L }}  |\det D_zF^{m}|^{-1},
\end{split}
\end{equation*}
where we have used (\ref{Trasfer Operator iteration}). The above inequality implies
\begin{equation*}
\mathcal{N}(n)\le \|\mathcal{L}^{n-m} 1\|_{\infty} \widetilde{\mathcal{N}}(m)\le\|\mathcal{L}^{\lceil n(1-\alpha)\rceil} 1\|_{\infty} \widetilde{\mathcal{N}}(\lceil \alpha n \rceil). \qed
\end{equation*}
\section{A second Lasota-Yorke inequality: Results}\label{section A second Lasota-Yorke inequality}
To state the main result we need a few definitions. From Appendix \ref{appendix space Hs } we recall that, for positive integers $N\in\bN$ and $s\ge 1,$ and for $u\in \cC^r(\bT^2),$
\begin{equation}
\label{Hs norm of L^N}
	\|\cL^N u\|_{\cH^s}^2=\sum_{\xi\in \bZ^2} |\langle \xi \rangle^{s}\cF\cL^{N}u(\xi)|^2,
\end{equation}
where $\langle \xi \rangle=\sqrt{1+\|\xi\|^2}.$ Since we will work in Fourier space, it is convenient to introduce the notion of the dual of a cone in $\bR^2$ by:
\begin{equation}
\label{dual cone}
\mathbf{C}^{\perp}=\lbrace {v}\in \mathbb{R}^2 \;:\;  \exists\, {u}\in \mathbf{C} : \langle {v},{u}  \rangle =0 \rbrace,
\end{equation} 
and if $\xi\in\bZ^2\setminus\{0\}$ we define $\xip:=(\xip_1, \xip_2)=(-\xi_2,\xi_1)\|\xi\|^{-1}$. In addition,  we define $\rho(\xip)=|\xip_2|/|\xip_1|$, for $\xip_1\neq 0$, $\rho(\pm e_2)=\infty$, and
\begin{equation}\label{def of varthetaxi*}
\vartheta(\xip):=\{\rho(\xip), \chi_u\}^+.
\end{equation}
Let us also define the sequence
\begin{equation}\label{def of supLn1}
\mathbbm{L}_n:=\|\cL^n 1\|_{\infty}.
\end{equation}
Finally, to state the main result one last key assumption is needed. Let us define
\begin{align}
&n_0(F):=\min \{ n\in\bN : \forall p\in \bT^2 \quad \exists\; z_1,z_2\in F^{-n}p : z_1\pitchfork z_2 \}.\label{n_0 in the general case}
\end{align}
We will always assume that the map $F$ satisfies
\begin{equation}\label{transversality-assumption}
n_0(F)<\infty.
\end{equation}
For simplicity, in the following we will just use the notation $n_0$ instead of $n_0(F).$
\begin{oss}\label{rmk:Tsuji-generic}
In \cite{Tsu 2} it is proven that assumption (\ref{transversality-assumption}) is generic. More precisely, the author proves that for surface partially hyperbolic systems $F$, $\cN_F(n)$ is generically strictly smaller than 1, for $n$ large. 
\end{oss}
The goal of this Section is to prove the following Theorem.
\begin{thm}
\label{thm: LY in Hs and Bs}
 {\bf(SVPH)} Let $F$ be a SVPH.  Let $m_{\chi_u}$ and $n_0$ be the integers given in (\ref{def of mrho and m}) and \eqref{n_0 in the general case} respectively.{ There exist uniform constants  $C_1, \const>0$, $\Lambda>2$ and $\sigma>1$ such that,  for each $q_0\geq n_0$ and any $1\le s<r$, if $M\geq \sigma m_{\chi_u}$ and $N= M+q_0$,}
\begin{equation}
\label{LY Hs and Bs+2}
\|\cL^N u\|_{\cH^s}\le C_1\left( \sqrt{[\bL_M\cN(q_0)]^{\frac{1}{N}}\mu^{2s}}\right)^{N}\|u\|_{\cH^s}+\Omega_{\chi_u}(M, s)\|u\|_{s+2}.
\end{equation} 
where $\Omega_{\chi_u}(M, s)\leq C_{\eps,q_0} \mu^{\const M}Q(M,s)C_{\mu,M}^{\const}\Lambda^{\const M}\bbd^{-4s}\bL_M$ and $Q(M,s)$ is given in Lemma \ref{cor norm Hs of cL}. \\
\end{thm}
In the case of SVPH$^{\,\sharp}$ we have a sharper control on the constants as specified by the next theorem.
\begin{thm}
\label{thm: LY in Hs and BsS}
{\bf(SVPH$^\sharp$)}
If the map $F$ is a SVPH$^{\,\sharp}$ (see Definition \ref{def:SVPHsharp}) and satisfies
\begin{equation}\label{key-conditions-for-F}
\begin{split}
&\chi_u^{-1}\|\omega\|_{\cC^r}\le C_5,
\end{split}
\end{equation}
for a uniform constant $C_5>0$, then there exist $\beta_3\in\bR^+$, depending only on $C_5$, such that
\begin{equation}
\label{Theta (N, 1)}
\Omega_{\chi_u}(M,1)\le C_{\eps, q_0}  C_{\mu,M}^{\beta_3}\mu^{\beta_3 M} (\ln \chi_u^{-1})^{\beta_3}\chi_u^{\frac{-13}{2}-\const \ln\mu^{-1}}\bL_M^{\frac 12}.
\end{equation}
\end{thm}
\noindent We will prove Theorem \ref{thm: LY in Hs and Bs} in Section \ref{section:Proof of Theorem thm: LY in Hs and Bs} and Theorem \ref{thm: LY in Hs and BsS} in Section \ref{section:Proof of Theorem thm: LY in Hs and Bs_sharp}, after several preliminary steps in which first we show how to localise the relevant objects and then we show how to estimate some of the Fourier coefficients using the $\|\cdot\|_s$ norms, for $s$ large enough.
\subsection{Partitions of unity}\label{section Decompos of Transfer Oper}\ \\
We will use notations and definitions introduced in Section \ref{section Some further notation}. First of all we want to decompose the transfer operator using suitable partitions of unity. For each point $z\in\bT^2$, and $q_0\ge n_0,$ let us 
set $\delta_{q_0}(z):=\mu_{q_0}^-(z)\lambda_{q_0}^+(z)^{-1}$,\footnote{ The functions $\mu_n^-$ and $\lambda_n^+$ 
are defined in (\ref{def of lamb^+ mu^+}).} and define,\footnote{ The constant $\eps$ is the one introduced just before equation \eqref{definition of C_epsu}.}
\begin{equation}
\label{U_z}
\cU_{z,q_0}=\{y\in\bT^2\;:\; \|y-z\|<\min\{ 1/2, \bbd\epsilon\delta_{q_0}(z) \}\},
\end{equation}
where\footnote{  $L_\star$ is the Lipschitz constant given in \eqref{Lstar} with $v=\chi_u$.}
\begin{equation}
\label{constant bbd}
\bbd=\bbd(\chi_u)=C_{\mu,\bbb\ln\chi_u^{-1}}^{-1}\mu^{-\bbb\ln\chi_u^{-1}}L_\star(\chi_u, q_0)^{-1}C_0\chi_u,
\end{equation}
for some uniform constants $\bbb, C_0$ to be chosen later.\footnote{ The choice of is $\bbb$, made in Proposition \ref{prop on four coeff} and depends on the parameter $t$ specifying the Sobolev space we are interested in. The choice of $C_0$ is made in Lemma \ref{lem on transv in terms of partitions}.}\\
By Besicovitch covering theorem  there exists a finite subset $\cA$ and points $\{z_\alpha\}_{\alpha\in\cA}$ such that $\bT^2\subset \bigcup_{\alpha}\cU_\alpha$ where $\cU_\alpha=5\cU_{z_\alpha, q_0}$, and such that the number of intersections is bounded by some fixed constant $C_\sharp$. We then define a family of smooth function $\{\psi_\alpha\}_\alpha$ supported on $\cU_{\alpha}$ such that $\sum_{\alpha}\psi_{\alpha}=1.$\\
 Note that we can choose the $\psi_\alpha$ such that
\begin{equation}\label{eq:part_der}
\|\psi_\alpha\|_{\cC^t}\leq \Const  \bbd^{-t}\epsilon^{-t}\mu^{q_0 t}\lambda_+^{q_0t}.
\end{equation}
Next, we construct a refinement of the above partition  using the inverse branches introduced in Section \ref{subsec admiss curv}.
For each $\alpha\in\cA$ we pick a curve $\gamma_\alpha\in\Upsilon$, $\gamma_\alpha'=e_2$, such that $\cU_\alpha \cap \gamma_\alpha=\{\emptyset\}$ and let 
$\tilde\gamma_\alpha=\gamma+(1/2,0)$. Recalling, from  equation \eqref{def:frh^star-frH}, that $\frH_{\gamma_\alpha,1}=\{\frh\in \frH : \cD_\frh=\bT^2\setminus\{\tilde \gamma_\alpha\}\}$, for each $\frh\in \frH_{\gamma_\alpha,1}\cup \frH_{\tilde\gamma_{\alpha},1}$ either $\frh(\bT^2)\cap\gamma_\alpha=\emptyset$ or $\frh(\bT^2)\cap\tilde \gamma_\alpha=\emptyset$. Note that the cardinality of $\frH_{\alpha,0}:=\frH_{\gamma_\alpha,1}$ and $\frH_{\tilde\gamma_{\alpha},1}$ is exactly $d$. We can then consider the set
$\frH^n_{\alpha}=\{(\frh_1,\cdots,\frh_n)\in \frH^{n}: \frh_{j}\in \frH_{\gamma_{j-1},\alpha}, j\in \{1,\cdots, n\}\}$ where $\gamma_0=\tilde\gamma_\alpha$ and, for $j>0$, $\gamma_j=\gamma_\alpha$ if $\frh_{j}(\bT^2)\cap \tilde \gamma_\alpha= \emptyset$ and $\gamma_j= \tilde\gamma_\alpha$ if $\frh_{j-1}(\bT^2)\cap \tilde \gamma_\alpha\neq \emptyset$.
Note that $\frH^n_{\alpha}$ has exactly one element for each equivalence class of $\frH^n_{*,\gamma_{\alpha}}$, defined in equation \eqref{def:frh^star-frH}, hence it is isomorphic to $\frH_{\gamma_\alpha,n}$ and has exactly $d^n$ elements. To simplify notation, given $\alpha \in \cA$ and $q_0\in \bN$, in the following we will denote $\frH^{q_0}:=\frH^{q_0}_{*,\gamma_{\alpha}}$ which is then a set with cardinality $d^{q_0}$.

Next, let $\alpha, \alpha' \in \cA$ and define
\begin{equation}
\label{partition of unity psi_alpha}
\psi_{\alpha{\alpha'}, \frh}(z)=\psi_\alpha\circ F^{q_0}(z)\mathds{1}_{\frh,\alpha}(z){\psi_{\alpha'}(z)}, \; \forall \frh\in \frH^{q_0}, z\in \bT^2,
\end{equation}
where $\mathds{1}_{\frh,\alpha}:=\mathds{1}_{\cU_{\alpha,\frh}}$. Note that \eqref{partition of unity psi_alpha} defines a $\cC^r$ partition of unity, supported on $\{\cU_{\alpha,{\alpha'}, \frh}\}_{\frh\in\frH^{q_0}}$, with $\cU_{\alpha{\alpha'},\frh}:=\frh(\cU_{\alpha})\cap \cU_{{\alpha'}}$, and intersection multiplicity bounded by {$\Const$}.\\
\indent In order to keep the notation simple, we use the index ${\balpha}=\alpha\alpha'$ to indicate quantities which depend both on $\alpha$ and $\alpha'$ and we write $\sum_{\balpha}$ for $\sum_{\alpha\in \cA}\sum_{\alpha'\in \cA}$.\\

The following result is similar to \cite[Lemma 9]{AvGoTs}, but the proof is adapted to our case.\footnote{ Similar inequalities hold more generally for some anisotropic norms, as used originally in \cite{BT} and go under the name of {\it fragmentation-reconstitution} inequalities in \cite{Babook2}.}
\begin{lem}
\label{lem H^s partition of uni}
For each $u\in\cC^r(\bT^2)$
\begin{align}
\|u\|_{\cH^s}^2&\le C_\sharp\sum_{\alpha\in\cA}\|u \psi_{\alpha} \|^2_{\cH^s} \label{partit of unit estim 2}\\
\sum_{\balpha}\sum_{\frh\in \frH^{q_0}}\|u\psi_{\balpha, \frh}\|^2_{\mathcal{H}^{s}}&\le \Const\|u \|^2_{\mathcal{H}^{s}}+C_\psi(s)\|u\|_{L^1}^2, \label{partit of unit estim 1}
\end{align}
where $C_\psi(s)=C_{\eps,q_0}\bbd^{-4s}$.
\end{lem}
\begin{proof} For the first inequality note that
\[
\|u\|_{\cH^s}^2=\big\| \sum_{\alpha\in \cA}u\psi_{\alpha}\big\|^2_{\cH^s}= \sum_{(\alpha, \alpha')\in \cA\times \cA}\langle \psi_{\alpha}u, \psi_{\alpha'}u\rangle_s.
\]
By the definition of the $\langle \cdot, \cdot \rangle_s$ the above sum is zero if the supports of $\psi_{\alpha}$ and $\psi_{\alpha'}$ do not intersect.  For the other terms, denoting with $\cA^*$ the set of elements in $\cA\times \cA$ for which the above supports intersect, we have:
\[
\sum_{\cA^*}\langle \psi_{\alpha}u, \psi_{\alpha'}u\rangle_s\le \sum_{\cA^*}\frac{\| \psi_{\alpha}u\|_{\cH^s}^2+ \|\psi_{\alpha'}u\|_{\cH^s}^2}{2}\le C_\sharp\sum_{\alpha\in \cA}\|\psi_{\alpha}u\|_{\cH^s}^2.
\]
We now prove (\ref{partit of unit estim 1}). By  formula \eqref{eq:horm} we have, recalling \eqref{eq:part_der},
\[
\begin{split}
\sum_{\balpha, \frh}\|u\psi_{\balpha, \frh}\|^2_{\mathcal{H}^{s}}&=\sum_{\balpha, \frh}\sum_{|\beta|\le s}C_\beta\langle \partial^{\beta}(u\psi_{\balpha, \frh}),  \partial^{\beta}(u\psi_{\balpha, \frh})\rangle_{L^2}\\ 
&\leq \Const\sum_{\balpha, \frh}\sum_{|\beta|\le s}\langle (\partial^{\beta}u)\psi_{\balpha, \frh},  (\partial^{\beta}u)\psi_{\balpha, \frh}\rangle_{L^2}
+\langle (\partial^{\beta-1}u)\partial\psi_{\balpha, \frh},  (\partial^{\beta}u)\psi_{\balpha, \frh}\rangle_{L^2}\\
&\phantom{\leq}+C_{\eps, q_0}\bbd^{-2s}\|u\|^2_{\cH^{s-1}}
\end{split}
\]
where we have usded the fact that $\sum_{\balpha, \frh}|\partial^\alpha\psi_{\balpha, \frh}| |\partial^\gamma u|\leq \Const \sup_{\balpha, \frh}\|\psi_{\balpha, \frh}\|_{\cC^{|\alpha|}}  |\partial^\gamma u|$ due to the uniform finite intersection multiplicity of the partition of unity. Hence, since $\psi_{\balpha, \frh}^2\leq \psi_{\balpha, \frh}$,
\[
\begin{split}
\sum_{\balpha, \frh}\|u\psi_{\balpha, \frh}\|^2_{\mathcal{H}^{s}}&\leq \Const \|u\|^2_{\cH^s}+C_{\eps, q_0}\bbd^{-1}\|u\|_{\cH^s}\|u\|_{\cH^{s-1}}+C_{\eps, q_0}
\bbd^{-2s}\|u\|^2_{\cH^{s-1}}\\ 
&\leq \Const \|u\|^2_{\cH^s}+C_{\eps, q_0}\bbd^{-2s}\|u\|^2_{\cH^{s-1}}\leq \Const \|u\|^2_{\cH^s}+C_{\eps, q_0}\bbd^{-4s}\|u\|^2_{L^1},
\end{split}
\]
where in the last line we used Lemma \ref{lem on norm s-1 and s}. 
\end{proof}
{
\begin{oss}\label{rmk constant bbd}
For future purposes we note that, under condition \eqref{key-conditions-for-F}, recalling equation \eqref{Lstar}, we have
\[
L_\star( \chi_u, q_0)=\Const C_{q_0}C_{\mu,\ln\chi_u^{-1}}\chi_u^{1-\const \ln \mu},
\]  
which implies that,
by \eqref{constant bbd}, $C_{\psi}(1)\le  \epsilon^{- 2} C_{q_0}C_{\mu,\const \ln\chi_u^{-1}}\chi_u^{\const \ln \mu}.$
\end{oss}
}
\subsection{Fourier basic estimate: case SVPH}\label{sec:fourier_first}\ \\
The next Proposition is the main ingredient for the proof of Theorem \ref{thm: LY in Hs and Bs}.
\begin{prop}
\label{prop on four coeff}
Let $n_0\in\bN$, $\xi \in \bZ^2$, $q_0\ge n_0$,  $\alpha,\alpha' \in\cA$, and $\frh\in\frH^{q_0}$ be such that  $D_p{\frh_{n_0}} \xip \notin \fC_{u}$, for each $p\in \text{supp }\psi_{\balpha, \frh}$,  then, for each $t\geq 2$, there exists $M_{\xi}\equiv M_{\xi,t},$ such that
\begin{equation}
\label{control four coeff}
\langle \xi \rangle^{t}| \mathcal{F}\mathcal{L}^{{q_0}}(\psi_{\balpha, \frh} \cL^{M_{\xi}}u) (\xi)|\le K_1(t, M_{\xi})\|u\|_{t},
\end{equation}
where $K_1(t,M_\xi)\le C_{q_0} C_{\mu, M_\xi}^2\mu^{\const M_\xi}\chi_u^{-\const}\Lambda^{\const M_\xi}$.
Moreover, there exists  $\sigma>1$ such that, with $m_{\chi_u}$ as in \eqref{def of mrho and m},
 \begin{equation}\label{def of ovmchiu}
\ovm_{\chi_u}:=\sigma m_{\chi_u}\geq \sup_{\{\xi: \xip\notin \fC_u\}}M_{\xi}.
\end{equation} 
 \end{prop}
\begin{proof}
Let $\xi=(\xi_1, \xi_2)$, chose $j\in\{1,2\}$ such that $\|\xi\|\le 2|\xi_j|$, and $M_\xi>0$ to be chosen later. Since $\xi_j\cF u=-i\cF\partial_{x_j} u$, and $\|\cF u \|_\infty \lesssim \|u\|_{L^1}$, using \eqref{geom norm} we have, for each $t\ge 1$ and setting $u^{M_\xi}_{\balpha,{\frh}}=\psi_{\balpha, \frh}\cL^{M_\xi} u$,
\begin{equation}\label{eq: Fourier 1}
\begin{split}
\langle\xi \rangle^t|\mathcal{F}\mathcal{L}^{{q_0}}(\psi_{\balpha, \frh} \cL^{M_\xi}u) (\xi)|&\lesssim \|\cL^{q_0}(u^{M_\xi}_{\balpha,{\frh}})\|_{L^1}+|\xi_j|^t|\cF\cL^{q_0}(u^{M_\xi}_{\balpha,{\frh}})|\\
&\lesssim \|u\|_0+  |\cF\partial^{t}_{x_j}\cL^{q_0}(u^{M_\xi}_{\balpha,{\frh}})|.
\end{split}
\end{equation}
Letting $J_{k}(p)=(\det D_pF^{k})^{-1}$ and recalling \eqref{partition of unity psi_alpha}, we have
\begin{equation}
\label{Fourier 2}
\begin{split}
&\left[\cF\partial^{t}_{x_j}\cL^{q_0}(u^{M_\xi}_{\balpha,{\frh}})\right](\xi)=\int_{\bT^2} dz\; e^{-2\pi i  \langle z,\xi\rangle}\partial_{z_j}^t\left\{\left[ J_{q_0} \psi_{\balpha, \frh}\cL^{M_\xi} u\right] \circ\frh\right\}(z)\\
&=\sum_{|\eta_1|+|\eta_2|=t}\hskip-12 pt C_{\eta_1,\eta_2} \int_{\cU_\alpha{ \cap F^{q_0}(\cU_{\alpha'})}}\hskip-24pt dz\; e^{-2\pi i  \langle z,\xi\rangle}\partial^{\eta_1}\left[ \psi_{\alpha}\psi_{\alpha'}\circ\frh\right](z)
\cdot\partial^{\eta_2}\left\{\left[J_{q_0}\cL^{M_\xi}u\right] \circ\frh\right\}(z).
\end{split}
\end{equation}
To simplify notation let $\bar \psi_{\bar \alpha,\frh}:=\psi_{\alpha}\psi_{\alpha'}\circ\frh$.
Using the change of variables $\gamma_\ell (\tau)=z_\alpha+\ell{\xi}+\tau{\xi}^{\perp}$, where
$\ell, \tau \in I_{{q_0}}=[-\bbd\epsilon \delta_{q_0}(z_\alpha),\bbd\epsilon \delta_{q_0}(z_\alpha)]$,\footnote{ By definition, the curves $\gamma_\ell$ cover $\cU_{z_{\alpha}, q_0}$, given in \eqref{U_z},  which contain the support of $\psi_{\balpha}$.} we have
\begin{equation}\label{eq:disintegration}
\left|\cF\partial^{t}_{x_j}\cL^{q_0}(u^{M_\xi}_{\balpha,{\frh}})\right|\le 
\Const \!\!\!\sup_{|\eta_1|+|\eta_2|=t}\int_{I_{q_0}}\!\!\!\!\!\! d\ell \left| \int_{I_{q_0}}\!\!\!\!\!\! d\tau\left\{\partial^{\eta_1}\bar \psi_{\bar \alpha,\frh}
\cdot\partial^{\eta_2}\left[J_{q_0}\cL^{M_\xi}u\right]\circ\frh \right\}(\gamma_\ell(\tau))\right|. 
\end{equation}

For each $\tilde\frh \in \frH^\infty$ let $\bar \frh=\tilde \frh\circ \frh$, provided it  is well-defined. Let $m:=m(p,\frh)$ be the smallest integer such that, for each $\bar \frh$ and $p\in \text{supp }\psi_{\balpha, \frh}$, $D_p{ \bar\frh_{m}} \xip \in \fC_{\epsilon, c}$. Note  that, by hypothesis and recalling \eqref{expansion with mchi_u}, $m-n_0\leq c^+_2(\ln \chi_u^{-1}+1)$.\\
For $z\in\gamma_\ell$ let $\overline m_\ell(z,\tilde \frh\circ \frh):=\sigma m(z, \frh)$, with $\sigma$ as in \eqref{def of sigma}, and set  $\ovm_\ell(\tilde \frh)=\sup_{z\in\gamma_\ell}\ovm_\ell(z,\tilde \frh\circ \frh)$.\footnote{ Notice that $\ovm$ depends on $\xi$ through $\gamma_\ell$. Also, it would be more precise to call it $\ovm_\ell(\tilde\frh\circ \frh)$, but we keep the notation as simple as possible.} 
We then define 
\begin{equation}\label{def of Mxi2}
    M_\xi \equiv M_{\xi,t}= \sup_\ell\sup_{\tilde\frh \in \frH^\infty}\ovm_\ell (\tilde\frh).
\end{equation}

Observe that, the assumption $D_p{ \bar\frh_{n_0}} \xip \notin \fC_{u}$ and condition \eqref{invariance of cone} imply \eqref{def of ovmchiu} for $t\ge 2$.\\

In the following we estimate the inner integrals of \eqref{eq:disintegration} for each fixed $\ell$, since this does not create confusion we thus drop the $\ell$ subscript in $\ovm_\ell$ to ease notation.

Define $\hat{\frH}_\alpha^*=\{\frh_{\overline{m}(\bar \frh)}\}_{\bar \frh\in \frH^\infty}$, $\hat{\frH}_\alpha=\{\hat \frh\;:\; \hat \frh\circ \frh\in \hat{\frH}^*_\alpha \}$, $v_{\balpha,\hat\frh}=\cL^{M_\xi-\ovm(\hat\frh)+q_0}u$ and write 
\begin{equation}\label{eq:splitting}
\frH^{M_\xi}=\bigcup_{\hat\frh\in \hat{\frH}_\alpha}\left\{ \frh'\circ\hat\frh\;:\;\frh'\in\frH^{M_\xi-\ovm(\hat\frh)+q_0}\right\}.
\end{equation}
This allows to define the decomposition
\[
\cL^{M_\xi} u=\sum_{\hat\frh\in \hat{\frH}_\alpha}J_{\ovm(\hat\frh)-q_0}\circ \hat\frh\cdot \left[\cL^{M_\xi-\ovm(\hat\frh)+q_0}u\right]\circ \hat\frh=\sum_{\hat\frh\in \hat{\frH}_\alpha}J_{\ovm(\hat\frh)-q_0}\circ \hat\frh \cdot v_{\balpha,\hat\frh}\circ \hat\frh.
\]
Thus, recalling \eqref{def operator Pfrh},
\[
\begin{split}
&\left|\cF\partial^{t}_{x_j}\cL^{q_0}(u^{M_\xi}_{\balpha,{\frh}})\right|\\
&\le 
\Const \!\!\!\sup_{|\eta_1|+|\eta_2|=t}\sum_{\hat\frh\in \hat{\frH}_\alpha}\int_{I_{q_0}} d\ell \left| \int_{I_{q_0}} d\tau\left\{\partial^{\eta_1}\bar \psi_{\bar \alpha,\frh}
\cdot J_{\ovm(\hat\frh)}\circ \hat\frh\circ \frh\left[ P^{\eta_2}_{\ovm(\hat\frh), |\eta_2|} v_{\balpha,\hat\frh}\right]\circ \hat\frh\circ\frh \right\}(\gamma_\ell(\tau))\right|.
\end{split}
\]
Next, we apply Lemma \ref{lem extension of curve}  to  $\gamma_\ell$ with $\delta=\bbd\epsilon\delta_{q_0}(z_\alpha)$, note that the hypotheses of the Lemma are satisfied thanks to the assumptions on $\xi$. We thus obtain closed curves $\tilde \gamma_\ell$ with $j+1$ derivative bounded by $C_{q_0,j}\Delta_{\tilde\gamma}^j$. 
It follows
\[
\begin{split}
&\left|\cF\partial^{t}_{x_j}\cL^{q_0}(u^{M_\xi}_{\balpha,{\frh}})\right|\\
&\leq\Const \!\!\!\sup_{|\eta_1|+|\eta_2|=t}\sum_{\hat\frh\in \hat\frH_\alpha}\int_{I_{q_0}} d\ell \left| \int_{\bT} d\tau\left\{\partial^{\eta_1}\bar\psi_{\bar\alpha,\frh}
\cdot J_{\ovm(\hat\frh)}\circ \hat\frh\circ \frh\left[ P^{\eta_2}_{\ovm(\hat\frh), |\eta_2|} v_{\balpha,\hat\frh}\right]\circ \hat\frh\circ\frh \right\}(\tilde\gamma_\ell(\tau))\right|.
\end{split}
\]
Next, we apply, for each inverse branch $\hat\frh\circ \frh$, Lemma \ref{lem unst curve pull back} to the curves $\tilde\gamma_\ell$ and obtain admissible central curves $\hat{\nu}_{\ell}=\nu_{\ell}\circ h_{\ell,\ovm}$.\footnote{ Notice that $\nu_{\ell}= \hat\frh\circ \frh\circ \gamma_\ell$ depends on $\hat \frh$, but we drop this dependence for simplicity.} Thus, we can rewrite the inner integrals in the right hand side of the above equation as follows
\[
\begin{split}
&\int_{\bT} d\tau\left\{\partial^{\eta_1}\bar\psi_{\bar\alpha,\frh}
\cdot J_{\ovm(\hat\frh)}\circ \hat\frh\circ \frh\left[ P^{\eta_2}_{\ovm(\hat\frh), |\eta_2|} v_{\balpha,\hat\frh}\right]\circ \hat\frh\circ\frh \right\}(\tilde\gamma_\ell(\tau))\\
&=\int_{\bT} d\tau \Psi_{\hat\nu_\ell}(\tau)\left\{(\partial^{\eta_1}\bar\psi_{\bar\alpha,\frh})\circ F^{\ovm(\hat\frh)}
\cdot \left[ P^{\eta_2}_{\ovm(\hat\frh), |\eta_2|} v_{\balpha,\hat\frh}\right]\right\}(\hat \nu_\ell(\tau)),
\end{split}
\]
where $\Psi_{\hat\nu_\ell}(\tau)=h_{\ell,\ovm}'[\det D_{\hat\nu_\ell(\tau)}F^{\overline{m}(\hat \frh)}]^{-1}$. 
By Proposition \ref{operator bound rho norm} applied with $n=\ovm(\hat \frh)$, $\varphi=\Psi_{\hat \nu_\ell}(\partial^{\eta_1}\bar\psi_{\bar\alpha,\frh})\circ F^{\ovm(\hat\frh)}\circ \hat \nu_\ell\|\Psi_{\hat \nu_\ell}(\partial^{\eta_1}\bar\psi_{\bar\alpha,\frh})\circ F^{\ovm(\hat\frh)} \|^{-1}_{\cC_{\hat\nu_\ell}^{|\eta_2|}}$, $\psi=1$, $u=v_{\balpha,\hat \frh}$ and $U=\bT^2$, the above integral is bounded by
\begin{equation}
\label{eq: estimate befor sum on tilde frh}
\tilde C(t,\ovm(\hat \frh), m)\|\Psi_{\hat \nu_\ell}(\partial^{\eta_1}\bar\psi_{\bar\alpha,\frh})\circ F^{\ovm(\hat\frh)} \|_{\cC_{\hat \nu_\ell}^{|\eta_2|}}\|v_{\balpha,\hat\frh}\|_{|\eta_2|},
\end{equation}
where $\tilde C(t,\ovm(\hat \frh),m)\le \Const\Lambda^{\const M_\xi}$.\\
Accordingly,
\begin{equation}\label{eq:horror}
\begin{split}
\left|\cF\partial^{t}_{x_j}\cL^{q_0}(u^{M_\xi}_{\balpha,{\frh}})\right|\leq&\Const\Lambda^{\const M_\xi}\!\!\!\sup_{|\eta_1|+|\eta_2|=t}\sum_{\hat\frh\in \hat\frH_\alpha} \|\Psi_{\hat \nu_\ell}\cdot(\partial^{\eta_1}\bar\psi_{\bar\alpha,\frh})\circ F^{\ovm(\hat\frh)} \|_{\cC_{\hat \nu_\ell}^{|\eta_2|}}
\|v_{\balpha,\hat\frh}\|_t|I_{q_0}|,
\end{split}
\end{equation}
where  $|I_{q_0}|\le 2\bbd\epsilon \delta_{q_0}(z_\alpha)\le 2\bbd\epsilon \lambda_-^{-q_0}\mu^{q_0}$.
Let us estimate the terms in the above sum, setting temporarily $\ovm=\ovm(\hat \frh)$. By Corollary \ref{Corollario 1}
\begin{equation}
\label{uNk norm}
\begin{split}
\|v_{\balpha, \hat \frh}\|_t=\|\cL^{M_\xi-\overline m(\frh)+q_0}u\|_t &\le C_{\mu, M_\xi}\mu^{M_\xi}C_{q_0} \|u\|_t, 
\end{split}
\end{equation}
and, by Lemma \ref{lem on det^-1-for-unstable-curve}, for each $\alpha\in\cA$
\begin{equation}
\label{sum on M of Psi}
\sum_{\hat \frh\in \hat \frH_\alpha}\|\Psi_{\hat{\nu}(\hat \frh)}\|_{\cC^{t}}\le \sum_{\hat \frh\in \frH^{M_\xi}}\|\Psi_{\hat{\nu}_\ell(\hat \frh)}\|_{\cC^{t}}\le A_u(t,\ovm,m),
\end{equation}
where,
\begin{equation}
\label{Amu}
A_u (t,M_\xi,m):=\Const (\chi_u^{-1} \Delta_{\tilde\gamma} )^{\const} \vartheta_{\tilde\gamma}^{-1} \Lambda^{\const M_\xi}, t \ge 2.
\end{equation}
Lemma \ref{lem extension of curve} implies
\begin{equation}\label{eq:Delta-estim}
\begin{split}
&\Delta_{\tilde \gamma}\le C_{q_0,\epsilon}\vartheta(\xip)^{-1}\mu^{5m}C_{\mu,m}( \|\omega\|_{\cC^2}+\vartheta(\xip))\{1,\mu^{5m}C_{\mu,m}( \|\omega\|_{\cC^2}+\vartheta(\xip)\}^+\mu^m,\\
 &\vartheta_{\tilde\gamma}\geq \vartheta_\gamma =\{\rho(\xip), \chi_u\}^+=:\vartheta(\xip). 
\end{split}
\end{equation}
With the choice $\bbb=1$, then by equations \eqref{partition of unity psi_alpha} and \eqref{eq:part_der}, we have the bound $\|(\partial^{\eta_1}\bar\psi_{\bar\alpha,\frh})\circ F^{\ovm(\hat\frh)} \|_{\cC_{\hat \nu_\ell}^{|\eta_2|}}\le  C_{\eps,q_0}\Lambda^{\const t M_\xi}\mu^{\const M_\xi}\chi_u^{-\const t}$, 
by \eqref{uNk norm}, \eqref{sum on M of Psi}, \eqref{Amu}, \eqref{eq:Delta-estim} and \eqref{eq:horror} we conclude. 
\end{proof}

\subsection{Decomposition of Fourier space}\ \\
 Let $\cZ_u=\{\xi: \xip \in \fC_u\}$ and $\cZ_u^c=\bZ^2\setminus \cZ_u.$ Recalling that  $\rho(\xip)={|\xip_2|}{|\xip_1|^{-1}}, \rho(e_2)=\infty$,
\[
\cZ_u=\{\xi: \rho(\xip)\le \chi_u\} \quad; \quad \cZ_u^c=\{\xi: \rho(\xip)> \chi_u\}.
\]
Next, take $N=q_0+M$, for some $M\in\bN$ to be chosen shortly. 
For simplicity, it is convenient to introduce the following notation for $ A\subset \bZ^2, \frh,\frh'\in \frH^{q_0}$:
\begin{equation}\label{def-of-S}
S^{\balpha}_{q_0,M}(A,\frh,\frh')=\sum_{\xi\in \bZ^2}\mathds{1}_{A}(\xi)\langle \xi \rangle^{2s}  [\cF\cL^{q_0}(u_{\balpha,\frh}^M)](\xi)[\overline{\cF \cL^{q_0}(u_{\balpha,\frh'}^M)}](\xi),
\end{equation}
where $u_{\balpha,\frh}^M=\psi_{\balpha,\frh}\cL^Mu$.
Then,  using equation (\ref{partit of unit estim 2}) we have 
\begin{equation}
\label{decomp norm Z tranf}
\begin{split}
&\|  \mathcal{L}^N u \| _{\mathcal{H}^{s}}^2 \le \Const  \sum_{\balpha}\|\psi_\alpha\cL^{q_0}(\psi_{\alpha'}\cL^M u)\|^2_{\mathcal{H}^{s}} 
\le C_\sharp\sum_{\balpha} \big\| \sum_{\frh\in \frH^{q_0}}\cL^{q_0}(u_{\balpha,\frh}^M) \big\|^2_{\mathcal{H}^{s}}\\&= C_\sharp\sum_{\balpha}\sum_{(\frh, \frh')\in \frH^{q_0}\times \frH^{q_0}}\langle\cL^{q_0}(u_{\balpha,\frh}^M) , \cL^{q_0}(u_{\balpha,\frh'}^M)  \rangle_{s}\\
&= C_\sharp\sum_{\balpha}\sum_{(\frh, \frh')\in \frH^{q_0}\times \frH^{q_0}}\sum_{\xi\in \bZ^2}\langle \xi \rangle^{2s}  [\cF\cL^{q_0}(u_{\balpha,\frh}^M)](\xi)[\overline{\cF \cL^{q_0}(u_{\balpha,\frh'}^M)}](\xi) \\
&= C_\sharp\sum_{\balpha}\sum_{(\frh, \frh')\in \frH^{q_0}\times \frH^{q_0}}S^{\balpha}_{q_0,M}(\cZ_u,\frh,\frh')+ C_\sharp\sum_{\balpha}\sum_{(\frh, \frh')\in \frH^{q_0}\times \frH^{q_0}}S^{\balpha}_{q_0,M}(\cZ_u^c,\frh,\frh').
\end{split}
\end{equation} 
We start with the second term in the above equation, next we will treat the term with $\xi\in\cZ_u$.
    \begin{lem}[Bound on $\cZ_u^c$]\label{lem on bound on cZc} Recall by \eqref{def of ovmchiu} that $ \ovm_{\chi_u}:=\sup_{\xi\in \cZ^c_u} M_{\xi}$.\footnote{ Recall that this is finite by Proposition \ref{prop on four coeff}.} For each $M\ge \ovm_{\chi_u}$, $1\le s \le r-1$, $\frh\in\frH^{q_0}$ and $N=q_0+M,$
\begin{equation}
\label{norm bound in cZc}
\sum_{\xi\in\bZ^2}| \langle \xi \rangle^s \mathds{1}_{\cZ_u^c}(\xi)  [\cF \cL^{q_0}(u_{\balpha,\frh}^M)](\xi)|^2 \lesssim \Theta_s \|u\|^2_{s+2},
\end{equation}
where $\Theta_s=C_{q_0}C_{\mu,M}^6 \mu^{\const M}\Lambda^{\const M}$.

\end{lem}
\begin{proof} Remark that $\Lambda^M\geq \mu^{-M}\chi_u^{-1}$. Since $\xip\notin \fC_u$ we can apply Proposition \ref{prop on four coeff} with $n_0=0$. For each $M\ge M_\xi$, by \eqref{control four coeff}, we have
\begin{equation}
\label{final estimate for the case s>1}
\begin{split}
&\sum_{\xi\in\bZ^2}| \langle \xi \rangle^s \mathds{1}_{\cZ_u^c}   \cF \cL^{q_0}(u_{\balpha,\frh'}^M)|^2=\sum_{\xi\in\bZ^2} \langle \xi \rangle^{-4}|\langle \xi \rangle^{s+2} \mathds{1}_{\cZ_u^c}   \cF \cL^{q_0}(\psi_{\bar\alpha,\frh'}\cL^{M_{\xi}}(\cL^{M-M_\xi} u))|^2\\
&\lesssim C_{q_0}C_{\mu,M}^4 \mu^{\const M}\Lambda^{\const M}\|\cL^{M-M_\xi}u\|^2_{s+2}.
\end{split}
\end{equation}
The statement (\ref{norm bound in cZc}) for $s>1$ follows since, by Corollary \ref{Corollario 1},
\begin{equation}\label{eq:MMxi}
\|\cL^{M-M_\xi}u\|^2_{t}\le C_{\mu, M}^2\mu^{2M} \|u\|^2_{t}, \quad\forall t\ge 1.
\end{equation}
\end{proof} 
\subsection{The case \texorpdfstring{$\xip \in \fC_u$}{Lg} \ (\texorpdfstring{$\xi\in\cZ_u$}{lgg})}\ \\
 In this case we cannot apply Proposition \ref{prop on four coeff} directly as we did in the previous section. The reason is that $\xip$ could belong to the unstable direction, so its preimage never enters the central cone. Here transversality plays a major role.
\begin{lem}[Bound on $\cZ_u$]\label{lem bound-on-Zu}
{If there exists $q_0\in\bN$ such that for each $\xi\in\bZ$ the hypothesis of Proposition \ref{prop on four coeff} are satisfied, then} there exist $C_{q_0}$ such that, for each $M\ge \ovm_{\chi_u}$ and each $\delta\in (0,1)$,
\begin{equation*}
\begin{split}
\sum_{(\frh, \frh')\in \frH^{q_0}\times \frH^{q_0}}\sum_{\xi\in \bZ^2}&\mathds{1}_{\cZ_u}(\xi)\langle \xi \rangle^{2s}  [\cF\cL^{q_0}(u_{\balpha,\frh}^M)](\xi)[\overline{\cF \cL^{q_0}(u_{\balpha,\frh'}^M)}](\xi)\\
\le&  \left(\cN(q_0)\mu^{2sq_0}+\delta\right) \sum_{\frh\in\frH^{q_0}}\|u^M_{\balpha,\frh}\|_{\cH^s}^2+ C_{q_0,s} \delta^{-1}\sum_{\frh\in\frH^{q_0}}\|u^M_{\balpha,\frh}\|^2_{\cH^{s-1}}\\
&+ C_{q_0}Q(M, s)\sqrt{\Theta_s}\|u\|_{s+2}\|u\|_{\cH^s},
\end{split}
\end{equation*}
where $Q(M,s)$ is given in (\ref{LY Hs and Hs-1})
and $\Theta_s$ in Lemma \ref{lem on bound on cZc}.
\end{lem} 
The rest of this Section is devoted to the proof of the above Lemma. We argue in three Steps.
\subsubsection{\bfseries Step I (Local transversality)} 
First we need a definition of transversality uniform on the elements of the partition of unity (\ref{partition of unity psi_alpha}):
\begin{defn}
\label{def of transv part un}
Given $n\in\bN$ and $\frh, \frh' \in \frH^n$ we say that $\frh\pitchfork_{\alpha}^n \frh'$ ($\frh$ is transversal to $\frh'$ on $\alpha$ at time $n$) if for every $z\in \frh(\cU_\alpha)$ and $w\in \frh'(\cU_\alpha)$ such that $F^{n}(z)=F^{n}(w)\in \cU_\alpha:$
\begin{equation}
\label{transv cond part of un}
  D_{z} F^n \mathbf{C}_{\epsilon, u} \cap D_{w} F^n \mathbf{C}_{\epsilon, u}=\lbrace 0 \rbrace . 
\end{equation}   
\end{defn}
\noindent Next, we relate the (pointwise) Definition \ref{transversality} to the (local)  Definition \ref{def of transv part un}. 
\begin{lem}
\label{lem on transv in terms of partitions}
The constant $C_0$ in \eqref{constant bbd} can be chosen such that: for all $\alpha\in \cA$, $p\in \cU_\alpha \subset \bT^2$ and $\frh, \frh' \in\frH^{q_0}$ if $z_1=\frh(p)$ and $z_2=\frh'(p)$, then $z_1\pitchfork z_2$ implies $\frh\pitchfork^{q_0}_{\alpha} \frh'$.
\end{lem}
\begin{proof} Let $p\in \cU_\alpha \subset \bT^2$ and $\frh, \frh' \in\frH^{q_0}$, if $z_1=\frh(p)$ and $z_2=\frh'(p)$, recall that $z_1 \pitchfork z_2$ means
\begin{equation}
\label{z_1 transv z_2}
D_{z_1}F^{q_0}\fC_{u}\cap D_{z_2}F^{q_0}\fC_{u}=\{0\}.
\end{equation}
As $\fC_{u, \epsilon}\Subset \fC_u$, clearly $D_{z_1}F^{q_0}\fC_{u, \epsilon}\Subset D_{z_1}F^{q_0}\fC_{u}$. So the above implies also
\begin{equation*}
D_{z_1}F^{q_0}\fC_{u, \epsilon}\cap D_{z_2}F^{q_0}\fC_{u, \epsilon}=\{0\}.
\end{equation*}
Let $\tilde p\in\cU_\alpha$, $\tilde p\neq p$, and define $\tilde z_1=\frh(\tilde p)$ and $\tilde z_1=\frh'(\tilde p)$. We claim that, for each $v\in \fC_{u, \epsilon}$,  the difference between $D_{z_1}F^{q_0}v$ and $D_{\tz_1}F^{q_0}v$ is smaller than the difference between $D_{z_1}F^{{q_0}}\fC_u$ and $D_{z_1}F^{{q_0}}\fC_{u,\eps}$, provided we choose $\cU_\alpha$ small enough. This suffices to conclude the argument. 

We compute a lower bound for the opening of the connected components of  $D_{z_1}F^{{q_0}}\fC_u\setminus D_{z_1}F^{{q_0}}\fC_{u,\epsilon}$. By Proposition \ref{prop on the det}, and by formula \eqref{wadge formula}, we deduce that, for each unitary vectors $v\in \fC_{u, \epsilon}$ and $w\not\in \fC_{u}\cup\fC_c$,
\[
\begin{split}
\measuredangle(D_{z_1}F^{{q_0}}v, D_{z_1}F^{{q_0}}w)&=\frac{|\det D_{z_{1}} F^{q_0}| \measuredangle(v,w)}{\|D_{z_1} F^{q_0} v\| \|D_{z_1} F^{q_0} w\|}\ge \frac{ C_*\chi_u\epsilon}{\mu_{q_0}^-(z)\lambda_{q_0}^+(z)}= { C_*\chi_u\epsilon}\delta_{q_0}(z_1).
\end{split}
\]
On the other hand let us recall that $u_{\frh, q_0}(p)$ defined in (\ref{slope of the cone general case}) gives the slope of the boundary of the cone $D_{\frh_\alpha(p)}F^{q_0} \fC_u$, and it is a Lipschitz function of $p$. In particular, Lemma \ref{lem extension of curve} provides an estimate for the Lipschitz constant $L_{\star}(q_0)$ given in \eqref{Lstar}.
Then, by the definition of $\cU_{z,{q_0}}$ in (\ref{U_z}) and \eqref{constant bbd}, we have the claim, since
\[
\begin{split}
\|D_{z_1}F^{q_0}v-D_{\tilde z_1}F^{q_0}v\|&\le L_\star (q_0)\|z_1-\tilde{z}_1\|\le L_\star (q_0)L_\star(\chi_u, q_0)^{-1}C_0\chi_u\epsilon\delta_{q_0}(z_1)\\
&\le \Const C_0\chi_u\epsilon\delta_{q_0}(z_1).
\end{split}
\]
Clearly the same is true replacing $z_1, \tz_1, \frh$ with $z_2, \tz_2, \frh'$, and the result follows. 
\end{proof}

It is thus natural to work with the local transversality.
Using Definition \ref{def of transv part un}, and recalling notation \eqref{def-of-S},
we introduce the following decomposition into transversal and non transversal terms, which we will estimate separately,
\begin{equation}
\label{transversal + non transversal}
\begin{split}
\sum_{(\frh, \frh')\in \frH^{q_0}\times \frH^{q_0}}&S^{\balpha}_{q_0,M}(\cZ_u, \frh, \frh')\\
&=\sum_{ \frh {\pitchfork}^{q_0}_{\alpha} \frh'}S^{\balpha}_{q_0,M}(\cZ_u, \frh, \frh')+\sum_{ \frh \not{\pitchfork}^{q_0}_{\alpha} \frh'}S^{\balpha}_{q_0,M}(\cZ_u, \frh, \frh').
\end{split}
\end{equation}
\subsubsection*{\bfseries Step II (Estimate of transversal terms)} 
In this step we will prove that
\begin{equation}\label{transversal terms inequality} 
\sum_{ \frh {\pitchfork}^{q_0}_{\alpha} \frh'} S^{\balpha}_{q_0,M}(\cZ_u, \frh, \frh') \le C_{q_0}Q(M, s)\sqrt{\Theta_s}\|u\|_{s+2}\|u\|_{\cH^s},
\end{equation}
where $\Theta_s$ is given in Lemma \ref{lem on bound on cZc}.
\\
If $\frh \pitchfork_\alpha^{q_0}\frh'$, then for any $\xi\in \bZ^2\setminus\{0\}$, either $\xi^\perp\not \in  D_{\frh(p)} F^{q_0} \mathbf{C}_{\epsilon ,u}$ or $ \xi^\perp\not \in D_{\frh'(p)} F^{q_0}\mathbf{C}_{\epsilon, u}$, for all $p\in \mbox{supp}({\psi_{\balpha, \frh}})$.
We can then decompose $\cZ_u=Z(\frh) \cup Z(\frh')$, where
\begin{align}\label{Zfrh}
&Z(\frh)=\lbrace \xi \in \cZ_u : \xi^\perp\not \in  D_{\frh(p)} F^{q_0} \mathbf{C}_{\epsilon ,u} \quad \forall p\in \text{supp}\psi_{\balpha,\frh} \rbrace,
\end{align}
and we write, recalling \eqref{def-of-S},
\begin{equation}
\label{split in Z1 and Z2}
\begin{split}
&S^{\balpha}_{q_0,M}(\cZ_u, \frh, \frh')=S^{\balpha}_{q_0,M}(\cZ_u\cap Z(\frh),\frh, \frh')+S^{\balpha}_{q_0,M}(\cZ_u\cap Z( \frh'),\frh, \frh').
\end{split}
\end{equation}
It is enough to estimate the first addend, the second being analogous. By the Cauchy-Schwarz inequality we have
\begin{equation}
\label{after CS ineq}
\begin{split}  	
&|S^{\balpha}_{q_0,M}(\cZ_u\cap Z(\frh), \frh, \frh')|\leq  \left(\sum_{\xi\in\bZ^2}| \langle \xi \rangle^s \mathds{1}_{\cZ_u\cap Z(\frh)}   \cF \cL^{q_0}(u_{\balpha,\frh}^M)|^2 \right)^{\frac{1}{2}} \| \mathcal{L}^{\nzero} (u_{\balpha,\frh'}^M)\|_{\cH^s}.
 \end{split}
\end{equation} 
Moreover, by (\ref{LY Hs and Hs-1}), $\| \mathcal{L}^{\nzero} (u_{\balpha,\frh'}^M)\|_{\cH^s}\le \Const Q(M, s)\|u\|_{\cH^s}$. We can bound the sum inside the square root following exactly the same argument of the proof of Lemma \ref{lem on bound on cZc}, since the key condition $\xi\in\cZ_u^c$ is now replaced by $\xi\in Z(\frh)$,  with the difference that this time $\vartheta(\xip)=\chi_u$, since $\xi\in \cZ_u$.
We thus have 
\[
|S^{\balpha}_{q_0,M}(\cZ_u\cap Z(\frh), \frh, \frh')| \lesssim Q(M, s)\sqrt{\Theta_s}\|u\|_{s+2}\|u\|_{\cH^s}.
\]
Summing over $\frh\pitchfork^{q_0}_\alpha \frh'$, we conclude the proof of (\ref{transversal terms inequality}). 

\subsubsection*{\bfseries Step III (Estimate of non-transversal terms).}
We now want to estimate the $\frh\not\pitchfork_{\alpha}^{q_0} \frh'$ terms in \eqref{transversal + non transversal}. Our aim is to prove that
\begin{equation}
\label{non trans inequa}
\begin{split}
\sum_{\frh \not{\pitchfork}^{\nzero}_{\alpha} \frh'} \left|\langle \mathcal{L}^{q_0} (u^M_{\balpha,\frh}), \mathcal{L}^{q_0} (u^M_{\balpha,\frh'}) \rangle_{s}\right| \lesssim & \,\cN(q_0)\mu^{2sq_0} \sum_{\frh\in\frH^{q_0}}\|u^M_{\balpha,\frh}\|_{\cH^s}^2\\
&+ C_{q_0} \sum_{\frh\in\frH^{q_0}}\|u^M_{\balpha,\frh}\|^2_{\cH^{s-1}}.
\end{split}
\end{equation}
Keeping the same notation used previously, we write
\begin{equation}
\label{crude estimate}
\begin{split}
&\sum_{\frh \not{\pitchfork}^{\nzero}_{\alpha} \frh'} \langle \mathcal{L}^{q_0} (u^M_{\balpha,\frh}), \mathcal{L}^{q_0} (u^M_{\balpha,\frh'}) \rangle_{s}= \sum_{\frh\in \frH^{\nzero}}\sum_{\frh':\frh' \not{\pitchfork}^{\nzero}_{\alpha} \frh} \langle \mathcal{L}^{q_0} (u^M_{\balpha,\frh}), \mathcal{L}^{q_0} (u^M_{\balpha,\frh'}) \rangle_{s}. \\
\end{split}
\end{equation}
By equation  \eqref{eq:horm}, there are $C_{\gamma, \beta}$ such that 
\begin{equation}
\label{crude estimate 2}
\langle \mathcal{L}^{q_0} u^M_{\balpha,\frh},  \mathcal{L}^{q_0} (u^M_{\balpha,\frh'}) \rangle_{s}
= \sum_{\gamma+\beta\leq s}C_{\gamma, \beta} \langle \partial_{x_1}^\gamma \partial_{x_2}^{\beta} \mathcal{L}^{q_0} (u^M_{\balpha,\frh}),  \partial_{x_1}^\gamma\partial_{x_2}^{\beta}   \mathcal{L}^{q_0} (u^M_{\balpha,\frh'})\rangle_{L^{2}}.
\end{equation}
We then use equation \eqref{bound of pull backs} and we have,  for every $\gamma, \beta$ such that $\gamma+\beta\leq s$
\[
\begin{split}
|\partial_{x_1}^\gamma \partial_{x_2}^{\beta} (  \mathcal{L}^{\nzero} u^M_{\balpha,\frh})| 
&\le \|(DF^{q_0})^{-1}\|^{s}_\infty \cL^{q_0}(|\partial_{x_1}^\gamma \partial_{x_2}^{\beta} u^M_{\balpha, \frh}|)+\cL^{q_0}(P^{q_0}_{s-1} u^M_{\balpha, \frh})
\end{split}
\]
where $P^{q_0}_{s-1} $ is a differential operator of order $s-1$.
By \eqref{lambda_+<Clambda_-} $\|(DF^{q_0})^{-1}\|^{s}_\infty\le C\mu^{sq_0}$. Clearly the same inequality holds for $\frh'$ and we use this in (\ref{crude estimate 2}) to obtain 
\[
\begin{split}&
\sum_{\gamma+\beta\leq s}C_{\gamma, \beta}\left|\langle \partial_{x_1}^\gamma \partial_{x_2}^{\beta} \mathcal{L}^{q_0} (u^M_{\balpha,\frh}),  \partial_{x_1}^\gamma\partial_{x_2}^{\beta}   \mathcal{L}^{q_0} (u^M_{\balpha,\frh'})\rangle_{L^{2}}\right|\\
&\lesssim\mu^{2sq_0}\sum_{\gamma+\beta=s}C_{\gamma, \beta} \langle \cL^{q_0}( |\partial_{x_1}^\gamma\partial_{x_2}^{\beta} u^M_{\balpha, \frh}|), \cL^{q_0}(|\partial_{x_1}^\gamma \partial_{x_2}^{\beta} u^M_{\balpha, \frh'}|) \rangle_{L^2}\\
&\phantom{\lesssim}
+C_{q_0}\sum_{\gamma+\beta=s}\| \cL^{q_0}( |\partial_{x_1}^\gamma \partial_{x_2}^{\beta} u^M_{\balpha, \frh}|)\|_{L^2}\|u^M_{\balpha, \frh'}\|_{\cH^{s-1}}\\
&+C_{q_0}\sum_{\gamma+\beta=s}\| \cL^{q_0}( |\partial_{x_1}^\gamma \partial_{x_2}^{\beta} u^M_{\balpha, \frh'}|)\|_{L^2}\|u^M_{\balpha, \frh}\|_{\cH^{s-1}}+C_{q_0}\|u^M_{\balpha, \frh}\|_{\cH^{s-1}}\|u^M_{\balpha, \frh'}\|_{\cH^{s-1}}.
\end{split}
\]
It follows that, for each $\delta>0$,
\begin{equation}
\label{crude estimate 3}
\begin{split}&
\sum_{\gamma+\beta\leq s}C_{\gamma, \beta} \left|\langle \partial_{x_1}^\gamma \partial_{x_2}^{\beta} ( \mathcal{L}^{q_0} (u^M_{\balpha,\frh}), \partial_{x_1}^\gamma \partial_{x_2}^{\beta}  \mathcal{L}^{q_0} (u^M_{\balpha,\frh'})\rangle_{L^{2}}\right|\\
&\lesssim\mu^{2sq_0}\sum_{\gamma+\beta=s}C_{\gamma, \beta} \langle \cL^{q_0}( |\partial_{x_1}^\gamma \partial_{x_2}^{\beta} u^M_{\balpha, \frh}|), \cL^{q_0}(|\partial_{x_1}^\gamma \partial_{x_2}^{\beta} u^M_{\balpha, \frh'}|) \rangle_{L^2}\\
&\phantom{\lesssim}
+\delta\left[\| u^M_{\balpha, \frh}\|_{\cH^{s}}^2+\| u^M_{\balpha, \frh'}\|_{\cH^{s}}^2\right]+C_{q_0,s}\delta^{-1}\left[\| u^M_{\balpha, \frh}\|_{\cH^{s-1}}^2+\|u^M_{\balpha, \frh'}\|_{\cH^{s-1}}^2\right].
\end{split}
\end{equation}
Since $u^M_{\balpha,\frh}$ and $u^M_{\balpha,\frh'}$ are supported on invertibility domains of $F^{q_0},$ 
\begin{equation}
\label{defin of chi e g}
\cL^{q_0}|(\partial_{x_1}^\gamma \partial_{x_2}^{\beta} u^M_{\alpha, \tau}|)=\frac{ |\partial_{x_1}^\gamma \partial_{x_2}^{\beta} u^M_{\alpha, \tau}|\circ \tau}{|\det DF^{q_0}|\circ \tau}, \quad \tau\in\{\frh, \frh'\}.
\end{equation}
We define $\chi_\tau:=|\partial_{x_1}^\gamma \partial_{x_2}^{\beta}u^M_{\alpha, \tau}|\circ {\tau}$ and $g_\tau:=|\det DF^{N}|	\circ {\tau}$ and we have
\begin{equation}
\label{integrals in chi and g}
\begin{split}
\langle \cL^{q_0}( |\partial_{x_1}^\gamma \partial_{x_2}^{\beta} u^M_{\balpha, \frh}|), \cL^{q_0}(|\partial_{x_1}^\gamma \partial_{x_2}^{\beta} u^M_{\balpha, \frh'}|) \rangle_{L^2}&=\int_{\bT^2} \frac{\chi_\frh \chi_{\frh'}}{\sqrt{g_{\frh}g_{\frh'}}\sqrt{g_\frh g_{\frh'}}}\\
&\le \frac{1}{2}\int_{\bT^2}\frac{\chi_{\frh}^2}{g_{\frh'} g_{\frh}}+\frac{1}{2}\int_{\bT^2} \frac{\chi_{\frh'}^2}{g_{\frh'}g_{\frh}},
\end{split}
\end{equation}
where we used the elementary inequality $ab\le \frac{1}{2}(a^2+b^2)$ with $a=\frac{\chi_\frh}{\sqrt{g_\frh g_{\frh'}}}, b=\frac{\chi_{\frh'}}{\sqrt{g_\frh g_{\frh'}}}.$ 
In order to obtain \eqref{non trans inequa}, we need to sum equation \eqref{crude estimate 3} over $\frh\in \frH^{q_0}$ and $\frh'\not\pitchfork^{q_0}_\alpha \frh$. Let us begin with the first term. Consider one of the integrals in (\ref{integrals in chi and g}), for example the first one. Recalling $\cN(q_0)$ as in Definition \ref{transversality} and equation  \eqref{def of N}, Lemma \ref{lem on transv in terms of partitions} implies
\begin{equation}
\label{estimate of integrals in chi and g}
\begin{split}
\sum_{\frh}\sum_{\frh':\frh' \not{\pitchfork}^{q_0}_{\alpha} \frh}\int_{\bT^2}&\frac{\chi_\frh^2}{g_{\frh'} g_\frh}\le  \sum_{\frh}\int_{\bT^2} \frac{\chi^2_{\frh}}{g_{\frh}}\sum_{\frh':\frh' \not{\pitchfork}^N_{\alpha} \frh} \frac{1}{g_{\frh'}}\le \cN(q_0)\sum_\frh \int_{\bT^2} \frac{|\partial_{x_1}^\gamma \partial_{x_2}^{\beta}u^M_{\balpha,\frh}|^2\circ \frh}{|\det DF^{N}	|\circ\frh}\\
 &=\cN(q_0)\sum_{\frh}\|\cL^{q_0} |\partial_{x_1}^\gamma \partial_{x_2}^{\beta}u^M_{\balpha, \frh}|^2\|_{L^{1}}= \cN{(q_0)}\sum_\frh \|\partial_{x_1}^\gamma \partial_{x_2}^\beta u^M_{\balpha, \frh}\|_{L^2}^2.
\end{split}
\end{equation}
By symmetry we have
\begin{equation}
\begin{split}
& \mu^{2sq_0} \sum_{\frh \not{\pitchfork}^{q_0}_{\alpha} \frh'}\sum_{\gamma+\beta=s}C_{\gamma, \beta}\left| \langle \cL^{q_0}( |\partial_{x_1}^\gamma \partial_{x_2}^{\beta} u^M_{\balpha, \frh}|), \cL^{q_0}(|\partial_{x_1}^\gamma \partial_{x_2}^{\beta} u^M_{\balpha, \frh'}|) \rangle_{L^2}\right|\\
&\leq \mu^{2sq_0}\cN(q_0)\sum_{\frh\in\frH^{q_0}}\sum_{\gamma+\beta=s}\|\partial_{x_1}^\gamma \partial_{x_2}^\beta u^M_{\balpha, \frh}\|_{L^2}^2 \le \mu^{2sq_0}\cN(q_0) \sum_{\frh\in\frH^{q_0}}\|u^M_{\balpha, \frh}\|_{\cH^s}^2,
\end{split}
\end{equation}
which corresponds to the first addend of the r.h.s. of (\ref{non trans inequa}).\\
Summing the other terms of \eqref{crude estimate 3} over $\frh$ yields \eqref{non trans inequa} which, together with  \eqref{transversal terms inequality}, conclude the proof of Lemma \ref{lem bound-on-Zu}. \\

\subsection{Proof of Theorem \ref{thm: LY in Hs and Bs}}\label{section:Proof of Theorem thm: LY in Hs and Bs}\ \\
By (\ref{decomp norm Z tranf}) and Lemmata \ref{lem on bound on cZc} and \ref{lem bound-on-Zu}, we have 
\begin{equation}
\label{final step}
\begin{split}
&\|\mathcal{L}^N u \| _{\mathcal{H}^{s}}^2 \leq C_{q_0} \Theta_s\|u\|_{s+2}^2+ C_{q_0}Q(M,s)\sqrt{\Theta_s}\|u\|_{s+2}\|u\|_{\cH^{s}} \\
&+(\cN(q_0)\mu^{2sq_0}+\delta)\sum_{\balpha}\sum_{\frh\in\frH^{q_0}}\|u^M_{\balpha, \frh}\|_{\cH^s}^2+ C_{q_0,s}\delta^{-1} \sum_{\balpha}\sum_{\frh\in\frH^{q_0}}\|u^M_{\balpha, \frh}\|^2_{\cH^{s-1}}.
\end{split}
\end{equation}
Recalling that $u^M_{\balpha, \frh}=\psi_{\balpha, \frh}\cL^Mu,$ we can use equations \eqref{partit of unit estim 1} and \eqref{LY Hs and Hs-1} to write,\footnote{ We also use repeatedly $\|\cL^nu\|_{L^1}\le \|u\|_{L^1}.$}
\begin{equation}
\label{final step 1}
\begin{split}
\sum_{\balpha}\sum_{\frh\in\frH^{q_0}}&\|u^M_{\balpha, \frh}\|_{\cH^s}^2=\sum_{\balpha}\sum_{\frh\in\frH^{q_0}}\|\psi_{\balpha, \frh}\cL^Mu\|_{\cH^s}^2
\le \Const\|\cL^M u\|_{\cH^s}^2+C_\psi(s)\|\cL^M u\|_{L^1}^2\\
&\le C_{q_0}A_s\|\cL^M1\|_{\infty}\mu^{2sM}\|u\|_{\cH^s}^2+Q(M,s)\|u\|^2_{\cH^{s-1}}+C_\psi(s) \|u\|_{L^1}^2
\end{split}
\end{equation}
and
\begin{equation}
\label{final step 2}
\begin{split}
\sum_{\balpha}\sum_{\frh\in\frH^{q_0}}\|u^M_{\balpha,\frh}\|^2_{\cH^{s-1}}&\le \Const \|\cL^M u\|_{\cH^{s-1}}^2+C_\psi(s-1)\|\cL^M u\|_{L^1}^2\\
&\le \|\cL^M1\|_{\infty}Q(M,s-1)\|u\|^2_{\cH^{s-1}}+C_\psi(s-1) \|u\|_{L^1}^2.
\end{split}
\end{equation}
Next, choosing $\delta= \cN(q_0)$, using Lemma \ref{lem on norm s-1 and s} with $\varsigma=\cN(q_0)^2\mu^{2sq_0}Q(M, s)^{-1}C_{\eps,q_0}^{-1}$, substituting equations (\ref{final step 1}) and (\ref{final step 2}) in (\ref{final step}), setting 
\begin{equation}
\label{eq:barQ}
\overline Q(M,s)=\{\{1,\cN(q_0)^{-3}\}^+\mu^{2sq_0}Q(M, s)^2C_{\eps,q_0}\bL_M, C_\psi(s)\}^+,
\end{equation}
and recalling \eqref{def of supLn1} for the definition of $\bL_M$, we obtain
\begin{equation}
\label{gather 1}
\begin{split}
&\|  \mathcal{L}^N u \| _{\mathcal{H}^{s}}^2 \le \Const \bL_M\cN(q_0)\mu^{2sN}\|u\|^2_{\cH^s}+C_{q_0}\Theta_s\|u\|_{s+2}^2\\
&\phantom{\|  \mathcal{L}^N u \| _{\mathcal{H}^{s}}^2 \leq}
+C_{q_0}Q(M, s)\sqrt{\Theta_s}\|u\|_{s+2}\|u\|_{\cH^{s}}+\overline Q(M,s)\|u\|_{L^1}^2\\
&\leq \Const \bL_M\cN(q_0)\mu^{2sN}\|u\|^2_{\cH^s}+C_{\eps,q_0}\left[Q(M, s)^2 \bL_M\Theta_s\cN(q_0)^{-1} +\overline Q(M,s)\right] \|u\|_{s+2}^2,
\end{split}
\end{equation}
where, in the last line, we used \eqref{relation norm rho rho' and L1} to estimate $\|u\|_{L^1}\leq \Const \|u\|_{s+2}$. Recalling \eqref{eq:barQ}, it follows
\begin{equation*}
\begin{split}
\|\cL^N u\|_{\cH^s}&\le \Const\left(\sqrt{[\bL_M\cN(q_0)]^{\frac{1}{N}}\mu^{2s}}\right)^{N}\|u\|_{\cH^s}+C_{\eps,q_0}\sqrt{\overline Q(M,s)\Theta_s}\|u\|_{s+2},
\end{split}
\end{equation*}
from which, by equations \eqref{norm bound in cZc}, \eqref{partit of unit estim 1} and Lemma \ref{cor norm Hs of cL}, we obtain \eqref{LY Hs and Bs+2} in the case $s>1$.
\subsection{Fourier basic estimate: case \texorpdfstring{SVPH$^{\,\sharp}$}{SPVHs}}
 To prove Theorem \ref{thm: LY in Hs and BsS} we must improve the constant $K_1(s,M_\xi)$ in Proposition \ref{prop on four coeff} for $t=2$. This is done in the following proposition.
\begin{prop}
\label{prop on four coeff_sharp}
Under the assumption of Proposition \ref{prop on four coeff} if, in addition, the map $F$ is a SVPH$^{\,\sharp}$ and satisfies condition \eqref{key-conditions-for-F} for a uniform constant $C_5>0$, then there exist $ C_{n_0,q_0,\epsilon},\beta_1,\beta_2>0$ such that the conclusions of Proposition \ref{prop on four coeff} hold with
\begin{equation}
\label{K_1(N,M,t)}
K_1(2,M_\xi)\le  C_{n_0,q_0,\epsilon}(\ln\chi_u^{-1})^{\beta_1}\mu^{\beta_2 \ln\chi_u^{-1}} \vartheta(\xip)^{-7} .
\end{equation}

\end{prop}
\begin{proof}
We follow the proof of Proposition \ref{prop on four coeff} word by word until the definition of $\overline m_\ell$ which we define here differently in the following way.\\
Let $\overline m_\ell(z,\tilde \frh\circ \frh, n_\star)$ {be the smallest integer satisfying} \eqref{eq:cond for bar m}, for  $z\in\gamma_\ell$, with $n_\star= \{\hat n,\min\{c^-_2\log\chi_u^{-1}, m\}\}^+$, where $m$ was defined right after \eqref{eq:disintegration}. In addition, set  $\ovm_\ell(\tilde \frh)=\sup_{z\in\gamma_\ell}\ovm_\ell(z,\tilde \frh\circ \frh, n_\star)$ and define 
\begin{equation}\label{def of Mxi3}
    M_\xi \equiv M_{\xi,2}= \sup_\ell\sup_{\tilde\frh \in \frH^\infty}\ovm_\ell (\tilde\frh).			
\end{equation}
We can follow the proof of Proposition \ref{prop on four coeff} literally and we see that the main task is to estimate $ \sum_{\hat\frh \in \hat\frH_\alpha}\|\Psi_{\hat \nu_\ell}(\partial^{\eta_1}\psi_{\alpha})\circ F^{\ovm(\hat\frh)} \|_{\cC_{\hat \nu_\ell}^2}$ in  \eqref{sum on M of Psi}, and so, to have a sharper bound of $A_u(t,M_\xi,m)$ for $\tau=2$.\\
To this goal, we use the improvements Proposition \ref{operator bound rho norm_sharp} and Lemma \ref{lem on det^-1-for-unstable-curve_sharp} instead of  Proposition \ref{operator bound rho norm} and Lemma \ref{lem on det^-1-for-unstable-curve} respectively to get,  by \eqref{eq: constant in the s=2 case}, setting $I_{\hat\nu_\ell}=h_{\ell\bar m}(I_{q_0})$,
\begin{equation}\label{eq:tildeC}
\tilde C(2,\ovm,m)\le C_{\mu,\ovm}^3\mu^{4\ovm}\sup\limits_{\zeta\in \hat\nu_\ell(I_{\hat\nu_\ell})}(1+C_F\lambda^+_{\ovm}(\zeta))\{\lambda^+_m(F^{\ovm-m}(\zeta))+C_{q_0}\Delta_{\tilde \gamma}\},
\end{equation}
and, accordingly, we have equation \eqref{sum on M of Psi} with
\begin{equation}
\label{AmuS}
A_u (\tau,M_\xi,m):=\begin{cases}
\Const\bI_{\tilde\gamma, M_\xi}\bJ_{\tilde\gamma,M_\xi}\quad &\tau=0\\
\Const(\bI_{\tilde\gamma, M_\xi})^2\bJ_{\tilde\gamma,M_\xi}\quad &\tau=1\\
C_{\mu,M_\xi}^4\mu^{5M_\xi}(1+C_F\mu^{m}\vartheta_{\hat\nu_0}^{-1})O^\star_m\bI_{\tilde\gamma,m}^3\bJ_{\tilde\gamma,M_\xi} &\tau=2
\end{cases}
\end{equation} 
First note that, by \eqref{eq:part_der}, \eqref{constant bbd} and \eqref{Lstar} we can choose $\bbb$ large enough so that $\bbb\ln\chi_u^{-1}>m$ and hence
\begin{equation}\label{eq:psia_est}
\|\bar\psi_{\bar\alpha, \frh}\|_{\cC^s}\le  C_{\epsilon,q_0}C_{\mu,\bbb\ln\chi_u^{-1}}^{4s}\mu^{4s\bbb\ln\chi_u^{-1}}.
\end{equation}
Next, let $g_\alpha=\partial^{\eta_1}\bar\psi_{\bar\alpha, \frh}$, and $G_\alpha(s)=g_{\alpha}\circ \tilde\gamma_\ell \circ h_{\ell,\ovm}(s)$. Then, since $F^\ovm \hat \nu_\ell=\tilde\gamma_\ell \circ h_{\ell,\ovm}$, we have $g_\alpha\circ F^\ovm\circ \hat\nu_\ell=G_\alpha$ and if we choose 
\[
\begin{split}
G_\alpha'&=\langle \nabla g_\alpha \circ \tilde\gamma_\ell \circ h_{\ell,\ovm}, \tilde\gamma_\ell' \circ h_{\ell,\ovm}h_{\ell,\ovm}' \rangle\\
G_{\alpha}''&=\langle (D^2 g_\alpha)\circ\tilde\gamma_\ell' \circ h_{\ell,\ovm}\cdot  \tilde\gamma_\ell' \circ h_{\ell,\ovm} , \tilde\gamma_\ell' \circ h_{\ell,\ovm} \rangle(h_{\ell,\ovm}')^2\\
&+\langle\nabla g_\alpha\circ \tilde\gamma_\ell \circ h_{\ell,\ovm} , \tilde\gamma_\ell'' \circ h_{\ell,\ovm}(h_{\ell,\ovm}')^2+\tilde\gamma_\ell'\circ h_{\ell,\ovm}h_{\ell,\ovm}'' \rangle.
\end{split}
\]
Note that by \eqref{derivative of extended curve tildegamma} and the first of \eqref{eq:Delta-estim} we have
\[
\|\tilde \gamma_\ell''\|_\infty\le C_{q_0,\epsilon}\mu^{11m}C_{\mu,m}^2.
\]
Moreover, \eqref{key-conditions-for-F} and \eqref{def:C_F}, imply $C_F\le \Const \chi_u$.
We can use this to compute, using \eqref{eq:Delta-estim}, \eqref{def of eta_nstar and M(t,n_0,m)},\footnote{ Recall that $q_0\ge n_0$.}
\begin{equation}\label{eq:M-estimate}
\begin{split}
&\Delta_{\tilde \gamma}\leq C_{q_0,\eps}\mu^{11 m}C_{\mu,m}^2\\
&M_{m,q_0}\le C_{q_0,\epsilon}C_{\mu,m}^2\mu^{15m}\vartheta_{\tilde \gamma}^{-1}\\
&\overline M_{m,q_0}\le C_{q_0,\epsilon}C_{\mu,m}^4\mu^{32m}\vartheta_{\tilde \gamma}^{-2}. 
\end{split}
\end{equation}
Accordingly, using also \eqref{eq: h_ovm' h_ovm''}(with $h_{\ovm}=h_{\ovm,\ell}$), \eqref{eq:Delta-estim}, \eqref{eq:psia_est} and \eqref{key-conditions-for-F}, we have
\begin{equation}\label{eq:C2-norm-psialpha}
\begin{split}
&\|(\partial^{2}\bar\psi_{\bar\alpha, \frh})\circ F^{\ovm(\hat\frh)} \|_{\cC_{\hat \nu_\ell}^0}\le   C_{\epsilon,q_0}C_{\mu,\bbb\ln\chi_u^{-1}}^{8}\mu^{8\bbb\ln\chi_u^{-1}} \\
&\|(\partial \bar\psi_{\bar\alpha, \frh})\circ F^{\ovm(\hat\frh)} \|_{\cC_{\hat \nu_\ell}^1}\le C_{\epsilon,q_0}C_{\mu,\bbb\ln\chi_u^{-1}}^{4}
\vartheta_{\tilde \gamma}^{-1}\mu^{M_\xi+4s\bbb\ln\chi_u^{-1}} \\
&\|\bar\psi_{\bar\alpha, \frh}\circ F^{\ovm(\hat\frh)} \|_{\cC_{\hat \nu_\ell}^2}\le C_{\eps,q_0}C_{\mu,\bbb\ln\chi_u^{-1}}^8 \mu^{3M_\xi+24\bbb\ln\chi_u^{-1}}\vartheta_{\tilde \gamma}^{-1}
\left\{\vartheta_{\tilde \gamma}^{-1}+{M}_{m,n_0}\right\}\\
&\phantom{\|\bar\psi_{\bar\alpha, \frh}\circ F^{\ovm(\hat\frh)} \|}
\leq C_{\eps,q_0}C_{\mu,\bbb\ln\chi_u^{-1}}^8 \mu^{3M_\xi+25\bbb\ln\chi_u^{-1}}\vartheta_{\tilde \gamma}^{-1}
\left\{C_{\mu,n}^2\mu^{15m}+C_{\mu,m}\mu^{6m}\vartheta_{\tilde \gamma}^{-1}\right\}\\
&\phantom{\|\bar\psi_{\bar\alpha, \frh}\circ F^{\ovm(\hat\frh)} \|}
\leq C_{\eps,q_0}C_{\mu,\bbb\ln\chi_u^{-1}}^{10} \mu^{3M_\xi+40 \bbb\ln\chi_u^{-1}}\vartheta_{\tilde \gamma}^{-2}.
\end{split}
\end{equation}
Since $(\Psi_{\hat \nu_\ell}G_\alpha)''=\Psi_{\hat \nu_\ell}''G_\alpha+2\Psi_{\hat \nu_\ell}'G_\alpha'+\Psi_{\hat \nu_\ell}G''_\alpha,$ by equations \eqref{sum on M of Psi}, \eqref{AmuS}, \eqref{eq:ThetagS} and \eqref{eq:C2-norm-psialpha}
\begin{equation}\label{eq:C2-norm-PsiG}
    \begin{split}
 \sum_{\hat\frh \in \hat\frH_\alpha}\|\Psi_{\hat \nu_\ell}(\partial^{\eta_1}\psi_{\alpha})\circ F^{\ovm(\hat\frh)} \|_{\cC_{\hat \nu_\ell}^2}\le &C_{\epsilon,q_0}
C_{\mu,\bbb\ln\chi_u^{-1}}^{10} \mu^{3M_\xi+40 \bbb\ln\chi_u^{-1}}\Big\{A_u(2,M_\xi,m)\\
 &+A_u(1,M_\xi,m)\vartheta_{\tilde \gamma}^{-1}+A_u(0,M_\xi,m)\vartheta_{\tilde \gamma}^{-2}\Big\}.
    \end{split}
\end{equation}
To conclude, we need to relate all the quantities to $\vartheta_{\tilde \gamma}$. First we notice that, by \eqref{eq:Delta-estim} and recalling \eqref{eq:chiulambda}, it follows that for each $\zeta\in {\bT}$.\\
\begin{equation}\label{eq:lambda_one}
\lambda_{m}^+\circ F^{\ovm-m}\circ \hat \nu_\ell(\zeta)\le \Const\vartheta_{\tilde \gamma}^{-1}\mu^{m}. 
\end{equation}
Notice that the expansion is almost constant for the points of interest. Indeed, by Lemma \ref{lem local expans}, for any $p, p_*\in \overline U=\cU_{\alpha,\frh }\cap \operatorname{Im}(\hat \nu_\ell)$ and $n\leq M_\xi$,
\[
\lambda^+_{n}(p) \le \lambda^+_{n}(p_*)e^{ \mu^{M_\xi} C_{\mu,M_\xi}\|p-p_*\|}
\leq \lambda^+_{n}(p_*)e^{\mu^{M_\xi} C_{\mu,M_\xi}\bbd\epsilon\delta_{q_0}(p_*)}.
\]
If we choose $\bbb$ so that $ \bbb\ln\chi_u^{-1}\geq M_\xi$, then \eqref{constant bbd} and \eqref{Lstar} imply
\begin{equation}\label{eq:distortion_lambda}
\lambda^+_{n}(p) \le\Const\lambda^+_{n}(p_*).
\end{equation}

Next, recall the choice $n_\star= \{\hat n,\min\{c^-_2\log\chi_u^{-1}, m\}\}^+$ in Lemma  \ref{lem stable vertic curves_sharp}.
To continue we need to check the conditions \eqref{eq:cond for bar m}. Recalling \eqref{eq:M-estimate}, the first of the \eqref{eq:cond for bar m} is implied by
\[
\lambda_{\ovm-m}^-(\hat \nu_m\circ  h_{\ovm-m}(t))\geq c_\flat c_{n_\star}^{\ovm}\mu^{2\ovm+15m-3n_\star}C_{q_0,\epsilon}C_{\mu,m}\vartheta_{\tilde \gamma}^{-1}
\]
The second is implied by
\[
c_\flat^2b_{n_\star}^{n_\star} \mu^{2\ovm+2n_\star}\lambda_{\ovm-m}^-(\gamma\circ h_{\ovm-m}(t))^{-1}C_{q_0,\epsilon}C_{\mu,m}\mu^{15m}\vartheta_{\tilde \gamma}^{-1}\leq \frac 12 b_{n_\star}^{2n_\star}\mu^{6n_\star}C_3
\]
which follows by
\[
\lambda_{\ovm-m}^-(\gamma\circ h_{\ovm-m}(t))\geq 2c_\flat^2 \mu^{2\ovm+15 m-4n_\star}C_{q_0,\epsilon}C_{\mu,m}\vartheta_{\tilde \gamma}^{-1}.
\]
And the third of  \eqref{eq:cond for bar m} is implied by
\[
c_\flat^2 c_{n_\star}^{\ovm}\mu^{3\ovm}(\lambda_{\ovm-m}^{-}(\gamma\circ h_{\ovm-m}(t)))^{-1} C_{q_0,\epsilon}C_{\mu,m}^4\mu^{32m}\vartheta_{\tilde \gamma}^{-2}\le \frac 12 C_{\mu,n_\star}^2\mu^{6n_\star} C_3^2
\]
which follows by
\[
\lambda_{\ovm-m}^{-}(\gamma\circ h_{\ovm-m}(t))\geq 2c_\flat^2 c_{n_\star}^{\ovm}\mu^{3\ovm+32 m-6n_\star} C_{q_0,\epsilon}C_{\mu,m}^2\vartheta_{\tilde \gamma}^{-2}.
\]
The above, due to  \eqref{eq:lambda_one} and recalling \eqref{Cmu,n}, are implied by
\[
\begin{split}
&\lambda_-^{\ovm}\geq \Const m c_{n_\star}^{\ovm}\mu^{15\ovm}\vartheta_{\tilde \gamma}^{-2}\\
&\lambda_-^{\ovm}\geq \Const m\mu^{14\ovm}\vartheta_{\tilde \gamma}^{-2}\\
&\lambda_-^{\ovm}\geq \Const m^2 c_{n_\star}^{\ovm}\mu^{30\ovm} \vartheta_{\tilde \gamma}^{-3},
\end{split}
\]
which, in turn, are satisfied if
\begin{equation}\label{eq:ovm_choice}
\left(\frac{\lambda_-}{\mu^{30}c_{n_\star}}\right)^{\ovm}\geq\Const m^2\vartheta_{\tilde \gamma}^{-3}.
\end{equation}
Recalling \eqref{eq:condition-hatn} and noting that $B_{n}\geq c_{n}^3$, the above is implied by $2^{\ovm}\geq \Const m^2\vartheta_{\tilde \gamma}^{-3}$
which implies $\ovm\leq C_6 \ln\vartheta(\xi^\perp)\leq C_6\ln\chi_u^{-1}$ for some constant $C_6$, hence
\begin{equation}\label{eq:Mxi}
M_\xi\leq C_6\ln\vartheta(\xi^\perp)^{-1}.
\end{equation}
Note that this implies $\sup_{\xip\notin \fC_u}M_{\xi}\leq C_6\ln\chi_u^{-1}<\infty$, which proves \eqref{def of ovmchiu} for $t= 2$.

The above allows us to make the choice $\bbb=C_6$. In addition, since $c_{n_\star}\leq e^{\const/\ln\vartheta_{\tilde\gamma}^{-1}}$ we have $c_{n_\star}^\ovm\leq \Const$. Hence, we can choose $\ovm$ so that 
\begin{equation}\label{eq:eq-for-lambda_ovm}
\begin{split}
\lambda_{\ovm}^+\circ\hat \nu_\ell(s)\leq  C_{\mu,\ovm}^2\mu^{30\ovm}\vartheta(\xip)^{-3}.
\end{split}
\end{equation}
By the above, \eqref{eq:tildeC}, \eqref{eq:lambda_one} and \eqref{eq:M-estimate}, we then have the estimate of the first term in the sum \eqref{eq:horror},
\begin{equation}\label{eq:estim tildeC2}
    \tilde{C}(2,\ovm(\frh),m)\le  C_{q_0,\epsilon}C_{\mu, M_\xi}^7 \mu^{45M_\xi}\vartheta(\xip)^{-3}.
\end{equation}
We now proceed with the bound of the other terms in \eqref{eq:horror}.
We want to use Lemma \ref{lem on det^-1-for-unstable-curve}, hence we need to estimate the quantities in \eqref{eq:ThetagS}. Our first task is to study $\vartheta_{\hat\nu_0, m}(s)$. By definition $\hat\nu_{n_0}(s)=(\sigma(s),s)$ for some function $\sigma$ and \eqref{def of vartheta_gamma} implies that $\vartheta_{\hat\nu_0, m}(s)=|\sigma'(s)|^{-1}$. Since, using the notation of Lemma \ref{lem unst curve pull back}, $F^{n_0}\hat\nu_{n_0}=\hat\nu_0\circ \bar h_{n_0}$, we have 
\[
C_{n_0} |\pi_2(\hat\nu_0'\circ \bar h_{n_0}(s))|\geq |\pi_2([D_{\hat\nu_{n_0}(s)}F^{n_0}]^{-1}\bar\nu_0'\circ \bar h_{n_0}(s)\bar h'_{n_0}(s))|=|\sigma'(s)|\geq C_{n_0} |\pi_2(\bar\nu_0'\circ \bar h_{n_0}(s))|.
\]
By the proof of Lemma \ref{lem extension of curve}, see \eqref{eq:hat gamma'}, we have that $\pi_2(\tilde \gamma_0')=\pi_2(\nu_0')$ is monotone, hence so is $\pi_2(\hat\nu_0'(s))=\pi_2(\nu_0'\circ \bar h_{n_0}(s))$. Thus, if $|t|\geq |s|$, we have
\begin{equation} \label{eq:maybe}
|\sigma'(s)|\leq C_{n_0}|\pi_2(\hat\nu_0'\circ \bar h_{n_0}(s))|\leq C_{n_0}|\pi_2(\hat\nu_0'\circ \bar h_{n_0}(t))|\leq C_{n_0}|\sigma'(t)|.
\end{equation}
Also the proof of Lemma \ref{lem extension of curve} implies that, for $|s|\geq \frac 14$, $|\sigma'(s)|\leq \const$. Accordingly,
\begin{equation}\label{eq:superpalla}
\begin{split}
\bJ_{\tilde\gamma, M_\xi}&=\int_{{\bT}}\left[\vartheta_{\hat \nu_{n_0},M_\xi}(s)\right]^{-1} ds\leq \Const+\int_{-\frac 14}^{\frac 1 4}\sup_{|s-t|\leq \const \chi_u\ln\chi_u^{-1}}|\sigma'(t)| dt\\
&\leq
\Const+\int_{-\const \chi_u\ln\chi_u^{-1}}^{\const \chi_u\ln\chi_u^{-1}}\vartheta(\xi^\perp)^{-1}+C_{n_0}\int_{-\frac 14}^{-\const \chi_u\ln\chi_u^{-1}}|\sigma'(s+\const \chi_u\ln\chi_u^{-1})|ds\\
&\phantom{\leq}
+C_{n_0}\int_{\const \chi_u\ln\chi_u^{-1}}^{-\frac 14}|\sigma'(s-\const \chi_u\ln\chi_u^{-1})|ds\\
&\leq C_{n_0}\left[1+\vartheta(\xi^\perp)^{-1}\chi_u\ln\chi_u^{-1}\right]\leq C_{n_0}\ln\chi_u^{-1},
\end{split}
\end{equation}
since, $\int_{{\bT}}|\sigma'|=\left|\int_{{\bT}}\sigma'\right|\leq 1$, since $\hat\nu_{n_0}'\not \in \fC_c$ and the curve does not wrap around the torus horizontally.
By the proof of Lemma \ref{lem extension of curve} we have that $\vartheta_{\hat\nu_0}(s)$ is monotone for $|s|\leq \frac 14$.

Next, by \eqref{eq:M-estimate} we have
\[
\begin{split}
&\bI_{\tilde\gamma,M_\xi}\le C_{\mu,M_\xi}^3\mu^{16 M_\xi}\vartheta_{\tilde \gamma}^{-1}\\
&O^\star_{M_\xi}\le C_{q_0,\epsilon}C_{\mu,M_\xi}^8 \mu^{60 M_\xi}\vartheta_{\tilde \gamma}^{-1}.
\end{split}
\]
Using the above estimates in \eqref{AmuS} it follows
\begin{equation}
    \begin{split}
        A_u(0,M_\xi,m)&\le C_{n_0}C_{\mu,M_\xi}^3\mu^{16M_\xi}\vartheta_{\tilde \gamma}^{-1}\ln\chi_u^{-1},\\
        A_u(1,M_\xi,m)&\le C_{n_0}C_{\mu,M_\xi}^6\mu^{32 M_\xi}\vartheta_{\tilde \gamma}^{-2}\ln\chi_u^{-1}\\
        A_u(2, M_\xi, m)&\le   C_{n_0,q_0,\epsilon}C_{\mu,M_\xi}^{21} \mu^{114 M_\xi}\vartheta_{\tilde \gamma}^{-4}\ln\chi_u^{-1}.
    \end{split}
\end{equation}

Substituting the above in \eqref{eq:C2-norm-PsiG} yields
\begin{equation}\label{eq:C2normGpsi_final}
\begin{split}
    \sum_{\hat\frh \in \hat\frH_\alpha}\|\Psi_{\hat \nu_\ell}(\partial^{\eta_1}\psi_{\alpha})\circ F^{\ovm(\hat\frh)} \|_{\cC_{\hat \nu_\ell}^2}\le&  C_{n_0,q_0,\epsilon}C_{\mu, \bbb\ln\chi_u^{-1}}^{10}C_{\mu,M_\xi}^{21}\\
    &\times\mu^{114 M_\xi+4\bbb\ln\chi_u^{-1}} \vartheta(\xip)^{-4}\ln\chi_u^{-1}.
\end{split}
\end{equation}
By \eqref{eq:horror}, \eqref{uNk norm}, \eqref{eq:estim tildeC2}, \eqref{eq:C2normGpsi_final}, ams since $|I_0|\leq C_{q_0,\eps}C_{\mu, \bbb\ln\chi_u^{-1}}^{-1}
\mu^{-\bbb\ln\chi_u^{-1}}$, there exist $ \beta_1,  \beta_2>0$ such that
\[
\begin{split}
\left|\cF\partial^{2}_{x_j}\cL^{q_0}(u^{M_\xi}_{\balpha,{\frh}})\right|\le C_{n_0,q_0,\epsilon}C_{\mu, \bbb\ln\chi_u^{-1}}^{9}C_{\mu,M_\xi}^{\beta_1}\mu^{\beta_2 M_\xi+3\bbb\ln\chi_u^{-1}} \vartheta(\xip)^{-7}\ln\chi_u^{-1} \|u\|_2,
\end{split}
\]
which concludes the proof of \eqref{K_1(N,M,t)}, recalling equations \eqref{eq: Fourier 1}, \eqref{eq:Mxi}.\\
\end{proof}
Thanks to the previous result,  we can now conclude the proof of Theorem \ref{thm: LY in Hs and BsS}.
\subsection{Proof of Theorem \ref{thm: LY in Hs and BsS}}\label{section:Proof of Theorem thm: LY in Hs and Bs_sharp}\ \\
Under the additional assumption of  $F$ being a SVPH$^{\,\sharp}$ and satisfying condition \eqref{key-conditions-for-F}, we are going to prove \eqref{Theta (N, 1)}
following word by word the SVPH case with the following three additions:\\
{\bf 1)} By Remark \ref{rmk constant bbd} we have a sharper estimate of $C_\psi(1)$ given in Lemma \ref{lem H^s partition of uni}: \[ C_\psi(1)\le C_{\eps,q_0} C_{\mu,\const\ln\chi_u^{-1}}\chi_u^{\const \ln\mu}.\]\\
{\bf 2)} Here we improve the constant $\Theta_1$ of Lemma \ref{lem on bound on cZc} for $s=1$. In the proof of Lemma \ref{lem on bound on cZc}, after \eqref{eq:MMxi}, we add the following computation for $s=1$:
 For any $\ray>0$ let $B_{\ray}=\{\xi\in \bZ^2: \|\xi\|\le \ray\}$ and $B_{\ray}^c=\bZ^2\setminus B_{\ray}.$ Then
\[
\begin{split}
\sum_{\xi\in\bZ^2}| \langle \xi \rangle \mathds{1}_{\cZ_u^c}   \cF \cL^{q_0}(u_{\balpha,\frh'}^M)|^2 =& \sum_{\xi\in \cZ_u^c\cap B_{\ray}}\langle \xi \rangle^{-2}| \langle \xi \rangle^{2} \cF \cL^{q_0}(u_{\balpha,\frh'}^M)|^2\\
&+ \sum_{\xi\in \cZ_u^c\cap B^c_{\ray}} \langle \xi \rangle^{-3} |\langle \xi \rangle^{2}\cF \cL^{q_0}(u_{\balpha,\frh'}^M)|\,| \langle \xi \rangle^{3}\cF \cL^{q_0}(u_{\balpha,\frh'}^M)|.
\end{split}
\]
Arguing as in the proof of Proposition \ref{prop on four coeff} and by \eqref{eq:MMxi} yield
\begin{equation}
\label{sum in Bray}
 \sum_{\xi\in \cZ_u^c\cap B_{\ray}}| \langle \xi \rangle \cF \cL^{q_0}(u_{\balpha,\frh'}^{M})|^2\le  C_{\mu, M}^2\mu^{2M}\|u\|_{2}^2 \sum_{\xi\in \cZ_u^c\cap B_{\ray}}\langle\xi \rangle^{-2} K_1(2, M_\xi)^2
\end{equation}
and
\begin{equation}
\label{sum in Brayc}
\begin{split}
&\sum_{\xi\in \cZ_u^c\cap B^c_{\ray}}| \langle \xi \rangle \cF \cL^{q_0}(u_{\balpha,\frh'}^M)|^2\\
&\le C_{\mu, M}^2\mu^{2M} \|u\|_{2}\|u\|_{3} \sum_{\xi\in \cZ_u^c \cap B^c_{\ray}}\langle\xi \rangle^{-3} K_1(2, M_\xi) K_1(3, M_\xi).
\end{split}
\end{equation}
We use the estimate of $K_1(2, M_\xi)$ in \eqref{K_1(N,M,t)} for the sum in (\ref{sum in Bray}), with $\vartheta(\xip)=\rho(\xip),$ since $\xi\in \cZ_u^c$,
\begin{equation}
\label{sum in bray for s=1}
\begin{split}
\sum_{\xi\in \cZ_u^c\cap B_{\ray}}\langle\xi \rangle^{-2} K_1(2,M_\xi)^2 &\leq C_{n_0,q_0,\epsilon}(\ln\chi_u^{-1})^{\const}\mu^{\const \ln\chi_u^{-1}} \hskip-6pt\sum_{\xi\in \cZ_u^c\cap B_{\ray}}\langle\xi \rangle^{-2} \rho(\xip)^{ -14}\\
&\leq C_{n_0,q_0,\epsilon}(\ln\chi_u^{-1})^{\const}\mu^{\const \ln\chi_u^{-1}} \chi_u^{-13} \log{\ray},
\end{split}
\end{equation}
since
\begin{equation}\label{eq:integral-gain_chiu}
\sum_{\xi\in \cZ_u^c\cap B_{\ray}}\langle\xi \rangle^{-2} \rho(\xip)^{ -14}\leq \int_{0}^{\ray} \int_{\{\tan\theta>\chi_u\}}
\frac{1}{1+\rho^2}\frac{1}{(\tan\theta)^{14}}\rho d\rho d\theta\lesssim \chi_u^{-13}\log\ray.
\end{equation}
Similarly, for the sum in (\ref{sum in Brayc}), we have
\begin{equation}
\label{sum in brayc for s=1}
\begin{split}
& \sum_{\xi\in \cZ_u^c\cap B^c_{\ray}}\langle\xi \rangle^{-3} K_1(2,M_\xi) K_1(3,M_\xi)\\
&\leq C_{n_0,q_0,\epsilon}\mu^{\const \ln\chi_u^{-1}}\chi_u^{-\const}
\sum_{\xi\in \cZ_u^c\cap B^c_{\ray}}\langle\xi \rangle^{-3} \rho(\xip)^{ -7} \Lambda^{\const M_\xi}\\
&\leq C_{n_0,q_0,\epsilon}\mu^{\const \ln\chi_u^{-1}}\chi_u^{-\const} \ray^{-1}.
\end{split}
\end{equation}
Choosing $\ray=\chi_u^{-\const}\Lambda^{\const M_\xi}$ by (\ref{sum in Bray}) and (\ref{sum in Brayc}) we have the following estimate:
\begin{equation}
\label{final estimate for case s=1}
\begin{split}
&\sum_{\xi\in\bZ^2}| \langle \xi \rangle \mathds{1}_{\cZ_u^c}   \cF \cL^{q_0}(u_{\balpha,\frh'}^M)|^2 \leq  C_{n_0,q_0,\epsilon} C_{\mu,M}^{2}\mu^{2M}(\ln\chi_u^{-1})^{\const}\chi_u^{ -13-\const\ln\mu} \|u\|_3^2.
\end{split}
\end{equation}
In conclusion, under the assumption that  the map is a SVPH$^{\, \#}$ and satisfies condition \eqref{key-conditions-for-F}, the constant $\Theta_1$ of Lemma \ref{lem on bound on cZc} for $s=1$ becomes
\begin{equation}
\label{Theta}
\Theta_1= C_{n_0,q_0,\epsilon} C_{\mu,M}^{2}\mu^{2M}(\ln\chi_u^{-1})^{\const}\chi_u^{ -13-\const\ln\mu}.
\end{equation}
{\bf 3)} To estimate the right hand side of \eqref{after CS ineq} in the case $s=1$ we use
\[
\sum_{\xi\in \cZ_u\cap Z(\frh)\cap B_R}\langle \xi \rangle^{-2}\chi_u^{-14}\le \Const \chi_u^{-13} \log R
\]
instead of \eqref{eq:integral-gain_chiu} where $Z(\frh)$ is defined in \eqref{Zfrh}.\\
By Lemma \ref{cor norm Hs of cL},  $Q(M,1)\leq C_{\mu, M}^{3}\mu^{3M}$. Finally, using \eqref{eq:barQ}, \eqref{Theta}, we can find $\beta_3>0$ such that
\[
\begin{split}
\sqrt{\overline Q(M,1)\Theta_1}&\le C_{\eps, q_0} C_{\mu,M}^{\beta_3}\mu^{\beta_3 M} (\ln \chi_u^{-1})^{\beta_3}\chi_u^{\frac{-13}{2}-\const \ln\mu^{-1}}\bL_M^{\frac 12},
\end{split}
\]
which concludes the proof of Theorem \ref{thm: LY in Hs and BsS}.


\section{The final Lasota-Yorke Inequality}\label{section:The final Lasota-Yorke inequality} We state and prove our main technical Theorem which implies the Theorems stated in section \ref{sec:results-tr}.
For each integer $1\le s\le r-1$ we define the following norm  
\[
\|\cdot\|_{s,*}:=\|\cdot\|_{\cH^s}+\|\cdot\|_{s+2}.
\]
\begin{thm}
\label{thm LY for SVPH}
Let $F\in \cC^r(\bT^2,\bT^2)$ be a SVPH and ${\alpha= \frac{\log(\lambda_-\mu^{-2})}{\log (\lambda_+)}}$. Let $\ovm_{\chi_u}$ be as in \eqref{def of ovmchiu}, {$n_0$ as in \eqref{n_0 in the general case} and $C_1>0$ provided in Theorem \ref{thm: LY in Hs and Bs}.  We assume that there exist: constants $c, K>0$, 
$\nu_0\in (0,1)$, integer ${n_1\ge n_0}$, and uniform constants $\tau_0\ge 1,  \kappa_1\geq \kappa_0\in\bN$ such that, for some  $1\le s\le r-3$,}\footnote{ Recall $\zeta_s$ defined in \eqref{def of zetar}, while \eqref{main assumption 0}, \eqref{main assumption} and \eqref{main assumption 1} constrain $\kappa_0$, $\kappa_1$ and $n_1$.}
\begin{align}
&\sup_{m\le n}\|\cL^m 1\|_{\infty} \le K\mu^{cn^{\tau_0}}, \qquad \forall n<\kappa_1 n_1+\ovm_{\chi_u},\label{main assumption 0}\\
&\Big\{ \mu^{\zeta_s}{\lambda_-^{-1}}, \sqrt{ \widetilde{\cN}_F(\lceil \alpha n_1 \rceil)\mu^{\alpha_s n_1^{\tau_0}\kappa_1^{\tau_0-1}+ \beta_s\ovm_{\chi_u}^{\tau_0}\kappa_0^{-1}} } \Big\}^+\le \nu_{0}<1, \label{main assumption}\\
& (C_1K)^{\frac1{\kappa_0 n_1+\ovm_{\chi_u}}}\nu_{0}^{\frac{\kappa_0}{n_1\kappa_0+\ovm_{\chi_u}}}< 1,\label{main assumption 1}
\end{align} 
where $\widetilde{\cN}_F$ is given in (\ref{def of tildeN}), $\alpha_s=c[(1-\alpha)^{\tau_0}+1]+2s$, $ \beta_s=2(s+c)$ and $\zeta_s$ given in \eqref{alphas betas zetas}. Moreover, for $\kappa\in (\kappa_0,\kappa_1)$, choose
\begin{equation}
\label{ess spectrum}
\sigma_\kappa\in (\{\mu^{\zeta_s}\lambda_-^{-1}, (C_1K)^{\frac1{\kappa n_1+\ovm_{\chi_u}}}\nu_{0}^{\frac{\kappa}{n_1\kappa+\ovm_{\chi_u}}}\}^+,1).
\end{equation}
Then, for each $n\in\bN$ and $\bar\sigma_\kappa\in (\sigma_\kappa,1)$ we have, $\Omega_{\chi_u}$ being as in Theorem \ref{thm: LY in Hs and Bs},
\begin{align} 
&\|\cL^{n}u\|_{s,*}\le C_\sharp \Omega_{\chi_u}(\kappa n_1+\ovm_{\chi_u}, s)\left\{ \sigma_\kappa^{n}\|u\|_{s,*}+(1-\sigma_\kappa)^{-1}C_{\mu,n}  \mu^n\|u\|_{0}  \right\}\label{LY B and 0}\\
&\|\cL^{n}u\|_{s,*}\le C_\sharp \Omega_{\chi_u}(\kappa n_1+\ovm_{\chi_u}, s) \overline\sigma_\kappa^{n}\|u\|_{s,*}    \label{LY B and L1}\\
&\phantom{\|\cL^{n}u\|_{s,*}\le }
+ C_{\bar\sigma_\kappa} \Omega_{\chi_u}(\kappa n_1+\ovm_{\chi_u}, s)^3 C_{\mu,n}^3\mu^{3n}\|u\|_{L^1},\nonumber
\end{align}
In particular, $\Omega_{\chi_u}(\kappa n_1+\ovm_{\chi_u}, 1)$ is given by \eqref{Theta (N, 1)} if $F$ is also a SVPH$^{\,\sharp}$ satisfying \eqref{key-conditions-for-F}.
\end{thm}
\begin{proof} We use Theorem \ref{thm: LY in Hs and Bs} with $q_0=\kappa n_1\ge n_0$ and $\kappa\in (\kappa_0,\kappa_1)$. First, by  conditions \eqref{main assumption 0} and \eqref{main assumption} and Lemma \ref{lemma on the relation N and tild N} , we observe that, setting $N=q_0+M$, 

\begin{equation}\label{eq:comput Nq0-Ntildeq0}
\begin{split}
 [\bL_{\ovm_{\chi_u}}\cN(q_0)]^{\frac{1}{N}}\mu^{2s}&\le [K\mu^{c \ovm_{\chi_u}^{\tau_0}}\cN(q_0)]^{\frac{1}{N}}\mu^{2s}  \\
 &\le  \left(K\mu^{c \ovm_{\chi_u}^{\tau_0}}\bL_{ q_0-\lceil\alpha q_0\rceil}\widetilde \cN(\lceil\alpha q_0\rceil)\right)^{\frac{1}{N}}\mu^{2s}\\
  &\le (K^2\widetilde{\cN}(\lceil \alpha q_0 \rceil)\mu^{q_0^{\tau_0}\alpha_s+\beta_s\ovm_{\chi_u}^{\tau_0}})^{\frac{1}{N}}.
 \end{split}
 \end{equation}
Therefore, by equation (\ref{LY Hs and Bs+2}),
\begin{equation}
\label{eq cond on tN implies cond on N}
\begin{split}
&\|\cL^N u\|_{\cH^s}\le C_1K\left( \widetilde{\cN}(\lceil \alpha q_0 \rceil)\mu^{q_0^{\tau_0}\alpha_s+\beta_s\ovm_{\chi_u}^{\tau_0}}\right)^{\frac{1}{2}}\|u\|_{\cH^s}+\Omega_{\chi_u}(M, s)\|u\|_{s+2}.
\end{split}
\end{equation}
Moreover, by the sub-multiplicativity of $\widetilde \cN$
\[
\widetilde{\cN}(\lceil \alpha q_0 \rceil)=\widetilde{\cN}(\lceil \alpha \kappa n_1 \rceil)\le \widetilde{\cN}(\lceil \alpha n_1 \rceil)^\kappa.
\]
It follows by the definition of $\nu_0$ that
\[
\sqrt{ \widetilde{\cN}(\lceil \alpha q_0 \rceil)\mu^{q_0^{\tau_0}\alpha_s+\beta_s\ovm_{\chi_u}^{\tau_0}}}\le \sqrt{\left[\widetilde{\cN}(\lceil \alpha n_1 \rceil)\mu^{\frac{\alpha_sq_0^{\tau_0}+\beta_s\ovm_{\chi_u}^{\tau_0}}{\kappa}}\right]^{\kappa}}\le \nu_0^{\kappa}.
\]
Accordingly
\begin{equation}\label{eq:almost}
\|\cL^N u\|_{\cH^s}\le \sigma_\kappa^{N}\|u\|_{\cH^s}+\Omega_{\chi_u}(M, s)\|u\|_{s+2}.
\end{equation}
On the other hand, the assumption $\mu^{\zeta_s}\lambda_-^{-1}\le \nu_0$ implies (\ref{mu lambda^rho<lambda^ -1/2}), so that we can choose $\delta_\ast$ in (\ref{L-Y B2}) such that, for all $n\in\bN$,
\begin{equation}\label{LY nstar s+2}
	\|\mathcal{L}^n u \|_{s+2} \le \Const  \sigma_\kappa^{n} \|u\|_{s+2}  +C_{\mu, n}\mu^{n}\|u\|_{0}.
\end{equation}
Iterating \eqref{eq:almost} by multiple of $N$ and using \eqref{LY nstar s+2} yields \eqref{LY B and 0}.\\
\indent Next, we want to compare the norm $\|\cdot\|_{0}$ with the $L^1$-norm. Let us fix $\ell>0$. Take an admissible central curve $\gamma$ and notice that, for any $\phi\in \cC^0(\bT)$ with $\|\phi\|_\infty=1$, we have
\[
\left| \int_{\mathbb{T}} \phi(t) (u)(\gamma(t)+\ell e_1)dt-\int_{\mathbb{T}} \phi(t) u(\gamma(t))dt\right|
=\left|\int_0^\ell ds \int_{\mathbb{T}}\phi(t) \partial_zu(\gamma(t)+se_1) dt\right|.
\]
Writing $\gamma(t)=(\sigma(t),t)$ we can make the change of variables $\psi(s,t)=\gamma(t)+s e_1=(\sigma(t)+s,t)$. Since $\det(D\psi)=-1$ and setting $D_\ell=\{\psi(s,t)\;;\; t\in\bT, s\in [0,\ell]\}$, we have
\[
\begin{split}
\left| \int_{\mathbb{T}} \phi(t) u(\gamma(t)+\ell e_1)dt-\int_{\mathbb{T}} \phi(t) u(\gamma(t))dt\right|
&=\left| \int_{D_\ell}\phi(z) \partial_zu(x,z) dx dz\right|\\
&\leq \|\phi\|_{L^\infty}\sqrt\ell \|u\|_{\cH^1}.
\end{split}
\]
Hence
\[
\begin{split}
\left|\int_{\mathbb{T}} \phi(t) u(\gamma(t)+s e_1)dt\right|\geq  \left|\int_{\mathbb{T}} \phi(t) u(\gamma(t))dt\right|-\sqrt s \|u\|_{\cH^1}.
\end{split}
\]
Integrating in $s\in [0,\ell]$ and taking the sup on $\gamma$ and $\phi$ yields
\begin{equation}\label{eq:L1rho}
\|u\|_0\leq \ell^{-1} \|u\|_{L^1}+ \frac {2\ell^{\frac 12}}3\|u\|_{\cH^1}.
\end{equation} 
Applying \eqref{eq:L1rho} to \eqref{LY B and 0} with the choice $\ell^{-1}=C_{\sigma_\kappa}^2 C_{\mu,n}^2\sigma_\kappa^{-2n}\mu^{2n}$, where $C_{\sigma_\kappa}=\Const(1-\sigma_\kappa)^{-1}$, yields
\[
\|\cL^{n}u\|_{s,*}\le C_\sharp \Omega_{\chi_u}(\kappa n_1+\ovm_{\chi_u}, s)\left\{ \sigma_\kappa^{n}\|u\|_{s,*}+C_{\sigma_\kappa}^3C_{\mu,n}^3  \mu^{3n}\sigma_{\kappa}^{-2n}\|u\|_{0}  \right\}.
\]
Next, for each $\bar \sigma_\kappa\in (\sigma_\kappa,1)$, let $n_\kappa$ be the smallest integer such that 
\[
C_\sharp \Omega_{\chi_u}(\kappa n_1+\ovm_{\chi_u}, s) \sigma_\kappa^{n_\kappa}\leq\bar\sigma_\kappa^{n_\kappa}.
\]
For each $n\in\bN$, write $n=kn_\kappa+m$ with $m<\kappa$, then iterating the above equation yields
\[
\begin{split}
\|\cL^{n}u\|_{s,*}&\le \bar\sigma_\kappa^{kn_\kappa} \|\cL^{m}u\|_{s,*}+\Const  \Omega_{\chi_u}(\kappa n_1+\ovm_{\chi_u}, s)\mu^{3kn_\kappa}C_{\sigma_\kappa}^3C_{\mu,n_\kappa}^3\sigma_\kappa^{-2n_\kappa}\sum_{j=0}^{k-1}\bar\sigma_\kappa^{n_\kappa}\|u\|_{L^1}\\
&\le \Const  \Omega_{\chi_u}(\kappa n_1+\ovm_{\chi_u}, s)\bar\sigma_\kappa^{n}\|u\|_{s,*}
+\Const \frac{ \Omega_{\chi_u}(\kappa n_1+\ovm_{\chi_u}, s)^3\mu^{3n}C_{\sigma_\kappa}^3C_{\mu,n}^3}
{\bar\sigma_\kappa^{2n_\kappa}(1-\bar\sigma_\kappa^{n_\kappa})}\|u\|_{L^1}
\end{split}
\]
which implies \eqref{LY B and L1}. 
\end{proof}
\begin{cor}
\label{cor conseq  of thm for SVPH}
Under the assumptions of Theorem \ref{thm LY for SVPH} there exists a Banach space $\cB_{s,*}$ such that $\cC^{r-1}(\bT^2)\subset \cB_{s,*}\subset \cH^{s}(\bT^{2})$ on which the operator $\cL_F:\cB_{s,*}\to\cB_{s,*}$ has spectral radius one and is quasi compact with essential spectral radius bounded by $\sigma_\kappa$. 
\end{cor}
\begin{proof} 
Let $\cB_{s,*}$ be the completion of $\cC^{r-1}(\bT^2)$ with respect to  $\|\cdot\|_{s,*}$, then $\cC^{r-1}(\bT^2)\subset \cB_{s,*}\subset \cH^{s}(\bT^{2})$. Iterating \eqref{LY B and L1}, and since $\cL_F$ is a $L^1$ contraction, implies that the spectral radius is bounded by one, but since the adjoint of $\cL_F$ has eigenvalue one, so does $\cL$, hence the spectral radius is one.

To bound the essential spectral radius note that the immersion $\cB_{s,*}\hookrightarrow \cH^{s}$ is continuous by definition of the norm. Moreover the immersion $\cH^s\hookrightarrow L^1$ is compact for every $s$ by Sobolev embeddings theorems, hence $\cB_{s,*}\hookrightarrow L^1$ is compact. Thus, by (\ref{LY B and L1}) and Hennion theorem \cite{Hen} follows that the essential spectral radius is bounded by $\bar\sigma_\kappa$ and hence the claim by the arbitrariness of $\bar\sigma_\kappa$.
\end{proof}

\begin{proof}[\bf Proof of Theorem \ref{Main Theorem 1}] According to Corollary \ref{cor conseq  of thm for SVPH}, it is enough to check the conditions of Theorem \ref{thm LY for SVPH}. Since $\mu>1$, Corollary \ref{cor of inf norm of transfer} implies $\sup_{k\le n}\|\cL^k1\|_\infty\le K_\mu\mu^{n}$ for each $n\in\bN$, for some constant $K_\mu$ depending on $\mu$. Hence \eqref{main assumption 0} is satisfied with $c=1$ and $\tau_0=1$ and arbitrary $\kappa_1\in \bN$.  Next, $\mu^{\zeta_s}\lambda_-^{-1}<1$ is implied by hypothesis {\bf (H3)}. Therefore,
condition \eqref{main assumption with N} coincide{s} with \eqref{main assumption} with $\alpha_s, \beta_s, \zeta_s$ given in \eqref{alphas betas zetas}. Finally, choosing any $\kappa_0$ such that
\begin{equation}\label{kappa0}
\kappa_0>\frac{\ln(C_1K_\mu^{\frac 12})}{\ln\nu_0^{-1}},
\end{equation}
we have also \eqref{main assumption 1}, whereby we conclude.
\end{proof}

\section{The map \texorpdfstring{$F_\ve$}{Lg}}\label{sec:8}
\label{section application to map F} 
Here we apply Theorem \ref{thm LY for SVPH} to the maps $F_\ve$, see \eqref{Fve}, and we prove Theorems \ref{thm check assump of thm SVPH}, \ref{thm:peripheral} and \ref{finitely many SRB}.

\subsection{The maps \texorpdfstring{$F_\ve$}{Lg} are \texorpdfstring{SVPH$^{\,\sharp}$}{SVPHs}}\label{section:Fve is SVPH}\  \\
We use Lemma \ref{lem check SVPH},  applied with $\omega$ replaced by $\ve \omega$, to check that the maps given in \eqref{Fve}  are SVPH for $\ve$ small enough. Condition \eqref{condition ftheta 1} is part of the assumptions on $F_\ve$ and implies $\lambda>2\|\partial_\theta f\|_\infty$, which implies \eqref{cond phi positive} and, since $\lambda>2$, \eqref{eq:condition ftheta 0} for $\ve$ small enough.
Finally, conditions \eqref{item:no-vertical},  \eqref{eq:omegatheta-cond} and \eqref{eq:pinching} are also immediate for $\ve$ small. Hence, there exists $\ve_0>0$ such that, for all $\ve<\ve_0$, $F_\ve$ satisfies the hypothesis of Lemma \ref{lem check SVPH} and hence it is a SVPH.

In particular, we choose $\mathbf{C}_\varepsilon^u=\lbrace (\xi, \eta)\in \mathbb{R}^2: |\eta|\le \ve  u_\star |\xi|  \rbrace$,\footnote{ Observe that in this special case $\chi_u(\ve)=\ve u_\star$, thus we have an unstable cone of size $\ve$.} and $\mathbf{C}^c=\lbrace (\xi, \eta)\in \mathbb{R}^2: |\xi|\le  \eta|  \rbrace$ with
\begin{equation}\label{eq:u_partt}
 u_\star=2\|\partial_x \omega\|_\infty=:\ve^{-1}\chi_u.
\end{equation}

Let $(1,\varepsilon u)\in \mathbf{C}^u_\varepsilon$, for $p=(x,\theta)\in \mathbb{T}^2$. In this case equation \eqref{eq: def of Xi} yields
\begin{equation}\label{eq:expansion}
D_pF_\ve (1,\varepsilon u)=(\partial_x f+\varepsilon u \partial_{\theta} f)(1,\varepsilon \Xi_\varepsilon({p, u})),
\end{equation}
where 
\begin{equation}
\label{def of Xi}
\Xi_\varepsilon({p, u})=\frac{\partial_x\omega+\varepsilon u \partial_{\theta}\omega +u}{\partial_x f+\varepsilon u \partial_{\theta} f}.
\end{equation}
We have also a more explicit formula for iteration of the  map $\Xi_\varepsilon$. For any $k\ge0$ and $p\in \mathbb{T}^2$, let us denote $p_k=F^k_\varepsilon (p)$. Then {we have} the recursive formula:
\begin{equation}\label{iteration of Xi}
\Xi^{(n)}_\varepsilon(p,u)={\Xi_\ve(p_{n-1},\Xi^{(n-1)}(p,u))}.
\end{equation}
On the other hand, recalling \eqref{u derivative of Xi}:
\begin{equation}
\label{u derivative of Xive}
\partial_u\Xi_{\varepsilon}(p, u)=\frac{\partial_xf+\varepsilon (\partial_\theta\omega\partial_x f-\partial_\theta f\partial_x \omega)}{(\partial_x f+\varepsilon u\partial_\theta f)^2}.
\end{equation}
In fact, with the above choice of $u_\star$ and using \eqref{def of Xi}, for $\ve$ small enough, 
\[
D_p F_\varepsilon(\mathbf{C}_\varepsilon^u)\subset\{(\eta,\xi): |\xi |\le \frac 32\lambda^{-1} \ve\chi_u|\eta|\},
\]
hence condition \eqref{invariance of cone} holds with $\iota_\star=\frac 32\lambda^{-1}\leq \frac 34<1.$\footnote{ Recall that $\lambda=\inf \partial_x f> 2$.}

Moreover, in \eqref{mu_n<e^b} it is shown that for some $\bar c>0$, $\mu_{\pm}=e^{\pm \bar c \ve}$. Finally, by \eqref{defn:lambda+-}, it follows that condition \eqref{partial hyperbolicity 2} holds with $C_\star=1$, $\lambda_+=2\sup_{\bT^2}\partial_x f$ and $\lambda_-=2\lambda/3$. \\
The above discussion shows that all the quantities $\chi_c$, $\iota_\star$ and $C_\star$ are independent of $\ve$, and $\mu_+$ and $(\lambda_--\mu_+)^{-1}$, $(1-\chi_u)^{-1}$ are bounded uniformly in $\ve$. Therefore, the uniform constant in the sense of the Notation at the end of Section \ref{sec:intro} are, in fact, uniform in the usual sense, hence the name. \\
Finally, it is immediate to check that, possibly taking $\ve_0$ smaller, the extra conditions \eqref{eq:eps-cond} and \eqref{eq:condition-hatn} for the invariance of the curves are also satisfied. Indeed, choosing $\hat n=c_3^- \log \ve^{-1}$, with $c_3^- <c^-_2$ (which implies  \eqref{eq:eps-cond}), and recalling \eqref{Cmu,n}-\eqref{def:C_F}, we have $C_{\mu, \hat n}\lesssim \log \ve^{-1}$ and $\sigma_{\hat n}\lesssim \log \ve^{-1}+ e^{c_\sharp \ve}$, which implies \eqref{eq:condition-hatn} if $\ve$ is small enough. Thus,  $F_\ve$ is a SVPH$^{\,\sharp}$ for each $\ve \le \ve_0$.\footnote{ See Remark \ref{rmk:extra-conditions} and Definition \ref{def:SVPHsharp} for details.}

To conclude the preliminaries, it is useful to note that, if we set $\psi(p)=\langle \nabla \omega, (-\frac{\partial_\theta f}{\partial_x f}, 1) \rangle(p)$, for every $p\in \bT^2$ and $n\in \bN$ we have
\[
\det D_pF^{n}_\ve=\prod_{k=0}^{n-1}\det D_{F^{k}_\ve p}F_\ve=\prod_{k=0}^{n-1}\left[ \partial_xf(F^k_\ve p)(1+\ve \psi(F^k_\ve p))  \right],
\]
hence
\begin{equation}
\label{condition on detF_ve}
e^{-\bar c\ve n}\lambda^n \le \det D_pF_\ve^n\le e^{\bar c\ve n}\Lambda^n, \quad \forall p\in \bT^2, \forall n\in \bN.
\end{equation}

\subsection{A non-transversality argument}\ \\
Here the aim is to prove Proposition \ref{main thm} which guarantees that, after some fix time which does not depend on $\ve$, for each point we have at least a couple of pre-images with transversal unstable cones, provided $\omega$ is not $x$-constant. This will imply the existence of the integer $n_1$ required by Theorem \ref{thm LY for SVPH}.\\
\indent In the following we denote as $\frH_\ve$ the set of the inverse branches of $F_\varepsilon$.\footnote{ Accordingly $\frH_0$  is the set of inverse branches of $F_0$.} Moreover, $\frH_{\varepsilon}^n$ will be the set of elements of the form $\frh_1\circ\cdots\circ \frh_n$, for $\frh_j\in\frH_\ve$ and $\frH^{\infty}_{\varepsilon}:=\frH^{\mathbb{N}}_\varepsilon$. For $\frh\in \frH_\varepsilon^{\infty}$ the symbol $\frh_n$ will denote the restriction of $\frh$ on $\frH_\varepsilon ^n$. 
\begin{oss}
\label{rmk on H}
 Since $F_0$ and $F_\ve$ are homotopic coverings they are isomorphic, that is there exist $I_\ve:\bT^2\to\bT^2$ such that $F_\ve=F_0\circ I_\ve$. This induces an isomorphism $\cI_\ve:\frH_0\to \frH_\ve$ defined by $\cI_\ve \frh=I_\ve^{-1}\circ \frh$. Hence, the same is true for the sets $\frH^n_\varepsilon=\frH^n$ and $\frH^{\infty}_\varepsilon$. In the following we will then identify inverse branches of $F_\varepsilon^n$ and $F_0^n$ by these isomorphisms, and drop the script $\varepsilon$ from the notation when it is not necessary.\\
\end{oss}

\begin{prop}
\label{main thm}
If $\omega$ is not $x$-constant with respect to $F_0$ (see Definition \ref{x-constant}), then there exist $\varepsilon_0>0$ and $n_0\in \mathbb{N}$ such that, for every $\varepsilon\le \varepsilon_0$, $p\in\mathbb{T}^2$ and vector $v\in \mathbb{R}^2 $, there exists ${q}\in F^{-n_0}_{\varepsilon}(p)$ such that $v\not\in D_{{q}}F_{\varepsilon}^{n_0} \mathbf{C}^u_\varepsilon $.
\end{prop}
\begin{proof}
We argue by contradiction. Suppose that for every  $\varepsilon_0>0$ and  $\ell\in \mathbb{N}$ there exist $\varepsilon_\ell\in [0,\varepsilon_0]$, $p_{\ell}\in\bT^2$ and $v_\ell=(1, \ve_\ell u_\ell)$ with $|u_\ell|\le u_\star$ such that\footnote{ We use the notation with subscript $\ell$ for a generic object that depends on $\ell$ through $\varepsilon_\ell$, but we keep the notation as simple as possible when there is no need to specify.}
\begin{equation}
\label{Non empty intersection}
D_{q}F_{\varepsilon_{\ell}} ^{{\ell}}\mathbf{C}^u_{\varepsilon_\ell} \supset v_\ell, \quad \forall q\in F_{\ve_{\ell}}^{-{\ell}}(p_\ell),
 \end{equation}
namely, all the above cones have a common direction. Since the sequence $\{p_\ell, u_\ell\}\subset \mathbb{T}^2\times [-u_\star, u_\star]$, it has an accumulation point $(p_\ast, u_\ast)$. In analogy with \eqref{Phifrh in the gener case}, for $p\in \mathbb{T}^2$ and $u\in [-u_\star, u_\star]$ we define
\begin{equation}\label{eq:Phive}
\Phi_\varepsilon^{n}(p,u)=\left( F^{n}_\varepsilon(p), \Xi_\varepsilon ^{(n)}(p,u) \right),
\end{equation}
where $\Xi_\varepsilon ^{(n)}$ is given by formula (\ref{iteration of Xi}).
Condition (\ref{Non empty intersection}) in terms of this dynamics says that the slope $u_\ell$ is contained in the interval $\Xi_{{\varepsilon}_{\ell}}^{({\ell})}(q, [-u_\star, u_\star])$ for every $\ell\in \mathbb{N}$ and $q\in F_{{\varepsilon}_{\ell}}(p_\ell)$. Hence, it can be written as:
\begin{equation}
\label{project contrary}
\forall \ell\in \mathbb{N}, \quad \exists (p_\ell,u_\ell): \quad\pi_2\circ\Phi^{{\ell}}_{\varepsilon_\ell}\left(q, [-u_\star, u_\star]\right) \supset \lbrace u_\ell \rbrace, \quad \forall q\in F_{\varepsilon_{\ell}}^{-{\ell}}(p_\ell),
\end{equation}
where $\pi_2:\mathbb{T}^2\times [-u_\star,u_\star]\to [-u_\star, u_\star]$ is the projection on the second coordinate.
Now, for $m\in \mathbb{N}$, $\varepsilon\in [0, \varepsilon_0]$, $u_0\in [-u_\star, u_\star ]$ and $\frh\in \frH^{\infty}$, let us define
\begin{equation}
\label{u_m, h^eps}
u^{\varepsilon}_{\frh,m}(p)=\pi_2\circ \Phi_{\varepsilon}^m(\frh_m(p), u_0):\mathbb{T}^2\to [-u_\star,u_\star].
\end{equation}
Next, we prove the following result, which will allow us to conclude the proof.

\begin{lem}
\label{lemma on  contin of u^eps_h}
The sequence of functions defined in \eqref{u_m, h^eps} satisfies:
\begin{itemize}
\item[(\textit{i})] For every $\varepsilon\in [0,\varepsilon_0]$ and $\frh\in \mathcal{\frH}^{\infty}$, there exists  $u^{\varepsilon}_{{\frh, \infty}}(\tilde p):=\lim_{m\to \infty} u^{\varepsilon}_{\frh, m}(\tilde p),$ and the limit is uniform in $\tilde p\in \mathbb{T}^2$.
\item[(\textit{ii})] For every $\frh\in \mathcal{H}^{\infty}$, the sequence $\lbrace u^{\varepsilon}_{\frh, \infty} \rbrace_{\varepsilon}$ converges to $\overline{u}_{\frh,\infty}$ uniformly for $\ve\to 0$.
\item[(iii)]
The functions $\overline{u}_{\frh,\infty}$ are independent of $\frh$, we call them $\tilde u$. In addition, $\tilde u$ satisfies
\begin{equation}
\label{u(F0)=Xi(u)}
\tilde{u}(F_{0}(q))=\Xi_{0}(q,\tilde{u}(q)), \quad \forall q\in \mathbb{T}\times \lbrace \theta_{\ast} \rbrace.
\end{equation}
\end{itemize}
\end{lem}
\begin{proof}
Applying Lemma \ref{lem:vec_regularity} with $u=u'\equiv u_0 \in[-u_\star, u_\star]$, $\ve_0=1$, $A=2\chi_c u_\star$ and $B=0$ we have that there exists $\nu\in (0,1)$ such that, for each $\frh\in \frH^{\infty}$, $q\in \mathbb{T}^2$, $\varepsilon, \varepsilon' \in [0, 1)$, $m\in\bN$ and $n>m$,\footnote{ The second equation of \eqref{eq: lip eq foe u_h} is a direct consequence of \eqref{eq:Xi_der_est} which implies that $\Xi_\ve(p, \cdot)$ is a contraction.}
\begin{equation}\label{eq: lip eq foe u_h}
\begin{split}
&|u_{\frh,m}^{\ve}(q)- u^{\ve'}_{\frh, m}(q)|\leq 
C_\sharp \mu^{3m}|\ve-\ve'|\\
&|u_{\frh,n}^{\ve}(q)- u^{\ve}_{\frh, m}(q)|\le \Const \nu^m.
\end{split}
\end{equation}
It follows that there exists $u^{\varepsilon}_{{\frh, \infty}}(q):=\lim_{m\to \infty} u^{\varepsilon}_{\frh, m}(q),$ and the limit is uniform in $q$. Next, for each $\delta>0$, we choose $\ve_*$ and $m$ such that $\Const\mu^{3m}\ve_*\le \frac{\delta}{4}$ and $\nu^m\le \frac{\delta}{4}$, then, for each $\ve,\ve'\le \ve_*$ and $q\in\bT^2$
\[
\begin{split}
|u^{\varepsilon}_{{\frh, \infty}}(q)-u^{\varepsilon'}_{{\frh, \infty}}(q)|&\le |u^{\varepsilon}_{{\frh, \infty}}(q)-u^{\varepsilon'}_{{\frh, m}}(q)|+|u^{\varepsilon}_{{\frh, m}}(q)-u^{\varepsilon'}_{{\frh, m}}(q)|+|u^{\varepsilon'}_{{\frh, m}}(q)-u^{\varepsilon'}_{{\frh, \infty}}(q)|\\
&\le 2\nu^m+\Const\mu^{3n}|\ve-\ve'|\le \delta.
\end{split}
\]
The above proves the first two items. Let us proceed with the third one.\\
First we claim that, for $q\in \mathbb{T}^2$, if ${\frh_q}$ is such that $q={\frh_q}(F_\varepsilon(q))$, then
\begin{equation}
\label{main equation}
u^{\varepsilon}_{\frh\circ {\frh_q},\infty}(F_{\varepsilon}(q))=\Xi_{\varepsilon}(q,u^{\varepsilon}_{\frh,\infty}(q)), \quad \forall q\in \mathbb{T}^2.
\end{equation}
Indeed, since $u^{\varepsilon}_{\frh, \infty}$ belongs to the unstable cone,  by  \eqref{u_m, h^eps}, for every $\frh\in\frH^{\infty}$  and $q\in \mathbb{T}^2$,
\begin{equation*}
\left( F_\varepsilon(q),\Xi_\ve (q, u^{\varepsilon}_{\frh, \infty}(q))\right)=\Phi_\varepsilon(q, u^{\varepsilon}_{\frh, \infty}(q))=\left( F_\varepsilon(q), u^{\varepsilon}_{\frh\circ 
{\frh_q}, \infty}(F_\varepsilon(q)) \right), 
\end{equation*}
which implies the claim taking the projection on the second coordinate.

For every $\ell\in \mathbb{N}$, let us now consider $\varepsilon_\ell$, $p_\ell$ and $u_\ell$ as given in (\ref{project contrary}) and let $\ell_j$ so that $(p_{\ell_j}, u_{\ell_j})$ is a convergent sequence. Equation \eqref{Non empty intersection} implies 
\begin{equation}
\label{|u_inf-u_m|}
|u_{\ell_j}-u^{\varepsilon_{\ell_j}}_{\frh, n_{\ell_j}}(p_{\ell_j})| \le \Const\nu^{n_{\ell_j}}.
\end{equation}
Taking the limit for $j\to \infty$ in the above inequality yields\footnote{ Recall that $(p_\ast, u_\ast)$ is an accumulation point of the sequence $(p_\ell, u_\ell)$ given in (\ref{project contrary})}
\begin{equation}
\label{u_*=u_h(p_*)}
u_\ast=\lim_{j\to\infty} u_{\ell_j}=\lim_{j\to\infty}u^{\varepsilon_{\ell_j}}_{\frh, n_{\ell_j}}(p_{\ell_j})=\overline{u}_{\frh,\infty}(p_\ast),
\end{equation}
regardless of the choice of the inverse branch $\frh\in\frH^{\infty}$. Let ${\frh_q}$ be the inverse branch such that $q={\frh_q}(F_\varepsilon(q))$, and set $q_\ell={\frh_q}(p_\ell)$ in equation (\ref{main equation}) to obtain:
\begin{equation}
\label{main equation with eps_ell}
u^{\varepsilon_{\ell}}_{\frh\circ {\frh_q},\infty}(p_\ell)=\Xi_{\varepsilon_{\ell}}(q_\ell,u^{\varepsilon_{\ell}}_{\frh,\infty}(q_\ell)).
\end{equation}
By item (\textit{ii}) above, and by the continuity of the map $F_\varepsilon$, we can take the limit as {$\ell_j\to\infty$} in the last equation and obtain 
\begin{equation*}
\overline{u}_{\frh\circ {\frh_q},\infty}(p_\ast)=\Xi_{0}(q_\ast,\overline{u}_{h, \infty}(q_\ast)),
\end{equation*}
where $q_\ast$ is such that $F_0(q_\ast)={p_\ast}.$ By (\ref{u_*=u_h(p_*)}), the above equation becomes $u_\ast=\Xi_{0}(q_\ast,\overline{u}_{\frh, \infty}(q_\ast))$, and, since $\Xi_{0}(q_\ast,\cdot)$ is invertible, this implies that there exists $u_*(q_\ast)$ independent of $\frh\in\cH^\infty$ such that
\[
u_*(q_\ast)=\bar u_\frh(q_\ast)=\lim_{j\to \infty} u^{\ve_{\ell_{j}}}_{\frh,\infty}(q_{\ell_{j}}).
\]
Hence, by induction, $\overline{u}_{\frh,\infty}(q)$ is independent on $\frh$ for each $q\in\bigcup_{k\in\bN}F^{-k}_0(p_\ast)=:\Lambda_{\theta_{\ast}}$, let us call it $u_*(q)$.
Taking the limit in equation \eqref{main equation} we have, for each $q\in\Lambda_{\theta_{\ast}}$, 
\begin{equation}\label{eq:ustar}
u_{*}(F_{0}(q))=\Xi_{0}(q,u_{*}(q)).
\end{equation}
Note that the $ \overline{u}_{\frh,\infty}$ are uniform limits of continuous functions and hence are continuous functions such that $ \overline{u}_{\frh,\infty}|_{\Lambda_{\theta_{\ast}}}=u_*$. Since $\Lambda_{\theta_{\ast}}$ is dense in $\bT\times \lbrace \theta_{\ast}\rbrace$.\footnote{ It follows from the expansivity of $f(\cdot,\theta_\ast)$ that the preimmages of any point form a dense set.} It follows that the $ \overline{u}_{\frh,\infty}$ equal some continuous function $\tilde{u}$  defined on $\bT\times \lbrace \theta_{\ast}\rbrace$ and independently of $\frh$. In addition, $\tilde{u}$ satisfies \eqref{u(F0)=Xi(u)}.\footnote{ Just approximate any point with a sequence $\{q_j\}\subset \Lambda_{\theta_{\ast}}$ and take the limit in \eqref{eq:ustar}.} 
\end{proof}
\begin{flushleft}
We can now conclude the proof of Proposition \ref{main thm}. By Lemma \ref{lemma on  contin of u^eps_h} we can find a function $\tilde{u}:\mathbb{T}^2\to \mathbb{R}$ and $\theta_{\ast}\in \mathbb{T}^1$ such that (\ref{u(F0)=Xi(u)}) holds, namely:
\end{flushleft}
\begin{equation}
\label{coboun equa}
\tilde{u}(F_{0}(q))=\frac{\partial_x\omega(q)+\tilde{u}(q)}{\partial_xf(q)}, \quad q\in \mathbb{T}^1\times {\lbrace \theta_{\ast}\rbrace}
\end{equation} 
Let us use the notation $g_{\theta}(x)$ for a function $g(x,\theta)$ and observe that, integrating (\ref{coboun equa}) and recalling that $\omega$ is periodic by hypothesis, we have
\begin{equation*}
\begin{split}
\int_{0}^{1} \tilde{u}_{\theta_{\ast}} (x)dx &=\int_{0}^{1}f'_{\theta_{\ast}}(x)\tilde{u}_{\theta_{\ast}}(f_{\theta_{\ast}}(x))dx-\int_0^1\partial_x\omega(x,\theta_\ast)dx=\sum_{i=0}^{d-1}\int_{U_i}f'_{\theta_{\ast}}(x)\tilde{u}_{\theta_{\ast}}(f_{\theta_{\ast}}(x))dx  \\
&= d\int_{0}^{1}\tilde{u}_{\theta_{\ast}}(t)dt,
\end{split}
\end{equation*}
where $U_i$ are the invertibility domains of $f_{\theta_{\ast}}$, and  $d>1$ its topological degree.  Hence $\int_{\mathbb{T}}\tilde{u}_{\theta_{\ast}}(x)dx=0$. So there is a potential given by
$
\Psi_{\theta_{\ast}}(x)=\int_{0}^{x}\tilde{u}_{\theta_{\ast}}(z)dz.
$
Finally, integrating equation (\ref{coboun equa}) from $0$ to $x$, there exists $c>0$ such that
\begin{equation*}
\omega_{\theta_{\ast}}(x)=\Psi_{\theta_{\ast}}(f_{\theta_{\ast}}(x))-\Psi_{\theta_{\ast}}(x)+c,
\end{equation*}
which contradicts the assumption on $\omega$ whereby proving the Proposition. 
\end{proof}

For reasons which will be clear shortly, we introduce a new quantity related to $\cN_{F_\ve}$ and $\widetilde{\cN}_{F_\ve}$ which can be interpreted as a kind of normalization of the latter one. The definition is inspired by \cite{BuEs}.
\begin{defn}
\label{defn of tildephi}
For each $p=(x, \theta)\in \bT^2$, $v\in \bR^2$, $n\in\bN$ and $\ve>0$ we define 
\begin{equation}
\label{tildephi}
\tN(x,\theta, v, n):= \frac{1}{h_*(x,\theta)}\sum_{\substack{(y,\eta)\in F_{\ve}^{-n}(x,\theta) \\ DF_\ve^{n}(y, \eta)\mathbf{C}^{u}_\ve \supset v }}  \frac{h_*(y, \theta)}{|\det DF_\ve^{n}(y, \eta)|},
\end{equation}
where, for every $\theta\in \bT$, $h_*(\cdot, \theta) =:h_{*\theta}(\cdot)$ is the density of the unique invariant measure of $f(\cdot, \theta).$ As before we will denote $\widetilde{\mathfrak{N}}(n):=\sup_{p}\sup_{v}\tN(p,v, n).$
\end{defn}
The motivation to introduce this quantity is twofold. One reason lies in Lemma \ref{shadowing} below in which, using a shadowing argument similar to \cite[Appendix B]{DeLi2}, we exploit the following fact: for each $\theta\in \bT,$ setting $f_\theta (\cdot)=f(\cdot, \theta), $ we have
\begin{equation}
\label{Lf h=h}
\frac{1}{h_{*\theta}(x)}\sum_{y\in f_\theta (x)} \frac{h_{*\theta}(y)}{(f_\theta^{n})'(y)}=1, \quad \forall x\in \bT.
\end{equation}
On the other hand it is easy to see that $\tN$ has the same properties of $\widetilde{\cN}_{F_\ve}$. In particular, arguing exactly in the same way as in Proposition \ref{submultiplicativity} and Lemma \ref{lemma on the relation N and tild N}, one can show that
\begin{align}
&\tN(n) \text{ is submultiplicative},\label{submult of tN}\\
&\cN_{F_\ve}(n)^{\frac{1}{n}}\le \Const\|\cL^{n\lfloor 1-\alpha \rfloor}_{F_\ve} 1\|_{\infty}^{\frac{1}{n}}\left( \tN({\lfloor \alpha n \rfloor})^{\frac{1}{ {\lfloor \alpha n \rfloor} }}\right)^{\alpha}, \quad \textrm{ for some }\alpha \in (0,1) \label{N<tN} \\
& \tN(n)\le \sup_{(x,\theta)\in\bT^2} \frac{1}{h_*(x,\theta)}(\cL^n_{F_\ve} h_*)(x,\theta). \label{tN<Lh/h}
\end{align}
This implies that we can check condition (\ref{main assumption}) of Theorem (\ref{thm LY for SVPH}) with $\widetilde{\cN}$ replaced by $\tN$.\\
\noindent To ease notation in the following we set $\cL_{F_\ve}=:\cL_{\ve}$.
\begin{lem}
\label{shadowing}
There are constants $C, c_\ast>0$ such that, for each $n<C\ve^{-\frac{1}{2}},$
\begin{equation}
\label{Lh/h<e^cn}
\sup_{(x,\theta)\in\bT^2} \frac{1}{h_*(x,\theta)}(\cL^n_{\ve} h_*)(x,\theta)\le e^{c_\ast n^2 \ve}.
\end{equation} 
\end{lem}
\begin{proof}
Let $F_\ve^{n}(q)=(x,\theta)$ and define $q_k=(x_k,\theta_k)=F_\ve^k(q)$, for every $0\le k \le n.$
Then,
\begin{equation}
\label{theta-theta_ast}
|\theta-\theta_{k}|\leq  \sum_{j=k}^{n-1}\ve\|\omega\|_\infty\leq \Const (n-k)\ve.
\end{equation}
Let us set $f_\theta(y)=f(y,\theta)$. Since $f_\theta$ is homotopic to $f_{\theta_k}$, for each $k$, there is a correspondence between inverse branches, hence there exists $x_*$ such that $|f_{\theta}^n(x_*)-x_n|=0$.
Moreover, let $\xi_k=f_{\theta}^k(x_*)-x_k$. Since  $f$ is expanding, by the mean value theorem and (\ref{theta-theta_ast}), there is $(\bar{x}, \bar{\theta})$ such that 
\[
|\xi_{k+1}|=|\langle \nabla f(\bar{x},\bar{\theta}), (\xi_k, \theta_k-\theta)\rangle |\ge \lambda |\xi_k|-\Const n\ve.
\]
Since $\xi_n=0$, we find by induction $|\xi_k|\le \sum_{j=k}^{n-1}\lambda^{-j+k}\Const\ve n\le \Const\ve n$.
Moreover, since $h_*$ is differentiable\footnote{ See \cite{DeLi1} for the details.} we also have
\[
|h_*(x_k, \theta_k)-h_{*\theta}(f_{\theta}^k(x_*))|\le \Const\ve n.
\]
Next, since $|\det D_q F_\ve-\partial_xf(q)|\leq \Const\ve$,
\[
\begin{split}
\frac{(f_\theta^n)'(x_*)}{\det DF^n_\ve(x_0, \theta_0)}&=\prod_{k=0}^{n-1}\frac{f_\theta'(f_\theta^k(x_*))}{\det DF_\ve(x_k,\theta_k)}
\leq \prod_{k=0}^{n-1}\frac{f_\theta'(f_\theta^k(x_*))}{\det DF_\ve(f_\theta^k(x_*),\theta)}\left[1+\Const n\ve\right]\leq e^{\const n^{2}\ve}.
\end{split}
\]
It follows that,
\[
 \frac{1}{h_*(x, \theta)}\sum_{(y,\vartheta)\in F_{\ve}^{-n}(x,\theta)}\left( \frac{h_*(y,\vartheta)}{|\det DF_{\ve}^{n}(y,\vartheta)|}\right)\leq \frac{e^{\const n^2\ve}}{h_{*\theta}(x)}\sum_{x_*\in f_{\theta}^{-n}(x)}\frac{h_{*\theta}(x_*)}{(f^{n}_\theta)'(x_*)}= e^{c_\sharp{n^2}\ve},
\]
where we have used (\ref{Lf h=h}).
\end{proof}

\subsection{Proof of Theorem \ref{thm check assump of thm SVPH}.}\label{section:proof of thm check ass}\ \\
 By section \ref{section:Fve is SVPH} $F_\ve$ is SVPH$^{\,\sharp}$ for $\ve\leq \ve_0$. We check the other hypotheses of Theorem \ref{thm LY for SVPH} for $F_\ve$, under the assumption that $\omega$ is not $x$-constant. In this case the existence of ${n_0}$ independent of $\ve$ is guaranteed by Proposition \ref{main thm}. Notice that $\chi_u= u_\star \ve$, i.e the unstable cone $\fC^u_{\ve}$ is of order $\ve$ while the center cone $\fC^c$ is of order one. Hence, by \eqref{expansion with mchi_u}, there exists $c_0>0$ such that ${\ovm_{\chi_u}\le}\lfloor c_0\log\ve^{-1}\rfloor.$ \footnote{ For simplicity in the following we drop the $\lfloor \cdot \rfloor$ notation.} Provided $n_1 \kappa_1= c_1\log\ve^{-1}$, for some $c_1>0$ and, by Lemma \ref{shadowing}, we have
\[
\sup_{m\le n}\|\cL^m_\ve 1\|_{\infty}\le \frac{|\sup h_*|}{|\inf h_*|} \sup_{m\le n}\left\|\frac{1}{h_*}\cL^m_\ve h_*\right\|_\infty\le \frac{|\sup h_*|}{|\inf h_*|} e^{c_*\ve n^2}, \quad \forall n\le\{c_0+c_1\}\log\ve^{-1},
\]
hence condition \eqref{main assumption 0} with $K=\frac{|\sup h_*|}{|\inf h_*|}$, $\tau_0=2$ and $c=c_*\ve$. 
Next, we prove that there exists a uniform constant $\nu_0$ such that 
\begin{equation}\label{main cond for Fve}
\left\{e^{\bar c\ve \zeta_s}\lambda^{-\frac{1}{2}},  {\tN}(\lceil \alpha n_1 \rceil)e^{\const\ve[n_1^2\kappa_1+ m_{\chi_u}^2]} \right\}^+\le \nu_0<1,   
\end{equation}
i.e condition \eqref{main assumption 0} with $\widetilde \cN_F$ replaced by $\widetilde{\mathfrak{N}}$ which, as we already observed, is an alternative condition under which Theorem \ref{thm LY for SVPH} holds for $F_\ve$. Let $\ve_1\in (0,\ve_0)$,   $\ve_0$ given in Proposition \ref{main thm}, be such that, for each $1\le s\le r-1$
\begin{equation}
\label{main condition 1 for Fve}
\mu^{\zeta_s}\lambda^{-1}=e^{\bar c\ve \zeta_s}\lambda^{-1}<1,\quad \forall \ve\in (0, \ve_1).
\end{equation}
Accordingly, for every $p=(x,\theta)\in \bT^2$ and $v\in \bR^2$, there exists $q_\ast\in F^{-n_0}(p)$, $n_0$ provided by Proposition \ref{main thm}, such that
\begin{equation*}
\begin{split}
\frac{1}{h_*(x,\theta)}\sum_{\substack{(y,\theta)\in F_{\ve}^{-n_0}(p) \\ DF_{\ve}^{n_0}(y,\eta)\mathbf{C}_{\ve}^{u}\supset v }}\frac{h_*(y,\theta)}{|\det D_qF_{\ve}^{n_0}|}&\le \frac{1}{h_*(x,\theta)}(\cL^{n_0}_\ve h_*)(x,\theta)-{\frac{\mathbbm{k}}{|\det D_{q_\ast}F_{\ve}^{n_0}|}},
\end{split}
\end{equation*}
where $\mathbbm{k}=\frac{\inf h_*}{\sup h_*}$. By Lemma \ref{shadowing} and equation (\ref{condition on detF_ve}), the last expression is bounded by
$e^{c_*n_0^2\ve}-\frac{\bbk}{\Lambda^{n_0}e^{\bar c\ve}}$.
Choosing $\ve_2 <\min\left(\ve_1, \frac{1}{cn_0^2}\log(1+\frac{\bbk}{\Lambda^{n_0}e^{\bar c\ve}}\right),$ we have that ${\tN}(n_0)\le \bar \sigma<1$ for every $\ve \in [0, \ve_2].$
Consequently, choosing $n_1=\alpha^{-1}n_0$, for each $\nu_0>\bar\sigma$, there exists $\ve_*\in (0,\ve_2)$ such that
\begin{equation}
\label{main condition 2 for Fve}
{\tN}(\lceil \alpha n_1 \rceil)e^{\const\ve[n_1^2\kappa_1+ (\ln\ve^{-1})^2]} \le \bar \sigma e^{\const\ve(\ln\ve^{-1})^2}\leq\nu_0 <1, \quad \forall\ve\in (0,\ve_*).
\end{equation}
By \eqref{main condition 1 for Fve} and \eqref{main condition 2 for Fve} we obtain \eqref{main cond for Fve}.
Finally, condition \eqref{main assumption 1} is satisfied choosing 
\begin{equation}\label{eq:Kfinal}
\kappa_0\geq \frac{\ln C_1K}{\ln\nu_0^{-1}}
\end{equation}
Thus Theorem \ref{thm LY for SVPH} applies  and Theorem \ref{thm check assump of thm SVPH} follows by Corollary \ref{cor conseq  of thm for SVPH} choosing $\kappa=\kappa_1$. 

\subsection{Eigenfunctions regularity (quantitative)}\label{section conseq of theorem LY for SVPH}\ \\ 
As we have already seen in Corollary \ref{cor conseq of thm for SVPH}, the main consequence of Theorem \ref{thm LY for SVPH} is that there exists a Banach space $\cB_{s,*}\subset \cH^s$ on which the transfer operator $\cL_\ve$ is quasi compact for each $\ve<\ve_*$ with a uniform essential spectral radius. 

Using inequality \eqref{LY B and 0}, we can say much more about the constants. Indeed, by the previous section Theorem \ref{thm LY for SVPH} applies for $\kappa_0$ as in \eqref{eq:Kfinal} and $\kappa_1=n_0^{-1} \alpha c_1\log\ve^{-1}$. Hence,  for each $n, \kappa\in \bN$, $\kappa\in (\kappa_0,\kappa_1)$,
\[
\|\cL_\ve^{n}u\|_{s,*}\le  \Omega_{\chi_u}(\kappa n_1+\ovm_{\chi_u}, s)\left\{ \sigma_\kappa^{n}\|u\|_{s,*}+(1-\sigma_\kappa)^{-1}C_{\mu,n}  \mu^n\|u\|_{0}  \right\} ,
\]
where $m_{\chi_u}=c_0\log\ve^{-1}$ and $\sigma_\kappa$ is given in (\ref{ess spectrum}). The choice $\kappa=\Const\log\ve^{-1}$, in the proof of Theorem \ref{thm check assump of thm SVPH}, yielded a spectral radius uniform in $\ve$, but it provides no control on the constant $\Omega_{\chi_u}(\kappa n_1+\ovm_{\chi_u}, s)$. On the contrary, the choice $\kappa=2\kappa_0\in \bN$ (independent of $\ve$) implies, for some $c_\star>0$,
\[
\sigma_{\kappa_0}\in (e^{- c_\star\left(\log\ve^{-1}\right)^{-1}},1),
\]
 hence a weaker contraction, but it allows a better control of the constants. Indeed, observe that by \eqref{Cmu,n}
 \[
C_{\mu, n_1+m_{\chi_u}}\le \Const\min\{\log\ve^{-1}, \ve^{-1}\}=\Const\log\ve^{-1}.
 \]
Hence, it follows by \eqref{Theta (N, 1)} that we can find $\beta_3,\beta_3, \Const>0$ and $\epsilon>0$ such that 
\begin{equation}
\label{comp of Theta in s=1 case for Fve}
\Omega_{\chi_u}(2\kappa_0n_1+ \ovm_{\chi_u}, 1)\le \Const \ve^{-\frac{13}{2}}(\log \ve^{-1})^{\beta_1}e^{\beta_2\ve\log \ve^{-1}}.
\end{equation}
Thus, for each $\beta>\frac{13}{2}$, we have, for all $n\in\bN$,
\begin{equation}\label{eq:at-last?}
\begin{split}
&\|\cL_\ve u\|_0\leq Ce^{\bar c n\ve}\|u\|_0\\
&\|\cL_\ve^{n}u\|_{1,*}\le C_\alpha \ve^{-{\beta}}e^{-\frac{c_\star n}{\ln\ve^{-1}}}\|u\|_{1,*}+C_{\beta}\ve^{-\beta}\|u\|_{0}.
\end{split}
\end{equation}
\begin{proof}[{\bf Proof of Theorem \ref{finitely many SRB}}] 
Suppose $\cL_\ve u=\nu u$ with $\nu{ ^n}>e^{-\bbr\frac{c_\star n}{\ln\ve^{-1}}}$, $\bbr<1$, then
\[
\|u\|_{1,*}= \nu^{-n} \|\cL_\ve^{n}u\|_{1,*}\le C_\beta \ve^{-\beta}\nu^{-n}e^{-\frac{c_\star n}{\ln\ve^{-1}}}\|u\|_{1,*}+C_{\beta}\nu^{-n}\ve^{-\beta}\|u\|_{0}.
\]
We choose $n$ to be the smallest integer such that $C_\beta \ve^{-\beta}e^{-\frac{(1-\bbr)c_\star n}{\ln\ve^{-1}}}\leq \frac 12$, which yields
\[
\|u\|_{\cH^1}\leq \|u\|_{1,*}\leq C_\beta \ve^{-(1-\bbr)^{-1}\beta}\|u\|_{0}
\]
which concludes the proof.
\end{proof}

\subsection{Proof of Theorem \ref{thm:peripheral}}\label{sec:last}\ \\
Let $\sigma_{ph}(\cL_{F_\ve})=\{z\in\bC\;:\; |z|=1\}$ be the peripheral spectrum. By Theorem \ref{thm check assump of thm SVPH} the cardinality of $\sigma_{ph}(\cL_{F_\ve})$ is finite and if $e^{i\vartheta}\in \sigma_{ph}(\cL_{F_\ve})$, then the corresponding eigenspace is finite dimensional. In addition, since the operator is power bounded, there cannot exists Jordan blocks, thus the algebraic and geometric multiplicity are equal. 

Hence, there exist $N\in\bN$ and $\{\vartheta_j, h_j,\ell_j\}_{j=1}^N$ such that $\vartheta_0=1$, $\ell_0(\vf)=\int_{\bT^2}\vf$, $\vartheta_j\in [0,2\pi)$, $h_j\in\cB_{*,s}$, $\ell_j\in\cB_{*,s}'$ and $\cL_{F_\ve}h_j=e^{i\vartheta_j}h_j$, $\ell_j(\cL_{F_\ve}\vf)=e^{i\vartheta_j}\ell_j(\vf)$ for all $\vf\in\cB_{*,s}$.
It follows that we have the following spectral decomposition
\begin{equation}\label{eq:spectral-decomp}
\cL_{F_\ve}=\sum_{j}e^{i\vartheta_j} \Pi_j+ Q
\end{equation}
where $\Pi_j h=h_j\ell_j(h)$, $\Pi_j\Pi_k=\delta_{jk}\Pi_j$ and $Q$ has spectral radius strictly smaller than one.
Moreover, see \cite[Section 5]{BauL} for a proof which applies verbatim to the present context, the eigenvectors associated to the eigenvalue one are the physical measures and since they are absolutely continuous with respect to Lebesgue they are SRB as well. Also note that
\begin{equation}\label{eq:Pij}
\lim_{n\to\infty}\frac 1n \sum_{k=0}^{n-1}e^{-i\vartheta_j k}\cL_{F_\ve}^k=\Pi_j,
\end{equation}
where the limit is meant in the $\cB_{s,*}$ topology.  By Lemma \ref{LY B and L1} it follows the $\Pi_j$ are bounded operators from $L^1\to\cB_{*,s}$.\\
Also, choosing $\vf\in\cC^\infty$ is such that $\alpha:=\int_{\bT^2}\vf h_j> 0$, 
\[
\begin{split}
|\ell_j( h)|&=\alpha^{-1}\left|\int_{\bT^2}\vf \Pi_j h\right|=\alpha^{-1}\left|\lim_{n\to\infty}\frac 1n \sum_{k=0}^{n-1}e^{-i\vartheta_j k}\int_{\bT^2}\vf \cL_{F_\ve}^k h\right|\\
&\leq\alpha^{-1} \lim_{n\to\infty}\frac 1n \sum_{k=0}^{n-1}\int_{\bT^2} |\vf\circ F_\ve^k|\,|h|\leq \alpha^{-1}\|\vf_j\|_\infty\|h\|_{L^1}.
\end{split}
\]
Which implies that there exists $\tilde\ell_j\in L^\infty$ such that
\[
\ell_j(h)=\int_{\bT^2}\tilde\ell_j h.
\]
Note that the above also implies $\tilde\ell_j\circ F_\ve=e^{i\vartheta_j}\tilde\ell_j$.
The above means that, for all $l\in\bN$,
\[
\int_{\bT^2}\tilde \ell_j^l \cL_{F_\ve} h=\int_{\bT^2}\tilde  \ell_j^l\circ F_\ve h= e^{i\vartheta_j l}\int_{\bT^2} \tilde \ell_j^l h.
\]
This implies that $e^{i\vartheta_j l}$ belongs to the spectrum of $(\cL_{F_\ve})'$, hence of $\cL_{F_\ve}$.
Since there can be only finitely many elements of $\sigma_{ph}(\cL_{F_\ve})$, it must be $\vartheta_j=\frac{2\pi p}{q}$ for some $p,q\in\bN$, that is the $\{\vartheta_j\}$ form a collection of finite groups.

Next we would like to better understand the structure of the peripheral spectrum, and prove equations  \eqref{eq:projectors1} and  \eqref{eq:projectors3}.

Let $(x_k,\theta_k)=F_\ve^k(x,\theta)$ and $f_\theta(x')=f(x',\theta)$. By \cite[Lemma 4.2]{DeLi2}, for each $\theta\in\bT$, there exists $Y_{\theta,n}$ such that $\pi_2(F_\ve^n(x,\theta))=f_\theta^n(Y_{\theta,n}(x))$ and, for all $k\leq n$, 
\[
\begin{split}
&\|x_k-f_\theta^k\circ Y_{\theta,n}\|_{\infty}\leq \Const \ve k\\
&|\theta_k-\theta|\leq \Const k\ve\\
&\|1-\partial_xY_{\theta,n}\|_{\infty}\leq \Const \ve n^2.
\end{split}
\]
For each $n\in\{\Const\ln\ve^{-1},\dots,\Const\ve^{-\frac 12}\}$ we have
\[
\begin{split}
\int_{\bT^2}\vf \cL_{F_\ve}^nh&=\int_{\bT^2}\vf\circ F_\ve^n h=\int_{\bT^2}\vf(f_\theta^n(Y_{\theta,n}(x)),\theta) h(x,\theta)+\cO(\ve  n\|\partial_\theta\vf\|_{\cC^0}\|h\|_{L^1})\\
&=\int_{\bT^2}\vf(f_\theta^n(x,\theta) h(Y_{\theta,n}^{-1}(x),\theta)+\cO(\ve  n^2\|\partial_\theta\vf\|_{\cC^0}\|h\|_{L^1})\\
&=\int_{\bT^2}\vf(x,\theta)[\cL_\theta^n( h\circ Y_{\theta,n}^{-1})](x,\theta)+\cO(\ve  n^2\|\partial_\theta\vf\|_{\cC^0}\|h\|_{L^1}).
\end{split}
\]
Let $\cL_\theta$ be the transfer operator associated to $f_\theta$ and let $h_*(\cdot,\theta)$ be the associated unique invariant density.
Then, for each $\theta\in\bT$, $\cL_\theta$ has a uniform spectral gap $\sigma\in (0,1)$ on the Sobolev space $\cH^{1}(\bT)$. Thus
\[
\int_{\bT^2}\left|[\cL_\theta^n( h\circ Y_{\theta,n}^{-1})](x,\theta)-h_*(x,\theta)\int_{\bT}  (h\circ Y_{\theta,n}^{-1})(y,\theta)dy\right|dx d\theta
\leq \Const\sigma^n\|h\|_{\cH^1},
\]
since $\|h\circ Y_n^{-1}(\cdot,\theta)\|_{\cH^1}\leq \Const \|h(\cdot,\theta)\|_{\cH^1}$. Hence,
\begin{equation}\label{eq:encore_un_effort}
\begin{split}
\int_{\bT^2}\vf \cL_{F_\ve}^nh=&\int_{\bT^2}dx\vf(x,\theta)h_*(x,\theta) \int_{\bT}dy h(y,\theta)+\cO(\ve  n^2\|\partial_\theta\vf\|_{\cC^0}\|h\|_{L^1})\\
&+\cO\left(\sigma^n \|\vf\|_{\cC^0}\|h\|_{\cH^1}\right).
\end{split}
\end{equation}
We can then choose $n=c\ln\ve^{-1}$, for $c$ large enough, and obtain
\begin{equation}\label{eq:Q_bound}
\|\cL_{F_\ve}^{\const\ln\ve^{-1}}h-Ph\|_{(\cC^1)'}\leq \Const \ve[\ln\ve^{-1}]^2\|h\|_{L^1}+\Const \ve^{3\alpha} \|h\|_{\cH^1}.
\end{equation}
Equation \eqref{LY B and L1} yields 
\begin{equation}\label{eq:L1LY}
\|\cL_{F_\ve}^k h\|_{\cH^1}\leq C_\alpha\ve^{-\alpha} e^{-\frac{\const k}{\ln\ve^{-1}}}\|h\|_{\cB_{1,*}}+C_\alpha \ve^{-3\alpha}\ln\ve^{-1}\|h\|_{L^1}.
\end{equation}
Hence, by equation \eqref{eq:Q_bound} we have that, for each $\vf\in\cC^1$ and $h\in\cB_{1,*}$, 
\begin{equation}\label{eq:hope}
\begin{split}
&\int_{\bT^2}\vf \Pi_0 h=\lim_{n\to\infty}\frac 1n\sum_{k=0}^{n-1}\int_{\bT^2}\vf \cL_{F_\ve}^k h\\
&=\lim_{n\to\infty}\frac 1n\sum_{k=0}^{n-\const\ln\ve^{-1}}\int_{\bT^2}\vf \cL_{F_\ve}^{\const\ln\ve^{-1}}\cL_{F_\ve}^k h\\
&=\lim_{n\to\infty}\frac 1n\sum_{k=0}^{n-\const\ln\ve^{-1}}\int_{\bT^2}\vf P\cL_{F_\ve}^k h+\cO(\ve[\ln\ve^{-1}]^2\|\vf\|_{\cC^1}\|h\|_{L^1})\\
&=\int_{\bT^2}\vf P\Pi_0 h+\cO(\ve[\ln\ve^{-1}]^2\|\vf\|_{\cC^1}).
\end{split}
\end{equation}
Hence, by the density of $\cB_{1,*}$ in $L^1$ and since $\Pi_0$ extends naturally to a bounded operator on $L^1$, we have
\begin{equation}\label{eq:projectors}
\|\Pi_0 -P\Pi_0 \|_{L^1\to(\cC^1)'}\leq \Const \ve[\ln\ve^{-1}]^2. 
\end{equation}
The same argument proves a similar result for the projectors $\Pi_j$ yielding \eqref{eq:projectors1}.

It remains to prove equation \eqref{eq:projectors3}.
 For each $\tau>0$ consider $h\in\cB_{*,1}$ such that $\cL_{F_\ve} h=\nu h$ with $|\nu|\geq e^{-\ve^\tau}$.  Then, for all $\varphi\in\cC^1$ and $n\in\bN$, we have
\[
\begin{split}
\int_{\bT^2} \vf h&=\nu^{-n}\int_{\bT^2} \vf \cL_{F_\ve}^n h=\nu^{-n}\int_{\bT^2} \vf\circ F_\ve^n h\\
&=\nu^{-n}\int_{\bT^2} \vf(f_\theta^n\circ Y_{\theta,n}( x), \theta) h(x,\theta)+
\cO\left(\nu^{-n}n\ve\|\vf\|_{\cC^1}\|h\|_{L^1}\right)\\
&=\nu^{-n}\int_{\bT^2} \vf(f_\theta^n(x), \theta) h(Y_{\theta,n}^{-1}(x),\theta)+
\cO\left(\nu^{-n}n^2\ve\|\vf\|_{\cC^1}\|h\|_{L^1}\right)\\
&=\nu^{-n}\int_{\bT^2} \vf( x, \theta) (\cL_\theta^n[ h_\theta\circ Y_{\theta,n}^{-1}])(x)+
\cO\left(\nu^{-n}n^2\ve\|\vf\|_{\cC^1}\|h\|_{L^1}\right)
\end{split}
\]
where $h_\theta(x)=h(x,\theta)$. Note that, by inequality  \eqref{eq:L1LY}, we have $\|h\|_{\cH^1}\leq \Const \ve^{-3\alpha}\ln\ve^{-1}\|h\|_{L^1}$. Thus
\[
\begin{split}
\int_{\bT^2} \vf( x, \theta) (\cL_\theta^n[ h_\theta\circ Y_{\theta,n}^{-1}])(x)&=\int_{\bT^2}dx d\theta \vf(x,\theta) h_*(x,\theta)\int_{\bT} dy h(Y_{\theta,n}^{-1}(y),\theta)
+\cO(\sigma^n\|\vf\|_{\cC^0}\|h\|_{\cH^1})\\
&=\int_{\bT^2} \vf Ph 
+\cO(\left[ \sigma^n\ve^{-3\alpha}\ln\ve^{-1}+n^2\ve\right]\|\vf\|_{\cC^0}\|h\|_{L^1}).
\end{split}
\]
To conclude we choose $n=c\ln\ve^{-1}$, with $c$ large enough, and obtain
\[
\begin{split}
\int_{\bT^2} \vf h&=\nu^{-n}\int_{\bT^2}\vf Ph+\nu^{-n}\cO\left(\ve(\ln\ve^{-1})^2\|\vf\|_{\cC^1}\|h\|_{L^1}\right)\\
&=\int_{\bT^2}\vf Ph+\cO\left(\ve^{\min\{1,\tau\}}(\ln\ve^{-1})^2\|\vf\|_{\cC^1}\|h\|_{L^1}\right).
\end{split}
\]
It follows that there exists a $\beta_h\in \cH^1(\bT)$, $\beta_h(\theta)=\int_{{\bT}} dy h(y,\theta)$, such that
\[
\|h-h_*\beta_h\|_{(\cC^1)'}\leq \Const \ve^{\min\{1,\tau\}}(\ln\ve^{-1})^2\| h\|_{L^1}.
\]
\appendix
\section {Proof of Lemma \ref{lem:gen_ex} }\label{sec:gen_ex}
To start, note that (1) coincides with the first part of {\bf (H4)}, which implies in particular that $\lambda>2$. We have thus to prove only {\bf (H0)} up to {\bf (H3)} and the second part of {\bf (H4)}.

We start with {\bf (H0)}. It suffices to show that $\partial_xf(p)>[\partial_\theta f\partial_x\omega-\partial_xf\partial_\theta\omega](p),$ for each $p\in\bT^2$. The latter, by \eqref{item:no-vertical} and \eqref{cond phi positive}, is implied by $\lambda(1-|\partial_\theta\omega|)>\lambda-1-\|\partial_\theta\omega\|_\infty-\|\partial_x\omega\|_\infty$ which, in turn, is implied by \eqref{eq:omegatheta-cond}.\\
Next we prove ({\bf H1}). Following \cite{DeLi1} we start by proving that $D_pF(\mathbf{C}_u)\Subset \mathbf{C}_u$ and $D_pF^{-1}(\mathbf{C}_c)\Subset \mathbf{C}_c$.
We consider a vector $(1, u)\in \fC_u$ and we write a formula for the unstable slope field
\begin{align}
\label{eq: def of Xi}
&D_pF (1, u)=(\partial_x f+ u \partial_{\theta} f)(1, \Xi({p,u })),
&\Xi({p,u })=\frac{\partial_x\omega(p)+ u \partial_{\theta}\omega(p) +u}{\partial_x f(p)+u\partial_{\theta} f(p)}.
\end{align}
Notice that
\begin{equation}
\label{u derivative of Xi}
\frac{d}{du}\Xi(p, \cdot)=\frac{\partial_xf+ (\partial_\theta\omega\partial_x f-\partial_\theta f\partial_x \omega)}{(\partial_x f+ u\partial_\theta f)^2}=\frac{\det DF(x,\theta)}{(\partial_x f+ u\partial_\theta f)^2}>0,
\end{equation}
since $\det DF>0$ by {\bf (H0)}. Hence, checking the invariance of $\fC_u$ under $DF$ is equivalent to showing that, for each $p\in \bT^2$,  $|\Xi(p,\pm\chi_u)|\le \chi_u$. That is
\begin{equation}
\label{inequality for invariance of C_u}
\|\partial_{\theta} f\|_\infty\chi_u^2-\left( \lambda- \|\partial_{\theta}\omega\|_\infty  -1 \right)\chi_u+\|\partial_x\omega \|\le 0.
\end{equation}
Setting $\phi=\lambda- \|\partial_{\theta}\omega\|_\infty-1$, inequality \eqref{inequality for invariance of C_u} has positive solutions since $\phi>0$ { by \eqref{cond phi positive}, which also implies
\[
\phi^2-4\|\partial_\theta f\|_{\infty}\|\partial_x\omega\|_\infty\geq (\|\partial_\theta f\|_\infty-\|\partial_x\omega\|_\infty)^2>0.
\] 
}
Setting $\Phi_\pm=\phi\pm\sqrt{\phi^2-4\|\partial_\theta f\|_{\infty}\|\partial_x\omega\|_\infty}$,
we can then choose
\begin{equation}
\label{chi_u}
\chi_u\in\left(\frac{\Phi_-}{2\|\partial_\theta f\|_\infty}, 1\right).
\end{equation}
Note that the interval it is not empty due to \eqref{cond phi positive}.\\
On the other hand, if $(c,1)\in \fC_c$ we consider the center slope field
\begin{equation}
\label{eq: def of Xi^-}
\Xi^-({p, c })=\frac{\left(1+\partial_{\theta} \omega(p)\right) c-\partial_{\theta} f(p)}{\partial_{x} f(p)- \partial_{x} \omega(p) c},
\end{equation}
and by an analogous computation we obtain $|\Xi^-(p, \pm \chi_c)|< \chi_c$ if
\begin{equation}
\label{chi_c}
\chi_c\in\left(\frac{\Phi_-}{2\|\partial_x \omega\|_\infty},1\right].
\end{equation}
This also proves the second part of {\bf (H4)}.

Again,  the interval it is not empty due to \eqref{cond phi positive}, we have thus proved (\ref{invariance of cone}).\\
Next, by the invariance of the cones we can define real quantities $\lambda_n, \mu_n, u_n$ and $c_n$ such that, for each $p\in \bT^2$,\footnote{ Note that the definition of $\lambda_n$ differs from the one of $\lambda^{\pm}_n$ in \eqref{def of lamb^+ mu^+}, since we are considering iteration of vectors inside the unstable cone. Nevertheless, they are related since there exists an integer $m$ such that $F^{m}(\bR^2\setminus \fC_c)\Subset \fC_u$.} 
\[
D_{p} F^{n}(1,0)=\lambda_{n}(p)\left(1, u_{n}{(p)}\right) \;{;}\quad D_{p} F^{n}\left(c_{n}{(p)}, 1\right)=\mu_{n}{(p)}(0,1),
\]
with $\|u_n\|_\infty\le \chi_u,$ $\|c_n\|_\infty\le \chi_c$. Moreover, by definition
\[
D_{p} F\left(c_{n}(p), 1\right)=\frac{\mu_{n}(p)}{\mu_{n-1}\left(F(p)\right)}\left(c_{n-1}\left(F(p)\right), 1\right),
\]
from which it follows, by \eqref{the map F},
\[
\mu_{n}(p)={\mu_{n-1}\left(F(p)\right)} (1+\partial_\theta \omega(p)+c_n(p)\partial_x \omega(p)).
\]
We set $b:=\|\partial_\theta \omega\|_{\infty}+\chi_c \|\partial_x \omega\|_{\infty}$. Since $\|c_n\|_\infty\le \chi_c$,  we have that the condition \eqref{item:no-vertical} of the Lemma implies $b<1$, and
\begin{equation}\label{mu_n<e^b} 
(1-b)^n\le\mu_n(p)\le  (1+b)^n.
\end{equation}
Similarly,
\[
\begin{split}
\lambda_n(p)&=\lambda_{n-1}(F(p))(\partial_x f(p)+\partial_{\theta}f(p)u_n(p)) \\
&={\prod_{k=0}^{n-1}\partial_x f(F^kp)\left(\partial_x f(F^kp)+\frac{\partial_{\theta}f(F^k)}{\partial_x f(F^kp)}u_{n-k}(F^kp)\right)},
\end{split}
\]
which, setting $a:=\chi_u\|\frac{\partial_\theta f}{\partial_x f}\|_\infty$, implies
\begin{equation}
\label{nu_n<e^a}
(1-a)^n\prod_{k=0}^{n-1}\partial_xf(F^k(p))\le \lambda_n(p)\le (1+a)^n\prod_{k=0}^{n-1}\partial_xf(F^k(p)),
\end{equation}
which yields the second of \eqref{partial hyperbolicity 2} with $C_\star=1$ and
\begin{equation}\label{defn:lambda+-}
\lambda_+=(1+a)\lambda \quad \text{and} \quad \lambda_-=(1-a)\lambda,
\end{equation}
since, by the definition of $\chi_u$ in \eqref{chi_u}, we can check that $\lambda_->1$.
To conclude, we need to check that $\frac{(1+b)}{(1-a)\lambda}<1,$ form which we deduce {\bf (H1)}.
This is implied by
\[
1+\|\partial_\theta\omega\|_\infty+\|\partial_x\omega\|_\infty+\|\partial_\theta f\|_\infty<\lambda
\]
which correspond to equation \eqref{cond phi positive}.\\
By \eqref{mu_n<e^b}, we can make the choice \eqref{rem:mu_choice} which, by \eqref{item:no-vertical}, implies $\mu \leq  e^{2\chi_c\|\partial_x\omega\|_\infty+2\|\partial_\theta\omega\|_\infty)}$ and hence we obtain ({\bf H3}) by \eqref{eq:pinching}.\\
It remains to prove ({\bf H2}).
Since $\lambda>2$, $F$  has rank at least two at each point, hence it is a covering map and each point has the same number of preimages, says $d$. Let then $\gamma:[0,1]\to \bT^2$ be a smooth closed curve $\gamma(t)=(c(t),t)$ such that $\gamma'\in \fC_c$ with homotopy class $(0,1)$. 
If $p=(x,\theta)\in \gamma(t)$ then  $F^{-1}(p)=\{q_1,\cdots, q_d\}$. Note that, by the implicit function theorem, locally $F^{-1}\gamma$ is a curve, also, due to the above discussion, it belongs to the central cone. If we call $\eta$ the local curve in $F^{-1}\gamma$ such that $\eta(0)=q_i$ we can extend it uniquely to a curve $\nu:[0,1]\to \bT^2$. We will prove that $\nu(1)=q_i=\nu(0)$. In turn this implies that $F^{-1}\gamma$ is the union of $d$ closed curves $\nu_1,\cdots, \nu_d$ with $\nu'_i\in \fC_c$, each one with homotopy class $(0,1)$, by the lifting property of covering maps (see \cite[Proposition 1.30]{All}). 
{We argue by contradiction: assume that $\nu(1)=q_j\neq q_i$. Let $q_k=(x_k,\theta_k)$, $k\in\{1,\dots d\}$,} then
\[
\theta_i+\omega(x_i,\theta_i)=\theta_j+\omega(x_j,\theta_j)
\]
implies
\begin{equation}
\label{eq: Fqi=Fqj}
|\theta_i-\theta_j|\leq \frac{ \|\partial_x\omega\|_\infty}{1-\|\partial_\theta \omega\|_\infty}|x_i-x_j|.
\end{equation}
Hence the segment joining $q_i$ and $q_j$ belongs to the unstable cone if
\begin{equation}\label{eq:separate}
\chi_u\geq \frac{ \|\partial_x\omega\|_\infty}{1-\|\partial_\theta\omega\|_\infty}
\end{equation}
which is possible since \eqref{item:no-vertical} implies that this condition is compatible with \eqref{chi_u}. It follows that the image of the segment 
$\ell=\{tq_i+(1-t)q_j\}$ is an unstable curve and hence it cannot join $p$ to itself without wrapping around the torus. In particular, if $q_i\neq q_j$, {then the horizontal length of $F(\ell)$ must be larger than one. Then, setting $\delta=|x_i-x_j|$,}
\begin{equation}
\label{eq: length of Fell}
\begin{split}
&1\leq \int_0^1\left|\langle e_1, D_{\ell(t)}F \ell'(t)\rangle\right|\leq \|\partial_x f\|_\infty\left(1+\chi_u\frac{\|\partial_\theta f\|_\infty}{\|\partial_x f\|_\infty}\right)|x_i-x_j|
\leq (1+a)\Lambda\delta.
\end{split}
\end{equation}
To conclude we must show that $\nu$ cannot move horizontally by $\delta$, whereby obtaining the wanted contradiction. Let $\nu(t)=(\alpha(t),\beta(t))$, then
\[
\begin{split}
\begin{pmatrix}c'(t)\\1\end{pmatrix}=\gamma'(t)=DF\nu'=\begin{pmatrix} \alpha'\partial_x f+\beta'\partial_\theta f\\ \alpha'\partial_x \omega+(1+\partial_\theta\omega)\beta'\end{pmatrix}.
\end{split}
\]
Since we know that $|c'|\leq \chi_c$ and $|\alpha'|\leq \chi_c|\beta'|$ we have
\[
\begin{split}
&|\beta'|\leq (1-\chi_c\|\partial_x\omega\|_\infty-\|\partial_\theta\omega\|_\infty)^{-1}\\
&|\alpha'|\leq \frac{\chi_c}{\lambda}+\frac{\|\partial_\theta f\|_\infty}{\lambda(1-\chi_c\|\partial_x\omega\|_\infty-\|\partial_\theta\omega\|_\infty)}.
\end{split}
\]
I follows that it must be
\[
\frac{1}{(1+a)\Lambda}\leq \delta\leq \int_0^1|\alpha'(t)|dt\leq \frac{\chi_c}{\lambda}+\frac{\|\partial_\theta f\|_\infty}{\lambda(1-\chi_c\|\partial_x\omega\|_\infty-\|\partial_\theta\omega\|_\infty)}.
\]
We thus have a contradiction if we can choose $\chi_c$ such that
\[
\left(1+\frac{\|\partial_\theta f\|_\infty}{\lambda}\right)\Lambda\left[\frac{\chi_c}{\lambda}+\frac{\|\partial_\theta f\|_\infty}{\lambda(1-\|\partial_x\omega\|_\infty-\|\partial_\theta\omega\|_\infty)}\right]<1
\]
which, by \eqref{chi_c}, is possible only if
\[
\frac{\Phi_-}{2\|\partial_x \omega\|_\infty}<\left(1+\frac{\|\partial_\theta f\|_\infty}{\lambda}\right)^{-1}\frac{\lambda}{\Lambda}-\frac{\|\partial_\theta f\|_\infty}{1-\|\partial_x\omega\|_\infty-\|\partial_\theta\omega\|_\infty}=:A.
\]
Note that if $A\geq 1$, then the inequality is trivially satisfied. We must then consider only the case $A<1$. A direct computation shows that the above inequality is implied by
\begin{equation}
\label{last eq for ftheta}
\|\partial_\theta f\|_\infty< A\left[\phi-A\|\partial_x\omega\|_\infty\right]=A\left[\lambda- \|\partial_{\theta}\omega\|_\infty-1-A\|\partial_x\omega\|_\infty\right]
\end{equation}
Let us set for simplicity $\varpi:=\|\partial_\theta \omega\|_\infty+\|\partial_x\omega\|_\infty$. Since $A<1$ the above equation is in turn implied by the following inequality
\begin{equation}
\label{inequality 1 for upsilon}
\|\partial_\theta f\|_\infty< \left[ \left(1+\frac{\|\partial_\theta f\|_\infty}{\lambda}\right)^{-1}\frac{\lambda}{\Lambda}-\frac{\|\partial_\theta f\|_\infty}{1-\varpi}\right]\left(\lambda-(1+\varpi) \right).
\end{equation}
By elementary algebra (\ref{inequality 1 for upsilon}) is equivalent to 
\begin{equation}
\label{inequality 2 for upsilon}
\|\partial_\theta f\|_\infty(\|\partial_\theta f\|_\infty+1)<\frac{\lambda^2}{\Lambda}\left(1-\frac{1}{\lambda+\varpi}\right) .
\end{equation} 
Since $\lambda>2$, (\ref{inequality 2 for upsilon}) is implied by $\|\partial_\theta f\|_\infty(\|\partial_\theta f\|_\infty+1)<\frac 12 \lambda^2\Lambda^{-1}$,
which is true if \\ $\|\partial_\theta f\|_\infty<\frac 12\left(-1+\sqrt{1+ 2\lambda^2\Lambda^{-1}}\right).$ Hence the conclusion by condition (\ref{condition ftheta 1}).

\section{Proof of Lemma \ref{properties of the C norm} } 
\label{appendix proof of lem prop of the C norm}
We start considering $\varphi, \psi \in \cC^{\rho}(\bT^2,\bR)$. First we prove, by induction on $\rho$, that
\begin{equation}
\label{banach alg property}
\sup_{|\alpha|=\rho}\|\partial^{\alpha}(\varphi \psi)\|_{\cC^0}\le \sum_{k=0}^{\rho}{\binom{\rho}{k}}\sup_{|\beta|=\rho-k}\|\partial^{\beta}\varphi\|_{\mathcal{C}^0} \sup_{|\gamma|=k}\|\partial^{\gamma}\psi\|_{\mathcal{C}^0}
.\end{equation}
Indeed, it is trivial for $\rho=0$ and
\[
\begin{split}
&\|\partial_{x_i}\partial^\alpha(\varphi\psi)\|_{\cC^0}=\|\partial^\alpha(\psi\partial_{x_i}\varphi+\varphi\partial_{x_i}\psi)\|
\\
&\le  \sum_{k=0}^{\rho}{\binom{\rho}{k}} \sup_{|\beta|=\rho-k}\|\partial^{\beta}\partial_{x_i}\varphi\|_{\mathcal{C}^0} \sup_{|\gamma|=k}\|\partial^{\gamma}\psi\|_{\mathcal{C}^0}+\sum_{k=0}^{\rho}{\binom{\rho}{k}}\sup_{|\beta|=\rho-k}\|\partial^{\beta}\partial_{x_i}\psi\|_{\mathcal{C}^0} \sup_{|\gamma|=k}\|\partial^{\gamma}\varphi\|_{\mathcal{C}^0}\\
&\le \sum_{k=0}^{\rho}{\binom{\rho}{k}}\sup_{|\beta|=\rho-k}\|\partial^{\beta}\varphi\|_{\mathcal{C}^0} \sup_{|\gamma|=k}\|\partial^{\gamma}\psi\|_{\mathcal{C}^0}+\sum_{k=0}^{\rho+1}{\binom{\rho}{\rho+1-k}}\sup_{|\beta|=\rho-k}\|\partial^{\beta}\varphi\|_{\mathcal{C}^0} \sup_{|\gamma|=k}\|\partial^{\gamma}\psi\|_{\mathcal{C}^0},
\end{split}
\]
from which (\ref{banach alg property}) follows taking the sup on $\alpha,i$ and since ${\binom{\rho}{k}}+{\binom{\rho}{\rho+1-k}}={\binom{\rho +1}{k}}$. We then have the first statement of the Lemma, indeed
\[
\begin{split}
\|\varphi\psi\|_{\cC^\rho}& \le \sum_{k=0}^{\rho}2^{\rho-k}\sum_{j=0}^{k}{\binom{k}{j}} \sup_{|\beta|=k-j}\|\partial^{\beta}\varphi\|_{\mathcal{C}^0} \sup_{|\gamma|=j}\|\partial^{\gamma}\psi\|_{\mathcal{C}^0}\\
&\le \sum_{j=0}^{\rho}\sum_{l=0}^{\rho-j}{\binom{\rho}{j}} 2^{\rho-j-l}\sup_{|\beta|=l}\|\partial^{\beta}\varphi\|_{\mathcal{C}^0} \sup_{|\gamma|=j}\|\partial^{\gamma}\psi\|_{\mathcal{C}^0}\le  \|\varphi\|_{\cC^\rho}\|\psi\|_{\cC^\rho}
\end{split}
\]
since ${\binom{\rho}{j}}\le 2^\rho.$ The extension to function with values in the matrices follows trivially since we have chosen a norm in which the matrices form a norm algebra.\\
\noindent To prove the second statement we proceed again by induction on $\rho$.
The case $\rho=0$ is immediate since $\cK_{0,0}$ contains only the zero string. Let us assume that the statement is true for every $k\le \rho$ and prove it for $\rho+1$. By equation \eqref{banach alg property} and the inductive hypothesis \eqref{formula 1 for norm Crho}, we have, for each $|\alpha|=\rho+1$, 
\[
\begin{split}
\left|\partial^\alpha(\vf\circ \psi)\right|&\leq\Const \sup_{|\beta|=\rho}\sup_{|\tau_1|,|\tau_2|=1}
\left|\partial^\beta\left[(\partial^{\tau_1}\vf)\circ \psi\cdot \partial^{\tau_2}\psi\right]\right|\\
&\leq C_{\rho}\sup_{|\tau_1|,|\tau_2|=1}\sup_{|\alpha_0|+|\alpha_1|=\rho}\|\partial^{\alpha_0}\left[(\partial^{\tau_1}\vf)\circ\psi\right]\|_{\cC^0}\|\partial^{\alpha_1}\partial^{\tau_2}\psi\|_{\cC^0}\\
&\leq C_{\rho}\sup_{|\tau_1|=1}\sup_{\alpha_0\leq\rho}
\|(\partial^{\tau_1}\vf)\circ\psi\|_{\cC^{\alpha_0}} \|D\psi\|_{\cC^{\rho-\alpha_0}}\\
&\leq C_\rho C^\star_{\rho} 
\sup_{\alpha_0\leq\rho}\sum_{s=0}^{\alpha_0}\|\vf\|_{\cC^{s+1}}
\sum_{k\in\cK_{\alpha_0,s}}\prod_{l=1}^{\alpha_0} \|D\psi\|_{\cC^{l-1}}^{k_l}\cdot \|D\psi\|_{\cC^{\rho-\alpha_0}}\\
&\leq C_\rho C^\star_{\rho} 
\sup_{\alpha_0\leq\rho}\sum_{s=0}^{\alpha_0}\|\vf\|_{\cC^{s+1}}
\sum_{k\in\cK_{\rho,s+1}}\prod_{l=1}^{\rho+1} \|D\psi\|_{\cC^{l-1}}^{k_l}\\
&\leq C_\rho C^\star_{\rho} 
\sum_{s=0}^{\rho+1}\|\vf\|_{\cC^{s}}
\sum_{k\in\cK_{\rho+1,s}}\prod_{l=1}^{\rho+1} \|D\psi\|_{\cC^{l-1}}^{k_l}.
\end{split}
\]
The result follows by choosing $C_{\rho+1}^\star$ large enough. $\qed$

\section{Black-box: iteration of curves}\label{appG}
\subsection{Proof of Lemma \ref{lem stable vertic curves}}
Fix $\gamma\in \Gamma_j(c_\star)$ and $n\in \bN$. Let $\nu_{n}$ be a pre-image of $\gamma$ under $F^{n}$ and consider $\frh\in \frH^\infty$ such that $\nu_{n}=\frh_{n}\circ \gamma$.
Let $h_{{n}}:\mathbb{T}\to \mathbb{T}$ be the diffeomorphism such that $\hat\nu_{n}=\nu_{n}\circ h_{n}$ is parametrized by vertical length. We then want to check properties  $c1),...,c3)$ for $\hat{\nu}_{n}$. The first two follow immediately by assumption ({\bf H2}), thus we only have to check property $c3)$ (from Definition \ref{adm cent curve}). By definition we have 
\begin{equation}
\label{eq: Fq nu=gamma h}
F^{ n} \hat \nu_{ n}=\gamma \circ h_{ n}.
\end{equation}
Differentiating equation \eqref{eq: Fq nu=gamma h} twice we obtain
\begin{equation}
\label{eq: hat nu''}
(\partial_t D_{\hat \nu_{n}}F^{ n})\hat \nu_{ n}'+D_{\hat \nu_{ n}}F^{ n}\hat \nu_{ n}''=\gamma''\circ h_{ n}(h'_{ n})^2+\gamma'\circ h_{ n}h_{ n}''.
\end{equation}
Similarly, if we differentiate equation \eqref{eq: Fq nu=gamma h} $j$-th times, we obtain
\begin{equation}
\label{eq: hat nu^j}
R_j(F^n,\hat \nu_n)+D_{\hat \nu_{ n}}F^{ n}\hat \nu_{ n}^{(j)}=\gamma^{(j)}\circ h_{ n}(h'_{ n})^j+Q_j(h_n, \gamma)+\gamma'\circ h_{ n}\cdot h_{ n}^{(j)},
\end{equation}
where, by \eqref{eq:ve-special-inv}, $R_j(F^n,\hat \nu_n)=A_{j,n}+B_{j,n}$ with $A_{j,n}\in \fC_u$ and $\|B_{j,n}\|\leq \|\omega\|_{\cC^r}\|\tilde B_{j,n}\|$. Moreover, $A_{j,n}, \tilde B_{j,n}$   are the sum of monomials, with coefficients depending only of $(\partial^\alpha F^{n})\circ \hat\nu_n$ with $|\alpha|\leq j$, in the variables $\hat\nu_n^{(s)}$, $s\in \{0,\dots j-1\}$, where, if $k_s$ is the degree of $\hat\nu^{(s)}$, then we have $\sum_{s=1}^{j-1} sk_s=j$.

Likewise the $Q_j$ are the sum of monomials that are linear in $\gamma^{(\sigma)}$, $\sigma\in \{2,\dots j-1\}$, and of degree $p_s$ in $h_n^{(s)}$, $s\in \{1,\dots j-\sigma+1\}$, such that $\sum_{s=1}^{j-\sigma+1}s\,p_s=j$.\footnote{ The reader can check this by induction (equation \eqref{eq: hat nu''} gives the case $j=2$). E.g., if a term $T$ in $R_j$ has the form $T=\prod_{s=0}^{j-1}\alpha_s(\hat\nu_n^{(s)})$ where $\alpha_s(x)$ is homogeneous of degree $k_s$ in $x$, then $\partial_t T$ will be a sum of terms of the same type with homogeneity degrees $k'_s$. Let us compute such homogeneity degrees: if the derivative does not hit a $\hat\nu_n^{(s)}$, $s>0$, then, by the chain rule, we will get a monomial with $k'_1=k_1+1$ while all the other homogeneity degree are unchanged: $k_s'=k_s$ for $s>1$. Hence, $\sum_{s=0}^j k_s'=j+1$. If the derivative hits one $\hat\nu_n^{(i)}$, then it produces a monomial with $k_s'=k_s$ for $s\not\in\{i,i+1\}$ while $k_i'=k_i-1$ and $k'_{i+1}=k_{i+1}+1$. Then $\sum_{s=0}^jk_s'=j-ik_i-(i+1)k_{i+1}+i(k_i-1)+(i+1)(k_{i+1}+1)=j+1$.} In order to obtain an estimate for $\|\hat \nu_n^{(j)}\|$ it is convenient to introduce the vectors $\eta_{n,j}=D_{\hat \nu_n}F^n \hat \nu_n^{(j)}$. We then define the unitary vectors $\eta_{n,j}^\bot,\hat \eta_{n,j}$ such that $\langle \eta_{n,j}^\bot, \eta_{n,j} \rangle=0$ and $\hat \eta_{n,j}=\frac{\eta_{n,j} }{\|\eta_{n,j}\|}.$ Multiplying equation \eqref{eq: hat nu^j} by $\eta_{n,j}^\bot$ and $\hat\eta_{n,j}$ respectively, we obtain the system of equations
\begin{equation}
\label{eq: system fo h'' and eta''}
\begin{split}
&\langle \eta_{n,j}^\bot, R_j(F^n,\hat \nu_n)\rangle=\left\langle \eta_{n,j}^\bot,\gamma^{(j)}\circ h_{ n}(h'_{ n})^j+Q_j(h_n,\gamma)+\gamma'\circ h_{ n}\cdot h_{ n}^{(j)}\right\rangle\\
&\langle \hat \eta_{n,j}, R_j(F^n,\hat \nu_n)\rangle+\|\eta_{n,j}\|=\left\langle \hat \eta_{n,j},\gamma^{(j)}\circ h_{ n}(h'_{ n})^{j}+Q_j(h_n,\gamma) +\gamma'\circ h_{ n} \cdot h_{ n}^{(j)}\right\rangle.
\end{split}
\end{equation}  
Notice that, since $\hat\nu_n^{(j)}$ and $\gamma^{(j)}$, $j>1$, are horizontal vectors, by the invariance of the unstable cone $\eta_{n,j}\in \fC_u.$ Moreover $\gamma'\in\fC_c$ by assumption and $\|\eta_{n,j}^\bot\|=1$, thus there exists $\vartheta\in (0,1)$ such that 
\begin{equation}
\label{eq: eta^bot scalar gamma'}
 |\langle \eta_{n,j}^\bot,\gamma'\circ h_{ n}\rangle|\geq \vartheta\|\gamma'\circ h_n\|\geq \vartheta.
\end{equation}
Using \eqref{eq: eta^bot scalar gamma'}, setting $\tilde R_{j,n}:=|\langle \eta_{n,j}^\perp, R_j(F^n,\hat \nu_n)\rangle|+\|Q_j(h_n,\gamma)\|$, and $R_{j,n}:=\|R_j(F^n,\hat \nu_n)\|+\|Q_j(h_n,\gamma)\|$, equation  \eqref{eq: system fo h'' and eta''} yields
\begin{equation}
\label{eq: system estimates}
\begin{split}
&|h_n^{(j)}|\le \frac{|h'_n|^j\|\gamma^{(j)}\circ h_n\|\chi_u+\tilde R_{j,n}}{\vartheta\|\gamma'\circ h_n\|}, \\
&\|\eta_{n,j}\|\le \|\gamma^{(j)}\circ h_n\||h'_n|^j+\|\gamma'\circ h_n\||h_n^{(j)}|+R_{j,n}.
\end{split}
\end{equation}
By equation \eqref{eq: Fq nu=gamma h} it follows that
\begin{equation}
\label{eq: hat nu '}
\|\hat{\nu}_n'\|=|h'_n| \|(D_{\hat{\nu}_n}F^n)^{-1} \gamma'\circ h_n\|,
\end{equation}
which yields, by \eqref{partial hyperbolicity 2} and the fact that $\hat \nu_n'=((\pi_1\circ \hat\nu_n)', 1)\in \fC_c,$
\begin{equation}
\label{eq: h_n'}
\frac{\mu^{-n}}{C_\star \sqrt{1+\chi_c^2}}\le |h_n'|\le \frac{C_\star\mu^n\|\hat\nu'_n\|}{\|\gamma'\circ h_n\|}\le {\sqrt{1+\chi_c^2}C_\star \mu^n}=:{\bar C_\star \mu^n}.
\end{equation}
Using this in \eqref{eq: system estimates} and observing that $\|\eta_{n,j}\|=\|D_{\hat \nu_n}F^n \hat \nu_n^{(j)}\|\ge \lambda_n^{-}\|\hat \nu_n^{(j)}\|$, we obtain
\begin{equation}
\label{eq: eta_nj}
\|\hat \nu_n^{(j)}\|\le  \|\gamma^{(j)}\circ h_n\|  (\lambda_n^{-})^{-1}(\bar C_\star \mu^{\bar n})^jA+R^\star_{j,n},
\end{equation}
where $A= (1+\vartheta^{-1})$ and $R_{j,n}^\star=(\lambda_n^{-})^{-1}AR_{j,n}$. 
Next we choose $\bar n$. The choice depends on a uniform constant $C_\flat\ge \{3\bar C_*,\frac{2\bar C_\star}\vartheta,\Const r\}^+ $ that will be chosen, large enough, in \eqref{eq:end_induction}:\footnote{ Note that this is possible due to hypothesis {\bf (H3)}, see \eqref{def of zetar}.}
\begin{equation}\label{def:eta and barn}
\begin{split}
    &A(\bar C_\star\mu^{\bar n})^{r+1}(\lambda_{\bar n}^-)^{-1}\left[1+C_\flat^{(r+1)^2}\mu^{\bar n(r+1)!}\right]<1,\textrm{ and set }\\
    &\eta:=\left(A(\bar C_\star\mu^{\bar n})^{r+1}(\lambda_{\bar n}^-)^{-1}\left[1+C_\flat^{(r+1)^2}\mu^{\bar n(r+1)!}\right]\right)^{\frac{1}{2\bar n(r+1)!}}<1.
\end{split}
\end{equation}
Note in particular that, as $\bar C_\star$ and $\vartheta$ are uniform constants, so are both $\bar n$ and $\eta$. 

We are ready to conclude. For $j=1$ the Lemma is trivial since $\|\hat\nu_{\bar n}'\|\leq \sqrt{1+\chi_c^2}$ and $h_{\bar n}'$ can be bounded by \eqref{eq: h_n'}, provided $C_\flat\geq \bar C_\star$. Equation \eqref{eq: hat nu''} 
implies that $\|R_2(F^{2\bar n},\hat\nu_{2\bar n})\|\leq C_{\bar n}$ while $\|\tilde R_2(F^{2\bar n},\hat\nu_{2\bar n})\|\leq C_{\bar n} C_F$ and $Q_2=0$,\footnote{ Recall that $C_F=\chi_u+\|\omega\|_{\cC^r}$, as defined just after \eqref{eq:ve-special-inv1}.} 
thus $R_{2,2\bar n}\leq C_{\bar n}$. Then the first of \eqref{eq: system estimates}, remembering equations \eqref{eq: h_n'} and \eqref{eq: eta_nj} imply
\begin{equation}\label{eq:nu_h2}
\begin{split}
&\|h_{n}^{(2)}\|\leq C_\sharp C_F( \bar C_\star^{2} c_\star\mu^{n} +C_{\bar n}) \quad\forall n\leq 2\bar n\\
&\|\hat \nu_n^{(2)}\|\le A(\lambda_n^{-})^{-1}\left\{ c_\star (\bar C_\star \mu^n)^2+C_{\bar n}\right\}.
\end{split}
\end{equation} 
Next, we proceed by induction on $j< \ell$ to prove that for each $\bar n\le n\le 2\bar n$
\begin{equation}\label{eq:nu_h_induction}
\begin{split}
&\|h_{n}^{(j)}\|\leq C_\flat^{j^2} C_F \left(c_\star^{(j-1)!}\mu^{j! n}+C_{\bar n}\right)\\ 
&\|\hat\nu_{n}^{(j)}\|\leq (\eta^{n}c_\star+\bbc/2)^{(j-1)!}.
\end{split}
\end{equation}
By \eqref{eq:nu_h2} we have the case $j=2$, let us assume it for all $2\leq s\leq j$.
Recalling the structure of $R_j, \tilde R_j, Q_j$, see equation \eqref{eq: hat nu^j} and comments thereafter, and provided $c_{\star}\geq C_{\bar n}$ we have\footnote{ The sum in the exponent of $c_\star$ starts from $2$ since bound of $h'$ does not contains $c_\star$.} 
\[
\begin{split}
&R_{j+1, n } \le C_\sharp\Bigg\{ \sum_{k}c_{\star}^{\sum_{s=1}^{j} { (s-1)!}k_s}+C_\flat^{(j+1)j}\sum_{\sigma=2}^{j}\sum_p c_\star^{ (\sigma-1)!+\sum_{s=2}^{j+2-\sigma}p_s (s-1)!}\mu^{n\sum_{s=1}^{j+2-\sigma}p_s (s-1)!}+C_{\bar n}\Bigg\}\\
&\tilde R_{j+1, n } \le C_\sharp C_F\Bigg\{ \sum_{k}c_{\star}^{\sum_{s=1}^{j} { (s-1)!}k_s}+C_\flat^{(j+1)j}\sum_{\sigma=2}^{j}\sum_p c_\star^{ (\sigma-1)!+\sum_{s=2}^{j+2-\sigma}p_s (s-1)!}\mu^{n\sum_{s=1}^{j+2-\sigma}p_s (s-1)!}+C_{\bar n}\Bigg\}.
\end{split}
\]
It is convenient to define
\[
\tau_j=\begin{cases} 2 &\text{ if } j=2\\
6&\text{ if } j=3\\
j!-1&\text{ if } j>3
\end{cases}
\]
Note if $j=2$, then  $\sum_{s=1}^{2} (s-1)!k_s= 2=\tau_2$, $\sigma=2$ and $1+\sum_{s=2}^{2}p_s (s-1)!=2=\tau_2$.  If $j>2$ note that $k_j\leq 1$, otherwise 
$\sum_{s=1}^{j} sk_s>j+1$, hence
\[
\begin{split}
\sum_{s=1}^{j} (s-1)!k_s&\leq {(j-3)!}\sum_{s=1}^{j-1} s k_s+(j-1)!k_j=(j-3)!(j+1-jk_j)+(j-1)!k_j\\
&\leq (j-3)!\{j+1, j^2-3j+3\}^+\leq \tau_j.
\end{split}
\]
If $\sigma=j$, then 
\[
(\sigma-1)!+\sum_{s=1}^{j+2-\sigma}p_s (s-1)!=(j-1)!+j+1\leq \tau_j.
\]
If $\sigma=j-1$, then 
\[
(\sigma-1)!+\sum_{s=1}^{j+2-\sigma}p_s (s-1)!=(j-2)!+j+1\leq (j-1)!+j+1\leq \tau_j.
\]
On the other hand if $\sigma<j-1$, then we have
\[
\begin{split}
(\sigma-1)!+\sum_{s=1}^{j+2-\sigma}p_s (s-1)!&\leq (j-3)!+(j-\sigma)!(j+1)\leq (j-3)![1+(j-2)(j+1)]\\
&\leq  (j-3)!(j^2-j-1)\leq \tau_j.
\end{split}
\]
Accordingly, since the sums in $k$ and $p$ have at most $j^j$ terms, and $j\leq r$, since $C_\flat\ge \{\frac{2\bar C_\star}\vartheta,\Const r\}^+$,
\begin{equation}\label{eq:palle_R}
\begin{split}
R_{j+1, n } &\le C_\flat^{(j+1)^2} \left\{ c_\star^{\tau_j}\mu^{n(j-1)!(j+1)}+C_{\bar n}\right\}\\
\tilde R_{j+1, n } &\le \frac 12C_\flat^{(j+1)^2} C_F\left\{c_\star^{\tau_j}\mu^{n(j-1)!(j+1)}+C_{\bar n}\right\}\\
R^\star_{j+1, n} &\le A(\lambda_n^-)^{-1} C_\flat^{(j+1)^2} \left\{c_\star^{\tau_j}\mu^{n(j-1)!(j+1)}+C_{\bar n}\right\}.
\end{split}
\end{equation}
Let us show the first of \eqref{eq:nu_h_induction}.
Substituting the above in the first of \eqref{eq: system estimates} and using \eqref{eq: h_n'} we have
\[
\|h_{ n }^{(j+1)}\|\leq  \frac{(\bar C_\star\mu^{n})^{j+1}}\vartheta c_\star^{j!}C_F+ \frac 12C_F C_\flat^{(j+1)^2}( c_\star^{j!}\mu^{n(j-1)!(j+1)}+C_{\bar n}).
\]
We can then write
\[
\|h_{n }^{(j+1)}\|\leq C_\flat^{(j+1)^2} C_F\left(c_\star^{j!}\mu^{n(j+1)!}+C_{\bar n}\right).
\]
which is the first of \eqref{eq:nu_h_induction} for $n\leq 2\bar n$.

Next, we substitute \eqref{eq:palle_R} in \eqref{eq: eta_nj}, using \eqref{def:eta and barn}, and choosing $2\bbc>C_{\bar n}$ we write
\[
\begin{split}
\|\hat \nu_{n}^{(j+1)}\|&\le \eta^{n j!}\left\{c_\star^{j!}  +C_{\bar n}\right\}\leq \left\{\eta^n c_\star +\bbc/2\right\}^{j!} .
\end{split}
\]
Hence also the second of \eqref{eq:nu_h_induction} is satisfied for $n\leq 2\bar n$.

In particular $\hat \nu_{n}\in \Gamma_{\ell}(\eta^n c_\star +\bbc/2)$ for each $\ell\le r$ and $\bar n\le n\le 2\bar n$. Next, let us set $c_{\star, 1}=\eta^{\bar n} c_\star +\bbc/2\le c_\star$, and, for each integer $k\ge 2$,
\[
\frac{\bbc}{2}\leq c_{\star,k}:=\eta^{\bar n}c_{\star, k-1}+\frac{\bbc}{2}\leq \eta^{\bar n k}c_{\star}+\frac{\bbc}{2(1-\eta^{\bar n})}.
\]
It follows that $\hat \nu_{k \bar n}\in \Gamma_\ell(c_{\star, k})$ and, for all $m\in\{\bar n,\dots 2\bar n\}$,\footnote{ Recall the definition of $\frh^*_n$ in \eqref{def:frh^star-frH}.}  
\[
\hat \nu_{ k\bar n+m}=\frh^*_{k\bar n+m-1}\circ \cdots\circ\frh^*_{k\bar n+1}\circ \hat\nu_{k \bar n}\circ h^*_{m,k+1},
\] 
and, by \eqref{eq:nu_h_induction}, 
\begin{equation}\label{eq:h_m*-Crho-norm}
\|h^*_{m,(k+1)}\|_{\cC^{j}}\leq C_\flat^{j^2} C_F \left(c_{\star,k}^{(j-1)!}\mu^{j! m}+C_{\bar n}\right).
\end{equation}
Hence, applying iteratively the above argument to $\hat \nu_{n}$ for $k\bar n\le n \le (k+1)\bar n$, we obtain the second of \eqref{eq:nu_h_induction} for each $n\ge \bar n$.\\
It remains to prove the estimate for $h_n$, $n\ge 2 \bar n$. We write $n=m+k\bar n$, $m\in\{\bar n,\dots, 2\bar n\}$ and
\begin{equation}
\label{eq: compos h_kbarn}
h_n=h^*_{m, k+1}\circ h^*_{\bar n, k}\circ \cdots \circ h^*_{\bar n, 1}
 =h^*_{m, k+1}\circ h_{k\bar n}.
\end{equation}
Note that \eqref{eq: h_n'} yields $\|h_{n}\|_{\cC^1}\leq C_\flat \mu^n$, since $C_\flat\geq 3 \bar C_\star$. It is then natural to start by investigating the second derivative. In fact, it turns out to be more convenient to study the following ratio
\begin{equation}\label{eq:two-der}
\begin{split}
\frac{h''_{n}}{h'_{n}}=\left(\log [(h^*_{m, k+1})'\circ h_{k\bar n}]\right)'+\frac{ h_{k\bar n}''}{ h_{k\bar n}'}=:Q_1+Q_2.
\end{split}
\end{equation} 
To estimate the norm of such a ratio we start estimating the sup norm, then we will proceed by induction.

Since \eqref{lambda_+<Clambda_-} and \eqref{eq: hat nu '} imply $|h_{\bar n, i}'|\geq c_0\mu^{-n}$ for each $i$, for some uniform constant $c_0$, formula \eqref{formula 1 for norm Crho} and \eqref{eq:h_m*-Crho-norm} yield 
\begin{equation}\label{eq:log_zero_b}
\|(\log h_{m, k+1}^{* \prime})^\prime\|_{\cC^{\ell}}\le \Const C_FC_\flat^{(\ell+2)^2} c_{\star, k}^{(\ell+1) !}\mu^{(\ell+2)!m}.
\end{equation}
It follows 
\[
\|Q_1\|_{\cC^0}\le \Const C_F C_\flat^4 c_{\star, k}\mu^{2 m}\mu^{k\bar n}\le C_\sharp C_FC_\flat^4 c_{\star}\mu^{n}.
\] 
To estimate $\|Q_2\|_{\cC^0}$ we write
\begin{equation}\label{eq:ratio h''/h'}
\begin{split}
Q_2=\frac{h''_{k\bar n}}{h'_{k\bar n}} &=\frac{ \left( \prod_{i=1}^{k}h^{* \prime}_{\bar n,i}\circ h_{i\bar n}\right)'}{\prod_{i=1}^{k}h^{* \prime}_{\bar n, i}\circ h_{i\bar n}} =  \left( \log \prod_{i=1}^{k}h^{* \prime}_{\bar n, i}\circ h_{i \bar n} \right)' =  \sum_{i=1}^{k} \left(\log  h_{\bar n, i}^{* \prime}\circ h_{i\bar n } \right)'.
\end{split}
\end{equation} 
Using formulae \eqref{eq: h_n'}, \eqref{eq:log_zero_b} and \eqref{eq:two-der} we have, since $\bar n\le m,$
\begin{equation}\label{eq:Q_0-C0-norm}
\begin{split}
\| Q_2\|_{\cC^{0}}&\le \sum_{i=1}^{k} \|(\log h_{\bar n, i}^{* \prime})^\prime\|_{\cC^0}\|h_{i \bar n}'\|_{\cC^0} 
\le  C_\sharp C_\flat^4 C_F c_{\star, 1}\mu^{2\bar n} \sum_{i=1}^{k} \mu^{i \bar n}\\
& \le C_\sharp C_\flat^4 C_F c_{\star, 1}\mu^{2\bar n}\frac{\mu^{k\bar n}-1}{\mu-1}\le \mu^{2\bar n} C_\flat^4 C_Fc_{\star, 1} C_{\mu, k\bar n}\mu^{k\bar n}\le C_\flat^4 C_F c_{\star, 1} C_{\mu, n}\mu^{n}.
\end{split}
\end{equation}
Hence, using the above and \eqref{eq: h_n'}, it follows by \eqref{eq:ratio h''/h'}
\begin{equation}\label{eq:hp2}
\begin{split}
\|h_{n}''\|_{\cC^0}&\le\Const \left\| \frac{h''_{n }}{h'_{n }} \right\|_{\cC^{0}}\|h'_{ n }\|_{\cC^{0}}\le \Const \mu^n\left\| \frac{h''_{n }}{h'_{n }} \right\|_{\cC^{0}}\le  C_\flat^4 C_F c_\star C_{\mu, n}\mu^{2n}.
\end{split}
\end{equation}
This proves the second of \eqref{Crho norm of h}.
We can now prove the general case by induction on $j\le \ell.$ Assume it true for all $i< j$. Using again \eqref{formula 1 for norm Crho}, by the inductive assumption and \ref{eq:log_zero_b} we have
\begin{equation}\label{eq:Q_1-Cj-norm}
\begin{split}
\|Q_1\|_{\cC^{j-1}}&=\|(\log [(h^*_{m, k+1})'])^\prime\circ h_{k\bar n}\|_{\cC^{j-1}}\| h_{k\bar n}'\|_{\cC^{j-1}}\\
&\le C_\flat^{4(j-2)!} (c_{\star, 1}^{j!}C_{\mu, n }^{a_j} +1) \mu^{n j!}.
\end{split}
\end{equation}
On the other hand, by formulae \eqref{eq:ratio h''/h'}, \eqref{formula 1 for norm Crho}, \eqref{eq:log_zero_b} and the inductive assumption 
\begin{equation}
\label{h''/h' 0}
\begin{split}
&\| Q_2\|_{\cC^{j-1}}\le C_\sharp \sum_{i=1}^{k} \|(\log [(h^*_{m, k+1})'])^\prime\circ h_{k\bar n}\|_{\cC^{j-1}} \|h_{i \bar n}^\prime\|_{\cC^{j-1}}\\
&\leq C_\sharp \sum_{i=1}^{k} \|(\log [(h^*_{m, k+1})'])^\prime\|_{\cC^{j-1}} \|h_{i \bar n}^\prime\|^{j-1}_{\cC^{j-2}}\|h_{i \bar n}^\prime\|_{\cC^{j-1}}\\
&\le C_\sharp C_\flat^{(j+1)^2+2\cdot (j-1)!+2\cdot j!} C_Fc_{\star, k}^{j!}\mu^{(j+1)! m}\sum_{i=1}^{k} \sum_{q=0}^{j-1} ([c_\star^{(j-1)!}C_F+1]C_{\mu, i\bar n}^{a_j}\mu^{j!i \bar n})^q\\
&\le C_\sharp C_\flat^{3\cdot j!}[c_\star^{2\cdot j!}C_F+1]\mu^{(j+1)! m}\sum_{i=1}^{k} \sum_{q=0}^{j-1} (C_{\mu, i\bar n}^{a_j}\mu^{j!i \bar n})^q.
\end{split}
\end{equation}

To estimate the last sum, notice that by definition
\begin{enumerate}[label=\roman{*}]
\item $ {1\leq} C_{\mu, \bar n i}\le C_{\mu, \bar n k}, \quad \forall i\le k,$
 \item $C_{\mu^{a}, n}\le C_{\mu, n}, \quad \forall a>1,$
 \end{enumerate}
Hence,
\[
\begin{split}
&\sum_{i=1}^{k} \sum_{q=0}^{j-1} (C_{\mu, \bar n  i}^{a_j}\mu^{j!i\bar n})^q
\le C_{\mu, \bar n k}^{a_j j}\sum_{i=1}^{k} \mu^{j!(j-1)\bar n i}\leq C_{\mu, \bar n k}^{a_j (j-1)+1}\mu^{j!(j-1) \bar n k}.
\end{split}
\]
Using this in (\ref{h''/h' 0}) we obtain
\begin{equation}
\label{eq:Q_2-Cj-norm}
\begin{split}
\| Q_2\|_{\cC^{j-1}}\le C_\sharp C_\flat^{3\cdot j!} [C_Fc_{\star, i}^{2\cdot j!}+1]\mu^{(j+1)!m}C_{\mu, \bar n k}^{a_j (j-1)+1}\mu^{j!(j-1) \bar n k}.
\end{split}
\end{equation}
Therefore, by the inductive assumption, equations \eqref{eq:Q_1-Cj-norm}, \eqref{eq:Q_2-Cj-norm} and \eqref{eq:two-der}, and provided we choose $C_\flat$ large enough, we finally have
\begin{equation}
\label{eq:end_induction}
\begin{split}
\|h_{ n}''\|_{\cC^{j-1}}&\le \left\|\frac{h''_{n}}{h_{n}'}\right\|_{\cC^{j-1}}{\|h_{n}\|_{\cC^{j}}}\\
&\le C_\flat^{2(j+1)!} [C_Fc_{\star, 1}^{(j+1)!}+1] C_{\mu, n}^{a_{j+1}}\mu^{(j+1)!n}.
\end{split}
\end{equation}
\subsection{Proof of Lemma \ref{lem unst curve pull back}}\label{secG3}
Let $\frh\in\frH^\infty$ be such that $\nu_m= \frh_m \gamma$.
Recalling \eqref{definition of Lambda}, we can apply $\eqref{formula 1 for norm Crho}$ and we have for each $j\le r$
\begin{equation}
\label{nu'<Lambda}
\|{\nu}_{m}\|_{\cC^{j+1}}=\|\frh_{m}\circ \gamma\|_{\cC^{j+1}}\le C_\sharp (\|\Delta_{\gamma}\|_{\infty} \Lambda^{m})^j.
\end{equation}
We set $\phi(t):=(\pi_2\circ \nu_m)(t)$. By \eqref{condition for m} there exists $c_{u,\gamma}\geq \chi_{u}\mu^{-m}$ such that we have $|\phi'|>c_{u,\gamma}>0$,
so it is well defined the diffeomorphism
$h_m(t)=\phi^{-1}(t)$, and $\hat{\nu}_m=\nu_m\circ h_m$ is parametrized by vertical length. We want to estimate the higher order derivatives of $h_m$ using a formula for inverse functions given in \cite{Joh}. For the reader convenience we write it down here for our case:
\begin{equation}
\label{Joh formula}
\begin{split}
h^{(j+1)}_m(t)&=\frac{d^{j+1}\phi^{-1}(t)}{dt^{j+1}}\\
&=\sum_{k=0}^{j} [\phi'\circ \phi^{-1}(t) ]^{-j-k-1} \sum_{\substack{b_{1}+\cdots+b_{k}=j+k\\ b_l\ge 2}} B_{j,k,\{b_l\}_{l=1}^k} \prod_{l=1}^k \phi^{(b_l)}\circ \phi^{-1}(t),
\end{split}
\end{equation}
where $B_{j,k,\{b_l\}_{l=1}^k}= \frac{(j+k)!}{k!b_1!\cdots b_k!}$. It follows by \eqref{nu'<Lambda} and \eqref{Joh formula} that for each $t$
\begin{equation}
\label{hig der of h_m with joh formula}
|h^{(j+1)}_m(t)|\le C_\sharp \left(c_{u,\gamma}^{-2}\|\Delta_{\gamma}\|_{\infty}\Lambda^{m} \right)^{j}.
\end{equation}
By \eqref{nu'<Lambda}, \eqref{hig der of h_m with joh formula} and formula \eqref{formula 1 for norm Crho} for the composition,
\begin{equation}\label{eq:this_I_need}
\begin{split}
\|\hat \nu_m\|_{\cC^{j+1}}&=\|\nu_m\circ h_m\|_{\cC^{j+1}}\le
\Const\sum_{s=0}^{j+1}\|\hat\nu_m\|_{\cC^{s}}
\sum_{k\in\cK_{\rho,s}}\prod_{l\in\bN}\|h_m\|_{\cC^{l}}^{k_l}\\
&\le C_\sharp(\bar c_{u,\gamma} \|\Delta_{\gamma}\|_{\infty}\Lambda^{c_\sharp m})^{2j}\le (\bar c_{u,\gamma}\|\Delta_{\gamma}\|_{\infty}\Lambda^{c_\sharp m})^{(j+1)!} ,
\end{split}
\end{equation}
where $\bar c_{u,\gamma}=\{c_{u,\gamma}^{-2}, 1\}^+$. Hence, setting $c_\star(m)=\bar c_{u,\gamma} \|\Delta_{\gamma}\|_{\infty}\Lambda^{c_\sharp m}$ we have that $\hat \nu_m\in \Gamma_j(c_{\star}(m))$. Since $ \ovm>m>\bar n$ we can apply Lemma \ref{lem stable vertic curves} and we have that the curve $\hat \nu_{\ovm}=\nu_m \circ h_{\ovm}$ belongs to $\Gamma_j(\eta^\ovm c_{\star}(m)+\frac \bbc{2})$.  By definition, $c_\star(m)\le \chi_u^{-2} \|\Delta_{\gamma}\|_{\infty}(\mu\Lambda)^{m}$ and by Corollary \ref{cor:stable_curves_inv}, $\bbc\ge \bar c_{u,\gamma}$ (since $j\ge 3$), having chosen $\varpi$ large enough. The statement then follows choosing $\ovm=\sigma m$, with $\sigma$ defined in \eqref{def of sigma}.

\subsection{Proof of Lemma \ref{lem stable vertic curves_sharp}}\label{secG2}
We use the notation of \eqref{eq: Fq nu=gamma h}.
To prove the first of \eqref{eq:nu2-sharp} it is convenient to go back to equation \eqref{eq: hat nu''} and, recalling \eqref{eq:dtDF-dt^2DF}, \eqref{eq:ve-special-inv1}, for each $v\in \bR^2$, $\|v\|=1$, we have 
\begin{equation}
\label{eq: local equation for hat nu''}
\begin{split}
&\left|\langle v,\hat \nu_{ n}''\rangle-\langle v,\hat \nu_{ n}'\rangle \frac{h_{ n}''}{h_n'}\right|\leq|\langle v,(D_{\hat \nu_{ n}}F^{ n})^{-1}\gamma''\circ h_{n}(h'_{ n})^2\rangle|\\
&+\sum_{k=0}^{n-1}\sum_{i=1}^2
\left|\langle v, (D_{\hat\nu_n}F^{k+1})^{-1}
\left[\partial_{x_i}D_{F^k(\hat\nu_{n})}F\right]
D_{\hat\nu_n}F^{k}\hat\nu_n'\rangle\right|
\|(D_{\hat\nu_n}F^k)\hat\nu_n'\|\\
&\leq|\langle v,(D_{\hat \nu_{ n}}F^{ n})^{-1}\gamma''\circ h_{n}(h'_{ n})^2\rangle|\\
&\phantom{\leq \ }+\Const\sum_{k=0}^{n-1}\left(\frac{1}{\lambda_{k+1}^{-}(\hat \nu_n(t))}+C_F\mu^{k+1}\right)
\|(D_{\hat\nu_n}F^k)\hat\nu_n'\|^2.
\end{split}
\end{equation}
Note that, recalling \eqref{eq:barc2}, for each $n\le n_\star\le c_2^- \log\chi_u^{-1}$ we have  $(D_{\hat \nu_{ n}(t)}F^{ n})^{-1}e_1\notin \fC_c$.

Consequently
\begin{equation}\label{eq:contract-gamma''}
|\langle v,(D_{\hat \nu_{ n}}F^{ n})^{-1}\gamma''\circ h_{n}\rangle|\leq (\lambda_n^{-}(\hat \nu_n(t)))^{-1}\|\gamma''\circ h_{n}(t)\|, \quad \forall n\le n_\star.
\end{equation}
Next, if $v$ is perpendicular to $\hat\nu_n'$, then it must be $|v_2|\leq \chi_c|v_1|$, hence
\begin{equation}\label{eq:v-perp-nu''}
|\langle v, \hat\nu_n''\rangle|=|v_1|\|\hat\nu_n''\|\geq (1+\chi_c^2)^{-\frac 12}\|\hat\nu_n''\|.
\end{equation}
On the other hand, if $v$ is perpendicular to $\hat\nu_n''$, then $v=e_2$ and
$|\langle v, \hat\nu_n'\rangle|=1$. 
Accordingly, using \eqref{eq: local equation for hat nu''} first with $v$ orthonormal to $\hat \nu'_n$ and then with $v=e_2$, and recalling Proposition \ref{prop on DF^-1S} and equations \eqref{partial hyperbolicity 2}, \eqref{eq: h_n'} we have for $n\le n_\star$,\footnote{ The constant $C_\flat$ is introduced just before \eqref{def:eta and barn}.}
\begin{equation}\label{eq:nuh2}
\begin{split}
\|\hat\nu_n''(t)\|&\leq (1+\chi_c^2)^{\frac 12}(\lambda_n^{-}(\hat \nu_n(t)))^{-1}C_\flat^2 \mu^{2n}\|\gamma''\circ h_{n}(t)\|+\sum_{k=0}^{n-1}
C_\star^2 \mu^{2k}C_{F,k},\\
\|h_n''/h_n'\|&\le (\lambda_n^{-}(\hat \nu_n(t)))^{-1}C_\flat^2 \mu^{2n}\|\gamma''\circ h_{n}(t)\|+\sum_{k=0}^{n-1}
C_\star^2 \mu^{2k}C_{F,k}, 
\end{split}
\end{equation}
where $C_{F,k}=\Const\left(\frac{1}{\lambda_{k+1}^{-}(\hat \nu_n(t))}+C_F\mu^{k+1}\right)$.
Setting $\bar c_{n_\star}=\left[(1+\chi_c^2)^{\frac 12} C_\flat^2\right]^{\frac 1{n_\star}}$ we obtain
\begin{equation}\label{eq:hatnu''_nstar}
\begin{split}
\|\hat\nu_{n_\star}''(t)\|&\leq \bar c_{n_\star}^{n_\star} \mu^{2n_\star}(\lambda_{n_\star}^{-}(\hat \nu_{n_\star}(t)))^{-1}\|\gamma''\circ h_{n_\star}(t)\|+\sum_{k=0}^{n_\star-1} C_\star^2 \mu^{2k}C_{F,k}.
\end{split}
\end{equation}
Condition \eqref{eq:condition-hatn} implies $ \bar c_{\hat n}^{\hat n} \mu^{2\hat n}\lambda_{-}^{-\hat n}\leq \frac 12$, so we can proceed by induction since, setting $h_{l,m}^*=h_{ln_\star+m}\circ h_{l n_\star}^{-1}$, if $n=ln_\star+m$, $m\leq n_\star$, then
\begin{equation}\label{eq:nu''-iterative}
\begin{split}
&\|\hat\nu_{n}''(t)\|\leq \bar c_{n_\star}^{n_\star} \mu^{2m}(\lambda_{m}^{-}(\hat \nu_{n}(t)))^{-1}\|\hat\nu_{ln_\star}''\circ h^*_{l,m}(t)\|+\sum_{k=0}^{n_\star-1} C_\star^2 \mu^{2k}C_{F,k}\\
&\leq \bar c_{n_\star}^{n+n_\star} \mu^{2n}(\lambda_{m}^{-}(\hat \nu_{n}(t)))^{-1}(\lambda_{n_\star}^{-}(\hat \nu_{ln_\star}\circ h_{l,m}^*(t)))^{-1}\dots (\lambda_{n_\star}^{-}(\gamma\circ h_{n}(t)))^{-1}c_\star\\
&\phantom{\leq}
+\sum_{s=1}^l \bar c_{n_\star}^{sn_\star}\mu^{2s n_\star}\lambda_-^{-sn_\star}\sum_{k=0}^{n_\star-1} C_\star^2 \mu^{2k}C_{F,k}\\
&\leq \bar c_{n_\star}^{n+n_\star} \mu^{2n}\bar c_\flat \bar c_\flat^{\frac n{n_\star}} (\lambda_{n}^{+}(\gamma\circ h_{n}(t)))^{-1}c_\star
+C_{\mu, n_\star}\mu^{3n_\star}C_3,
\end{split}
\end{equation}
for some appropriate $C_3$.
 This implies the first of the \eqref{eq:nu2-sharp}.

It remains to bound the third derivative of $\hat \nu_n$. The strategy is basically the same. Recalling that $\hat\nu_n'=(D_{\hat \nu_n} F^{n})^{-1}\gamma'\circ h_n h_n',$ we differentiate this expression twice and multiply by a unitary vector $v$ orthogonal to $\hat\nu_n'$:
\begin{equation}\label{eq:nu'''}
\begin{split}
\langle \hat \nu_n''', v\rangle&=\Big\langle [(D_{\hat \nu_n} F^{n})^{-1}]''\gamma'\circ h_n h_n'+2[(D_{\hat \nu_n} F^{n})^{-1}]'(\gamma''\circ h_n (h_n')^2+\gamma'\circ h_n h_n'')\\
&+[(D_{\hat \nu_n} F^{n})^{-1}]\left(\gamma'''\circ h_{n}(h^{\prime}_{n})^3+ 3\gamma''\circ h_n h_n' h_n''\right),v\Big\rangle.
\end{split}
\end{equation}
We will estimate the norms of the terms in the first line  of the above equation one at a time, for each $n\le n_\star$. First, using \eqref{norm of DF^N^-1} with $m=0$ and $\bbc=\|\hat\nu_n''\|$ (that we have estimated in the first of the \eqref{eq:nu2-sharp}), and \eqref{Crho norm of h} we have, for some uniform $A_1>0$
\[
\begin{split}
\|[(D_{\hat \nu_n} F^{n})^{-1}]''\gamma'\circ h_n h_n'\|\le&  C_\flat\mu^{4n}\varsigma_n^2C_{\mu,n}+
A_1 c_\flat C_\flat\varsigma_n  c_{n_\star}^{n_\star} \mu^{4n}(\lambda_{n}^{-}(\hat \nu_{n}(t)))^{-1}\|\gamma''\circ h_{n}(t)\|\\
&+  C_\flat\varsigma_n \mu^{2n} C_{\mu, n_\star}\mu^{3n_\star}C_3.
\end{split}
\]
Next, notice that $(D_{\hat \nu_{n_\star}} F^{n_\star})^{-1}\gamma''\notin \fC_c$, hence by the second of  \eqref{eq:dtDF-dt^2DF} and subsequent, there is uniform $A_2>0$ such that
\begin{equation}\label{eq:anothereq}
\|[(D_{\hat \nu_n} F^{n})^{-1}]'\gamma''\circ h_n (h_n')^2\|\le A_2 C_\flat ^2 \mu^{3n}(\lambda_{n}^{-}(\hat \nu_{n}(t)))^{-1} \varsigma_n \|\gamma''\circ h_{n}(t)\|. 
\end{equation}
It is convenient to write the third term as 
\[
\begin{split}
[D_{\hat \nu_n} F^{n})^{-1}]'\gamma'\circ h_n h_n''&=\frac{h_n''}{h_n'}[D_{\hat \nu_n} F^{n})^{-1}]'\gamma'\circ h_n h_n'\\
&=\frac{h_n''}{h_n'}\left( \hat \nu_n''- [(D_{\hat \nu_n} F^{n})^{-1}]\gamma''\circ h_n (h_n')^2-\hat \nu_n'\frac{h_n''}{h_n'} \right).
\end{split}
\]
The last term vanishes when we multiplied by $v$; hence, by  \eqref{eq:contract-gamma''} and \eqref{eq:nuh2}, we have\footnote{ Recall also the lower bound for $|h_n'|$ in \eqref{eq: h_n'}.}
\[
\begin{split}
\left|\langle [D_{\hat \nu_n} F^{n})^{-1}]'\gamma'\circ h_n h_n'',v \rangle\right|&\le \left|\frac{h_n''}{h_n'}\right|\left\{\|\hat \nu_n''\|+\|[D_{\hat \nu_n} F^{n})^{-1}]\gamma''\circ h_n (h_n')^2\|\right\}\\
&\le c_{n_\star}^{n_\star}C_\flat^4\mu^{4n}(\lambda_n^{-}(\hat \nu_n(t)))^{-2}\|\gamma''\circ h_{n}(t)\|^2\\
&+ 2c_{n_\star}^{n_\star} C_{\mu,n}C_\flat^2 C_3 \mu^{5n}(\lambda_n^{-}(\hat \nu_n(t)))^{-1}\|\gamma''\circ h_{n}(t)\|+C_{\mu,n}^2\mu^{6n}C_3^2.
\end{split}
\]
For the two terms in the second line of \eqref{eq:nu'''}, when the matrix hits $\gamma''$ or $\gamma'''$, we can use \eqref{eq:contract-gamma''} for $n\le n_\star$ and \eqref{Crho norm of h} with $\|\gamma''\circ h_n(t)\|$ instead of $c_\star$.
Collecting all the above estimates in \eqref{eq:nu'''} we finally have, recalling also \eqref{eq:v-perp-nu''},
\[
\begin{split}
(1+\chi_c^2)^{-\frac 12}\|\hat\nu_n'''\|&\le  C_\flat \mu^{4n}\varsigma_n^2C_{\mu,n}+A_1 c_\flat C_\flat\varsigma_n  c_{n_\star}^{n_\star} \mu^{4n}(\lambda_{n}^{-}(\hat \nu_{n}(t)))^{-1}\|\gamma''\circ h_{n}(t)\|\\
&+  C_\flat\varsigma_n \mu^{2n}C_{\mu, n_\star}\mu^{3n_\star}C_3\\
&+ A_2  C_\flat ^2 \mu^{3n} \varsigma_n (\lambda_{n}^{-}(\hat \nu_{n}(t)))^{-1}\|\gamma''\circ h_{n}(t)\|\\
&+ c_{n_\star}^{n_\star}C_\flat^4\mu^{4n}(\lambda_n^{-}(\hat \nu_n(t)))^{-2}\|\gamma''\circ h_{n}(t)\|^2\\
&+ 2c_{n_\star}^{n_\star} C_{\mu,n} C_\flat^2C_3\mu^{5n}(\lambda_n^{-}(\hat \nu_n(t)))^{-1}\|\gamma''\circ h_{n}(t)\|+C_{\mu,n}^2\mu^{6n}C_3^2\\
&+C_\flat^3\mu^{3n} (\lambda_n^{-}(\hat \nu_n(t)))^{-1}\|\gamma'''\circ h_{n}(t)\|\\
&+ 3(\lambda_n^{-}(\hat \nu_n(t)))^{-1}C_\flat^5\mu^{3n}C_{\mu, n}\|\gamma''\circ h_n(t)\|(\|\gamma''\circ h_n(t)\|C_F+1).
\end{split}
\]
Setting $C_4=\{A_1 c_\flat+ A_2+3+2C_3\}^+(1+\chi_c^2)^{1/2}$, and recalling the second of \eqref{def of mathbbms}, yields
\[
\begin{split}
\|\hat\nu_n'''\|\le& a_{n_\star}^{n_\star}\mu^{3n}(\lambda_n^{-}(\hat \nu_n(t)))^{-1}\|\gamma'''\circ h_{n}(t)\|+6c_{n_\star}^{n_\star}\mu^{4n}C_\flat^5(\lambda_n^{-}(\hat \nu_n(t)))^{-2}\varsigma_{n}\|\gamma''\circ h_n(t)\|^2 C_{\mu,n} \\
&+b_{n_\star}^{n_\star}(\lambda_{n}^{-}(\hat \nu_{n}))^{-1} \mu^{5n}\|\gamma''\circ h_{n}(t)\|+\overline{\mathbbm{s}}_{n_\star}.
\end{split}
\]
Using the first of \eqref{eq:nu2-sharp}, we can write
\[
\begin{split}
&\|\hat\nu_{kn_\star}'''\|\le a_{n_\star}^{n_\star}\mu^{3n_\star}(\lambda_{n_\star}^{-}(\hat \nu_{(k-1)n_\star}(t)))^{-1}\|\nu_{(k-1)n_\star}'''\circ h^*_{n_\star}(t)\|\\
&+6c_{n_\star}^{n_\star}\mu^{4n_\star}C_\flat^5(\lambda_{n_\star}^{-}(\hat \nu_{(k-1)n_\star}(t)))^{-2}\varsigma_{n_\star}C_{\mu,n_\star} 
\left[ c_\flat c_{n_\star}^{(k-1)n_\star}\mu^{2(k-1)n_\star}\lambda_{(k-1)n_\star}^-(\gamma\circ h_{(k-1)n_\star}(t))^{-1}c_\star\right]^2\\
&+b_{n_\star}^{n_\star}(\lambda_{n_{\star}}^{-}(\hat \nu_{(k-1)n_{\star}}))^{-1} \mu^{5n_\star}\left[ c_\flat c_{n_\star}^{(k-1)n_\star}\mu^{2(k-1)n_\star}\lambda_{(k-1)n_\star}^-(\gamma\circ h_{(k-1)n_\star}(t))^{-1}c_\star\right]\\
&+\overline{\mathbbm{s}}_{n_\star}+6c_{n_\star}^{n_\star}\mu^{4n_\star}C_\flat^5(\lambda_{n_\star}^{-}(\hat \nu_{(k-1)n_\star}(t)))^{-2}\varsigma_{n_\star}C_{\mu,n_\star}^3\mu^{6n_\star}C_3^2\\
&+b_{n_\star}^{n_\star}(\lambda_{n_{\star}}^{-}(\hat \nu_{(k-1)n_{\star}}))^{-1} \mu^{5n_\star}C_{\mu,n_\star}\mu^{3n_\star}C_3.
\end{split}
\]
We can then proceed by induction as in \eqref{eq:nu''-iterative}, and since  condition \eqref{eq:condition-hatn} implies both $a_{\hat n}^{\hat n}\mu^{3\hat n}\lambda_-^{-\hat n}\le  \frac 12$ and $ \bar c_{n_\star}^{n_\star} \mu^{2n_\star}\lambda_{-}^{-n_\star}\leq \frac 12$, we have
\[
\begin{split}
\|\hat\nu_{kn_\star}'''\|\le& c_\flat^2(1+6\varsigma_{n_\star}C_{\mu,n_\star})c_{n_\star}^{kn_\star}\mu^{3kn_\star}(\lambda_{kn_\star}^{-}(\gamma\circ h_{kn_\star}(t)))^{-1}c_\star^2\\
&+c_\flat^2b_{n_\star}^{n_\star} \mu^{2kn_\star+2n_\star}\lambda_{kn_\star}^-(\gamma\circ h_{kn_\star}(t))^{-1}c_\star+\mathbbm{s}_{n_\star},
\end{split}
\]
which implies the second of \eqref{eq:nu2-sharp}.
\subsection{Proof of Lemma \ref{lem unst curve pull back_sharp}}\label{secG4} Let $\nu_{n}=\frh_{n}\circ\gamma$ for each $n\in\bN$. Then, $C_\sharp\vartheta_\gamma(t)|\pi_1\circ \nu_{n_0}'(t)|\geq|\pi_2\circ \nu_{n_0}'(t)|\geq \vartheta_\gamma(t)|\pi_1\circ \nu_{n_0}'(t)|>0$, and we can reparametrize $\nu_{n}$, $n\ge n_0$, by vertical length $\hat\nu_{n}(t)=\nu_n(h_{n}(t))$.
Note that 
\begin{equation}\label{eq:nun0}
C_\sharp\vartheta_\gamma(t)^{-1}\leq\|\hat \nu_{n_0}'(t)\|\le C_\sharp\vartheta_\gamma(t)^{-1}.
\end{equation}
If $n_0=0$, then ${ C_\sharp\vartheta_{\hat\nu_0}(t)^{-1}=C_\sharp\vartheta_\gamma\circ h_0(t)^{-1}
\leq}|h_0'(t)|\leq  C_\sharp\vartheta_{\hat\nu_0}(t)^{-1}$ and  \eqref{eq: local equation for hat nu''} yields\footnote{ Note that \eqref{eq: local equation for hat nu''} holds also if $\gamma$ is not parametrized vertically.}
\begin{equation*}
\begin{split}
&|h_{n_0}''(t)|\leq \frac{\|\gamma''\circ h_{n_0}(t)\| |h_{n_0}'(t)|^3}{|\langle e_1, \hat\nu_{n_0}'(t)\rangle|}\leq C_\sharp \Delta_{\gamma}\circ h_{n_0}(t)\vartheta_{\hat\nu_{n_0}}(t)^{-2}\\
&\|\hat \nu_{n_0}''(t)\|\le  C_\sharp \Delta_{\gamma}\circ h_{n_0}(t)\vartheta_{\hat\nu_{n_0}}(t)^{-1}.
\end{split}
\end{equation*}
If $n_0>0$, then 
\begin{equation}\label{eq:hn0}
\begin{split}
&C_\sharp\Lambda^{-n_0}
\vartheta_{\hat\nu_{n_0}}(t)^{-1}\leq |h_{n_0}'(t)|
\leq C_\sharp\Lambda^{n_0}\vartheta_{\hat\nu_{n_0}}(t)^{-1}, \\
&\|\nu_{n_0}''(t)\|\le C_\sharp\Lambda^{2n_0}\|\gamma''(t)\|
\end{split}
\end{equation}
 and
\begin{equation}\label{eq:nunzero}
\begin{split}
&|h_{n_0}''(t)|\leq C_\sharp \Lambda^{3n_0}\Delta_{\gamma}\circ h_{n_0}(t)\vartheta_{\hat\nu_{n_0}}(t)^{-2}\\
&\|\hat \nu_{n_0}''(t)\|\le  C_\sharp \Lambda^{2n_0}\Delta_{\gamma}\circ h_{n_0}(t)\vartheta_{\hat\nu_{n_0}}(t)^{-1}\\
&\|\hat \nu_{n_0}'''(t)\|\le  C_\sharp \Lambda^{3n_0} \Delta_{\gamma}^2\circ h_{n_0}(t)\vartheta_{\hat\nu_{n_0}}(t)^{-2}.
\end{split}
\end{equation}
Remark that 
\begin{equation}\label{eq:nu0'}
\|(D_{\hat\nu_m(t)}F^k)\hat\nu_{m}'(t)\|\leq  \sqrt{1+\chi_c^2} C_\star\lambda^+_k(\hat\nu_{m}(t)).
\end{equation}
Let $m_\star(t)$ be the largest integer for which $\hat\nu'_{m_\star}(t)\not\in\fC_{c}$, then, recalling \eqref{def of lamb^+ mu^+} and setting $F^{m-n_0} \hat \nu_{m}=\hat\nu_{n_0} \circ \bar h_{m-n_0}$,
we have
\begin{equation}\label{eq:h'_m0}
\begin{split}
|\bar h'_{m-n_0}(t)|&= |\langle e_2, D_{\hat{\nu}_m(t)}F^{m-n_0} \hat\nu_{m}'(t)\rangle|\\
&\leq C_\sharp\mu^{m-m_\star} \lambda_{m_\star-n_0}^+(\hat\nu_{m_\star}(t))
\vartheta_{\hat\nu_{n_0}}(\bar h_{m-n_0}(t))\\
{|\bar h'_{m-n_0}(t)|}&\geq C_\sharp \mu^{m_\star-m}\lambda_{m_\star-n_0}^-(\hat\nu_{m_\star}(t)) 
\vartheta_{\hat\nu_{n_0}}(\bar h_{m-n_0}(t)).
\end{split}
\end{equation}
Next, we want to use equation \eqref{eq: local equation for hat nu''}, with $\gamma$ replaced by $\hat\nu_{n_0}$,  $n$ by $m-n_0$ and $\hat \nu_k$ by $\hat\nu_{k+n_0}$ and. Setting $\eta_m=DF_{\hat\nu_{m-n_0}}e_1$, $\hat\eta_m=\eta_m\|\eta_m\|^{-1}$, if we write $\nu_0''= a\eta_m+b e_2$, then $|a|\leq \Const\|\nu_0''\|$ and $|b|\leq \Const\chi_u\|\nu_0''\|$. Accordingly,  setting $m_0=m-n_0$, equations  \eqref{eq:nunzero} and  \eqref{eq:h'_m0}, yield,
\begin{equation}\label{eq:nuh_prelim}
\begin{split}
\|\hat\nu_m''(t)\|&\leq C_\sharp  \left\{\frac{1} {\lambda_{m_0}^{-}(\hat \nu_{m_0}(t))}+\mu^{m_0} C_F\right\}\\
&\phantom{\leq}
\times \Lambda^{2n_0}\Delta_{\gamma}\circ h_{n_0}(t)
\vartheta_{\hat\nu_{n_0}}(\bar h_{m_0}(t)) \lambda_{m_\star-n_0}^{+}(\hat \nu_{m_\star}(t))^2\mu^{2(m-m_\star)}\\
&\phantom{\leq}
+\sum_{k=0}^{m_0-1}\Const\left\{1+ \mu^{k}\lambda_{m_\star-n_0}^+(\hat\nu_m(t))C_F\right\}
\lambda_{m_\star-n_0}^+(\hat\nu_m(t))^{-1}\\
&\phantom{\leq}
\times\mu^{2\{k,m-m_\star\}^-}\lambda_{\{0,k-m+m_\star\}^+}^+(\hat\nu_{m_\star}(t))^2\\
|\bar h_{m_0}''(t)|&\leq  C_\sharp  \left\{\frac{1} {\lambda_{m_0}^{-}(\hat \nu_{m_0}(t))}+\mu^{m_0} C_F\right\}\\
&\phantom{\leq}
\times \Lambda^{2n_0}\Delta_{\gamma}\circ h_{n_0}(t)
\vartheta_{\hat\nu_{n_0}}(\bar h_{m_0}(t))^2 \lambda_{m_\star-n_0}^{+}(\hat \nu_{m_\star}(t))^3\mu^{3(m-m_\star)}\\
&\phantom{\leq}
+\sum_{k=0}^{m_0-1}\Const\left\{1+ \mu^{k}\lambda_{m_\star-n_0}^+(\hat\nu_m(t))C_F\right\}
\mu^{2\{k,m-m_\star\}^-}\\
&\phantom{\leq}
\times \lambda_{\{0,k-m+m_\star\}^+}^+(\hat\nu_{m_\star}(t))^2\vartheta_{\hat\nu_{n_0}}(\bar h_{m_0}(t)) \mu^{m-m_\star}.
\end{split}
\end{equation}
To continue we need the following
\begin{sublem}\label{sublem:chiulambda}
We have
\begin{equation}\label{eq:chiulambda}
\Const \chi_c^{-1}\mu^{-m_\star+n_0}\leq \vartheta_{\hat \nu_{n_0}}(\bar h_{m_0}(t))\lambda_{m_\star-n_0}^+(\hat\nu_{m_\star}(t))\leq\Const \chi_c^{-1}\mu^{m_\star-n_0}, \quad \forall t\in\bT.
\end{equation}
\end{sublem}
\begin{proof}
Let $w$, $\|w\|=1$, such that $DF^{m_\star}w=\|DF^{m_\star}w\|e_2$.
Also, let $v\in\bR^2$, $\|v\|=1$ such that $DF^{m_\star-n_0}v=\frac{\|DF^{m_\star-n_0}v\|}{(1+\vartheta_{\hat \nu_{n_0}}^2)^{1/2}}(\qdr(\hat\nu_{n_0}'),\vartheta_{\hat \nu_{n_0}})$, where, for $z=(z_1,z_2)\in\bR^2$, $\qdr(z)=\sign(z_1)\cdot \sign(z_2)$. Note that $v\not\in\fC_c$, otherwise we would have $\hat\nu_{m_\star}'\in\fC_c$, contrary to the hypothesis. We can then write $v=a e_1+bw$. Note that $w\in (DF)^{-1}\fC_c$, moreover by \eqref{def of lamb^+ mu^+} it follows $|w_1|\leq \iota_\star\chi_c |w_2|$ while $\Const \chi_c|v_2|\geq |v_1|\geq \chi_c|v_2|$. In addition, $v_2=b w_2$ and 
 \[
 \begin{split}
& |a|\geq \chi_c|v_2|-|bw_1|\geq \chi_c(1-\iota_\star)|b||w_2|
 \geq \chi_c(1-\iota_\star)(1+\chi_c^2\iota_\star^2)^{-\frac 12}|b|\\
& |a|\leq \Const\chi_c|v_2|+|bw_1|\leq\Const \chi_c |b||w_2|
 \leq \Const \chi_c |b| 
 \end{split}
 \]
 which implies $\Const\chi_c^{-1}\leq \frac{|b|}{|a|}\leq\Const\chi_c^{-1}$. Finally, by equations \eqref{def of lamb^+ mu^+} and \eqref{lambda_+<Clambda_-}, we can write
 \[
 \begin{split}
 \vartheta_{\hat \nu_{n_0}}&=\frac{|\langle e_2,  DF^{m_\star-n_0}v\rangle|}{|\langle e_1,  DF^{m_\star-n_0}v\rangle|}\leq \frac{|b|\mu^{m_\star}+|a|\,|\langle e_2,  DF^{m_\star-n_0}e_1\rangle|}{|a|\,|\langle e_1,  DF^{m_\star-n_0}e_1\rangle|}\\
 &\leq \Const\frac{|b|}{|a|}\mu^{m_\star-n_0}(\lambda_{m_\star-n_0}^+\circ \hat \nu_{m_\star})^{-1} +\iota_\star\chi_u
 \leq  \Const\frac{|b|}{|a|}\mu^{m_\star-n_0}(\lambda_{m_\star-n_0}^+\circ \hat \nu_{m_\star})^{-1} +\iota_\star\vartheta_{\hat \nu_{n_0}}\\
\vartheta_{\hat \nu_{n_0}}&\geq  \frac{|b|\mu^{-m_\star+n_0}-|a|\,|\langle e_2,  DF^{m_\star-n_0}e_1\rangle|}{|a|\,|\langle e_1,  DF^{m_\star-n_0}e_1\rangle|}\geq \Const \chi_c^{-1}\mu^{-m_\star+n_0}(\lambda_{m_\star-n_0}^+\circ \hat \nu_{m_\star})^{-1} -\vartheta_{\hat \nu_{n_0}}
\end{split}
\]
that is \eqref{eq:chiulambda}.
\end{proof}
Note that equation \eqref{eq:h'_m0} and Sublemma \ref{sublem:chiulambda} imply
\begin{equation}\label{eq:h'_m01}
|\bar h'_{m_0}|\leq \Const \mu^{m_0}.
\end{equation}
Thus, recalling \eqref{eq:nuh_prelim} and \eqref{eq:h'_m0}, and the definition of $M_{m,n_0}$ in \eqref{def of eta_nstar and M(t,n_0,m)},
\begin{equation}\label{eq:h_m0''-hat nu''_m}
\begin{split}
\|\hat\nu_m''(t)\|&\leq \Const\left[ \Lambda^{2n_0}\mu^{2m}\Delta_{\gamma}\circ h_{n_0}(t)+ 
C_{\mu,m}\mu^{3m}\vartheta_{\hat\nu_{n_0}}^{-1}\right] \left\{1+ \mu^{2m}\vartheta_{\hat\nu_{n_0}}^{-1}C_F\right\}\\
&\leq M_{m,n_0}(t)\\
|\bar h_{m_0}''(t)|&\leq  \Const\left[
\Lambda^{2n_0}\mu^{3m}\Delta_{\gamma}\circ h_{n_0}(t)+
C_{\mu,m}\mu^{4m}\vartheta_{\hat\nu_{n_0}}^{-1}\right]\left\{1+ \mu^{2m}\vartheta_{\hat\nu_{n_0}}^{-1}C_F\right\}\\
&\leq\mu^m M_{m,n_0}(t).
\end{split}
\end{equation}  

Our next task is to estimate  $\hat\nu_\ovm'''$. To this end we first estimate $\hat\nu_{m_\star}'''$, to do so we use \eqref{eq:nu'''} where $\hat \nu_n, \gamma, h_n$ are replaced by $\hat \nu_{m_\star}, \hat \nu_{n_0}$ and $\bar h_{m_0^{\star}}$, $m_0^{\star}=m_\star-n_0$, respectively. In this case $\hat \nu_{n_0}'\notin \fC_c$, and so is $\frh_k(\hat \nu_{n_0})'$ for each $k< m_\star$, $\frh_k\in \frH^k$. We will estimate the terms in \eqref{eq:nu'''} one by one. We will use Proposition \ref{prop on DF^-1S} for the first one and \eqref{eq:dtDF-dt^2DF} with the equivalent of \eqref{eq:ve-special-inv1} for the second.\footnote{ We use Proposition \ref{prop on DF^-1S}, with $\|\hat \nu_{m_\star}''\|$ instead of $\bbc$, $\hat\nu_{m_\star}$ for $\nu$, $n$ and $m$ replaced by $m_0^\star$.}
Recalling Sublemma \ref{sublem:chiulambda} we have the following estimates
\[
\begin{split}
&\left\|\left[(D_{\hat \nu_{m_\star}}F^{m^\star_0})^{-1}\right]''\hat \nu_{n_0}'\bar h_{m^\star_0}'\right\|
\leq \Const \mu^{m_0^\star}C_{\mu,m_0^\star}(1+C_F\lambda^+_{m_0^\star})\Big\{\mu^{2m_0^\star}C_{\mu,m_0^\star}\lambda_{m_0^\star}^+ 
+\|\hat \nu_{m_\star}''\|\Big\}\|\hat \nu_{n_0}'\||\bar h_{m^\star_0}'|\\
&\phantom{\|(D_{\hat \nu_{m_\star}}}
 \leq C_{\mu,m_0^\star}\mu^{5m_0^\star} (1+C_F\mu^{m_0^\star}\vartheta_{\hat\nu_{n_0}}^{-1})^2\left[C_{\mu,m_0^\star}\mu^{m_0^\star}+\Lambda^{2n_0}\Delta_\gamma\circ h_{n_0}(t)\vartheta_{\hat\nu_{n_0}}\right]\vartheta_{\hat\nu_{n_0}}^{-2} \\
&\left\|\left[(D_{\hat \nu_{m_\star}}F^{m^\star_0})^{-1}\right]'\hat \nu_{n_0}'\bar h_{m^\star_0}''\right\|\leq\Const  \mu^{{m_\star}} \left(1 +\lambda_{m^\star_0}^+(\hat\nu_{m_\star})C_F\right)\vartheta_{\hat\nu_{n_0}}^{-1}| \bar h_{m^\star_0}''|\\
&\phantom{\|(D_{\hat \nu_{m_\star}}}
 \leq \mu^{5{m_0^\star}}C_{\mu,{m_0^\star}}(1+C_F\mu^{m_0^\star}\vartheta_{\hat\nu_{n_0}}^{-1})^2\left[C_{\mu,m_0^\star}\mu^{m_0^\star}+\Lambda^{2n_0}\Delta_\gamma\circ h_{n_0}(t)\vartheta_{\hat\nu_{n_0}}\right]\vartheta_{\hat\nu_{n_0}}^{-2},
\end{split}
\]
where in the second and last inequality we  have used equations \eqref{eq:h'_m01}, \eqref{eq:h_m0''-hat nu''_m}, \eqref{eq:nun0}.
Analogously, recalling \eqref{eq:dtDF-dt^2DF}, \eqref{eq:ve-special-inv1}, \eqref{eq:nunzero}, \eqref{eq:h'_m0}, \eqref{eq:h_m0''-hat nu''_m} and   Sublemma \ref{sublem:chiulambda} we have\footnote{ We write $\hat \nu_{n_0}''=a v+be_2$ where $v=DF^{m_0^\star}e_1 \|    DF^{m_0^\star}e_1\|^{-1}$, thus $|a|\leq \Const \|\hat \nu_{n_0}''\|$ and $|b|\leq \Const\chi_u \|\hat \nu_{n_0}''\|$ and then we use \eqref{eq:ve-special-inv}. Also, we treat  $\hat \nu_{n_0}'''$ in the same manner.}
\[
\begin{split}
&\left\|\left[(D_{\hat \nu_{m_\star}}F^{m_0^\star})^{-1}\right]'\hat \nu_{n_0}''(\bar h_{m_0^\star}')^2\right\|\leq C_{\mu, m_\star}\mu^{m_\star} \left(1 +\lambda_{m_0^\star}^+(\hat\nu_{m_\star})C_F\right)\|\hat \nu_{n_0}''\||\bar h_{m_0^\star}'|^2,\\
&\phantom{\|\left[(D_{\hat \nu_{m_\star}}F^{m_0^\star})^{-1}\right]'\hat \nu_{n_0}''(\bar h_{m_0^\star}')^2\|}
\leq \mu^{3m_\star} C_{\mu, m_\star}(1+C_F\mu^{m_0^\star}\vartheta_{\hat\nu_{n_0}}^{-1})\Lambda^{2n_0}\Delta_\gamma\circ h_{n_0} \vartheta_{\hat\nu_{n_0}}^{-1}\\
&\left\|\left[(D_{\hat \nu_{m_\star}}F^{m_0^\star})^{-1}\right]\hat \nu_{n_0}'''(\bar h_{m_0^\star}')^3\right\|\leq  C_{\mu, m_\star}\mu^{m_\star} \left(1 +\lambda_{m_0^\star}^+(\hat\nu_{m_\star})C_F\right)\lambda_{m_0^\star}^-(\hat \nu_{m_\star})^{-1} \|\hat \nu_{n_0}'''\||\bar h_{m_0^\star}'|^3\\
&\phantom{\|\left[(D_{\hat \nu_{m_\star}}F^{m_0^\star})^{-1}\right]\hat \nu_{n_0}'''(\bar h_{m_0^\star}')^3\|}
\leq\mu^{5m_\star} (1+C_F\mu^{m_0^\star}\vartheta_{\hat\nu_{n_0}}^{-1}) \Lambda^{3n_0}\Delta_\gamma^2 \vartheta_{\hat\nu_{n_0}}^{-1}  \\
& \left\|\left[(D_{\hat \nu_{m_\star}}F^{m_0^\star})^{-1}\right]\hat \nu_{n_0}''\bar h_{m_0^\star}'\bar h_{m_0^\star}''\right\|\leq  C_{\mu, m_\star}\mu^{m_\star} \left(1 +\lambda_{m_0^\star}^+(\hat\nu_{m_\star})C_F\right)\lambda_{m_0^\star}^-(\hat \nu_{m_\star})^{-1} \|\hat \nu_{n_0}''\||\bar h_{m_0^\star}'||\bar h_{m_0^\star}''|\\
&\leq  C_{\mu, m_\star} \mu^{7m_\star}(1+C_F\mu^{m_0^\star}\vartheta_{\hat\nu_{n_0}}^{-1})^2\Lambda^{2n_0}\Delta_\gamma\circ h_{n_0}\left[C_{\mu,m_0^\star}\mu^{m_0^\star}+\Lambda^{2n_0}\Delta_\gamma\circ h_{n_0}(t)\vartheta_{\hat\nu_{n_0}}\right]\vartheta_{\hat\nu_{n_0}}^{-1}.
\end{split}
\]
Using the above estimates in \eqref{eq:nu'''}, and Sublemma \ref{sublem:chiulambda} again, we conclude
\[
\begin{split}
\|\hat \nu_{m_\star}'''\|\le& \mu^{8m_\star}C_{\mu,m_\star}^2(1+C_F\mu^{m_0^\star}\vartheta_{\hat\nu_{n_0}}^{-1})^2
\left[ 1+\Lambda^{2n_0}\Delta_\gamma\circ h_{n_0}\vartheta_{\hat\nu_{n_0}}+\Lambda^{3n_0}(\Delta_\gamma\circ h_{n_0})^2\vartheta_{\hat\nu_{n_0}}\right]\vartheta_{\hat\nu_{n_0}}^{-2}\\
&=\overline{M}_{m_\star,n_0}(t).
\end{split}
\]
Next, for each $\ovm>m_\star$ let  $F^{\ovm-m_\star}\hat \nu_{\ovm}=\hat \nu_{m_\star}\circ \bar h_{\ovm-m_\star}$. Then,
\begin{equation}\label{eq:h'_m1}
\begin{split}
|\bar h'_{\ovm-m_\star}(t)|&= |\langle e_2, D_{\hat{\nu}_\ovm(t)}F^{\ovm-m_\star} \hat\nu_{\ovm}'(t)\rangle|\le \Const\chi_c^{-1}\mu^{\ovm-m_\star}\\
|\bar h'_{\ovm-m_\star}(t)|& \geq \Const\chi_c^{-1}\mu^{-\ovm+m_\star}.
\end{split}
\end{equation}
We can now apply Lemmata \ref{lem stable vertic curves} and \ref{lem stable vertic curves_sharp}, in particular \eqref{eq:nu2-sharp}, to $\hat\nu_{\ovm}$ and $\hat h_{\ovm-m_\star}$ with $\gamma$ replaced by $\hat \nu_{m_\star}$, $n$ by $\ovm-m_\star$, and $c_\star$ and $c_\star^2$ replaced by $M_{m,n_0}(t)$ and $\overline M_{m,n_0}(t)$ respectively. We thus obtain
\begin{equation}\label{eq:h_m1''-hat nu''_ovm}
\begin{split}
&\|\hat\nu_{\ovm}''\|\leq c_\flat c_{n_\star}^{\ovm-m_\star}\mu^{2(\ovm-m_\star)}\lambda_{\ovm-m_\star}^+(\hat \nu_m\circ \bar h_{\ovm-m_\star})^{-1}M_{m,n_0}+C_{\mu,n_\star}\mu^{3n_\star}C_3,\\
&\|\hat \nu_\ovm'''\|\le   c_\flat^2(1+6\varsigma_{n_\star}C_{\mu,n_\star})c_{n_\star}^{\ovm-m_\star}\mu^{3(\ovm-m_\star)}(\lambda_{\ovm-m_\star}^{-}(\gamma\circ h_{\ovm-m_\star}(t)))^{-1}\overline M_{m,n_0}\\
&\phantom{\|\hat \nu_n'''\|\le }
+c_\flat^2b_{n_\star}^{n_\star} \mu^{2(\ovm-m_\star)+2n_\star}\lambda_{\ovm-m_\star}^-(\gamma\circ h_{\ovm-m_\star}(t))^{-1}M_{m,n_0}+\mathbbm{s}_{n_\star},\\
&|\bar h''_{\ovm-m_\star}|\le \Const (M_{m,n_0} C_F+1)\mu^{2(\ovm-m_\star)}C_{\mu,\ovm-m_\star}.
\end{split}
\end{equation}
We are ready to conclude. Recalling Corollary \ref{cor:stable_curves_inv}, the first two of the above equations plus condition \eqref{eq:cond for bar m} give $\hat \nu_\ovm\in \Gamma_3(\bbc)$. Next we set $m_1=\ovm-m$. If $F^{\ovm}\hat \nu_{\ovm}=\gamma\circ h_{\ovm}$, by definition we have
\begin{equation}\label{eq:h_ovm}
h_{\ovm}=h_{n_0}\circ \bar h_{m_0}\circ \bar h_{m_1}.
\end{equation}
Hence, differentiating \eqref{eq:h_ovm} and recalling \eqref{eq:h'_m01}, \eqref{eq:h'_m1} and \eqref{eq:hn0} we have the first of \eqref{eq: h_ovm' h_ovm''}. 
Taking two derivatives of \eqref{eq:h_ovm} and using the second lines of \eqref{eq:nunzero}, \eqref{eq:h_m0''-hat nu''_m} and the third of \eqref{eq:h_m1''-hat nu''_ovm}, we have\footnote{ Here we drop the dependence on $t$ to ease notations.}
\[
\begin{split}
|h''_\ovm|&\le |h''_{n_0}\circ \bar h_{m_0}\circ \bar h_{m_1}\cdot\left[ \bar h'_{m_0}\circ h_{m_1}\cdot \bar h'_{m_1}\right]^2|\\
&+| h'_{n_0}\circ \bar h_{m_0}\circ \bar h_{m_1}\left[h''_{m_0}\circ \bar h_{m_1}\cdot (\bar h'_{m_1})^2+ \bar h''_{m_1}\cdot \bar h'_{m_0}\circ \bar h_{m_1}\right]|\\
&\leq  C_\sharp \mu^{3\ovm}\vartheta_{\hat\nu_{n_0}}^{-1}\Lambda^{3n_0}\left\{\Delta_\gamma\vartheta_{\hat\nu_{n_0}}^{-1}+M_{m,n_0}(C_FC_{\mu,\ovm}+1)+C_{\mu,\ovm}\right\}, 
\end{split}
\]
form which the second of \eqref{eq: h_ovm' h_ovm''} follows and the Lemma is proven.

\section{Proof of Lemma \ref{lem on decomposition}}
\label{appendix proof of lem on decomposition}
{ This appendix is devoted to the proof of Lemma \ref{lem on decomposition}.

As before we use the notation $F^k\hat \nu_k=\gamma\circ h_k$, $F^k \nu_k=\gamma$ .
As the computation is local it suffices to consider $p_n\in\hat\nu_n$ and $p_0\in \gamma$ such that $F^n(p_n)=p_0$. Let $p_k=F^{n-k}p_n$. To ease notation we use a translation to reparmetrize the curves so that $\nu_k(0)=\hat\nu_k(0)=p_k$,  note that $h_k(0)=0$. 
Before discussing the splitting of the vector field we need some notations and few estimates.

It is convenient to perform the changes of variables $\phi_k^{-1}(x,y)=(x,0)+\hat\nu_k(y)$ and set
\[
\wF^k=\phi_0\circ F^k \circ \phi_{k}^{-1}\;; \quad \wF_k=\phi_{k-1}\circ F \circ \phi_{k}^{-1}
\]
Note that $\wF^k=\wF_k\circ \cdots \circ \wF_1$ and $\wF^n(0,y)=\phi_0\circ F^n(\hat\nu_n(y))=\phi_0(\gamma\circ h_n(y))=(0,h_n(y))$, this implies that
\[
	D_{(0,y)}\wF^n=\begin{pmatrix}
	a^n(y) & 0 \\
	c^n(y) & d^n(y)
	\end{pmatrix}
	\;;\quad 	D_{(0,y)}\wF_k=\begin{pmatrix}
	a_k(y) & 0 \\
	c_k(y) & d_k(y)
	\end{pmatrix}
		\;;\quad
	D\phi_k^{-1}=\begin{pmatrix}
	1&(\hat\nu_k')_1\\ 
	0&1
	\end{pmatrix},
\]
with $d^n(y)=h_n'(y)$ and $d_k(y)=h_k^*(y)$. Thus, we have the estimates on the $\cC^\rho$ norms of $d^k$ by Lemma \ref{lem stable vertic curves}, also the changes of coordinates $\phi_k$ have uniformly bounded $\cC^\rho$ norms. 
From the above we easily get the formulae:
\begin{align}
&a^{k+1}(y)=a^{k}(y)a_{k+1}(h_k(y))\label{eq:a^k} \\
&d^{k+1}(y)=d_{k+1}(h_k(y))d^{k}(y)\label{eq:d^k} \\
&c^{k}(y)=\sum_{j=1}^kd_k(h_{k-1}(y))\cdots d_{j+1}(h_{j}(y)) c_j(h_{j-1}(y))
a_{j-1}(h_{j-2}(y))\cdots a_1(y). \label{eq:c^k}
\end{align}
Moreover, 
\[
DF^k=\begin{pmatrix}
a^k+(\hat\nu_0')_1c^k& 
(\hat\nu_0')_1 d^k-(\hat\nu_k')_1\left[a^k+(\hat\nu_0')_1 c^k\right]\\
c^k&d^k-(\hat\nu_k')_1c^k
\end{pmatrix}
\]
which, setting $y_k=h_k(y)$, yields the alternative representations and estimates
\begin{equation}\label{eq:adck}
\begin{split}
&c_k(y_{k-1})=\langle e_2, D_{(0,y_{k-1})}F e_1\rangle\\
&a_k(y_{k-1})=\langle e_1, D_{(0,y_{k-1})}F e_1\rangle-\nu'_{k-1}(y_{k-1})_1 \langle e_2, D_{(0,y_{k-1})}F e_1\rangle\\
&|c^k(y)|=|\langle e_2, D_{(0,y)}F^k e_1\rangle|\leq \lambda_k^+\chi_u\\
&\frac{\lambda_k^-}{\sqrt{1+\chi_u^2}}-\chi_c\chi_u\lambda_k^+\leq |a^k(y)|\leq \lambda_k^+ +\chi_c\chi_u\lambda_k^+
\end{split}
\end{equation}
Also, for further use,
\begin{equation}\label{eq:inverseDwF}
\left(D\wF^k\right)^{-1}=\begin{pmatrix}
a^k(y)^{-1}&0\\
-d^k(y)^{-1}a^k(y)^{-1}c^k(y)&d^k(y)^{-1}
\end{pmatrix}.
\end{equation}

We are now ready to describe the splitting of the vector field.
We do it in the new coordinates.
Consider the subspace $E_n(y)= \{(\eta,u^n(y)\eta)\}_{\eta\in \bR}$, where $u_n(y)=a^n(y)^{-1}c^n(y)$, which is a $\cC^r$ approximation of the unstable direction. Given a vector $v\in\bR^2$ let us call $\tilde v=D\phi_0v$ the vector in the new coordinates. Next, we decompose a vector $\tilde v$ as
\[
\tilde v=(1,u_n\circ \frh_n)\tilde v_1+(\tilde v_2-\tilde v_1 u_n\circ \frh_n)e_2
\]
where $\frh_n\circ \wF^n(0,y)=(0,y)$.
Thus, setting $V(t)=v_1(\gamma(t))-\gamma'(t)_1 v_2(\gamma(t))$, we have the decomposition \eqref{eq:v=vu+vc}, restricted to $\gamma$, with
\begin{equation}
\label{vector fields decomposition}
\begin{split}
&v^u(\gamma(t))=V(t)(1+\gamma'(t)_1u_n\circ \frh_n(0,t), u_n\circ \frh_n(0,t) ) \\
&  v^c(\gamma(t))=[v_2(\gamma(t))-u_n\circ \frh_n(0,t)V(t)]\gamma'(t),
\end{split}
\end{equation}
which gives, in particular, $v^c(\gamma(t))=g(t)\gamma'(t)$ with $g(t)=v_2(\gamma(t))-u_n\circ \frh_n(0,t)V(t).$ \\
To extend the above decomposition in a neighborhood of $\gamma$ we will proceed as in \cite[Lemma 6.5]{GoLi}.\footnote{ In the mentioned paper the authors need more regularity for the extended vector field. Here it is enough to obtain a vector field which is $\cC^\rho$.}
First, we compute the derivatives along the curve, to this end note that in the new coordinates $t=y_n$.
Differentiating \eqref{eq:a^k} we have 
\begin{equation}\label{eq:deriv_ak}
\partial_y a^{k}(y)^{-1}=\left[\partial_ya^{k-1}(y)^{-1}\right]a_{k}(y_{k-1})^{-1}
+a^{k-1}(y)^{-1}\partial_{y_{k-1}}a_{k}(y_{k-1})^{-1}
\tilde{d}^{k-1},
\end{equation}
and, by \eqref{eq:adck} and and  Lemma \ref{lem stable vertic curves},
\[
\begin{split}
&|\partial_{y_{k-1}}a_k(y_{k-1})|\leq \Const(1+\|\nu_{k-1}''\|)\leq \Const (1+\bbc)\\
& \|\partial_y a_{k}\|_{\cC^\rho}\le \Const \|\nu_{k-1}\|_{\cC^{\rho+1}}\le \Const \bbc^{\rho!}. 
\end{split}
\]
Next, using \eqref{eq:deriv_ak}, we can prove by induction that $\|({a}^{n})^{-1}\|_{\cC^{\rho}}\leq \Const\lambda^{-n}_-
\bbc^{\rho\rho!} C_{\mu, n}^{\rho a_{\rho}}\mu^{\rho\rho! n}:$\footnote{ Here $a_\rho$ is the one given by Lemma \ref{lem stable vertic curves}.}
\begin{equation}\label{eq:a^n-Crho-norm}
\begin{split}
\|[a^{n}]^{-1}\|_{\cC^{\rho}}
&\leq \Const\lambda^{-n}_-+\lambda^{-1}\|[a^{n-1}]^{-1}\|_{\cC^{\rho}}+\Const \|[a^{n-1}]^{-1}\|_{\cC^{\rho-1}}\bbc^{\rho!} C_{\mu, n-1}^{a_\rho}\mu^{\rho!(n-1)}\\
&\leq \Const\lambda^{-n}_-+\bbc^{\rho!}\Const \sum_{j=0}^{n-1}\lambda^{j-n}_- \|[{a}^{j}]^{-1}\|_{\cC^{\rho-1}}C_{\mu, j}^{a_\rho}\mu^{\rho!j}\\
&\leq \Const\lambda^{-n}_-
\bbc^{\rho\rho!} C_{\mu, n}^{\rho a_{\rho}}\mu^{\rho\rho! n}.
\end{split}
\end{equation}
To compute $\|(d^n)^{-1}\|_{\cC^\rho}$ we can use formula \eqref{formula 1 for norm Crho} and recall \eqref{Crho norm of h} and \eqref{eq: h_n'}:
\begin{equation}\label{eq:d^n-Crho-norm}
\|(d^n)^{-1}\|_{\cC^\rho}=\|(h_n')^{-1}\|_{\cC^\rho}\le \Const \bbc^{\rho!}\mu^{(\rho+1)n}C_{\mu, n }^{a_{\rho+1}}\mu^{(\rho+1)! n}=\Const \bbc^{\rho!} C_{\mu, n }^{a_{\rho+1}}\mu^{(\rho+1)(\rho!+1) n}.
\end{equation}
Next, by \eqref{eq:a^k}, \eqref{eq:d^k} and \eqref{eq:c^k} we have
\[
\begin{split}
&[a^n(y)]^{-1}c^{n}(y)=\sum_{j=1}^n d_n(h_{n-1}(y))\cdots d_{j+1}(h_{j}(y)) c_j(h_{j-1}(y))
[a_{n}(h_{n-1}(y))\cdots a_j(h_{j-1}(y))]^{-1},\\
&[d^{n}(y)a^n(y)]^{-1}c^{n}(y)=\sum_{j=1}^n[d_{j-1}(h_{j-2}(y))\cdots d_{1}(y))]^{-1} c_j(h_{j-1}(y))
[a_{n}(h_{n-1}(y))\cdots a_j(h_{j-1}(y))]^{-1}.
\end{split}
\]
}
Hence, by \eqref{eq:a^n-Crho-norm}, \eqref{eq:d^n-Crho-norm} and the first of \eqref{eq:adck}, we obtain, using \eqref{banach alg property},
\begin{equation}\label{eq_d^nc^na^n}
\begin{split}
&\|[d^{n}a^n]^{-1}c^{n}\|_{\cC^{\rho}}\le \Const \bbc^{\rho \rho!} C_{\mu, n }^{2\rho a_{\rho}+1}\mu^{(\rho+1) (\rho!+1)+\rho!}\\
&\|[a^n]^{-1}c^{n}\|_{\cC^\rho}\le \Const \bbc^{\rho!}C_{\mu,n}^{a_\rho}\mu^{\rho!n}.
\end{split}
\end{equation}
We are ready to conclude. Since 
\[
(D_{\hat \nu_n(y)}F^n)^{-1}= D_{(0,y)}\phi_n^{-1}(D_{(0,y)}\wF^n)^{-1} D_{\gamma\circ h_n(y) }\phi_0,
\]
by \eqref{vector fields decomposition} and \eqref{eq:inverseDwF} it follows
\begin{align*}
&(D_{\hat \nu_n(y)}F^{n})^{-1}v^u(\gamma\circ h_n(y))=V(h_n(y))\left({a}^n(y)^{-1}, 0\right),\\
&(D_{\hat \nu_n(y)}F^{n})^{-1}v^c(\gamma\circ h_n(y))=d^n(y)^{-1}\cdot[v_2-u_n v_1] \circ\gamma(h_n(y))\big((\hat \nu_n')_1(y),1 \big).
\end{align*}
Recalling that $u_n(y)=a^n(y)^{-1}c^n(y),$ by \eqref{eq:a^n-Crho-norm}, \eqref{eq:d^n-Crho-norm}, \eqref{eq_d^nc^na^n}, and since $\gamma\in \Gamma(\bbc)$ and $\|v\|_{\cC^r}\le 1$, we have the result for the vector field along the curve. Finally, we extend $v^u$ to a neighborhood of $\gamma$. It turns out the be more convenient to define first the extension
\[
w(x,y)=F^{n^*}v^u(\hat \nu_n(y))
\]
then $\hat v^u=\frh_n^*w$ and $ F^{n^*}\hat v^u=w$.
By these definitions it follows
\[
\begin{split}
&\|F^{n^*}\hat v^u\|_{\cC^\rho(N(\nu))}=\|F^{n^*} v^u\|_{\cC^\rho_\nu}\le \lambda^{-n}_- \bbc^{\rho!\rho}
 C_{\mu, n}^{\rho a_{\rho}}\mu^{\rho\rho! n}\\
&\|\hat v^u\|_{\cC^\rho(M'(\gamma))}\le C_n.
\end{split}
\]
The definition of $\hat v^c$ and relative estimates are analogous.

\section{The space \texorpdfstring{$\cH^s$}{Lg}}
\label{appendix space Hs }
Let $u\in \mathcal{C}^\infty(\mathbb{T}^2)$. The \textit{Fourier Transform} of $u$ and its inverse are
\begin{align}
&\mathcal{F}u(\xi)=\int_{\mathbb{T}^2} e^{- i 2\pi x\xi}u(x)dx, \quad\xi\in \mathbb{Z}^2, \label{Fourier}\\
&u(x)=\sum_{\xi\in \mathbb{Z}^2}\mathcal{F}u(\xi)e^{ i 2\pi x\xi}, \quad x\in \mathbb{T}^2.\label{antiFourier}
\end{align}
Then $\cH^s$ is the completion of $\mathcal{C}^\infty(\mathbb{T}^2)$ with respect to the inner product
\begin{equation}
\label{s-scalar prod}
\langle u,v \rangle_s=\sum_{\xi \in \mathbb{Z}^2}\langle \xi \rangle^{2s}\mathcal{F}u(\xi)\overline{\mathcal{F}v(\xi)}, \quad \langle \xi \rangle:=\sqrt{1+\|\xi\|^{2}}.
\end{equation}
Notice that, arguing like in the proof of formula \cite[(7.9.2)]{Hor}, we have
\begin{equation}
\label{eq:horm}
\langle u,v \rangle_s= \sum_{\gamma+\beta\leq s}C_{\gamma, \beta}\langle\partial^\gamma_{x_1}\partial^\beta_{x_2} u,\partial^\gamma_{x_1}\partial^\beta_{x_2}v \rangle_0.
\end{equation} 
Hence, there are constants $C, C_{\gamma,\beta}>0$ such that  
\begin{equation}
\label{equivalence of norm H L}
C^{-1} \sum_{\gamma+\beta=s}C_{\gamma, \beta} \| \partial_{x_1}^\gamma \partial_{x_2}^{\beta} u \|^2_{L^{2}}\le \|u\|^2_{\cH^s}\le C\sum_{\gamma+\beta=s}C_{\gamma, \beta} \| \partial_{x_1}^\gamma \partial_{x_2}^{\beta} u \|^2_{L^{2}}.
\end{equation}
\begin{appxlem}
\label{lem on norm s-1 and s}
For every $\varsigma \in (0,1)$ and $1\le s<r$ there exists constants $C_s$ such that
\begin{equation*}
\|u\|^2_{\mathcal{H}^{s-1}}\le \varsigma\|u\|^2_{\mathcal{H}^s}+\frac{C_s}{\varsigma}\|u\|^2_{L^{1}}, \quad u \in \mathcal{C}^r(\mathbb{T}^2).
\end{equation*}
\end{appxlem}
\begin{proof} By definition of the norm we have, for all $\tau\in (1,2)$,
\begin{equation}
\label{equation of norm s-1}
\|u\|_{\mathcal{H}^{s-1}}^2=\sum_{\xi \in \mathbb{Z}^2}|\mathcal{F}u(\xi)|^2\langle \xi \rangle^{2(s-1)}=\sum_{\xi \in \mathbb{Z}^2}|\mathcal{F}u(\xi)|^2\langle \xi \rangle^{2s-2+\tau}\langle \xi \rangle^{-\tau}
\end{equation}
By Young inequality $ab\le \frac{\varsigma a^p}{p}+\frac{\varsigma^{-\frac{p}{q}}b^{q}}{q}$, for every $\varsigma>0$ and $\frac{1}{p}+\frac{1}{q}=1$. We apply this  with $a=\langle \xi \rangle^{2s-2+\tau}$, $b=\langle \xi \rangle^{-\tau}$ and $p=\frac{2s}{(2s-2+\tau)}, q=\frac{2s}{2-\tau}$ to obtain: 
\begin{equation*}
\langle \xi \rangle^{2s-2+\tau}\langle \xi \rangle^{-\tau}\le \left(1-\frac{2-\tau}{2s} \right)\varsigma \langle \xi \rangle^{2s}+\varsigma^{-1}\frac{(2-\tau)\langle \xi \rangle ^{-\frac{2s\tau}{2-\tau}}}{2s}.
\end{equation*}
Using this fact in (\ref{equation of norm s-1}) and recalling that $\|\mathcal{F}u\|_{\infty}\le C\|u\|_{L^{1}}$, we get
\begin{equation*}
\|u\|_{\mathcal{H}^{s-1}}^2\le \varsigma \sum_{\xi\in \mathbb{Z}^2} |\mathcal{F}u(\xi)|^2\langle \xi \rangle^{2s}+\frac{C_s}{\varsigma}\|\mathcal{F}u\|_{\infty}^2\le \varsigma \|u\|^2_{\mathcal{H}^s}+\frac{C_s}{\varsigma}\|u\|^2_{L^1}.
\end{equation*}
\end{proof}
\section{Vector Field regularity }
\label{appendix proof of lem on contin of ueps}
This appendix is devoted to proving the following regularity results on the iteration of a vector field. Note that the hypotheses \eqref{eq:cond_E1} are implied by assumption ({\bf H4}).
\begin{lem}
\label{lem:vec_regularity}
Let $\ve_0\in (0,1]$, $A\in [0,1/2]$, $B>0$ and $u,u'\in\cC^1(\bT^2,\bR)$ such that $\|u\|_\infty,\|u'\|_\infty\leq A\ve_0^{-1}$ and $\|\nabla u\|_\infty, \|\nabla u'\|_\infty\leq B\ve_0^{-1}$. Consider a family of vertically partially hyperbolic map $F_\ve$, $\ve\leq \ve_0$ such that
\begin{equation}\label{eq:cond_E1}
\begin{split}
&\left\|\frac{\partial_\theta f(p)}{\partial_xf(p)}\right\|_\infty\leq 1\\
&\partial_xf(p)\left[1-A\left\|\frac{\partial_\theta f(p)}{\partial_xf(p)}\right\|_\infty\right]\geq 2(1+\ve_0\|\partial_x\omega\|_\infty).
\end{split}
\end{equation}
For each $\frh\in \frH^{\infty}$ and $k\leq n\in \mathbb{N}$, we define the sequence of functions\footnote{ See \eqref{eq:Phive} for the definition of $\Phi^n_{\varepsilon}$.}
\[
\begin{split}
&\bar u_0(p,\ve)=u(\frh_n(p))\\
&\bar u_{k}(p,\ve)=\pi_2\circ \Phi^k_{\varepsilon}(\frh_n(p), u(\frh_n(p))),
\end{split}
\]
and similarly for $\bar u_{k}'$.
Then, for each  $p,p'\in\bT^2$ and $\ve,\ve'<\ve_0$,
\begin{equation}\label{eq:lip-eq-for-u}
\begin{split}
|&\bar u_n(p,\ve)-\bar u_n'(p',\ve')|\leq 
C_\sharp e^{4A}\mu^{3n} \bigg\{\lambda_{n}^+(\frh_n(p))^{-1}\|u-u'\|_\infty\\
&+( \|\omega\|_{\cC^2}+\mu^{2n}\lambda_{n}^+(\frh_n(p))^{-1} C_{\mu,n}|u'|)\|p-p'\|
+\left[1+\lambda_{n}^+(\frh_n(p))^{-1}|u'|^2 \right]|\ve-\ve'|\bigg\}.
\end{split}
\end{equation}
\end{lem}
\begin{proof}
Let $p_k(p,\ve)=\frh_k(p)$, for $\frh\in\cH^\infty$, $p\in\bT^2$. By \eqref{norm of DF^N^-1} (or see \cite{DeLi1} for details) we have
\begin{equation}\label{eq:p_der}
 \|\partial_p p_k\|\leq \|(D_{\frh_k(p)}F_\ve^k)^{-1}\|\leq\Const \mu^k\leq C_\sharp e^{c_\sharp\ve k}.
\end{equation}
For each $u>0$ and for $k\leq n$ let
\begin{equation}\label{eq:lambda_def}
\begin{split}
&\lambda(p,\ve)=
\frac{|\partial_x f(p)|}{1+\ve\left(\|\partial_\theta\omega\|_\infty
+\|\partial_x\omega\|_\infty\right)}\geq |\partial_xf(p)|\mu^{-1}\\
&u_{k}(p,\ve, u)=\Xi_\ve(p_{n-k+1}(p,\ve),u_{k-1}(p,\ve, u)),
\end{split}
\end{equation}
where in the first line we have used equation \eqref{mu_n<e^b}, which shows that this bound does not depend on the hypotheses of Lemma \ref{lem check SVPH}, and $\Xi_\ve$ is defined in \eqref{def of Xi}.
Note that $\bar u_n(p,\ve)=u_n(p,\ve,u_0(p,\ve))$.

Using \eqref{def of Xi} and \eqref{u derivative of Xive} a direct computation yields, for $|u|\leq A\ve_0^{-1}$,
\begin{equation}\label{eq:Xi_der_est}
\begin{split}
&|\Xi_\ve(p,u)|\leq \frac{|u|\left(1+\ve\|\partial_\theta\omega\|_\infty\right)}
{|\partial_xf(p)|\left[1-\ve|u|\|\frac{\partial_\theta f(p)}
{\partial_xf(p)}\|_\infty\right]}+
\frac{\|\partial_x\omega\|_\infty}
{\partial_xf(p)\left[1-A\|\frac{\partial_\theta f(p)}{\partial_xf(p)}\|_\infty\right]}\\
&\phantom{|\Xi_\ve(p,u)|}\leq \frac1{|\partial_xf(p)|}\mu e^{2\ve_0|u|}|u|+\frac 12\|\partial_x\omega\|_\infty\\
&|\partial_u\Xi_\ve(p,u)|\leq \frac{1}{\lambda(p,\ve)}
\left[1-\ve \left|u\frac{\partial_\theta f(p)}{\partial_xf(p)}\right|\right]^{-2}\\
&\|\partial_p\Xi_\ve(p,u)\|\leq C_\sharp( \|\omega\|_{\cC^2}+|u|)\\
&|\partial_\ve\Xi_\ve(p,u)|\leq C_\sharp(1+|u|)|u|.
\end{split}
\end{equation}
The first line of the \eqref{eq:Xi_der_est} and the second of \eqref{eq:cond_E1} imply
\[
|u_k(p,\ve, u)|\leq 2^{-k}|u|+\|\partial_x\omega\|_\infty.
\]
We can get a sharper bound defining
\[
\begin{split}
&\Lambda_{k,j}(p):=\prod_{i=k+1}^{j}\lambda(p_i,\ve)\;;\quad \overline\Lambda_{k,j}(p):=\prod_{i=k+1}^{j}|\partial_x f(p_i)|\\
&\Delta:=\|\partial_x\omega\|_\infty\left\|\frac{\partial_\theta f(p)}{\partial_xf(p)}\right\|_\infty,
\end{split}
\]
then
\begin{equation}\label{eq:uk_decay}
|u_k(p,\ve, u)|\leq \Lambda_{n-k,n}(p)^{-1}|u|+\|\partial_x\omega\|_\infty.
\end{equation}
Moreover, setting $u_j=u_{j}(p,u,\ve)$, $u'_j=u_{j}(p',u',\ve')$, with $|u|, |u'|\leq \frac A{\ve_0}$, and recalling \eqref{eq:p_der}, \eqref{eq:lambda_def}, \eqref{eq:uk_decay}, we have
\begin{equation}
\label{eq: diff up-up'}
\begin{split}
&|u_{k+1}(p,\ve, u)-u_{k+1}(p',\ve', u')|=|\Xi_\ve(p_{n-k},u_{k})-
\Xi_{\ve'}(p_{n-k}',u_{k}')|\\
&\leq C_\sharp ( \|\omega\|_{\cC^2}+|u_{k}'|)\| p_{n-k}-p'_{n-k}\|
+C_\sharp (1+|u_{k}'|)|u_{k}'||\ve-\ve'|\\
&\phantom{\leq}
+\Lambda_{n-k-1,n-k}(p)^{-1}e^{2^{-k+1}A+2\ve_0\Delta}|u_{k}-u'_{k}|\\
&\leq C_\sharp ( \|\omega\|_{\cC^2}+\overline \Lambda_{n-k,n}(p')^{-1}\mu^{k}|u'|)\big[ \mu^{n-k}\|p-p'\|\\
&\phantom{\leq}
+(1+\overline\Lambda_{n-k,n}(p')^{-1}\mu^k|u'|)|\ve-\ve'|\big]
+\Lambda_{n-k-1,n-k}(p)e^{2^{-k+1}A+2\ve_0\Delta}|u_{k}-u'_{k}|.
\end{split}
\end{equation}
We can then iterate the above equation and obtain
\[
\begin{split}
|&u_n(p,\ve,u)-u_n(p',\ve', u')|=
\overline\Lambda_{0,n}(p)^{-1}\mu^ne^{4A+2n\ve_0\Delta}|u-u'|\\
&+C_\sharp\sum_{k=0}^{n-1}\overline \Lambda_{0,n-k}(p)^{-1}\mu^{n-k}
e^{4A+2\ve_0(n-k)\Delta}(\|\omega\|_{\cC^2}+\overline \Lambda_{n-k,n}(p')^{-1}\mu^k|u'|) \mu^{n-k}\|p-p'\|\\
&+C_\sharp\sum_{k=0}^{n-1}\overline \Lambda_{0,n-k}(p)^{-1}\mu^{n-k}
e^{4A+2\ve_0(n-k)\Delta}(1+\overline \Lambda_{n-k,n}(p')^{-2}\mu^k|u'|^2)
|\ve-\ve'|.
\end{split}
\]
In addition equations \eqref{eq:expansion} and \eqref{lambda_+<Clambda_-} imply
\[
\begin{split}
&\overline\Lambda_{j,n}(p)\geq \Const\lambda_{n-j}^+(p_n)\\
&|\partial_p\overline \Lambda_{j,n}(p)|\leq\sum_{j=l+1}^{n}\left|\overline \Lambda_{l,n}(p)\frac{\partial^2_x f(p_{l})}{[\partial_x f(p_{l})]^2}\overline \Lambda_{j,l-1}(p)\right|\|\partial_{p}p_l \|
\leq C_\sharp C_{\mu,n}\mu^n\overline \Lambda_{j,n}(p).
\end{split}
\]
Thus,
\[
\begin{split}
|&u_n(p,\ve,u)-u_n(p',\ve',u')|\leq 
C_\sharp e^{4A+2n\ve_0\Delta}\mu^n \bigg\{\lambda_{n}^+(p_n)^{-1}|u-u'|\\
&+( \|\omega\|_{\cC^2}+\mu^{2n}\lambda_{n}^+(p_n)^{-1} C_{\mu,n}|u'|)\|p-p'\|
+\left[1+\lambda_{n}^+(p_n)^{-1}|u'|^2 \right]|\ve-\ve'|\bigg\}.
\end{split}
\]
The Lemma follows recalling that equation \eqref{rem:mu_choice} and our hypotheses imply $e^{\ve_0\Delta}\leq \mu$.
\end{proof}
\section{Extension of curves }\label{append lem exten of curv}
In this section we explain how to extend a segment to a close curve of homotopy class $(0,1)$ with precise dynamical properties and explicit bounds on the derivatives.
\begin{lem}
\label{lem extension of curve}
There exist constants $\delta_0$,  $C_{n_0,j}>0$, $j,n_0\in\bN\cup\{0\}$, and $L_\star\geq 1$ such that for each line segment $\gamma(t)=\gamma(0)+(1,v)t$, $t\in[-\delta,\delta]$ with $\delta\leq \delta_0$, such that, for each $\frh\in\frH^\infty_{\gamma_-}$, if $\gamma'(t)\not \in D_{\frh_{n_0}(\gamma(t))}F^{n_0}\fC_u$, then we  can extend $\gamma$ to a closed curve $\tilde \gamma$, parametrized by arclength, of homotopy class $(0,1)$ with the following properties:
\begin{itemize}
    \item let $\gamma_-(t)=\gamma(0)+\frac 12e_1+e_2t$, then for each $k\in\bN$ we have $\tilde\gamma\in \operatorname{Dom}(\frh_k)$ and $\frh_k\circ\tilde\gamma$ is a closed curve in the homotopy class $(0,1)$.
    \item  $\vartheta_\gamma\leq \vartheta_{\tilde\gamma}$\footnote{ Recall the definition of $\vartheta_\gamma$ in \eqref{def of vartheta_gamma}.}
    \item  Let $m\geq n_0$ be the smallest integer such that $D_{\gamma(t)}\frh_{m}\gamma'(t)\in  \mathbf{C}_{\epsilon,c}$, for all $t\in [-\delta,\delta]$, then  $D_p\frh_{n_0}\tilde \gamma'\not\in  \mathbf{C}_{C_\sharp\epsilon,u}$  and $D_p\frh_{m}\tilde \gamma'\in  \mathbf{C}_{c}$.
    \item For each $j\in \{1,\dots, r-1\}$ and $t\in\bR$,
\begin{equation}\label{derivative of extended curve tildegamma}
\|\tilde\gamma^{(j+1)}(t)\|\leq C_{n_0,j}
\frac{(L_\star\{L_\star,1\}^+\eps^{-1} \mu^{m})^j}{(\chi_u+|\pi_2(\tilde \gamma'(t))|)^{j}}:=C_{n_0,j}\Delta_{\tilde\gamma}^j.
\end{equation}
\end{itemize}
Moreover, if the conditions of Lemma \ref{lem:vec_regularity} are satisfied, then \eqref{derivative of extended curve tildegamma} holds true with
\begin{equation}\label{Lstar}
\begin{split}
&L_\star(n_0)=\sup_{|v|\leq 1}L_\star(v, n_0),\\
&L_\star(v, n_0)=\Const\mu^{5m}C_{\mu,m}(\|\omega\|_{\cC^2}+\bar \kappa)\;;\quad  \bar \kappa=|v|+\chi_u\;;\quad m-n_0\leq \Const \ln\bar\kappa^{-1}.
\end{split}
\end{equation}
\end{lem}
\begin{proof} 
By an isometric change of variables we can assume, without loss of generality, that $\gamma(0)=0$. Hence $\gamma(t)=(1,v)t$ for $t\in [-\delta, \delta]$ and $\gamma'(t)=(1,v)=:\bar v$. Note that we can assume $|v|\leq 1$ since otherwise the Lemma is trivial. 

Before getting to the extension per se, we need some results on the dynamics of the tangent vectors seen as elements of a projective space.
We write a vector outside the central cone as $(1, \vcs)$, so $\vcs$ can be interpreted as a projective coordinate. Then, in analogy with \eqref{eq: def of Xi}, we have, for each $p\in\bT^2$ and $\vcs\in\bR$,
\[
\begin{split}
&D_pF(1,\vcs)=(\partial_xF_1+\partial_\theta F_1 \vcs)(1, \Xi(p, \vcs))\\
&\Xi(p, \vcs)=\frac{\partial_xF_2+\partial_\theta F_2 \vcs}{\partial_xF_1+\partial_\theta F_1 \vcs}.
\end{split}
\]
Also, computing as in \eqref{u derivative of Xi},
\[
\partial_\vcs\Xi(p,\vcs)=\frac{\det(D_{p}F)}{\left(\partial_xF_1+\partial_\theta F_1 \vcs\right)^2}.
\]
Next, for each $q\in\bT^2$, $j\in \bN$, let $q_j=F^{j}(q)$, $z_0(\vcs)=\vcs$, $z_{1}(q,\vcs)=\Xi(q, z_0(\vcs))$ and, for $j\geq 1$,  $z_{j+1}(q_j,\vcs)=\Xi(q_j, z_j(q_{j-1},\vcs))$. In particular, if  $p\in\bT^2$ and $\Gamma_j(p)=D_{\frh_j(p)}F^j\mathbf{C}_c$, then $\Gamma_j(p)=\{ (1,\bar z_j(p,\vcs))\;:\; |\vcs|\leq \chi_c\}$ where $\bar z_j(p,\vcs):=z_j(\frh_j(p),\vcs)$. 
Note that if $\bar z_{j}(p,\chi_c)\not \in \fC_u$, then we have
\begin{equation}\label{eq:lower_der}
|\bar z_j(p,\chi_c)|\leq \Const \lambda_j^-(\frh_j(p))^{-1}\mu^j\chi_c.
\end{equation}
Moreover, 
\[
\begin{split}
&\partial_\vcs \bar z_j(p,\vcs)=\partial_{\bar z}\Xi(\frh_j(p),\bar z_{j-1}(p,\vcs))\partial_\vcs \bar z_{j-1}(p,\vcs).
\end{split}
\]
Iterating the above identities and recalling Propositions \ref{prop on the det}, \ref{prop on DF^-1} we have
\[
\begin{split}
&C_\sharp \frac{\mu^{-j}}{\lambda_j^-(\frh_j(p))}\leq|\partial_\vcs \bar z_j(p,\vcs)|\leq C_\sharp \frac{\lambda_j^+(\frh_j(p))\mu^j}{\lambda_j^-(\frh_j(p))^2}\leq C_\sharp \frac{\mu^{-j}}{\lambda_j^-(\frh_j(p))}.
\end{split}
\]
It follows, by \eqref{eq:lower_der}, that 
\[
|\bar z_j(p,\pm\chi_c)-\bar z_j(p,\pm\chi_{c}(1-\eps))|\geq C_\sharp \eps \mu^{-2j} |\bar z_j(p,\pm\chi_c)|. 
\]
Let $\fra_{j}(t,\frh)=\frac{\pi_2(D_{\gamma(t)}\frh_{j}\bar v)}{\pi_1(D_{\gamma(t)}\frh_{j}\bar v)}$, and $m(\frh,t)$ be the smallest integer $k$ such that $D_{\gamma(t)}\frh_k\bar v\in\fC_{\eps,c}$. 

In the following we consider only the case $\fra_{m(\frh,t)}(t,\frh)\geq \chi_c$. The case 
$\fra_{m(\frh,t)}(t,\frh)\leq -\chi_c$ is totally analogous. Also, we consider only $t>0$ since the construction for $t<0$ is exactly the same.

It follows $\Const \bar z_{m(\frh,t)}(\frh_{m(\frh,t)}(\gamma(t)),\chi_c )\geq v\geq \bar z_{m(\frh,t)}(\frh_{m(\frh,t)}(\gamma(t)),\chi_c (1-\eps))$. Thus, setting 
$m'(\frh,t)=m(\frh,t)-n_0$, we have
\begin{equation}\label{eq:space}
\begin{split}
\fra_{n_0}(\frh,t)-\bar z_{m'(\frh,t)}(\frh_{n_0}(\gamma(t)),\chi_c )&\geq C_\sharp \eps \mu^{-2m} \bar z_{m'(\frh,t)}(\frh_{n_0}(\gamma(t)),\chi_c )\\
&\geq C_\sharp \eps \mu^{-2m} v.
\end{split}
\end{equation}
We are finally ready to extend our segment. 
For $\vf\in\bR$, let $w(\vf)=(\cos\vf, \sin \vf)$, $\theta=\tan v$ and $ a=\sqrt{1+v^2}$. Then $\bar v=aw(\theta)$.
We start by extending the curve to the interval $(\delta, \delta+A)$, with $A=\frac {C_{n_0}}{2a} \eps \mu^{-2m}\bar \kappa L_\star^{-1}<1$.
Next, let $L_\star$ be the maximal Lipschitz constant of the $\bar z_j(\cdot,\pm\chi_c)$ for $j\leq m-n_0$.

Let $b\in\cC^{\infty}(\bR,[0,1])$ be a bump function with $b(t)=0$ for $t\leq 0$ and $b(t)= 1$ for $t\geq 1$.
Also, let $B=\{a L_\star , 16\theta_c,\bbk\}^+$, for some $\bbk\geq 1$ to be chosen later, and where $\theta_c:=\arctan (2\chi_c^{-1})$, and define, for $t\in[\delta, T]$, where $T$ will be chosen shortly,
\begin{equation}\label{eq:hat gamma'}
\hat \gamma'(t)=a w\big(\theta +b((t-\delta)A^{-1}) B(t-\delta)\big)=:aw(\tilde \theta(t)).
\end{equation}
Note that, by construction, $\tilde \theta(t)\geq \theta$. Moreover, for $t\in[\delta,\delta+A]$, we have 
\[
\|\hat \gamma(t)-\hat \gamma(\delta)\|\leq \int_{\delta}^{\delta+A}\|aw(\tilde \theta(ts)) \|ds\leq Aa.
\]
Thus, recalling  \eqref{eq:space}, for $t\in [\delta,\delta+A]$,
\[
\begin{split}
\arctan \bar z_{m(\frh, \gamma(\delta))}(\hat\gamma(t),\chi_c)&\leq \arctan \bar z_{m(\frh, \gamma(\delta))}(\hat\gamma(\delta),\chi_c) + L_\star aA\\
&\leq \theta -C_{n_0}\eps \bar \kappa\mu^{-2m}+L_\star Aa<\theta\leq \tilde\theta(t),
\end{split}
\]
which implies that $D_{\tilde\gamma(t)}\frh_m \hat \gamma'(t)\in {\mathbf C}_c$.
In addition, for $t\geq \delta+A$, we have 
\[
|\frac{d}{dt}\tan\tilde \theta(t)|\geq B\geq aL_\star\geq |\frac  d{dt}  \bar z_m(\tilde\gamma(t)|,
\]
thus $D_{\tilde\gamma(t)}\frh_m \hat \gamma'(t)\in {\mathbf C}_c$ also for $t\geq \delta+A$.

We now choose  $T>0$ be such that $\tilde\theta(T)=\theta_c$ so that $\hat \gamma'(T)$ is well inside the central cone. This implies $T\leq \delta+\theta_cB^{-1}$ and
\[
|\pi_1(\hat \gamma(T))|\leq \Const T\leq  \Const (\delta+ B^{-1})\leq  \Const(\delta_0+ \bbk^{-1})<1/4,
\]
provided $\delta_0$ and $\bbk^{-1}$ are small enough.
It is then a simple exercise to define $\hat \gamma:[T,S]\to \bT^2$ such that $\|\hat \gamma'\|=a$; $\hat\gamma'(t)\in{\mathbf C}_{c}$,  for all $t\in [T,S]$, and $\hat\gamma (S)=(0,1/2)$,
$|\pi_1(\hat\gamma)|\leq \Const(\delta_0+ \bbk^{-1})$, $\hat \gamma'(S)=(-\chi_c/2,1)$,
$\hat\gamma^{(j)}(S)=0$ for all $j>1$ and $\sup_{t\in [T,S]}\|\hat\gamma^{(j)}(t)\|\leq C_\sharp$. By symmetry we have a closed curve $\hat \gamma$ of homotopy class $(0,1)$.

Note that  $\hat\gamma\in \operatorname{Dom}(\frh_k)$ for each $\frh\in \frH_{\gamma_-}^\infty$ and $k\in\bN$. 
Then Lemma \ref{lem inverse branches} implies that there exists inverse branches $\{\frh_{k,i}\}_{i=1}^{d^k}$, where $d$ is the degree of $F$, such that $F^{-k}\hat \gamma=\bigcup_{i=1}^{d^k}\frh_{k,i}\circ \hat\gamma$. Since $\frh_{k,i}$ is a diffeomorphism, $\frh_{k,i}\circ \hat\gamma$ is a closed curve. In addition it must be of homotopy type $(0,1)$, otherwise it would intersect an horizontal segment in more than one point and the image, under $F^k$, of the interval between two intersection points  would be an unstable curve going from $\hat \gamma$ to itself. Since such a curve would be transversal to $\hat \gamma$ by hypothesis, it follows that it would have to wrap around the torus horizontally an hence intersect $\gamma_-$ contradicting the fact that it is in the domain of $\frh_{k,i}$.

Recalling \eqref{eq:hat gamma'}, formula \eqref{formula 1 for norm Crho} gives, for all $j\geq 2$,\footnote{ Notice that, as
$\|\tilde \theta\|_{\cC^l}\le C_\sharp A^{-l}B,$
recalling the definition of $\cK_{j,s}$ we have
\[
\begin{split}
\sum_{k\in \cK_{j,s}}\prod_{l\in\bN}(A^{-l}B)^{k_l}\lesssim \sum_{k\in \cK_{j,s}}A^{-\sum_{l=1}^\infty lk_l}B^{\sum_{l=1}^\infty k_l}\lesssim A^{-j}B^s.
\end{split}
\]
}
\[
\begin{split}
\|\hat \gamma\|_{\cC^{j+1}}\leq C_\sharp\|w\circ \tilde \theta\|_{\cC^j}&\le   C_\sharp\sum_{s=0}^{j} \|w\|_{\cC^s}
\sum_{k\in \cK_{j,s}}\prod_{l\in\bN}^j \|\tilde\theta\|^{k_l}_{\cC^l}\\
&\le C_\sharp \sum_{s=0}^{j}\sum_{k\in \cK_{j,s}}\prod_{l\in\bN}\left(A^{-l}B\right)^{k_l}\le A^{-j}\sum_{s=0}^{j}B^s.
\end{split}
\]
Thus, since $\|\hat\gamma'\|=a$, we can reparametrize the curve by arc-length. Calling $\tilde\gamma$ the reparametrized curve we obtain
\[
\|\tilde\gamma^{(j)}(t)\|\leq \begin{cases} 0 &\textrm{ if } |t|\leq \delta\\
C_\sharp A^{-j+1} B^{j-1} &\textrm{ if } \delta\leq  |t|\leq \delta+A\\
C_\sharp B^{j-1} &\textrm{ if } |t|\geq \delta+A,
\end{cases}
\]
which yields \eqref{derivative of extended curve tildegamma} since
\[
|\pi_2(\tilde \gamma'(t))|\geq \begin{cases} |v|   &\textrm{ if }  |t|\leq \delta+A\\
\Const(|v|+ B(t-\delta)) &\textrm{ if } |t|\geq \delta+A.
\end{cases}
\]
To conclude we estimate $L_\star$ when the hypotheses of Lemma \ref{lem:vec_regularity} are satisfied. In a finite number of steps $n_2$ (depending only on the derivatives of $F$) we can have $\bar z_{n_2}\leq 1/2$, we can thus apply Lemma \ref{lem:vec_regularity} $\ve_0=\ve=\ve'=1$, $A=1/2$, $B\leq \Const$ and $u=u'=\bar z_{n_2}(p)$, we have
\[
\begin{split}
&|\bar z_{m-n_0}(p,\chi_c)-\bar z_{m-n_0}(p',\chi_c)|\le 
L_{m-n_0}\|p-p'\|\\
&L_j=\Const \mu^{3j}(\|\omega\|_{\cC^2}+\mu^{2j}\lambda_{j}^+(p)^{-1}C_{\mu,j}/2).
\end{split}
\]
Next, note that, by usual distortion arguments, equation \eqref{eq:space} implies $\lambda_{m-n_0}^+\geq \Const \mu^{-n_0}\chi_c(\bar\kappa)^{-1}$ and $m-n_0\leq \Const \ln\bar\kappa^{-1}$, thus
\[
L_{m-n_0}\leq \Const\mu^{5m}C_{\mu,m}(\|\omega\|_{\cC^2}+\bar \kappa)=L_\star (v, n_0).
\]
\end{proof}


\end{document}